\documentclass[12pt]{amsart}
\usepackage{amssymb,epsfig,amsmath,latexsym,amsthm}
\usepackage{graphicx}
\usepackage{amsgen}
\usepackage{curvesls}
\usepackage{epic}
\usepackage[all]{xy}
\usepackage{psfrag}
\usepackage[T1]{fontenc}
\usepackage[sc]{mathpazo}
\usepackage{color, xcolor, colortbl}
\usepackage{float}
\usepackage{arydshln, multirow}
\theoremstyle{plain}
\newtheorem{theorem}{Theorem}[section]
\newtheorem{lemma}[theorem]{Lemma}

\newtheorem{proposition}[theorem]{Proposition}
\newtheorem{corollary}[theorem]{Corollary}
\newtheorem{question}[theorem]{Question}
\newtheorem{questions}[theorem]{Questions}
\newtheorem{conjecture}[theorem]{Conjecture}
\newtheorem{problem}[theorem]{Problem}
\theoremstyle{definition}
\newtheorem{definition}[theorem]{Definition}
\newtheorem{definition/construction}[theorem]{Definition/Construction}
\newtheorem{construction}[theorem]{Construction}
\newtheorem{remark}[theorem]{Remark}

\newtheorem{remarks}[theorem]{Remarks}

\newtheorem{notation}[theorem]{Notation}
\newtheorem{convention}[theorem]{Convention}
\newtheorem{example}[theorem]{Example}
\DeclareMathOperator{\Diff}{Diff}

\DeclareMathOperator{\Emb}{Emb}
\DeclareMathOperator{\Map}{Map}
\DeclareMathOperator{\Maps}{Maps}

\DeclareMathOperator{\id}{id}

\DeclareMathOperator{\length}{length}

\DeclareMathOperator{\cl}{cl}

\DeclareMathOperator{\inte}{int}

\DeclareMathOperator{\fix}{fix}

\DeclareMathOperator{\pr}{pr}
\newcommand{\ivm}{{\mathrm{ivmap}}}
\newcommand{\Dsupp}{\textrm{D}^\textrm{supp}}
\newcommand{\Rsupp}{\textrm{R}^\textrm{supp}}
\newcommand{\Jzero}{1_{J_0}}

\newcommand{\E}{\Emb(I,S^1\times B^3;I_0)}

\newcommand{\finv}{f^{-1}}

\newcommand{\BN}{\mathbb N}

\newcommand{\BR}{\mathbb R}
\newcommand{\BZ}{\mathbb Z}
\newcommand{\Zed}{\mathbb Z} 
\newcommand{\Real}{{\mathbb R }}
\newcommand{\Rat}{{\mathbb Q}}

\newcommand{\mB}{\mathcal{B}}

\newcommand{\mN}{\mathcal{N}}

\newcommand{\soneb}{S^1\times B^3}

\usepackage{epsfig,color}

\begin{document}

\title{Knotted 3-balls in $S^4$}
\author{Ryan Budney \\ David Gabai}

\address{
Mathematics and Statistics, University of Victoria PO BOX 3060 STN CSC, Victoria BC Canada V8W 3R4
\\ Fine Hall, Washington Road Princeton NJ 08544-1000 USA}
\email{rybu@uvic.ca \\ gabai@math.princeton.edu}

\thanks{Version 2.54, April 26, 2021.  Partially supported by NSF grants DMS-1607374 and DMS-2003892
\newline\noindent\noindent\emph{Primary class:} 57M99
\newline\noindent\emph{secondary class:} 57R52, 57R50, 57N50
\newline\noindent\emph{keywords:} 4-manifolds, 2-knots, isotopy}

\begin{abstract} 
The unknot $U$ in $S^4$ has non-unique smooth spanning $3$-balls up to isotopy fixing $U$. 
 Equivalently there are properly embedded non-separating 3-balls in $S^1 \times B^3$ not 
properly isotopic to $\{1\} \times B^3$.  More generally there exist  non-separating 3-spheres in $S^1\times S^3$ not isotopic to $\{1\} \times S^3$   and non trivial elements of $\Diff_0(S^1\times S^3)$.  Along the way we introduce barbell diffeomorphisms, implantations and twistings to construct and modify diffeomorphisms homotopic to the identity.  We also introduce a 2-parameter calculus of embeddings of the interval into 4-manifolds and introduce a framed cobordism method as well as a direct method for showing that certain 2-parameter families are homotopically non trivial and  diffeomorphisms are isotopically nontrivial.  Extensions to higher dimensional manifolds are obtained.  
\end{abstract}


\maketitle

\section{Introduction}\label{intro}

This paper introduces the study of knotted 3-balls in 4-manifolds, in particular the 4-sphere and $S^1\times B^3$.  Let $S^4$ be the unit sphere in $\BR^5$.
Define a \emph{standard} $3$-ball in $S^4$ to be a great 3-ball, i.e. a geodesic 3-ball with boundary a great 2-sphere.  A \emph{knotted} ball in $S^4$ means a 
smoothly-embedded $3$-ball $\Delta_1 \subset S^4$ whose boundary is a great $2$-sphere which
is not isotopic keeping the boundary fixed, to a standard 3-ball $\Delta_0$.   The requirement that the boundary be constrained throughout the isotopy is necessary since any two  
embedded $k$-balls in the interior of a connected $n$-manifold are ambiently isotopic \cite{Ce1} 
p. 231, \cite{Pa}.  The existence of knotted 3-balls in $S^4$ contrasts with the uniqueness of spanning discs for the unknot in $S^3$ and uniqueness for spanning discs for circles in $S^4$, \cite{Ga1}. 
This paper works in the smooth category and unless otherwise said, all mappings are smooth.


We say $N$ is a \emph{reducing} $3$-ball in $S^1 \times B^3$ if $N$ is a properly-embedded submanifold, diffeomorphic 
to $B^3$ such that the complement $(S^1 \times B^3) \setminus N$ is connected.  By properly-embedded we mean
that $N \cap \partial (S^1 \times B^3) = \partial N$. A reducing $3$-ball $N$ is \emph{knotted} if it is
not properly isotopic to the linear reducing $3$-ball, $\{1\} \times B^3$.   All reducing $3$-balls are
properly homotopic to $\{1\}\times B^3$.  The study of reducing 3-balls up to isotopy is equivalent to the study of such balls that coincide with $\{1 \}\times B^3$ near the boundary since Allen Hatcher has proven that the space of non-separating embeddings of $S^2$ in $S^1 \times S^2$ has the homotopy-type of  $S^1 \times O(3)$ \cite{Ha2}.

Isotopy classes of reducing $3$-balls in $S^1 \times B^3$ admit an abelian group structure coming from an operation similar to boundary
connect-sum, that we call {\it concatenation} defined as follows.  Starting with two 
reducing $3$-balls in $S^1 \times B^3$ whose boundary is $\{1\} \times S^2$, one glues the two 
copies of $S^1 \times B^3$ together along $S^1 \times H$ where $H$ is a hemisphere in $\partial B^3$. This produces a new $4$-manifold canonically 
diffeomorphic to $S^1 \times B^3$ together with a new reducing $3$-ball.  
We will see in Sections \ref{exseq} and \ref{knotball} that concatenation has inverses, i.e. it is a group, with the unit being the linear reducing sphere. 
 A less abstract way to describe 
concatenation would be to take $f_1, f_2: B^3 \to S^1 \times B^3$ and assume $B_1, B_2 \subset B^3$ are disjoint 
$3$-balls.  One can assume $f_i(p) \in S^1 \times B_i$ provided $p \in B_i$ and
$f_i(p) = (1,p)$ otherwise. Define the sum of $f_1$ and $f_2$ to be equal to 
$f_i$ on $B^3 \setminus B_j$ where $\{i,j\} = \{1,2\}$.  Isotopy classes of oriented 3-balls in $S^4$ where the embeddings are required to be linear on the
boundary also have a group structure, defined in essentially the same way, and these groups are 
isomorphic.  The key point is that the closed complement (the exterior) of the unknotted $S^2$ in $S^4$
is diffeomorphic to $S^1 \times B^3$. 

We use the convention that if $M$ is a manifold with
boundary, then $\Diff(M \fix \partial)$ denotes the diffeomorphisms of $M$ that are the identity on the boundary.  
Similarly, we use the notation $\Diff_0(M)$ to denote the subgroup of diffeomorphisms {\it homotopic} to the 
identity.  We shall see in \S3 and \S9 that the diffeomorphism groups $\Diff_0(S^4)$, $\Diff(S^1\times B^3 \text{ fix } \partial)$ and $\Diff_0(S^1\times S^3))$ 
are  abelian and act transitively respectively  on 3-balls with common boundary, reducing balls with common boundary and reducing 3-spheres. 
(Actually, $\Diff(S^1 \times B^n \text{ fix } \partial)$ is a $(n+1)$-fold loop space compatible with the group
multiplication \cite{cubes}.) This leads to the theorem $S^1 \times B^3$ and equivalently the closed complement of the unknot in $S^4$, have up to isotopy, infinitely many distinct fiberings over $S^1$ as does $S^1\times S^3$.

A diffeomorphism 
$\phi:S^1\times B^3\to S^1\times B^3$ properly homotopic to the identity, gives rise to the 3-ball 
$\Delta_1=\phi(\{1\} \times B^3)$ which is unknotted if and only if $\phi$ is properly isotopic to a map 
supported in a 4-ball.   The group of isotopy classes of oriented 3-balls that are linear on their boundary is 
isomorphic to $\pi_0(\Diff(S^1\times B^3 \text{ fix } \partial)/\Diff(B^4 \text{ fix } \partial))$.  Similarly, $ \Diff_0(S^1\times S^3)/
\Diff(B^4\fix \partial)$ is isomorphic to the group of \emph{reducing 3-spheres} in $S^1\times S^3$.   
See Theorem \ref{non-sep-s1sn} and Theorem \ref{diff-fibr}.

The main result of this paper is a construction of an infinite family of linearly independent elements of 
$\pi_0(\Diff(S^1\times B^3 \text{ fix } \partial)/\Diff(B^4 \text{ fix } \partial))$ with explicit constructions 
of the corresponding 
knotted 3-balls in $S^1\times B^3$ and hence $S^4$.   Furthermore, these diffeormorphisms extend to a linearly independent set in $\Diff_0(S^1\times S^3)/\Diff(B^4\fix\partial)$.  The techniques of this paper also construct subgroups
of $\pi_{n-3} \Diff(S^1 \times B^n \text{ fix } \partial)$ whenever $n \geq 3$.  

Denote the component of the unknot in $\Emb(S^2, S^4)$ by $\Emb_u(S^2, S^4)$.  A consequence of the above results is that
$\Emb_u(S^2, S^4)$ does {\it not} have the homotopy type of the subspace of linear embeddings.  The 
latter has the homotopy type of the Stiefel manifold $V_{5,3} = SO_5 / SO_2$ while the former has a non-finitely-generated
fundamental group. See Theorem \ref{appsecthm}.

We give a framework for approaching the smooth $4$-dimensional Sch\"onflies problem, describing the set of
counter-examples as the fixed points of an endomorphism 
$\pi_0 \Emb(B^3, S^1 \times B^3) \to \pi_0 \Emb(B^3, S^1 \times B^3).$
The endomorphism is given by lifting such an embedding to a non-trivial finite-sheeted 
covering space of $S^1 \times B^3$. The non-trivial 
elements of $\pi_0 \Emb(B^3, S^1 \times B^3)$ we construct in this paper all belong to the kernel of iterates of this
endomorphism.   

The paper is organized as follows.  In Section \S 2 we compute 
the homotopy group $\pi_{n-2} \Emb(S^1, S^1 \times S^n)$ and show that it contains an infinitely-generated free abelian group 
provided $n\geq 3$, giving explicit generators $\theta_k$, $k\ge 2$.  This result extends the work of Dax \cite{Da} 
who among other things computed $\pi_1(\Emb(S^1, S^1\times S^3))$ in terms of certain cobordism groups and Arone-Szymik \cite{AS}, 
who show $\pi_1$ and $\pi_2$ of $\Emb(S^1, S^1\times S^3)$ contain infinitely-generated free abelian groups. 
For $n=3$ we describe other generators in terms of embeddings of tori $T : S^1 \times S^1 \to S^1\times S^3$. 
In \S 3 we compute the three and five dimensional rational homotopy groups of $C_3[S^1\times B^3]$ and use them to 
define the $W_3$ invariant of $\pi_2 \Emb(I, S^1\times B^3)$ which takes values in a quotient of 
$\pi_5 C_3[S^1\times B^3]$.  We define $G(p,q)$ a 2-parameter family of $\Emb(I, S^1\times B^3)$  and show that 
$W_3[G(p,q)]$ is equal to the class of the standard Whitehead product $t_1^p t_2^q[w_{23}, w_{13}]$.  We also describe  
fibration sequences relating  various embedding spaces and diffeomorphism groups and also show that 
$\Diff(S^1\times B^3\fix\partial)$ and $\Diff_0(S^1\times S^3)$ respectively act transitively on reducing balls and spheres. 
 In \S \ref{2 parameter section} we introduce geometric methods for working with 2-parameter families of 
 $\Emb(I,M^4)$.  In particular, we introduce a \emph{bracket} operation that produces a 2-parameter family from two 1-parameter 
 families that are null homotopic and have disjoint domain and range supports.  In \S \ref{barbell section} we introduce the 
 barbell map $\beta$ of $S^2\times D^2\natural S^2\times D^2$ which we call the \emph{barbell neighborhood} $\mN\mB$.  
 $\beta$ fixes $\partial \mN\mB$ pointwise hence induces homotopically trivial diffeomorphisms of 4-manifolds called 
 \emph{implantations} when $\mN\mB$ is embedded in a 4-manifold and $\beta$ is pushed forward.  We describe the 
 implantations of $S^1\times B^3\fix \partial$ induced from the generating elements $\theta_k$.  We also compute  
 $\beta(\Delta_0)$ of the standard separating 3-ball $\Delta_0$.  This is used in \S \ref{family section} to produce two 
 parameter families $\hat \theta_k$ in $\Emb(I,S^1\times B^3)$ arising from the $\theta_k$'s.  We then  show how to modify 
 these families when twisting the implantation.  In \S \ref{factoring section} we compute the class 
 $[\hat\theta_k]\in \pi_2 \Emb(I,S^1\times S^3)$ as the sum of elements of a $(k-1)\times (k-1)$-matrix $A_k$ with entries a 
 sum of $G(p,q)$'s.   This matrix is skew symmetric and hence $[\hat\theta_k]=0$. In \S \ref{twisting section} we 
 show that the effect of twisting the $\theta_k$ implantation is to modify $A_k$ by row and column operations.  By 
 twisting the  $\theta_k$ implantations we produce the $\delta_k$ implantations, $k\ge 4$, whose homotopy classes are 
 shown to be linearly independent by the $W_3$ invariant.   Together with the triviality of the $[\hat\theta_k]$'s we 
 conclude that the $\delta_k$ implantations in $\Diff_0(S^1\times S^3)$ are linearly independent up to isotopy. 
  In \S 9 we discuss the relation between knotted 3-balls and the smooth 4-dimensional Schoenflies conjecture.  More 
  applications are given in \S 10 and questions and conjectures are given in \S11.  In the Appendix we give a direct 
  argument that $W_3[G(p,q)]$ is equivalent to the standard Whitehead product $t_1^p t_2^q[w_{23}, w_{13}]$ up to 
  sign independent of $p$ and $q$.

Independently, Tadayuki Watanabe \cite{Wa2} has constructed an invariant 
$$Z_1 : \pi_1 B\Diff(S^1 \times B^3 \text{ fix }\partial) / \pi_1 B\Diff(B^4 \text{ fix } \partial) \to \mathcal{A}_1(R[t^{\pm 1}])$$
and has shown it to be non-trivial on some diffeomorphisms created via his graph surgery construction.

{\bf Acknowledgements.} The authors would like to thank the Banff International Research
Station. The Unifying 4-Dimensional Knot Theory meeting at BIRS was crucial to this project. 
The authors would also like to thank Greg Arone and Markus Szymik for helpful discussions.
The authors also thank Allen Hatcher, Tadayuki Watanabe and Fran\c{c}ois Laudenbach for 
helpful comments on the initial version of the paper.  
The first author thanks BIRS for hosting the Spaces of Embeddings: Connections and
Applications meeting in the fall of 2019, where some useful developments occurred for this paper.
He also thanks Dev Sinha for helpful discussions, and Alan Mehlenbacher for pointing out mistakes
in early drafts of this paper. 
The second author's work on this project was initiated during  visits to Trinity College, Dublin 
and work was also carried out while visiting the Max Planck Institute for Mathematics, Bonn and the 
Mathematical Institute at Oxford University.  He thanks all these institutions for 
their hospitality.  He thanks Toby Colding for long conversations years ago and Maggie Miller for recent helpful conversations.  
Special thanks to Martin Bridgeman.  He was partially supported by NSF grants DMS-1607374 and DMS-2003892. 

\section{Embeddings of circles in $S^1 \times S^n$}\label{embsec}

In this section we describe a range of low-dimensional homotopy groups of 
$\Emb(S^1, S^1 \times S^n)$.  These results were essentially known to Dax \cite{Da},
who used a Haefliger-style parametrized double-point elimination process to 
describe the low-dimensional homotopy groups of a variety of embedding spaces.  
Given an element of an embedding space $f \in \Emb(S^1, S^1 \times S^n)$ we will denote
the path-component of $\Emb(S^1, S^1 \times S^n)$ containing $f$ by $\Emb_f(S^1, S^1 \times S^n)$.   

We begin with the least technical elements in Theorem \ref{embthm}, describing 
for $n \geq 3$, three epimorphisms:
$$W_0 : \pi_0 \Emb(S^1, S^1 \times S^n) \to \BZ$$
$$W_1 : \pi_1 \Emb_f(S^1, S^1 \times S^n) \to \BZ$$
$$W_2 : \pi_{n-2} \Emb_f(S^1, S^1 \times S^n) \to \Lambda^{W_0(f)}_{n}$$ 
The epimorphisms $W_1$ and $W_2$ are defined for all components of the embedding space. 
The group $\Lambda^{W_0}_{n}$ is defined as a quotient of the Laurent polynomial ring $\BZ[t^{\pm 1}]$, 
and contains a free abelian subgroup of infinite rank. It can also contain
$2$-torsion. 

For this definition we consider the Laurent polynomial ring $\BZ[t^{\pm 1}]$ to be only
a group. We define $\Lambda^{W_0}_{n}$ to be the quotient group, $\BZ[t^{\pm 1}]$ modulo
the subgroup generated by the relations  
$$\langle t^k + (-1)^n t^{W_0-1-k} = 0 \ \forall k, \ \ t^0 = 0, t^{-1} = 0\rangle.$$

The group $\Lambda^{W_0}_{n}$ is the free abelian group on the generators 
$$G_{W_0} = \{ t^k : k \in \BZ, k \geq \frac{W_0-1}{2}, k \notin \{-1,0,W_0,W_0-1\} \}$$
with the sole exception when $n$ is even, $|W_0|>2$ and $W_0$ is odd.  In this case one has the same generating
set $G_{W_0}$, but $t^{\frac{W_0-1}{2}}$ represents $2$-torsion.  The remaining $t^k$ are free generators, i.e. 
$\Lambda^{W_0}_{n} \simeq \BZ_2 \oplus \left(\oplus_{k > \frac{W_0-1}{2}, t^k \in G_{W_0} } \BZ\right)$.

The definitions of the maps $W_0, W_1$ and $W_2$ will be elementary applications of basic transversality
theory. 

\begin{definition}
Let $\pi : S^1 \times S^n \to S^1$ be defined as
$\pi(z,v) = z$.  Given an embedding  $f : S^1 \to S^1 \times S^n$, 
define $W_0(f) = deg(\pi \circ f) \in \BZ$.  This is the degree of the map 
$\pi \circ f : S^1 \to S^1$. 
$$\xymatrix{S^1 \ar[dr]_-{f} \ar[rr]^{\pi \circ f} && S^1 \\ & S^1 \times S^n \ar[ur]_-{\pi} & } $$
\end{definition}

The value $W_0(f)$ only depends on the homotopy class of $f$.  Provided $n \geq 3$, the
homotopy class of $f$ agrees with the isotopy class, by transversality.  Thus 
$$W_0 : \pi_0 \Emb(S^1, S^1 \times S^n) \to \BZ$$
is a bijection.  An embedding $f : S^1 \to S^1 \times S^n$ satisfying
$\pi(f(z)) = z^n \ \forall z \in S^1$ would have $W_0(f) = n$. 

\begin{definition}\label{w1def}
Given $F : S^1 \to \Emb(S^1, S^1\times S^n)$ we define
$$W_1(F) = deg(\hat F) \in \BZ $$
where $\hat F : S^1 \to S^1$ is defined as $\hat F(z) = \pi(F(z)(1))$, i.e. we consider 
$F(z) \in \Emb(S^1, S^1 \times S^n)$ and we evaluate it at $1 \in S^1$. 
\end{definition}

We can consider $W_1$ to be a function
$$W_1 : \pi_1 \Emb_f(S^1, S^1 \times S^n) \to \BZ.$$

As a thought experiment, argue that given $[F] \in \pi_1 \Emb(S^1, S^1\times S^n)$ satisfying 
$W_1(F) = k$, one can assume $\pi(F(z)(1)) = z^k \ \forall z \in S^1$. More generally, 
one can show $\Emb(S^1, S^1 \times S^n) \simeq S^1 \times \Emb^*(S^1, S^1 \times S^n)$
where $\Emb^*(S^1, S^1 \times S^n)$ is the subspace of $\Emb(S^1, S^1 \times S^n)$ where
$\pi \circ F(1) = 1$. From this perspective, $W_1$ is simply the induced map from the projection
onto the first factor, i.e. $\Emb(S^1, S^1 \times S^n) \to S^1$.

An appealing way to think of the invariants $W_0$ and $W_1$ is via the inclusion
$\Emb(S^1, S^1 \times S^n) \to \Map(S^1, S^1 \times S^n)$, i.e. we are including the embedding
space in the space of all continuous functions from $S^1$ to $S^1 \times S^n$.  Notice
that $W_0$ and $W_1$ extend to invariants of $\pi_0 \Map(S^1, S^1\times S^n)$ and $\pi_1 \Map(S^1, S^1 \times S^n)$
respectively. Moreover, as invariants of the homotopy-groups of $\Map(S^1, S^1\times S^n)$ they are 
isomorphisms, since $\Map(S^1, S^1 \times S^n)$ splits as a direct product of two free loop spaces,
$\Map(S^1, S^1 \times S^n) \equiv L(S^1)\times L(S^n)$.  A simple general position argument tells us that
the inclusion $\Emb(S^1, S^1 \times S^n) \to \Map(S^1, S^1 \times S^n)$ induces an epi-morphism 
on $\pi_{n-2}$ and an isomorphism on $\pi_k$ for $k<n-2$.  The rough idea is that if one has a map 
from a $k$-dimensional manifold $\phi: M \to \Map(S^1, S^1 \times S^n)$ one constructs the {\it track}
of the map $\overline{\phi} : M \times S^1 \to M \times S^1 \times S^n$ given by
$\overline{\phi}(p,z) = (p, \phi_p(z))$.  By transversality, this map can be uniformly approximated
by a smooth embedding if $2(k+1) < k+1+n$, i.e. $k \leq n-2$ (see for example Theorem 2.13 \cite{Hirsch}). 
Such an approximation is no longer a
track-type function on the nose, but given that the approximation is uniform in (at least) the $C^1$-topology, 
and the fact that diffeomorphisms of $M$ form an open subset of the space of maps $\Map(M, M)$, 
one can apply a diffeomorphism of $M$ (close to $Id_M$) to generate an approximation that is the track 
of an embedding.   The first
two non-trivial homotopy-groups of $\Map(S^1, S^1\times S^n)$ are $\pi_0$ and $\pi_1$, both infinite
cyclic.  The next non-trivial homotopy-group is $\pi_{n-1} \Map(S^1, S^1\times S^n) \simeq \BZ$. 

The next invariant $W_2$ has the form
$$W_2 : \pi_{n-2} \Emb_f(S^1, S^1 \times S^n) \to \Lambda^{W_0(f)}_{n}.$$
By the previous paragraph, it measures the lowest-dimensional deviation between the homotopy-types
of $\Emb(S^1, S^1\times S^n)$ and $\Map(S^1, S^1\times S^n)$. 

\begin{definition}\label{c2def}
Let $C_2(M)$ denote the configuration space of pairs of distinct points in $M$, 
$$C_2(M) = \{ (p_1, p_2) \in M^2 : p_1 \neq p_2 \}.$$

Denote the {\it cocircular pair} subspace of $C_2(S^1 \times S^n)$ by 
$\mathcal{CC} = \{ (z_1,p_1), (z_2, p_2) \in C_2 (S^1 \times S^n) : p_2 = p_1 \}$.
The cocircular pair subspace is $(n+2)$-dimensional, having co-dimension $n$ in $C_2 (S^1 \times S^n)$. 
Given $F : S^{n-2} \to \Emb(S^1, S^1 \times S^n)$, assume the induced map
$$\hat F : S^{n-2} \times C_2 S^1 \to C_2 (S^1 \times S^n)$$
is transverse to $\mathcal{CC}$, where $\hat F(v,z_1,z_2) = (F(v)(z_1), F(v)(z_2))$.  In 
such a situation we will associate $W_2(F) \in \Lambda^{W_0(f)}_{n}$.

Our polynomial will be akin to the transverse intersection number of $\hat F$ with $\mathcal{CC}$, but we 
include an enhancement into the definition. The set $\hat F^{-1}(\mathcal{CC})$ is 
$\Sigma_2$ invariant, and $\Sigma_2$ acts freely on $C_2 (S^1)$. The invariant $W_2(F)$ will be
a sum of monomials associated to the points of $\hat F^{-1}(\mathcal{CC}) / \Sigma_2$.  Given a point 
$q = (v, z_1, z_2) \in \hat F^{-1} (\mathcal{CC})$ we associate an 
element $\mathcal{L}_{q}(F) \in \BZ[t^{\pm 1}]$ and define 
$$W_2(F) = \sum_{[q] \in \hat F^{-1}(\mathcal{CC})/\Sigma_2 } \mathcal{L}_{q}(F)\in \Lambda^{W_0}_{n}.$$

We define $\mathcal{L}_{q}(F) = \epsilon t^k$, where $\epsilon \in \{\pm 1\}$ is the local oriented 
intersection number of $\hat F$ with $\mathcal{CC}$ at $q$.  Observe the map
$$S^n \times C_2(S^1) \ni (w, z_1,z_2) \longmapsto ((z_1,w), (z_2,w)) \in C_2(S^1 \times S^n)$$
is a diffeomorphism between $S^n \times C_2(S^1)$ and $\mathcal{CC}$.  This is how we give $\mathcal{CC}$
its orientation. This map is also $\Sigma_2$-equivariant.   The monomial degree $k$ is computed via a pair 
of conventions.  If $(z_1, z_2) \in C_2 (S^1)$, let $[z_1,z_2]$ 
denote the counter-clockwise oriented arc in $S^1$ that starts at $z_1$ and ends 
at $z_2$.   Similarly, given a point of $\mathcal{CC}$, $((z_1, p_1), (z_2, p_1))$, the {\it cocircular arc} 
with this boundary is denoted $[z_1,z_2]\times\{p_1\}$.  When thinking of $S^1 \times S^n$ we refer
to this as the {\it vertical} orientation. The monomial degree $k$ is obtained by concatenating 
$F(v)([z_1,z_2])$ with the opposite-oriented cocircular arc in $S^1 \times S^n$
associated to $\hat F(v,z_1,z_2)$, and taking the degree of the projection to the $S^1$ factor of
$S^1 \times S^n$.  We depict an example in Figure \ref{fig:Fig00}, with the concatenation appearing
in green.  The unused portion of the vertical circle is in red. In this example $W_0(f)=6$, and 
$\mathcal{L}_q(F) = \epsilon t^2$. 
\end{definition}

\begin{figure}[ht]
$$\includegraphics[width=6cm]{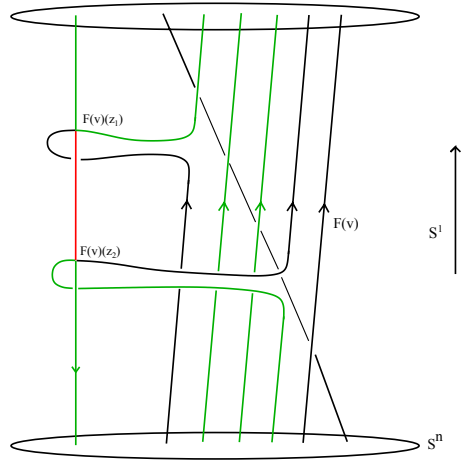}$$
\caption{\label{fig:Fig00}Example computation of monomial exponent, $\mathcal{L}_q(F) = \pm t^2$.}
\end{figure}

Notice that $W_2$ naturally factors as $W_2^* : \pi_{n-2} \Emb^*_f(S^1, S^1 \times S^n) \to \Lambda^{W_0(f)}_{n}$
after projecting-out the $S^1$ factor, using the product decomposition 
$\Emb(S^1, S^1 \times S^n) \simeq S^1 \times \Emb^*(S^1, S^1 \times S^n)$.

Given $(v, z_1, z_2) \in \hat F^{-1}(\mathcal{CC})$ then we also have 
$(v, z_2, z_1) \in \hat F^{-1}(\mathcal{CC})$ and one can check 
$$\mathcal{L}_{(v,z_2,z_1)}(F) = (-1)^{n+1} t^{W_0-1} \overline{\mathcal{L}_{(v,z_1,z_2)}(F)}.$$
We use the notation $\overline{\cdot}$ to denote 
the $\BZ$-linear mapping $\overline{\cdot} : \BZ[t^{\pm 1}] \to \BZ[t^{\pm 1}]$ satisfying $\overline{t^k} = t^{-k}$.
Thus $W_2(F)$ is well-defined for $F$.  The relation $t^0 = 0$ in $\Lambda^{W_0}_{n}$ was chosen to ensure $W_2(F)$ 
is a homotopy-invariant of $F$. We use a compactification of configuration spaces to check homotopy-invariance.

Our manifold compactification of $C_2 (S^1)$ is diffeomorphic to an annulus $S^1 \times [-1,1]$. The boundary circles 
correspond to `infinitesimal' configurations of pairs of points in $S^1$; one component where the direction vector 
from $z_1$ to $z_2$ agrees with the orientation of $S^1$, and the other being the reverse.

The Fulton-MacPherson compactified configuration space has the rather simple model of 
$M^2$ blown up along its diagonal $C_2[M] = Bl_{\Delta M} M^2$.  Typically this is made formally
precise by defining $C_2[M]$ to be the closure of the graph of a function \cite{Sin2}, such
as $\phi : C_2(M) \to S^k$ where $\phi(p,q) = \frac{p-q}{|p-q|}$, assuming $M \subset \BR^{k+1}$.
This compactification is functorial under embeddings of manifolds. The inclusion $C_2(M) \to C_2[M]$ is a
homotopy-equivalence, i.e. $C_2[M]$ is diffeomorphic to $M^2$ remove an open tubular neighbourhood of
the diagonal $\Delta M = \{(p,p) : p \in M\}$.  There is a canonically-defined onto smooth map $C_2[M] \to M^2$, where
the pre-image of $\Delta M$ is $\partial C_2[M]$, which is canonically isomorphic to the unit tangent bundle of $M$.
The interior of $C_2[M]$ is mapped diffeomorphicly to $M^2 \setminus \Delta M$.  

We now prove the homotopy-invariance of $W_2(F)$. Consider what happens in a homotopy of $F$.  
The boundary of $I \times S^{n-2} \times C_2 [S^1]$ consists of the {\it temporal part}
$(\partial I) \times S^{n-2} \times C_2 [S^1]$ and the {\it annular part}
$I \times S^{n-2} \times \partial C_2 [S^1]$. The only monomial degrees that run off the 
annular part of the boundary are $t^{-1}, t^{0}, t^{W_0-1}, t^{W_0}$.  For example, 
$t^0$ runs off the annular part if in our transverse family we have a tangent vector to 
our knot pointing in the vertical direction, oriented counter-clockwise.  Similarly, 
$t^{-1}$ can run off the boundary if we produce a tangent vector in the vertical direction, 
oriented clockwise.  The monomials $t^{W_0-1}$ and $t^{W_0}$ are symmetric, after re-labelling
the points of the domain $(z_1, z_2) \leftrightarrow (z_2,z_1)$.

Thus if we consider $W_2(F)$ to be an element of the quotient group $\Lambda^{W_0}_{n}$, it is 
a homotopy-invariant.

In Theorem \ref{embthm} and Section \ref{exseq} we need the notion of a half-ball. 

\begin{definition}\label{half-ball-def}
Let $H^n = \{ (x_1, \cdots, x_n) \in \BR^n : x_1 \leq 0 \}$ and define the half-ball
$HB^n = B^n \cap H^n$.  $HB^n$ is a manifold with corners.  As such, it is a stratified space
with two co-dimension one strata, the {\it round boundary} $HB^n \cap \partial B^n$ and
the {\it flat boundary} $HB^n \cap \partial H^n$.  These two boundaries meet at the
{\it corner (co-dimension two) stratum} $\{0\} \times S^{n-2}$. 
\end{definition}

\begin{figure}[ht]
$$\includegraphics[width=6cm]{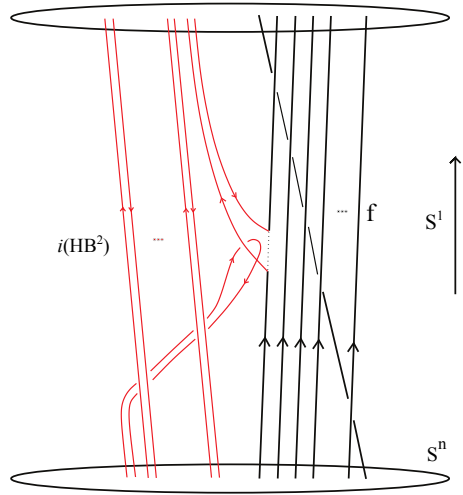}$$
\caption{\label{fig:Fig0}Constructing $\theta_{f,k}$.}
\end{figure}

\begin{theorem}\label{embthm}
Let $f \in \Emb(S^1, S^1 \times S^n)$. 
Provided $n \geq 3$ both $W_1$ and $W_2$ are epimorphisms. When $n=3$
the map 
$$W_1 \times W_2 : \pi_1 \Emb_f(S^1, S^1 \times S^3) \to \BZ \oplus \Lambda^{W_0}_{3}$$
is an epimorphism, i.e. $W_2^* : \pi_1 \Emb_f^*(S^1, S^1 \times S^3) \to \Lambda^{W_0}_{3}$ is
an epi-morphism.

\begin{proof}
That $W_1$ is an epimorphism follows from the splitting 
$\Emb_f(S^1, S^1 \times S^n) \simeq S^1 \times \Emb_f^*(S^1, S^1 \times S^n)$, as it
is the degree of the projection to the $S^1$ factor. 

To argue that $W_2$ is an epimorphism, we start with the fixed degree $W_0$ near-linear embedding 
$f : S^1 \to S^1 \times S^n$, depicted in black in Figure \ref{fig:Fig0}.

Imagine an immersed half-ball $i : HB^2 \to S^1 \times S^n$ that is an embedding with the exception
of a regular double point on the round boundary.  We demand $i^{-1}(f(S^1))$ coincides with the flat boundary of $HB^2$. 
As in Figure \ref{fig:Fig0} we demand that the map from the flat part of the boundary of $i(HB^2)$ to the 
vertical $S^1$ factor has no critical points. We further demand that the arc in $HB^2$ connecting the double
points projects to a degree $k$ map in the vertical $S^1$ factor, and that (as in Figure \ref{fig:Fig0})
the double-point occurs near the bottom of the flat boundary.  

We modify the embedding $f$, creating a new immersed curve $S^1 \to S^1 \times S^n$ by replacing the arc 
$f(S^1) \cap i(\partial^f HB^2)$ with $i(\partial^r HB^2)$, i.e. we cut out the flat boundary of $i(HB^2)$ from 
$f(S^1)$ and replace it with the round boundary.  This immersed curve is depicted in Figure \ref{fig:Fig0}
as the union of the solid black vertical broken curve (depicting $f$ with the flat boundary removed), 
with the solid red curve (depicting the round boundary of $i(HB^2)$). 
Call this immersed curve $\tilde\theta_{f,k} : S^1 \to S^1 \times S^n$. Given that $n \geq 3$, we can 
assume the projection of $\tilde\theta_{f,k}$ to $S^n$ is also an immersion, and embedding at all
but the single double point. 

The sum of the tangent spaces at the double-point of $\tilde\theta_{f,k}$ is $2$-dimensional, so the orthogonal 
complement is $(n-1)$-dimensional, having an $S^{n-2}$-parameter family of unit normal vectors.  Using a bump function, 
given a unit normal vector one can perturb one strand of $i(\partial^r HB^2)$ at the double-point, creating an 
embedded circle in $S^1 \times S^n$.  This gives us our family of resolutions,
$$\theta_{f,k} : S^{n-2} \to \Emb_f(S^1, S^1 \times S^n).$$

The fact that the projection of $\tilde\theta_{f,k}$ to $S^n$ was an embedding with the sole exception of 
the single double-point allows us to conveniently identify the cocircular points in our family 
$\theta_{f,k} : S^{n-2} \to \Emb_f(S^1, S^1\times S^n)$, 
giving us $W_2(\theta_{f,k}) = t^k - t^{k-1}$. 
\end{proof}
\end{theorem}

We have an involution of $\Emb(S^1, S^1 \times S^n)$ that negates the $W_0$ invariant.  One description 
is the process that sends the embedding $f \in \Emb(S^1, S^1 \times S^n)$ to $z \longmapsto f(z^{-1})$. 
Call this embedding $\overline{f}$.  Then we have $\overline{\theta_{f,k}} = \theta_{\overline{f}, -k-1}$, 
i.e. the $\theta$ elements are symmetric about $-1/2$. 

Another family $i : HB^2 \to S^1 \times S^n$ to consider is one where $i$ is an embedding.  
We demand that $i(HB^2)$ intersects $f$ along the flat boundary, and also at a single point
along the round boundary -- a regular double point.   Consider
the case where the projection of the embedding $i$ to the $S^1$ factor is not onto, i.e. it is constrained
to an interval in the $S^1$ factor.  Then $i$ connects one strand of $f$ to adjacent strands.  Let's say to the
$k$-th strand (using the cyclic ordering) with $k \in \{0, 1, \cdots, W_0-1\}$, assuming $W_0>0$.  Call the resolved 
family of knots
$\gamma_{f, k} : S^{n-2} \to \Emb(S^1, S^1 \times S^n)$.  Given that, we have $W_2(\gamma_{f,k}) = t^k-t^{k-1}$. 
Recall that $\Lambda^{W_0}_{n} = \BZ[t^{\pm 1}] / \langle t^{-1}, t^0, t^k + (-1)^k t^{W_0-1-k} \ \forall k \rangle $.
Thus $\{ \gamma_{f,k} : k \in \{0,1,\cdots, W_0-1\}\}$ spans the same subspace of $\Lambda^{W_0}_{n}$ as the monomials
$\{t, t^2, \cdots, t^{W_0-2}\}$, i.e. all the intermediate monomials that were not killed by the definining
relations of $\Lambda^{W_0}_{n}$.

We give an alternative way to visualize elements of $\pi_1\Emb_f(S^1, S^1 \times S^3)$ with $W_2 \neq 0$ by embedded tori, 
where $f$ denotes the standard generator of $\pi_1(S^1)$, i.e. the $W_0=1$ component. 
 Each generator will be  represented by an embedded torus 
$T\subset S^1\times S^3$ which contains the curve $\gamma_0=S^1\times \{y_0\}$.  Such a torus $T$ gives
rise to an element $z$ of $\pi_1 \Emb(S^1, S^1\times S^3)$ by fibering $T$ by parametrized smooth circles
$\{\gamma_t|t\in [0,1]\}$ with $\gamma_0=\gamma_1$.  Once $\gamma_0$ is chosen, what really matters is 
which way to go around the torus.  To do this and control $W_1$,  we choose an oriented simple
closed curve $\mu_w\subset T$, homotopically trivial in $S^1\times S^3$, that intersects $\gamma_0$ 
transversely once at some point $w=(x_0, y_0)\in S^1\times S^3$.  The homotopy condition implies that 
that $W_1(z)=0$ and the orientation informs us that 
$\gamma_0$ is required to spin about $T$ so 
that $w$ follows $ \mu_w$ in the oriented direction.  Denote by $(T,\mu_w)$ the represented element 
of $\pi_1 \Emb_f(S^1, S^1 \times S^3)$.

\begin{figure}[ht]
$$\includegraphics[width=8cm]{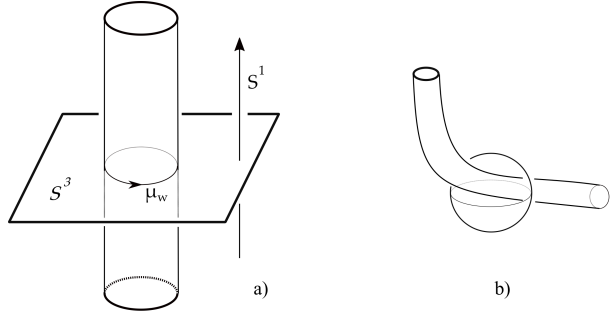}$$
 \caption[(a) X; (b) Y]{\label{fig:Fig1}\begin{tabular}[t]{ @{} r @{\ } l @{}}
 (a) & Vertical torus in $S^3 \times S^1$. Vertical fibers represent trivial \\
     & element in $\pi_1 \Emb(S^1, S^3 \times S^1)$,  in the component with $W_0=1$. \\
 (b) & Sphere linking tube in $4$-space.
 \end{tabular}}\end{figure}

The standard vertical torus $T^*$, shown in Figure \ref{fig:Fig1}(a) represents the trivial element of 
$\pi_1 \Emb(S^1, S^1 \times S^3)$.  
Figures \ref{fig:Fig2} (a) and (b) describe embedded tori corresponding to $t^2$ and $t^3$ respectively. 
 In our diagrams, $\mu_w\subset \{x_0\}\times S^3$, with $S^3$ being depicted 
horizontally as in Figure \ref{fig:Fig1} (b). In a similar manner we obtain a torus corresponding to $t^n, |n|\ge 1$.  Each of our tori is
constructed by tubing $T^*$ with an unknotted, unlinked 2-sphere as follows.  
Emanating from the boundary of a small disc on $T^*$ the tube first links the sphere, then goes $n\in \BN$ times around 
the $S^1$ factor before finally connecting to the 2-sphere.  Figures \ref{fig:Fig2} (a) and (b)
show the projection of $T$ to $S^1\times (S^2\times 0)$ where $S^3$ is identified with 
$S^2\times [-\infty,\infty]$, where each  component of  $S^2\times \{\pm \infty\}$ is identified to a point. 
 By construction $T\subset S^1\times (S^2\times \{0\})$, except for where the tube links the 2-sphere.  
See Figure  \ref{fig:Fig1}(b) for a detail.  The crossing convention for the tube and sphere informs us  
that the part of the tube that projects to the right side of the 2-sphere lives a bit in the past 
(i.e. in $S^1\times (S^2\times [-1,0)$) and the part of the tube on the left lives in the future.  By 
construction, the 2-sphere bounds a 3-ball $B\subset S^1\times(S^2\times \{0\})$ that intersects the tube 
is a single simple closed curve.  By either reversing the way the tube links the 2-sphere, or reversing 
the orientation on $\mu_w$ we obtain the inverse of the generator.  See Figure \ref{fig:Fig3}(b).  

Proposition \ref{alph-thet} relates the generators $\alpha_{f,k}$ and $\theta_{f,k}$, described above. 
To make this proposition precise, and not simply up to a choice of sign, we need to provide an 
orientation to the parametrizing sphere $S^{n-2}$ of our generators $\theta_{f,k} : S^{n-2} \to \Emb_f(S^1, S^1\times S^n)$. 
The parametrizing sphere came up as the unit normal sphere to the sum of the tangent spaces at the 
double point.  Given this proposition is only for $n=3$, our $S^{n-2}$ is a circle.  We choose the orientation
consistent with the rotation from above the double point (i.e. the positive vertical direction) to 
the {\it into the page} direction.   This gives us the formula below.

\begin{figure}[ht]
$$\includegraphics[width=8cm]{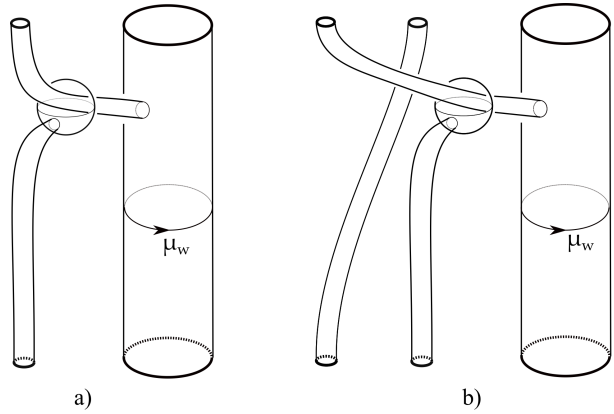}$$
 \caption[(a) X; (b) Y]{\label{fig:Fig2}\begin{tabular}[t]{ @{} r @{\ } l @{}}
 (a) & Torus representing $\alpha_{f,1}$ with $W_0(f)=1$, $W_2=t^2-t^0 = t^2 \neq 0$ \\
 (b) & Torus representing $\alpha_{f,2}$ with $W_0(f)=1$, $W_2=t^3-t^1 = t^3 \neq 0$
 \end{tabular}}
\end{figure}

\begin{proposition}\label{alph-thet} $\alpha_{1,i}=\theta_{1,i+1}-\theta_{1,i}.$\end{proposition}
\begin{proof}  The proof is by directly constructing a homotopy between representatives. 
 We demonstrate it for $i=1$ with the general case being similar.  As above, we view 
 $S^1\times S^3$ as a quotient of $S^1\times S^2 \times [-\infty,\infty]$.  In what follows all 
 the figures, except for f) live in $S^1\times S^2\times \{0\}$.  Figure \ref{fig:Fig1.6}(a) depicts 
 the constant loop $\kappa^a_t, t\in I$, with each $\kappa^a_t$ representing the standard $W_0=1$ 
 curve $\kappa$.  Now Figure \ref{fig:Fig1.6}(b) also represents the $W_2=0$ loop $\kappa^b_t$.

\begin{figure}[ht]
$$\includegraphics[width=8cm]{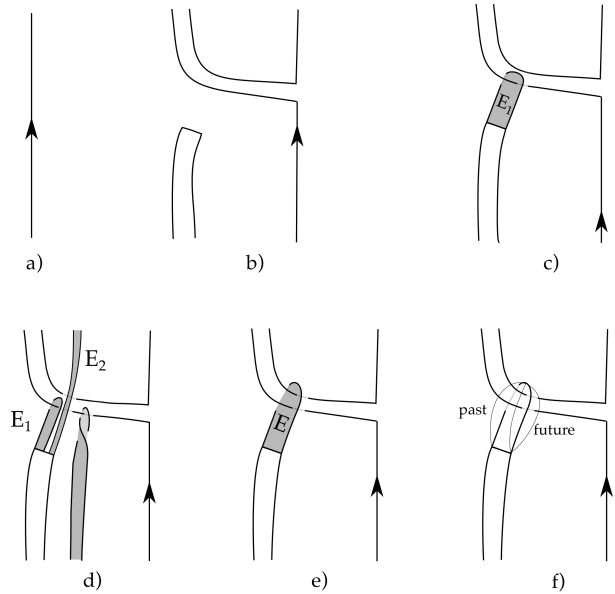}$$
\caption[]{Associated to Proposition \ref{alph-thet} \label{fig:Fig1.6}}
\end{figure}

 Here $\kappa^b_t=\kappa^a_t$ for $t\in [0,1/8]\cup [7/8,1]$.  During $t\in [1/8,1/4]$ it sweeps 
 to the curve shown in b) and stays there for $t\in [1/4,3/4]$ before sweeping back to $\kappa$.  
 Next we modify this loop to $\kappa^c_t $ representing  $-\theta_{1,1}$.  Consider the half disc 
 $E_1$.  Here $\partial^f(E_1)\subset \kappa^b_{1/2}$ and $\partial^r(E_i)$ locally \textit{links}
 $\kappa^b_{1/2}$.  Define $\kappa^c_t=\kappa^b_t, t\in[0,1/4]\cup[3/4,1]$.  During $t\in [1/4,1/2]$, 
 keeping endpoints fixed, $\partial^f(E_1)$ sweeps across $E_1$ to end at $\partial^r(E_1)$.  
 The interior of each arc, for $t\in (1/4, 1/2)$ is pushed slightly into the future, i.e. into 
 $S^1\times S^2\times s$ for $s>0$.  During $t\in [1/2, 3/4]$, keeping endpoints fixed $\partial^r(E_1)$
 is pushed back to $\partial^f(E_1)$.  Here for $t\in (1/2, 3/4)$, the interior of each arc is 
 pushed into the past.   We next modify this loop to $\kappa^d_t$ representing 
 $-\theta_{1,1}+ \theta_{1,2}$ as shown in Figure \ref{fig:Fig1.6}(d).  We have abused notation 
 by calling one of the half discs of d) also $E_1$. Again, keeping endpoints fixed the arcs 
 $\partial^f(E_1)\cup \partial^f(E_2)$ sweep to $\partial^r(E_1)\cup \partial^r(E_2)$ and then 
 back again with interiors of arcs in the first (resp. second) part of the motion in the future 
 (resp. past).  Note, that the twist in the half disc $E_2$, which lies in $S^1\times S^2\times 0$, 
 gives rise to $\theta_{1,2}$ rather than its inverse.  The loop $\kappa^d_t$ is homotopic to 
 $\kappa^e_t$ where here the half disc $E$ is used.  Again, the homotopy is supported in $[1/4,3/4]$ 
 and  $\kappa^e_t$ has the feature that keeping endpoints fixed $\partial^f(E)$ sweeps, pushed 
 slightly into the future, across $E$ to $\partial^r(E)$, and then sweeps back to $\partial^f(E)$ 
 again with interiors of arcs pushed slightly into the past. Thus Figure \ref{fig:Fig1.6}(f) also 
 represents $\kappa^e_t$, with the sphere being the image of the track of $\partial^f(E)$ as it 
 sweeps to $\partial^f(E)$ and back.  Finally, it is readily checked that $\kappa^e_t$ represents
 $\alpha_{1,1}$.\end{proof}
 
 \begin{figure}[ht]
$$\includegraphics[width=10cm]{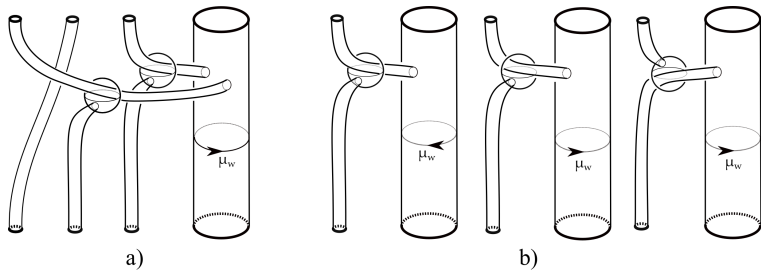}$$
 \caption[(a) X; (b) Y]{\label{fig:Fig3}\begin{tabular}[t]{ @{} r @{\ } l @{}}
 (a) & Torus with $W_2 = t^2+t^3$ \\
 (b) & All three tori with $W_2 = -t^2$
 \end{tabular}}
\end{figure}

We describe how to represent composition of generators when $f=\kappa$, the general case being similar.  First some terminology.  
Let $p:T^*\to \mu_w$ be the vertical projection.  By construction, each generator is obtained by 
removing a small disc $D$ from $T^*\setminus \kappa$ and replacing it by a disc $D'$.  Further each knotted disc lies 
in a small neighborhood of a 1-complex which itself lies in a neighborhood of $p^{-1}(\delta)$, for 
some interval $\delta\subset \mu_w$.  Squeezing, expanding or rotating this interval and correspondingly 
modifying the discs $D$ and $D' $ does not change the based homotopy class of $(T, \mu_w)$ provided 
that the expanding or rotating is supported away from $w$.  The composition of generators $\beta_0$  
and $\beta_1$ is represented as follows.    First find tori $(T_0,\mu_w), (T_1, \mu_w)$ constructed as 
above respectively representing $\beta_0, \beta_1$ so that $T_0$ coincides with $T_1$ near $\mu_w$ and 
the latter having a fixed orientation.  Further, assume that each $T_i$ is standard away from a neighborhood 
of $p^{-1}(\delta_i)$ where $\delta_0\cap\delta_1=\emptyset$ and $\delta_0$ proceeds $\delta_1$ when 
starting at $x\in \mu_w$.  To obtain $(T, \mu_w)$ representing $\beta_0*\beta_1$, modify $T^*$ near 
both $p^{-1}(\delta_i), i=0,1$ according to $T_0$ and $T_1$.  See Figure \ref{fig:Fig3}(a).  To see that $\beta_0*\beta_1$ 
is homotopic to $\beta_1*\beta_0$ observe that the two tori representing these classes are isotopic via an 
isotopy fixing $\kappa\cup \mu_w$ pointwise.  We conclude that any word in the generators is realizable 
by an embedded torus and $\pi_1\Emb_f(S^1, S^1 \times S^3)$ is abelian.

\begin{proposition} \label{passing to cover} Let $f:S^1\to S^1\times S^3$ be the standard vertical embedding with 
$W_0(f)=1$. Let $p:S^1\times S^3\to S^1\times S^3$ the $m$-fold cyclic cover, let $\{ \alpha_i|i>0\}$ 
denote the generators of $\pi_1(\Emb_f(S^1, S^1\times S^3))$ as in Theorem \ref{embthm} and let $\tilde\alpha_i$ 
denote the $p^*$ pull back of $\alpha_i$.  Then $\tilde\alpha_n=1$ if $m$ does not divide $n$ and 
$\tilde\alpha_n= m\alpha_{n/m} $ if $m$ divides $n$.\end{proposition}

\begin{proof}  Represent $\alpha_n$ by the torus $T_n$ as in \S2.  Then $\tilde\alpha_n$ is represented 
by $p^{-1}(T_n)=R_n$.  Now $T_n$ is constructed from the standard vertical torus $T^*$ and an unknotted
2-sphere $K$ by removing small discs $D^T, D^K$ from $T^*$ and $K$, and then adding a tube $Y$ that starts 
at $ \partial D^T$, links through $K$, goes $n$ times about the $S^1$ direction before connecting to
$\partial D^K$.  Therefore $R_n$ is obtained by removing $m$ small discs from $T^*$ and one from each
of the preimages of $K$ and then connecting their boundaries by the $m$ preimages $Y_1, \cdots, Y_m$ of 
$Y$.  I.e. $R_n$ is obtained by removing $m$ standard discs from $T^*$ and replacing them by $m$ other 
ones.  Now assume that $m$ divides $n$.   Then the sphere $K_i$ that $Y_i$ links is also the sphere to 
which $Y_i$ connects.  Note that if $D_i^T$ (resp. $D_i^K$) is the preimage of $D^T$ (resp. $D^K$) whose 
boundary is tubed to $Y_i$, then $((T^*\cup K_i)\setminus (D^T_i\cup D^K_i))\cup Y_i$ is a torus 
representing $\alpha_{n/m}$.  After an isotopy of $R_n$ supported near the $D^T_i$'s we can assume 
that the projection $\pi:T^*\to \mu_w$ has the property that $\pi(D_1^T), \cdots, \pi(D_n^T)$ are
disjoint intervals. (Recall that $\mu_w\subset T^*$ is an oriented  loop intersecting each vertical 
$S^1$ fiber once.)  It follows that $R_n$  represents an element $\beta$ of $\pi_1(\Emb(S^1, S^1\times S^3))$ 
which corresponds to the standard vertical circle sweeping around $T^*$ and going over one knotted
disc at a time.  Since there are $m$ such discs it follows that $\beta=m\alpha_{n/m}$.  

Now assume that $m$ does not divide $n$.  In that case each tube $Y_i$ links a sphere distinct from 
the one to which it connects.  Again isotope $R_n$  near the $D^T_i$'s so that the projection 
$ \pi:T^*\to \mu_w$ has the property that $\pi(D_1^T), \cdots, \pi(D_n^T)$ are disjoint intervals.  
Again let $\beta$ be the element represented by $R_n$.  Here the discs swept over by $\beta$ can be 
individually isotoped back to their $D_i^T$'s without intersecting $T^*$.  It follows that $\beta=1$. 
 \end{proof}
 
We return to the problem of determining if our invariants $W_0, W_1$ and $W_2$ are complete
invariants of the low-dimensional homotopy groups of $\Emb(S^1, S^1 \times S^n)$. 

\begin{lemma}\label{htpy-compute}
The homotopy groups of $C_2(S^1 \times S^n)$ are:
$$\pi_1 C_2(S^1 \times S^n) \simeq \BZ^2$$
$$\pi_m C_2(S^1 \times S^n) \simeq \pi_m S^n \oplus \pi_m(S^n)[t^{\pm 1}]$$
for $m \geq 2$. The symbol $\pi_m(S^n)[t^{\pm 1}]$ denotes the Laurent polynomials
with coefficients in the group $\pi_m(S^n)$, thus when $m=n$ it is the Laurent 
polynomial ring with integer coefficients. 

The boundary of $C_2(S^1 \times S^n)$ can be canonically identified with
$S^1 \times S^n \times S^n$ using our preferred trivialization of $T(S^1 \times S^n)$.
Thus $\pi_1 \partial C_2(S^1 \times S^n) \simeq \BZ$ and 
$\pi_n \partial C_2(S^1 \times S^n) \simeq \BZ^2$. 
We compute the induced map on the above homotopy groups for the inclusion map 
$\partial C_2(S^1 \times S^n) \to C_2(S^1 \times S^n)$. To make sense of this
map we need a common choice of basepoint. Identify $\partial C_2(S^1 \times S^n)$ with
the unit sphere bundle of $S^1 \times S^n$. Our basepoint will be the direction vector
pointing in the counter-clockwise direction of $S^1$, based at $(1, *)$ where $* \in S^n$
is any basepoint choice for $S^n$. 

The induced map on $\pi_1$ is identified with the diagonal map $\Delta : \BZ \to \BZ^2$, 
$\Delta(t) = (t,t)$.  The induced map on $\pi_n$ is identified with
$\BZ^2 \to \BZ \oplus \BZ[t^{\pm 1}]$, which in matrix form is
$$\begin{pmatrix}1 & 0 \\ 0 & 1-t^{-1} \end{pmatrix}.$$ 
The above computation requires a choice of common basepoint
in $\partial C_2(S^1\times S^n)$ and $C_2(S^1 \times S^n)$, and is valid for any such 
choice.
\end{lemma}
 
These isomorphisms follow from the fact that the fiber bundle
$$Bl_*(S^1 \times S^n) \to C_2[S^1 \times S^n] \to S^1 \times S^n$$
has a section.  There are several sections available: (1) using the
trivialization of $T(S^1 \times S^n)$ or (2) using the antipodal map of $S^1$ or
$S^n$ or the combination of the two. All of these sections are homotopic.  The 
section (1) is the only choice that allows for a common base-point in $C_2(S^1\times S^n)$
and its boundary.

\begin{theorem}\label{calc_thm}
The invariant
$$W_1 \oplus W_2 : \pi_1 \Emb_f(S^1, S^1 \times S^3) \to \BZ \oplus \Lambda^{W_0}_{3}$$
is an isomorphism for all $f \in \Emb(S^1, S^1 \times S^3)$. Stated another way, 
$W_2^* : \pi_1 \Emb_f^*(S^1, S^1 \times S^3) \to \Lambda^{W_0}_{3}$ is an isomorphism
for all $f \in \Emb_f^*(S^1, S^1 \times S^3)$, i.e. the components of $\Emb(S^1, S^1 \times S^3)$
have infinitely-generated, free-abelian fundamental groups. 
\end{theorem}

We also prove an analogous theorem for $\Emb(S^1, S^1 \times S^n)$ with $n \geq 4$. 

\begin{theorem}\label{embspace_thm} For $n \geq 4$, the first three non-trivial homotopy groups
of the embedding space $\Emb(S^1, S^1 \times S^n)$ are given by the maps:
\begin{enumerate}
\item $W_0 : \pi_0 \Emb(S^1, S^1 \times S^n) \to \BZ$ which is an isomorphism. 
\item $W_1 : \pi_1 \Emb_f(S^1, S^1 \times S^n) \to \BZ$ which is also an isomorphism, for any choice of 
path-component, i.e. $f \in \Emb(S^1, S^1 \times S^n)$. 
\item $W_2 : \pi_{n-2} \Emb_f(S^1, S^1 \times S^n) \simeq \Lambda^{W_0}_{n}$ and it is also an isomorphism, 
for any choice of path component $f \in \Emb(S^1, S^1 \times S^n)$. 
\end{enumerate}
\end{theorem}

As was described between Definitions \ref{w1def} and \ref{c2def}, 
a general-position argument tells us the forgetful map 
$\Emb(S^1, S^1 \times S^n) \to \Map(S^1, S^1 \times S^n)$ is an isomorphism
on all homotopy groups $\pi_k$ for $k<n-2$, and an epi-morphism on $\pi_{n-2}$.  Moreover, 
the space $\Map(S^1, S^1 \times S^n)$ splits as the product of two free loop-spaces
$\Map(S^1, S^1 \times S^n) \simeq L(S^1) \times L(S^n)$, which proves claims (1) and (2)
in Theorem \ref{embspace_thm}.  The primary role of Theorems \ref{calc_thm} and
\ref{embspace_thm} is the description of the first homotopy-group of the embedding space $\Emb(S^1, S^1 \times S^n)$
that differs from that of the mapping space $\Map(S^1, S^1 \times S^n)$.  By Theorem \ref{embthm} 
we know this happens in dimension $n-2$.  The homotopy group $\pi_{n-2} \Emb_f(S^1, S^1 \times S^n)$ contains a 
large abelian subgroup, detected by the $W_2$ invariant.   The purpose of Theorems
\ref{calc_thm} and \ref{embspace_thm} is to argue $W_2$ (and $W_1$ if $n=3$) detects all 
non-trivial elements of $\pi_{n-2} \Emb_f(S^1, S^1 \times S^n)$, provided $n \geq 3$.

We will use a tool called {\it functor calculus} in the context of embedding spaces $\Emb(M,N)$
to prove Theorems \ref{calc_thm} and \ref{embspace_thm}.  Although everything needed to prove these two theorems
is present in the work of Dax \cite{Da}, we choose to use embedding calculus to situate the proof in
a contemporary context. It should be noted that Theorems \ref{calc_thm} and \ref{embspace_thm} are not
essential to any of the results highlighted in the introduction. 

The embedding calculus gives us a sequence of maps out of embedding spaces
$$\Emb(M, N) \to T_k \Emb(M, N)$$
where $M$ is an $m$-dimensional manifold and $N$ is an $n$-dimensional manifold.  The $k$-th 
{\it evaluation map} $\Emb(M, N) \to T_k\Emb(M,N)$ is known
to be $k(n-m-2)+1-m$-connected. This means that for any choice of path-component 
of $\Emb(M,N)$ the induced map $\pi_j \Emb(M, N) \to \pi_j T_k\Emb(M,N)$
is an isomorphism for $j< k(n-m-2)+1-m$ and an epimorphism for $j=k(n-m-2)+1-m$.  
This connectivity result is only valid provided $n \geq m+3$, i.e. it requires embeddings to be of 
co-dimension $3$ or larger.  In our case, $M=S^1$ is
a $1$-manifold and $N = S^1 \times S^{n-1}$, thus the $k$-th evaluation map is $k(n-3)$-connected.  This 
tells us that we need only compute $\pi_{n-2} \left( T_2 \Emb(S^1, S^1\times S^n)\right)$ to verify Theorems \ref{calc_thm} and
\ref{embspace_thm}.  

Our invariant $W_2$ is {\it almost} defined on $T_2 \Emb(S^1, S^1 \times S^n)$.  Specifically, 
$T_2 \Emb(S^1, S^1 \times S^n)$ is described as a homotopy pull-back of three familiar spaces in
Corollary 4.3 of the paper of Goodwillie and Weiss \cite{GW2}.  Readers unfamiliar with homotopy pull-backs, 
or homotopy-limits of diagrams
of the form $\xymatrix{X \ar[r]^f & Z & Y \ar[l]_g}$, 
see Definition
3.2.4 of the book \cite{MV} which provides useful context.  In short, such a homotopy pullback is denoted
$holim(\xymatrix{X \ar[r]^f & Z & Y \ar[l]_g})$.  This is the space of triples
$$\{(x,\alpha,y) : x \in X, \alpha : [0,1] \to Z, y \in Y, \text{ s.t. } \alpha(0)=f(x), \alpha(1)=g(y) \}.$$
The element $\alpha$ is a continuous path between $f(x)$ and $g(y)$.  We will describe $T_2 \Emb(S^1, S^1 \times S^n)$
as a homotopy pullback of a diagram of three spaces, as in Corollary 4.3 of \cite{GW2}. 
\begin{enumerate}
\item $\Map(S^1, S^1 \times S^n)$, i.e. this is the space of continuous functions from $S^1$ to $S^1 \times S^n$.
\item $\Map^{\Sigma_2}( (S^1)^2, (S^1 \times S^n)^2 )$, this is the space of $\Sigma_2$-equivariant continuous functions
from $(S^1)^2$ to $(S^1 \times S^n)^2$, where the $\Sigma_2$-action on the two spaces comes from permuting coordinates.
\item $\ivm^{\Sigma_2}( (S^1)^2, (S^1 \times S^n)^2 )$, this is the space of {\it strictly isovariant maps}.  This is 
a subspace of (2) where the maps $f$ have the additional properties that $f^{-1}(\Delta_{S^1 \times S^n}) = \Delta_{S^1}$, i.e.
the diagonal subspace of $(S^1)^2$ is the only subspace sent to the diagonal subspace of $(S^1 \times S^n)^2$, where
$\Delta_M = \{ (p,p) : p \in M \} \subset M^2$. The other condition is the derivative of $f$ is fibrewise injective from
the normal bundle of $\Delta_{S^1}$ (in $(S^1)^2$) to the normal bundle of $\Delta_{S^1 \times S^n}$ (in $(S^1 \times S^n)^2$). 
\end{enumerate}

The result of Goodwillie and Weiss \cite{GW2} is that $T_2 \Emb(S^1, S^1 \times S^n)$ is the homotopy pullback
of the  diagram 
$$\xymatrix{\ivm^{\Sigma_2}( (S^1)^2, (S^1 \times S^n)^2 ) \ar[r] & 
           \Map^{\Sigma_2}( (S^1)^2, (S^1\times S^n)^2) & \Map(S^1, S^1 \times S^n) \ar[l]}$$
where the first map is set-theoretic inclusion. The second map 
$\Map(S^1, S^1 \times S^n) \to \Map^{\Sigma_2}( (S^1)^2, (S^1\times S^n)^2)$ is given by {\it repetition} i.e.
 if $f \in \Map(S^1, S^1 \times S^n)$ then the equivariant map of
pairs is given by $(z_1, z_2) \longmapsto (f(z_1), f(z_2))$. 

\begin{proposition}\label{t2htpy}
The forgetful map $T_2 \Emb(S^1, S^1 \times S^n) \to \ivm^{\Sigma_2}((S^1)^2, (S^1 \times S^n)^2)$ induces
an isomorphism on homotopy groups $\pi_k$ for $k \leq n-2$, on all path components. Moreover, the space
of isovariant maps $\ivm^{\Sigma_2}((S^1)^2, (S^1 \times S^n)^2)$ is homotopy-equivalent to the space of
stratum-preserving $\Sigma_2$-equivariant maps $C_2[S^1] \to C_2[S^1 \times S^n]$. 
\begin{proof}
The forgetful map $holim(\xymatrix{X \ar[r]^f & Z & Y \ar[l]_g}) \to X$ is the map that maps a triple
$(x,\alpha,y) \longmapsto x$.  As is described in Example 3.2.8 \cite{MV}, the fibre over a point
$x_* \in X$ is $holim(\xymatrix{\{x_*\} \ar[r]^f & Z & Y \ar[l]_g})$, and this can be identified with
the homotopy-fibre of $g : Y \to Z$ over $f(x_*)$. 

In our case we are interested in the forgetful map from the homotopy limit of
$$\xymatrix{\ivm^{\Sigma_2}( (S^1)^2, (S^1 \times S^n)^2 ) \ar[r] & 
           \Map^{\Sigma_2}( (S^1)^2, (S^1\times S^n)^2) & \Map(S^1, S^1 \times S^n) \ar[l]}$$
to $\ivm^{\Sigma_2}( (S^1)^2, (S^1 \times S^n)^2 )$. Let's investigate the homotopy-groups of
$$\mathrm{hofib}_{W_0}\left( \Map(S^1, S^1 \times S^n) \to \Map^{\Sigma_2}( (S^1)^2, (S^1\times S^n)^2) \right).$$ 

Given that this map is the repetition map, it is split.  The splitting comes from restriction to the
$\Sigma_2$-fixed subspaces of $(S^1)^2$ and $(S^1 \times S^n)^2$ respectively, i.e. the diagonals.   
This tells us that the homotopy-groups of $\mathrm{hofib}$ are the
kernel of the induced maps $\pi_k \Map(S^1, S^1\times S^n) \to \pi_k \Map^{\Sigma_2} ( (S^1)^2, (S^1\times S^n)^2) )$. 
These groups are trivial when $k \leq n-2$, since $\pi_k \Map(S^1, S^1 \times S^n)$, with the exceptions of $k=0$ or
$k=1$, in which case the repetition map is injective.  Thus the map
$$\pi_{n-2} T_2 \Emb(S^1, S^1\times S^n) \to \pi_{n-2} \left(\ivm^{\Sigma_2}( (S^1)^2, (S^1\times S^n)^2)\right)$$
is always an isomorphism.

Regarding the claim that the space of strictly isovariant maps is homotopy-equivalent to the space of stratum-preserving
$\Sigma_2$-equivariant maps $C_2[S^1] \to C_2[S^1 \times S^n]$, recall that a map of Fulton-Macpherson compactified
configuration spaces descends to a map $(S^1)^2 \to (S^1 \times S^n)^2$.  Given that our initial map was assumed to
be $\Sigma_2$-equivariant, the induced map will be as well.  Lastly, using the uniqueness of collar neighbourhoods
theorem, one can assume all our maps $C_2[S^1] \to C_2[S^1 \times S^n]$ are fibrewise linear with respect to the
distance parameter from the boundary -- this is enough to guarantee our induced map $(S^1)^2 \to (S^1\times S^n)^2$
is isovariant. 
Similarly, given a strictly isovariant map, one can lift it to a unique map of the Fulton-Macpherson compactified 
configuration spaces. 
\end{proof}
\end{proposition}

\begin{proof}(of Theorem \ref{calc_thm})
The space $C_2[S^1]$ is simply an annulus, i.e. diffeomorphic to
$S^1 \times [-1,1]$.  The space $C_2[S^1 \times S^3]$ is diffeomorphic to 
$(S^1 \times S^3) \times Bl_1(S^1 \times S^3)$, given by the
map $$C_2(S^1\times S^3) \ni (z_1,p_1), (z_2,p_2) \longmapsto (z_1,p_1), (z_1z_2^{-1}, p_1p_2^{-1}) \in 
S^1 \times S^3 \times (S^1 \times S^3)\setminus \{1\}.$$ 
The blow-up $Bl_1(S^1 \times S^3)$ deformation-retracts to its $3$-skeleton $S^1 \vee S^3$. 

The boundary of $C_2[S^1 \times S^3]$ is canonically diffeomorphic to the unit tangent
bundle of $S^1 \times S^3$.  Due to the triviality of $T(S^1 \times S^3)$ we can think of 
the unit tangent bundle of $S^1 \times S^3$ as $S^1 \times S^3 \times S^3$. 

The fundamental group $\pi_1 C_2[S^1 \times S^3]$ is free abelian on two generators, see Lemma \ref{htpy-compute}. 
The natural set of generators are given by the winding numbers of the first and second points of the configurations
about the $S^1$ factor of $S^1 \times S^3$.  Consider the covering space of $C_2[S^1 \times S^3]$
corresponding to the homomorphism $\pi_1 C_2[S^1 \times S^3] \to \BZ$ given by taking the difference
between the two winding numbers. We denote this covering space by $\tilde C_2[S^1 \times S^3]$. 
By design, any stratum-preserving map $C_2[S^1] \to C_2[S^1 \times S^3]$
lifts to this covering space, $C_2[S^1] \to \tilde C_2[S^1 \times S^3]$.  
Since $C_2[S^1 \times S^3]$ fibers over $S^1 \times S^3$, this covering
space does as well, but the fiber is the universal cover of $Bl_*(S^1 \times S^3)$, which could be
described as $Bl_{\BZ \times\{1\} }(\BR \times S^3)$, giving
$$ \tilde C_2[S^1\times S^3] \simeq Bl_{\BZ \times\{1\} }(\BR \times S^3) \times S^1 \times S^3.$$ 

Consider a map $S^1 \times C_2[S^1] \to \tilde C_2[S^1 \times S^3] \equiv Bl_{\BZ \times\{1\} }(\BR \times S^3) \times S^1 \times S^3$.
 We can assume this map is {\it null} in the rightmost $S^3$ factor, as it is homotopic to a map that
factors through a $2$-dimensional domain.  Thus we have reduced the computation
of the fundamental group of this mapping space to understanding the space of stratum-preserving maps 

$$C_2[S^1] \to Bl_{\BZ \times \{1\}} (\BR \times S^3) \times S^1.$$

If we take the degree of the projection to the $S^1$ factor we recover $W_0$.  Given a family
$S^1 \times C_2[S^1] \to Bl_{\BZ \times\{1\} }(\BR \times S^3) \times S^1$, 
the homotopy-class of the projection to the $S^1$ factor is determined by the $W_0$ and $W_1$
invariants.

Consider the projection $C_2[S^1] \to Bl_{\BZ \times \{1\}}(\BR \times S^3)$. By design, 
one boundary stratum is in the blow-up sphere corresponding to $k \times \{1\}$, and the other stratum is in 
the blow-up sphere corresponding to $(k+W_0)\times \{1\}$.  The space $Bl_{\BZ \times \{1\}}(\BR \times S^3)$ has
the homotopy type of an infinite wedge of $3$-spheres, perhaps best thought of as the $3$-skeleton of

$$\left(\BR \times S^3\right) \setminus \left(\BZ \times \{1\}\right) \simeq 
\left(\BR \times \{-1\}\right) \cup \left(\sqcup_{i \in \BZ} \{\frac{1}{2}+i\} \times S^3\right).$$

To describe the equivariance condition on our lift $C_2[S^1] \to Bl_{\BZ \times \{1\}} (\BR\times S^3)$,
the relevant $\Sigma_2$-action on the target space is induced by the map $(t,p) \longmapsto (2j+W_0-t, p^{-1})$.
Choosing $j=0$ gives us the same convention as in Theorem \ref{embthm}.  This computation is done by considering
the diffeomorphism $C_2(S^1 \times S^3) \to (S^1 \times S^3) \times ((S^1 \times S^3) \setminus \{1\})$. The
involution of $C_2(S^1 \times S^3)$ in the $\Sigma_2$-action sends 
$(z_1,p_1), (z_2,p_2)$ to $(z_2,p_2), (z_1,p_1)$.  Conjugating this involution by our identification gives
us the map $(z_1,p_1), (z_2,p_2) \to (z_2^{-1}z_1,p_2^{-1}p_1), (z_2^{-1},p_2^{-1})$, which, on the fiber lifts to the 
above map.

Homotopy classes of maps to wedges of spheres, via the Ponyriagin construction, are characterized by their 
intersection numbers with the points antipodal to the wedge point. Our maps are equivariant, so our framed 
points in the domain satisfy a symmetry condition.  Our space $Bl_{\BZ \times \{1\}} (\BR\times S^3)$ 
equivariantly deformation retracts to the above $3$-skeleton.  The $\Sigma_2$-stabilizer is a single point 
if $W_0$ is even, and a pair of points $(\frac{W_0}{2}, \pm 1) \subset \{\frac{W_0}{2}\} \times S^3$ if 
$W_0$ is odd. Since the $\Sigma_2$-action on the domain is fixed-point
free, all this tells us is the degree associated to this intermediate sphere is even. 
Thus, if $W_2$ is zero, we can equivariantly homotope our map so that its image is disjoint from
the antipodal points to all the wedge points of the $S^3$ factors. We can therefore assume our map 
$C_2[S^1] \to \BR \times \{-1\} \cup \left(\sqcup_{i \in \BZ} \{\frac{1}{2}+i\} \times S^3\right)$ 
is homotopic to a map to the interval $[0,W_0] \times \{-1\}$. 
\end{proof}

\begin{proof}(of Theorem \ref{embspace_thm})
The proof roughly mimics Theorem \ref{calc_thm}.  Unfortunately, the bundle $C_2(S^1 \times S^n) \to S^1 \times S^n$
is generally not trivial, so we do not have access to quite as simple an argument, but we take some inspiration
from it.

As with Theorem \ref{calc_thm} we need only consider equivariant stratum-preserving maps 
$C_2[S^1] \to C_2[S^1 \times S^n]$ to compute $\pi_{n-2} \Emb(S^1, S^1 \times S^n)$. 
So we consider a map 
$$S^{n-2} \times C_2[S^1] \to C_2[S^1 \times S^n]$$
with $n \geq 4$. 

The composite
with the bundle projection map $C_2[S^1 \times S^n] \to S^1 \times S^n$ factors through the projection to $S^1$, and
is given by the $W_1$ invariant.  Thus our map lifts
$$S^{n-2} \times C_2[S^1] \to \tilde C_2[S^1 \times S^n].$$
As with the proof of Theorem \ref{calc_thm} the fiber of the map $\tilde C_2[S^1 \times S^n] \to S^1 \times S^n$ can be
identified with 
$Bl_{\BZ \times\{1\} }(\BR \times S^n)$, which is similarly identified with 
$\BR \times \{-1\} \cup \left(\sqcup_{i \in \BZ} \{\frac{1}{2}+i\} \times S^n\right)$. 

Here our argument diverges from the proof of Theorem \ref{calc_thm}.  While the action of 
$\Sigma_2$ on $C_2(S^1)$ is free, it has the invariant subspace of antipodal
points on $S^1$. 

By restricting to the subspace of antipodal points, we get a fibration from the space of stratum-preserving
equivariant maps $C_2[S^1] \to C_2[S^1 \times S^n]$ to the space $Maps^{\Sigma_2}(S^1, C_2[S^1 \times S^n])$
of equivariant maps.  This mapping space can be thought of as the space of maps
$S^1 \to S^1 \times S^n$ where antipodal points are required to map to distinct points. By a transversality argument,
any $k$-dimensional family of maps to the free loop space $L(S^1 \times S^n)$ can be perturbed to have this property,
provided $k<n$.  Thus through dimension $n-2$, this space has the same homotopy groups as $L(S^1 \times S^n) \simeq L(S^1) \times L(S^n)$, which are the homotopy groups of $\BZ \times S^1$. i.e. this recovers our $W_0$ and $W_1$ invariants.

We are considering the fibration from the space of stratum-preserving equivariant maps
$$C_2[S^1] \to C_2[S^1 \times S^n]$$
to the space $Maps^{\Sigma_2}(S^1, C_2[S^1 \times S^n])$.  The fiber is precisely the
space of maps of an annulus $S^1 \times [0,1]$ to $C_2[S^1 \times S^n]$
that restrict to a fixed map on one boundary circle, and which send the other boundary
circle to $\partial C_2(S^1 \times S^n)$.  We lift this map to the fiber
$\BR \times \{-1\} \cup \left(\sqcup_{i \in \BZ} \{\frac{1}{2}+i\} \times S^n\right)$. 
From this perspective we can see that the $W_2$ invariant is well-defined for an $(n-2)$-parameter family,
and there are no further invariants. 
\end{proof}

\begin{theorem} To each element of $\pi_1 \Emb(S^1, S^1\times S^3)$, there is an explicitly constructible 
embedded torus that represents that element via the spinning construction. \qed
\end{theorem}

\section{Bundles of Embeddings and Diffeomorphism Groups}\label{exseq}

Rationally, the first three non-trivial homotopy groups of $\Emb(I, S^1 \times B^n)$ are in dimensions $0$, $n-2$ and
$2n-4$.  In this section we construct invariants of these homotopy groups, specifically the $W_2$ and $W_3$ invariants, 

$$W_2 : \pi_{n-2} \Emb(I, S^1 \times B^n) \to \Zed[t^{\pm 1}] / \langle t^0 \rangle$$

$$W_3 : \pi_{2n-4} \Emb(I, S^1 \times B^n) \to \Rat[t_1^{\pm 1}, t_3^{\pm 1}] / R$$

\noindent where $R$ is the {\it hexagon relation}. Note that $W_2$ was defined in Section \ref{embsec}. We re-use the notation
here as this also is an invariant of the $2^{nd}$ non-trivial homotopy group of an embedding space.  
We derive these invariants from a computation of the (rational)
homotopy of configuration spaces in $S^1 \times B^n$.   This allows us to detect diffeomorphisms of
$S^1 \times B^3$ via the scanning construction $\pi_0 \Diff(S^1 \times B^3 \text{ fix } \partial) \to \pi_2 \Emb(I, S^1 \times B^3)$. 
To conclude the section we relate homotopy-type of $\Diff(S^1 \times B^n \text{ fix } \partial)$ to that of the space of
co-dimension two unknots in $\Emb(S^{n-1}, S^{n+1})$, and the homotopy-type of $\Emb(B^n, S^1 \times B^n)$ 
via some simple fiber sequences. 

The rational homotopy groups of the configuration spaces of points in Euclidean space were first described by
Milnor and Moore \cite{MM}.  Their result is that the rational homotopy groups $\Rat \otimes \pi_* C_k(B^{n+1})$
are isomorphic to the primitives of $H_*(\Omega C_k(B^{n+1}) ; \Rat)$ via a Hurewicz-style map. The generators of 
$\pi_n C_k(B^{n+1})$ we denote
$w_{ij}$. The class $w_{ij}$ has all $k$ points stationary, with the exception of point $j$ that orbits around point $i$.  
The Whitehead bracket operation $[\cdot,\cdot] : \pi_n X \times \pi_m X \to \pi_{m+n-1} X$ is the obstruction to a 
map $f \vee g : S^n \vee S^m \to X$, extending to $S^n \times S^m \to X$.  We identify $S^n \vee S^m$ with all but the top-dimensional 
cell of $S^n \times S^m$, i.e. $S^n \vee S^m = S^n \times \{*\} \cup \{*\} \times S^m$.  Thus the Whitehead product is the characteristic 
map of the top-dimensional cell $S^{n+m-1} \to S^n \vee S^m$ composed with $f \vee g$.
The Whitehead bracket is bilinear, graded symmetric, i.e. $[y,x] = (-1)^{nm} [x,y]$ and it satisfies a Jacobi-like identity. 
$$(-1)^{pr}[[f,g],h] + (-1)^{pq}[[g,h],f] + (-1)^{rq}[[h,f],g] = 0,$$
$$\text{where } f \in \pi_p X, g \in \pi_q X, h \in \pi_r X \text{ with } p,q,r \geq 2.$$

The theorem of Milnor and Moore implies $\Rat \otimes \pi_* C_k(\Real^{n+1})$ is generated by the $w_{ij}$ classes, 
subject to the relations
\begin{itemize}
\item $w_{ii} = 0 \ \forall i$
\item $w_{ij} = (-1)^{n+1} w_{ji} \ \forall i \neq j$
\item $[w_{ij},w_{lm}] = 0$ when $\{i,j\} \cap \{l,m\} = \emptyset$. 
\item $[w_{ij}+w_{il}, w_{lj}] = 0$ for all $i,j,l$. 
\end{itemize}

The latter relation should be viewed a generalized `orbital system' map $S^n \times S^n \to C_k(B^{n+1})$ where there is an earth-moon
pair corresponding to points $l$ and $j$ respectively, orbiting around the sun corresponding to point $i$. 
This relation can be rewritten as
$$[w_{ij}, w_{jk}] - [w_{jk}, w_{ki}] = 0. $$
Thus we have the equality of the three cyclic permutations,
$$[w_{ij}, w_{jk}] = [w_{jk}, w_{ki}] = [w_{ki}, w_{ij}].$$

To compute the rational homotopy groups of $C_k[S^1 \times B^n]$ we start by considering the inclusion
$C_k[B^{n+1}] \to C_k[S^1 \times B^n]$ induced by an embedding $B^{n+1} \to S^1 \times B^n$.  This allows us
to define classes $w_{ij} \in \pi_n C_k[S^1 \times B^n]$.  The bundle $C_k(S^1 \times B^n) \to C_{k-1}(S^1 \times B^n)$ 
is split, with fiber the homotopy-type of a wedge of a circle and $(k-1)$-copies of $S^n$, we again have that
the homotopy groups of $C_k(S^1 \times B^n)$ are rationally generated by the elements $t_l.w_{ij}$ where 
$\{t_1,\cdots,t_k\} \subset \pi_1 C_k[S^1 \times B^n]$
are the standard generators of the fundamental group, the free abelian group of rank $k$. 

\begin{proposition}\label{htpyck} The $n$-th homotopy group of $C_k[S^1 \times B^n]$ has generators 
$t_l^m w_{ij}$ for $i,j,m \in \Zed, l \in \{1,2,\cdots,k\}$.  The relations are 
\begin{itemize}
\item $w_{ii} = 0$,
\item $w_{ji} = (-1)^{n+1} w_{ij}$,
\item $t_l.w_{ij} = w_{ij} \text{ if } l \notin \{i,j\}$,
\item $t_i.w_{ij} = t_j^{-1}.w_{ij}.$
\end{itemize}
The rational homotopy-groups of $C_k[S^1 \times B^n]$ are generated by the Whitehead products of 
the elements $t_l^m.w_{ij}$. These satisfy the relations
\begin{itemize}
\item $[w_{ij}, w_{lm}]=0$ if $\{i,j\} \cap \{l,m\} = \emptyset$,
\item $[w_{ij}, w_{jl}] = [w_{jl}, w_{li}] = [w_{li}, w_{ij}]$,
\item $t_l.[f,g] = [t_l.f,t_l.g]$.
\end{itemize}
\end{proposition}

The above computation should be viewed as a slight rephrasing of the argument given in Cohen-Gitler \cite{CG}. Observe
$$t_1^{\alpha_1} \cdots t_m^{\alpha_m}.w_{ij} = t_i^{\alpha_i-\alpha_j}.w_{ij}.$$
Thus we also have
$$t_1^{\alpha_1} \cdots t_m^{\alpha_m}[w_{ij}, w_{jl}] = [t_i^{\alpha_i-\alpha_j} w_{ij}, t_j^{\alpha_j-\alpha_l}w_{jl}].$$

As a $\Rat[t_1^{\pm 1},\cdots, t_k^{\pm 1}]$-module, we have (assuming $k \geq 3$)
$$\Rat \otimes \pi_{2n-1} C_k[S^1 \times B^n] \simeq 
\bigoplus_{i<j<l} \Rat[t_1^{\pm 1},\cdots, t_k^{\pm 1}] / \langle t_it_jt_l-1=t_m-1=0 \ 
  \forall m \notin \{i,j,l\}\rangle \oplus$$
$$ \bigoplus_{\substack{i<j \\ 0 \leq l \text{ if } n \text{ even} \\ 1 \leq l \text{ if } n \text{ odd} }}  
  \Rat[t_1^{\pm 1},\cdots, t_k^{\pm 1}] / \langle t_it_j-1 = t_m-1 = 0 \ \forall m \notin \{i,j\}\rangle.$$

The top summands, i.e. for each $i<j<l$ are generated by the $[w_{ij}, w_{jl}]$ brackets, 
the bottom summands are generated by the $[w_{ij}, t_i^l w_{ij}]$ brackets.  
\vskip 10pt

We outline a general `closure argument' that produce invariants of 
the homotopy groups $\pi_{n-2} \Emb(I, S^1 \times B^n)$ and $\pi_{2n-4} \Emb(I, S^1 \times B^n)$.
For this computation the $2^{nd}$ stage of the Taylor tower suffices.  We will use Dev Sinha's 
mapping-space model \cite{Si1}, analogously to how it is used in \cite{BCSS}. 

Sinha's model for the $k$-th stage of the tower for $\Emb(I, M)$ is denoted $AM_k$.  
This consists of the stratum-preserving aligned maps $C_k[I, \partial] \to C_k'[S^1 \times B^n, \partial]$ where 
\begin{itemize}
\item $C_k[I, \partial]$ is the subspace of $C_{k+2}[I]$ consisting of
points of the form $(0, t_1, \cdots, t_n, 1)$ with $0 \leq t_1 \leq \cdots \leq t_n \leq 1$.
In general $C_k[I, \partial]$ is the $(k+2)^{nd}$ Stasheff polytope, whose vertices correspond to the ways
of bracketing a word with $(k+2)$-letters into a tree of nested binary operations, i.e. all the ways of
expressing associativity in a word of length $(k+2)$. The edges are applications of the associativity rule. 
Thus the $3^{rd}$ Stasheff polytope is an interval. The $4^{th}$ Stasheff Polytope is the pentagon. 
In general it is a truncated simplex. 
\item $C_k[M, \partial]$ is the subspace of
$C_{k+2}[M]$ consisting of points of the form $(b_0, p_1, \cdots, p_n, b_1)$ where $b_0, b_1 \in \partial M$ are 
the basepoints of the embedding space, i.e. provided we demand the maps in the embedding space $\Emb(I, M)$ sends
$0 \longmapsto b_0$ and $1 \longmapsto b_1$. 
\end{itemize}

To visualize $C_k[I, \partial]$, one first considers the {\it naive compactification} of $C_k(I)$, i.e. the simplex. 
$$\overline{C_k(I)} = \{ (t_1,\cdots, t_k) : 0 \leq t_1 \leq \cdots \leq t_k \leq 1 \}$$
The space $\overline{C_k(I)}$ is the point-set topological closure of the path-component of $C_k(I)$ where the points are linearly ordered by $<$. 
To obtain the Stasheff polytope from the $n$-simplex, one iteratively truncates (blows up) strata corresponding to collisions of
more than two points. Thus for $C_2[I, \partial]$ we blow up the $0=t_1=t_2$ stratum, since $t_1$ and $t_2$ are colliding, but they are also colliding
with the initial point.  Similarly we blow up the $t_1=t_2=1$ stratum.  Given that no other collisions occur, these are the only additional 
strata in the compactification.  Similarly in $C_3[I, \partial]$, but now there are the blow-ups from $0=t_1=t_2$, $t_1=t_2=t_3$, $t_2=t_3=1$, and
the two relatively high co-dimension blow-ups $0=t_1=t_2=t_3$ and $t_1=t_2=t_3=1$. 

We will only be interested in the $2^{nd}$ and $3^{rd}$ stages of the embedding calculus in this paper, and given that the behaviour of our 
mappings will be fibrewise linear on the `truncations' of $\overline{C_k[I]}$, we will often simply consider $T_k \Emb(I, S^1 \times B^n)$ 
to simply be stratum-preserving aligned maps $C_k[I] \to C_k'[S^1 \times B^n]$, i.e. suppressing extraneous combinatorial details to keep the 
technicalities light.  Readers should be aware these additional constraints must always be considered, to ensure these simplified mapping spaces
have the desired homotopy-type. 

Elements of the $2^{nd}$ stage can be restricted to the faces of $C_2[I]$ giving loops in the three boundary sub-strata of $C_2'[S^1 \times B^n]$ 
corresponding to the collisions $t_1=0, t_1=t_2$ and $t_2=1$. 
These sub-strata are diffeomorphic to $C_1'[S^1 \times B^n]$, i.e. the unit tangent bundle of
$S^1 \times B^n$ which is diffeomorphic to $S^1 \times B^n \times S^n$. 
Given that elements of $\pi_{n-2} \Omega S^n$ are trivial, we can 
homotope any map $S^{n-2} \times C_2[I] \to C_2'[S^1 \times B^n]$ to standard linear maps on the
boundary facets.  These null-homotopies can be attached to the original map with domain $S^{n-2} \times C_2[I]$ 
to give us a map out of an adjunction.  This new map is standard on the boundary. 


In our case, we are only interested in the component of 
$\Emb(I, S^1 \times B^n)$ where the interval winds a net zero number of times about the $S^1$ factor. 
After the adjunction we have a map from a space diffeomorphic to $B^n$ to the space $C'_2[S^1 \times B^n]$, 
and this map is constant on the boundary. 

The important part of this construction is we used a choice of null-homotopies to construct the map
$S^n \to C_2'[S^1 \times B^n]$.  If we choose {\it different} null-homotopies, we should check 
how the resulting map may differ. 

\begin{proposition}({\bf Closure Argument 1}) \label{tor-clo-arg1}
Given an element of $[f] \in \pi_{n-2} \Emb(I, S^1 \times B^n)$ we form the closure of
the evaluation map $ev_2(f) \in \pi_{n-2} T_2 \Emb(I, S^1 \times B^n)$ which is
a map of the form
$$\overline{ev_2}(f) : S^n \to C_2'[S^1 \times B^n].$$
The homotopy group $\pi_n C_2'[S^1 \times B^n]$ is isomorphic to a direct sum $\Zed[t^{\pm 1}] \oplus \Zed^2$. 
One can think of this isomorphism via the splitting $C_2'[S^1 \times B^n] \simeq C_2[S^1 \times B^n] \times (S^n)^2$. 
There are three inclusions  of
$C_1'[S^1 \times B^n]$ into $C_2'[S^1 \times B^n]$ coming from the three facets $p_1 = \{1\} \times \{-1\}$, 
$p_1=p_2$ and $p_2=\{1\}\times \{1\}$ respectively, corresponding to the three facets of $C_2[I]$ via the 
stratum-preserving condition.  These inclusions induce subgroups generated by $(t^0, 1, 1)$, $(0,1,0)$ and $(0,0,1)$ respectively.  
Thus there is a well-defined homomorphism
$$\pi_{n-2} \Emb(I, S^1 \times B^n) \to \Zed[t^{\pm 1}] / \langle t^0 \rangle$$
by counting all monomials the Laurent-polynomial part of $\pi_n C_2'[S^1 \times B^n] \simeq \Zed[t^{\pm 1}] \oplus \Zed^2$ 
other than $t^0$.  This map is an epimorphism. 
\begin{proof}
Consider the construction of the closure $ev_2(f)$.  We attach null-homotopies 
to the three face maps of $ev_2(f) : S^{n-2} \times C_2[I] \to C_2'[S^1 \times B^n]$. 
We use the notation $C_2[I]_{1=2}$ to denote the $t_1=t_2$ substratum of $C_2[I]$, and similarly
$C_2'[S^1 \times B^n]_{1=2}$ will denote the substratum of $C_2'[S^1 \times B^n]$ where the two
points collide. Similarly, $1=0$ will be our notation to indicate the first point is at its initial point
in $C_2[I]$ and $C_2[S^1 \times B^n]$ respectively.  We will use the {\it simplicial identifications} between $C_1[I]$ and
the three boundary facets ($0=1, 1=2, 2=3$), that is we identify $C_1[I]$ with $C_2[I]_{0=1}$ via the map
$t \longmapsto (0, t)$, similarly we identify $C_1[I]$ with $C_2[I]_{1=2}$ via the map $t \longmapsto (t,t)$. 
We use the isomorphism $\pi_n C_2'[S^1 \times B^n] \simeq \Zed[t^{\pm 1}] \oplus \Zed^2$ to talk about elements
of $\pi_n C_2'[S^1 \times B^n]$. The generators of $\Zed[t^{\pm}]$ we denote $t^i \ \forall i$.  
The generators of the remaining $\Zed^2$ factor are denoted $\alpha_1, \alpha_2$. Notice if we attach a different null-homotopy
to the $i=i+1$ face of $C_2[I]$ we are modifying the homotopy-class of $\overline{ev_2}(f)$ by adding a multiple of:
\begin{itemize}
\item $(0,1,0)$, if $i=0$, 
\item $(t^0, 1, 1)$, if $i=1$,
\item $(0,0,1)$, if $i=2$.
\end{itemize}
Thus the closure is 
a well-defined element of a group isomorphic to $\Zed[t^{\pm 1}] / \langle t^0 \rangle$, with generators $t^i w_{12}$. 
Using an argument similar to Theorem \ref{embthm} we can argue these maps are epimorphisms. 
\end{proof}
\end{proposition}

We call the homomorphism from Proposition \ref{tor-clo-arg1} the $W_2$-invariant, 
$$W_2 : \pi_{n-2} \Emb(I, S^1 \times B^n) \to \Zed[t^{\pm 1}] / \langle t^0 \rangle.$$
The subscript $2$ indicates the domain is the $2^{nd}$ non-trivial (rational) homotopy-group of the space
$\Emb(I, S^1 \times B^n)$.  One can go a step further than Proposition \ref{tor-clo-arg1} and argue that
$W_2$ is an isomorphism.  Roughly speaking, the idea is that given any based map $S^n \to C_2'[S^1 \times B^n]$, 
we reinterpret the function as having domain $B^{n-2} \times C_2[I]$ with the boundary collapsed.  
One then appends three copies of the standard null-homotopy 
of $ev_1(i)$ where $i$ is the constant family $i : S^{n-2} \to \Emb(I, S^1 \times B^n)$.
This gives us a map back from $Maps(S^n, C_2'[S^1 \times B^n])$
to $Maps(S^{n-2}, Maps(C_2[I], C_2'[S^1 \times B^n])$, which is an element of 
$Maps(S^{n-2}, T_2 \Emb(I, S^1 \times B^n))$ provided we began with something in the image of $W_2$, proving
$W_2$ is injective.  

We can perform the same kind of analysis for $[f] \in \pi_{2n-4} \Emb(I, S^1 \times B^n)$. These elements are
detected by $3^{rd}$-stage of the Embedding Calculus, which are maps of the form $C_3[I] \to C_3'[S^1 \times B^n]$.  
In general, when we restrict these
maps to the boundary facets of $C_3[I]$, the resulting map $S^{2n-4} \times C_2[I] \to C_2'[S^1 \times B^n]$
may not be null-homotopic.  That said, such maps are torsion.  This allows us to perform the above construction
rationally, i.e. if the order of the map $ev_2(f) : S^{2n-4} \times C_2[I] \to C_2'[S^1 \times B^n]$  is $m$, we can
perform the same analysis to construct a closure of $ev_3(mf)$, thus $\frac{1}{m}\overline{ev_3}(mf)$ would be a well-defined
rational-homotopy invariant of $[f]$, {\it provided} we mod-out by the boundary subgroups, in this case they come from the 
inclusions of the four facets of $C_3[I]$, $t_1=0$, $t_1=t_2$, $t_2=t_3$ and $t_3=1$.  To do this we need to describe
the change in homotopy-class to $\frac{1}{m}\overline{ev_3}(mf)$ due to attaching different null-homotopies 
to $ev_3(mf) : S^{2n-4} \times C_3[I] \to C_3'[S^1 \times B^n]$. 

As we have seen $\pi_{2n-1} C_2'[S^1 \times B^n]$ is isomorphic to $\pi_{2n-1}(S^1 \vee S^n) \oplus \bigoplus_2 \pi_{2n-1} S^n$. 
Modulo torsion, the generators of $\pi_{2n-1}(S^1 \vee S^n)$ are the Whitehead products of elements $t^k w_{12}$ for $k \in \Zed$.  
This gives us the result that $\pi_{2n-1} C_2'[S^1 \times B^n]$, mod torsion, is isomorphic to 
$\Zed[t_1^{\pm 1}, t_2^{\pm 1}] / \langle t_1t_2-1 = 0 \rangle$ as a module over the group-ring 
of the fundamental group. 
The generator of $\pi_{2n-1} C_2'[S^1 \times B^n]$ corresponding to a monomial $t_1^\alpha t_2^\beta$ is
$t_1^\alpha t_2^\beta w_{12}$. By attaching a homotopy-class of maps $S^{2n-4} \times I \times C_2[I] \to C_2'[S^1 \times B^n]$
to a closed-off $S^{2n-4} \times C_3[I] \to C_3'[S^1 \times B^n]$ we change the homotopy class by adding:

\begin{enumerate}
\item $[t_2^\alpha w_{23}, t_2^\beta w_{23}]$. This comes from the $t_1=0$ face. 
Thus the generator $t_1^\alpha w_{12}$ is mapped to $t_2^\alpha w_{23}$, and a Whitehead
bracket $[t_1^\alpha w_{12}, t_1^\beta w_{12}]$ is mapped to $[t_2^\alpha w_{23}, t_2^\beta w_{23}]$.
\item $[t_1^\alpha w_{12}, t_1^\beta w_{12}]$ to 
$[t_1^\alpha w_{13}+ t_2^\alpha w_{23} + a_1 w_{21}, t_1^\beta w_{13}+ t_2^\beta w_{23} + a_1 w_{21}]$. 
This comes from the $t_1=t_2$ face map, i.e. the inclusion $C_2'[S^1 \times B^n] \to C_3'[S^1 \times B^n]$
that doubles the first point, i.e. $(p_1, p_2) \longmapsto (p_1, \epsilon^+ p_1, p_2)$, where the 
perturbation $\epsilon^+ p_1$ is in the direction of the velocity vector. The integer $a_1$ is the degree of this
velocity vector map.   
This map sends $w_{12}$ to $w_{13}+w_{23} + a_1 w_{21}$,  $t_1$ to $t_1t_2$ and $t_2$ to $t_2$.  The 2nd stage
of the Taylor tower induces a null-homotopy of the velocity vector map, so we can assume $a_1=0$, but it is of interest
that the following computation gives the same answer for $a_1 \neq 0$. 
Thus it sends
$[t_1^\alpha w_{12}, t_1^\beta w_{12}]$ to $[t_1^\alpha w_{13}+ t_2^\alpha w_{23} + a_1 w_{21}, t_1^\beta w_{13}+ t_2^\beta w_{23} + a_1w_{21}]$.  Expanding this bracket using bilinearity we get
$$ = \left( -t_1^{\alpha-\beta} t_3^{-\beta} + (-1)^{n-1} t_1^{\beta-\alpha} t_3^{-\alpha} \right)[w_{12},w_{23}] + [t_1^\alpha w_{13}, t_1^\beta w_{13}] + $$
$$ a_1 \left( (-1)^n t_3^{-\beta} + (-1)^{n-1} t_3^{-\beta} + t_1^{-\alpha} - t_1^{-\alpha} \right)[w_{12},w_{23}]$$
where the latter row comes from collecting the terms involving $a_1$, and clearly these terms sum to zero. 
\item $[t_1^\alpha w_{12}+t_1^\alpha w_{13}+a_2 w_{23}, t_1^\beta w_{12}+ t_1^\beta w_{13}+a_2w_{23}]$. 
This is for the $t_2=t_3$ facet.  This corresponds to the map $C_2'[S^1 \times B^n] \to C_3'[S^1 \times B^n]$ that
doubles the second point, i.e. $(p_1,p_2) \longmapsto (p_1, p_2, \epsilon^+ p_2)$.  This map sends
$w_{12}$ to $w_{12} + w_{13} + a_2w_{23}$, $t_1$ to $t_1$ and $t_2$ to $t_2t_3$.  Thus
$[t_1^\alpha w_{12}, t_1^\beta w_{12}] \longmapsto [t_1^\alpha w_{12}+t_1^\alpha w_{13} + a_2w_{23}, t_1^\beta w_{12}+ t_1^\beta w_{13}+a_2w_{23}]$.  Like the previous case, this simplifies to
$$ = \left( -t_1^\alpha t_3^{\alpha-\beta} + (-1)^{n+1} t_1^\beta t_3^{\beta-\alpha} \right)[w_{12},w_{23}] + [t_1^\alpha w_{13}, t_1^{\beta}w_{13}] + $$
$$ a_2 \left( t_1^\beta - t_1^\beta + (-1)^n t_1^\alpha + (-1)^{n+1} t_1^\alpha \right)[w_{12},w_{23}].$$
Again, the terms with $a_2$ cancel.  
\item $[t_1^\alpha w_{12}, t_1^\beta w_{12}]$. This is for the $t_3=1$ facet.  
This corresponds to the inclusion $C_2'[S^1\times B^n] \to C_3'[S^1\times B^n]$ that
maps $(p_1,p_2)$ to $(p_1,p_2,(1,0))$, thus it sends $w_{12} \longmapsto w_{12}$ and $t_1 \longmapsto t_1$, 
$t_2 \longmapsto t_2$, thus it acts trivially on $[t_1^\alpha w_{12}, t_1^\beta w_{12}]$. 
\end{enumerate}

Thus our invariant via closure $\frac{1}{m}\overline{ev_3}(mf)$ of $\pi_{2n-4} \Emb(I, S^1 \times B^n)$ takes values in
$$\Rat \otimes \pi_{2n-1} C_3'[S^1 \times B^n] / R$$
where $R$ is the subgroup generated by the above four inclusions.  Notice (1) kills the summand corresponding to
the $w_{23}$ brackets, and (4) kills the summands corresponding to the $w_{12}$ brackets. Using relation (1) and (4)
we can simplify (2) and (3) into relations between $w_{13}$ brackets and brackets of the form $[w_{12}, w_{23}]$, giving us
the proposition below. 

\begin{proposition}({\bf Closure Argument 2}) \label{tor-clo-arg2}
Given an element of $[f] \in \pi_{2n-4} \Emb(I, S^1 \times B^n)$ such that $ev_2(f) : S^{2n-4} \to T_2 Emb(I, S^1 \times B^n)$ is null,
we form the closure of the evaluation map $ev_3(f) : S^{2n-4} \to T_3 \Emb(I, S^1 \times B^n)$ which is
a based map of the form
$$\overline{ev_3}(f) : S^{2n-1} \to C_3'[S^1 \times B^n].$$
The homotopy-class of this map, as a function of the homotopy-class $[f]$ is well-defined modulo a subgroup we call $R$.  
Using the notation of Proposition \ref{htpyck}, $R$ is generated by the torsion subgroup of
$\pi_{2n-1} C_3'[S^1 \times B^n]$ together with the elements 
$$\left(t_1^{\alpha-\beta}t_3^{-\beta} - t_1^\alpha t_3^{\alpha-\beta} + (-1)^{n-1} \left( t_1^\beta t_3^{\beta-\alpha} - t_1^{\beta-\alpha} t_3^{-\alpha} \right)\right) [w_{12},w_{23}] \ \forall \alpha,\beta \in \Zed, $$
$$ [t_2^\alpha w_{23}, t_2^\beta w_{23}] \ \forall \alpha, \beta, $$
$$ [t_1^\alpha w_{12}, t_1^\beta w_{12}] \ \forall \alpha, \beta, $$
$$ [t_1^\alpha w_{13}, t_1^\beta w_{13}] + \left(t_1^{\alpha-\beta}t_3^{-\beta} + (-1)^n t_1^{\beta-\alpha}t_3^{-\alpha}\right)[w_{12},w_{23}] \ 
\forall \alpha, \beta.$$

Since $\pi_{2n-4} T_2 \Emb(I, S^1 \times B^n)$ is torsion, there is a homomorphism, called the {\bf closure operator}
$$\pi_{2n-4} \Emb(I, S^1 \times B^n) \to \Rat[t_1^{\pm 1}, t_3^{\pm 1}] / 
\langle 
t_1^{\alpha-\beta}t_3^{-\beta} - t_1^\alpha t_3^{\alpha-\beta} = (-1)^n \left( t_1^\beta t_3^{\beta-\alpha} - t_1^{\beta-\alpha} t_3^{-\alpha} \right)\ \forall \alpha,\beta \in \Zed \rangle$$
given by mapping $f \longmapsto \frac{1}{m} \overline{ev_3}(mf)$. 
\begin{proof}
The relations are given in the comments preceding the Proposition.  Relations (1) and (4) kill
$[t_2^\alpha w_{23}, t_2^\beta w_{23}]$ and $[t_1^\alpha w_{12}, t_1^\beta w_{12}]$ respectively. 
Using Relations (1) and (4) we can simplify relations (2) and (3) to 3-term relations, both expressing
$[t_1^\alpha w_{13}, t_1^\beta w_{13}]$ in the $\Zed[t_1^\pm, t_2^\pm]$-linear span of
$[w_{12}, w_{23}]$. Comparing the two gives the relation
$$\left(t_1^{\alpha-\beta}t_3^{-\beta} - t_1^\alpha t_3^{\alpha-\beta} + (-1)^{n-1} \left( t_1^\beta t_3^{\beta-\alpha} - t_1^{\beta-\alpha} t_3^{-\alpha} \right)\right) [w_{12},w_{23}] = 0.$$ 

\end{proof}
\end{proposition}

It is important to note that in  Proposition \ref{tor-clo-arg2} we are allowing the attachment of distinct null-homotopies on all
four boundary facets of $S^{2n-4} \times C_3[I]$.   This is because the elements of the $3^{rd}$ stage of the Taylor tower restrict to
potentially different maps on the four faces.  

The map that applies the rationalized closure operator to $[f] \in \pi_{2n-4} \Emb(I, S^1 \times B^n)$ we denote $W_3$, i.e.
$$W_3([f]) = \frac{1}{m} \overline{ev_3}(mf) \in \Rat \otimes \pi_{2n-1} C_3'[S^1\times B^n]/R.$$ 
By design, if $ev_2(f) \in \pi_{2n-1} T_2 \Emb(I, S^1 \times B^n)$ is null, the invariant $W_3$ lives in the
lift
$$W_3([f]) = \overline{ev_3}(f) \in \pi_{2n-4} \Emb(I, S^1 \times B^n) / R.$$
Proposition \ref{cor:r-rels} concerns the structure of this group. 

\begin{proposition}\label{cor:r-rels}  Let $\bar W$ denote the subspace of $\pi_{2n-1}(C_3(S^1\times B^n)) $ generated by the 
Whitehead products $t_1^p t_2^q[w_{13},w_{23}], p,q\in \BZ$.  Let $W$ be the quotient of $\bar W$ by the subspace $K$ generated 
by $(t_1^p t_2^q - t_1^q t_2^{q-p} +(-1)^n(t_1^p t_2^{p-q}-t_1^q t_2^p)) [w_{13},w_{23}]$.  
Then the induced map from $W$ to $\pi_{2n-1}(C_{3}(S^1\times B^n))/R$ is an isomorphism.  The quotient $\pi_{2n-1}(C_{3}(S^1\times B^n))/R$ 
is a direct sum of a free-abelian group and a $2$-torsion group. \end{proposition}

\begin{proof}  The map $\bar W \to \pi_{2n-1} C_3(S^1 \times B^n) / R$ being an isomorphism follows from Proposition \ref{tor-clo-arg2} 
after noting that $t_1^m t_3^n[w_{12},w_{23}]= t_1^m t_3^n [w_{23},w_{31}]=(-1)^n t_1^{m-n} t_2^{-n} [w_{13},w_{23}]$.  
Consider the relator
$$\left(t_1^{\alpha-\beta}t_3^{-\beta} - t_1^\alpha t_3^{\alpha-\beta} + (-1)^{n-1} \left( t_1^\beta t_3^{\beta-\alpha} - t_1^{\beta-\alpha} t_3^{-\alpha} \right)\right) [w_{12},w_{23}] = 0.$$
If we replace the indices $(\alpha, \beta)$ by $(\alpha-\beta, -\beta)$ the above relator becomes
$$\left(t_1^{\alpha}t_3^{\beta} - t_1^{\alpha-\beta} t_3^{\alpha} + (-1)^{n-1} \left( t_1^{-\beta} t_3^{-\alpha} - t_1^{-\alpha} t_3^{\beta-\alpha} \right)\right) [w_{12},w_{23}] = 0,$$
suppressing the Whitehead bracket we rewrite the polynomial relator as
$$t_1^{\alpha}t_3^{\beta} + (-1)^{n-1} t_1^{-\beta} t_3^{-\alpha} = t_1^{\alpha-\beta} t_3^{\alpha} + (-1)^{n-1} t_1^{-\alpha} t_3^{\beta-\alpha}.$$
Notice the map $(\alpha,\beta) \longmapsto (-\beta,-\alpha)$ is an involution of $\Zed^2$, while the
map $(\alpha,\beta) \longmapsto (\alpha-\beta, \alpha)$ has order $6$, moreover if we conjugate the latter by the former we get
$(\alpha, \beta) \longmapsto (\beta, \beta-\alpha)$, which is the inverse of the mapping $(\alpha,\beta) \longmapsto (\alpha-\beta, \alpha)$ 
i.e. these two mappings generate the dihedral group of the hexagon.  One can also see this directly, by observing the orbit of $(1,0)$ under
this group action is the planar hexagon
$$\{(1,0), (1,1), (0,1), (-1,0), (-1,-1), (0, -1)\}.$$
This allows us to quickly write-out the consequences of our relator.
\begin{align}
t_1^{\alpha}t_3^{\beta} + (-1)^{n-1} t_1^{-\beta} t_3^{-\alpha} & = t_1^{\alpha-\beta} t_3^{\alpha} + (-1)^{n-1} t_1^{-\alpha} t_3^{\beta-\alpha} \\
  & = t_1^{-\beta} t_3^{\alpha-\beta} + (-1)^{n-1} t_1^{\beta-\alpha} t_3^{\beta} \\
  & = t_1^{-\alpha} t_3^{-\beta} + (-1)^{n-1} t_1^\beta t_3^\alpha \\
  & = t_1^{\beta-\alpha} t_3^{-\alpha} + (-1)^{n-1} t_1^\alpha t_3^{\alpha-\beta} \\
  & = t_1^\beta t_3^{\beta-\alpha} + (-1)^{n-1} t_1^{\alpha-\beta} t_3^{-\beta}
\end{align}
Since the map $(\alpha,\beta) \longmapsto (\alpha-\beta,\alpha)$ is of order $6$ and conjugate to a rotation, all of its orbits have six elements
with the exception of the origin.  Thus the orbits of our dihedral group can be trivial, as in the case of $(0,0)$, or they have at least six elements.
The orbits have only six elements provided any of the following equations hold $\alpha \pm \beta = 0$, $\alpha = 0$, $\beta = 0$, 
$2\alpha = \beta$ or $2\beta = \alpha$. This could be thought of as any of the vertices of the hexagon, or mid-points of the edges of the hexagon.
Observe at the origin $(\alpha,\beta)=0$ the hexagon relation is trivial, thus for the $(\alpha,\beta)=0$ orbit, the quotient group is free abelian of rank one.  

In the case of a $12$-element orbit, no generator is mentioned more than once in the relators, so the quotient is a free abelian group of rank 
$7 = 12 - 5$.  For the $7$ generators one can take the vertices of the hexagon, i.e. the orbit of $(\alpha,\beta)$ under 
$(\alpha,\beta) \longmapsto (\alpha-\beta,\alpha)$, plus one generator obtained by taking a vertex of the hexagon and applying 
$(\alpha,\beta) \longmapsto (-\beta,-\alpha)$ to it. 

In the case of a $6$-element orbit the quotient is isomorphic to either $\Zed^4$ or $\Zed^3 \oplus \Zed_2$, depending on which $6$-element
orbit one considers, and the parity of $n$.  For example, 
consider the case $\alpha=0$, then our relations are 
\begin{align}
t_1^{0}t_3^{\beta} + (-1)^{n-1} t_1^{-\beta} t_3^{0} & = t_1^{-\beta} t_3^{0} + (-1)^{n-1} t_1^{0} t_3^{\beta} \\
  & = t_1^{-\beta} t_3^{-\beta} + (-1)^{n-1} t_1^{\beta} t_3^{\beta} \\
  & = t_1^{0} t_3^{-\beta} + (-1)^{n-1} t_1^\beta t_3^0 \\
  & = t_1^{\beta} t_3^{0} + (-1)^{n-1} t_1^0 t_3^{-\beta} \\
  & = t_1^\beta t_3^{\beta} + (-1)^{n-1} t_1^{-\beta} t_3^{-\beta}.
\end{align}

Thus for $n$ odd (still in the $\alpha=0$ case) this quotient is isomorphic to $\Zed^4$, while for $n$ even, it is isomorphic to $\Zed^3 \oplus \Zed_2$. 
Similarly if we take $2\alpha = \beta$ we get quotient $\Zed^3 \oplus \Zed_2$ if n is odd, and $\Zed^4$ if $n$ is even.  Notice these two cases suffice
as our relations are invariant under the rotations of the hexagon. We could further deduce the $2\alpha=\beta$ case using the mirror reflections of
the hexagon.  In this case, the symmetry does not preserve our system of equations, it preserves them after changing the parity of $n$.
\end{proof}

\begin{remark} For $n$ odd the coefficients of defining relations in $K$ have the form \newline $t_1^p t_2^q+t_1^p t_2^{p-q}=t_1^q t_2^p +t_1^q t_2^{q-p}$.\end{remark}

\begin{figure}[H]
$$\includegraphics[width=6cm]{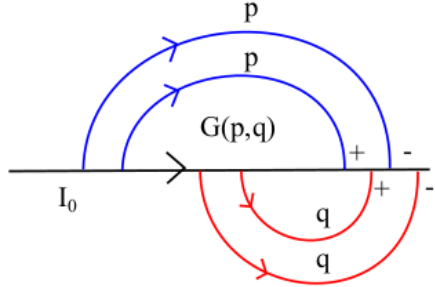}$$
\caption{\label{fig:Figcdgd} Chord diagram for $G(p,q)$}
\end{figure}

Example \ref{framed_cobord_eg} is a computation of $W_3(G(p,q))$. An alternative computation, given in greater detail appears 
in Appendix Section \ref{appendix section}. The techniques we use here are expanded upon in \cite{BG2}.
We begin by giving a careful definition of the homotopy-class $G(p,q) : S^{2n-4} \to \Emb(I, S^1 \times B^n)$. 

We interpret the chord diagram in Figure \ref{fig:Figcdgd} as defining an immersion
$I \to S^1 \times B^n$ with four regular double-point pairs that we will resolve to create families
of embeddings. Two of the double points are decorated 
in blue, and the other two are decorated in red.  The chord decorations $p$, $q$ indicate that the
`shortcut' loop $S^1 \to S^1 \times B^n$, when projected to the $S^1$ factor has degree $p$ or $q$ 
respectively.  The immersion is represented in Figure \ref{fig:Figimd} (left). 

Resolving this immersion would give us a map
$$\hat G(p,q) : \textcolor{red}{S^{n-2} \times S^{n-2}} \times \textcolor{blue}{S^{n-2} \times S^{n-2}} \to \Emb(I, S^1 \times B^n).$$
We pre-compose $\hat G(p,q)$ with the map $\Delta : S^{n-2} \times S^{n-2} \to S^{n-2} \times S^{n-2} \times S^{n-2} \times S^{n-2}$
given by $\Delta(v,w) = (M(v), v, M(w), w)$ where $M : S^{n-2} \to S^{n-2}$ is a map with $deg(M) = -1$.   The choice of degree is governed
by the signs in our chord diagram.  The composite
$\hat G(p,q) \circ \Delta$, when restricted to $S^{n-2} \vee S^{n-2}$ is null, giving us, after a small homotopy
of $\hat G(p,q) \circ \Delta$, a commutative diagram

$$ \xymatrix{ S^{n-2} \times S^{n-2} \ar[dr] \ar[rr]^-{\hat G(p,q) \circ \Delta} && \Emb(I, S^1 \times B^n) \\
 & S^{n-2} \times S^{n-2} / S^{n-2} \vee S^{n-2} \equiv S^{2n-4} \ar[ur]^-{G(p,q)} }. $$

We proceed computing the homotopy-class of $\frac{1}{m}\overline{ev_3}(m \cdot G(p,q))$ in steps. 
We will see that $ev_2(G(p,q))$ is null-homotopic, thus $m=1$.  In general, one can prove an analogous theorem to
Proposition \ref{tor-clo-arg1}, computing $\pi_{2n-4} T_2 \Emb(I, S^1 \times B^n)$ precisely. The
exponent of this group can be shown to be equal to $|\pi_{2n-2} S^n|$.  The group $\pi_{2n-2} S^n$ is one of the 
stable homotopy groups of spheres, and is known to be finite.  When needing to compute the order of $ev_2(f)$ precisely, 
one constructs the closure as in 
Proposition \ref{tor-clo-arg1}, giving the map $\overline{ev_2}(f) : S^{2n-2} \to C_2'[S^1 \times B^n]$.  
The homotopy group $\pi_{2n-2} C_2'[S^1 \times B^n]$ can be described in terms of homotopy-groups of spheres 
via the homotopy-equivalence
$C_2'[S^1 \times B^n] \simeq S^1 \times (S^1 \vee S^n) \times (S^n)^2$ giving 
$$\pi_{2n-2} C_2'[S^1 \times B^n] \simeq \pi_{2n-2} S^n[t^{\pm 1}] \oplus \bigoplus_2 \pi_{2n-2} S^n.$$

\begin{figure}[H]
{
\psfrag{p}[tl][tl][0.7][0]{$p$}
\psfrag{q}[tl][tl][0.7][0]{$q$}
\psfrag{+}[tl][tl][0.5][0]{$+$}
\psfrag{-}[tl][tl][0.5][0]{$-$}
$$\includegraphics[width=6cm]{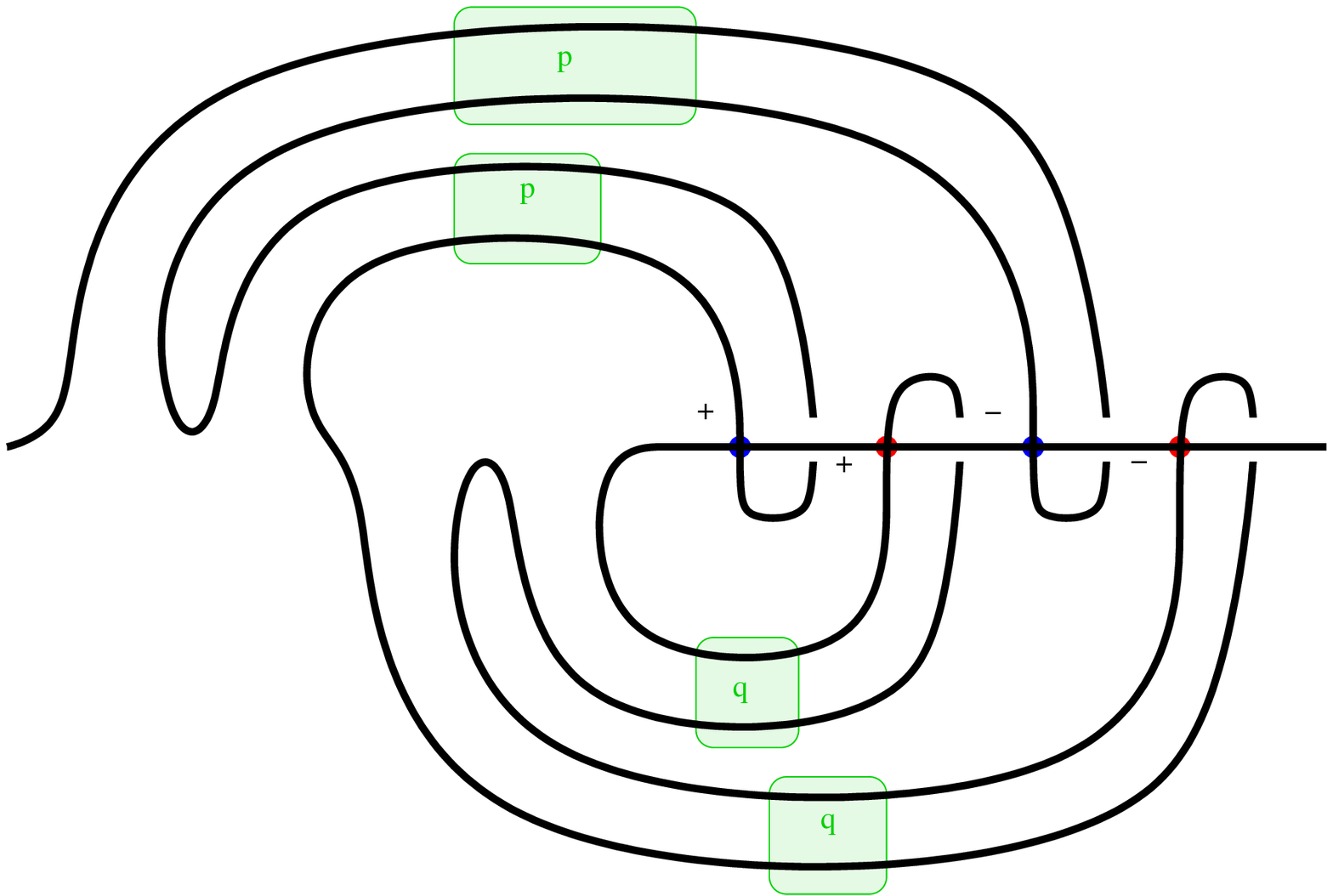} \hskip 1cm \includegraphics[width=6cm]{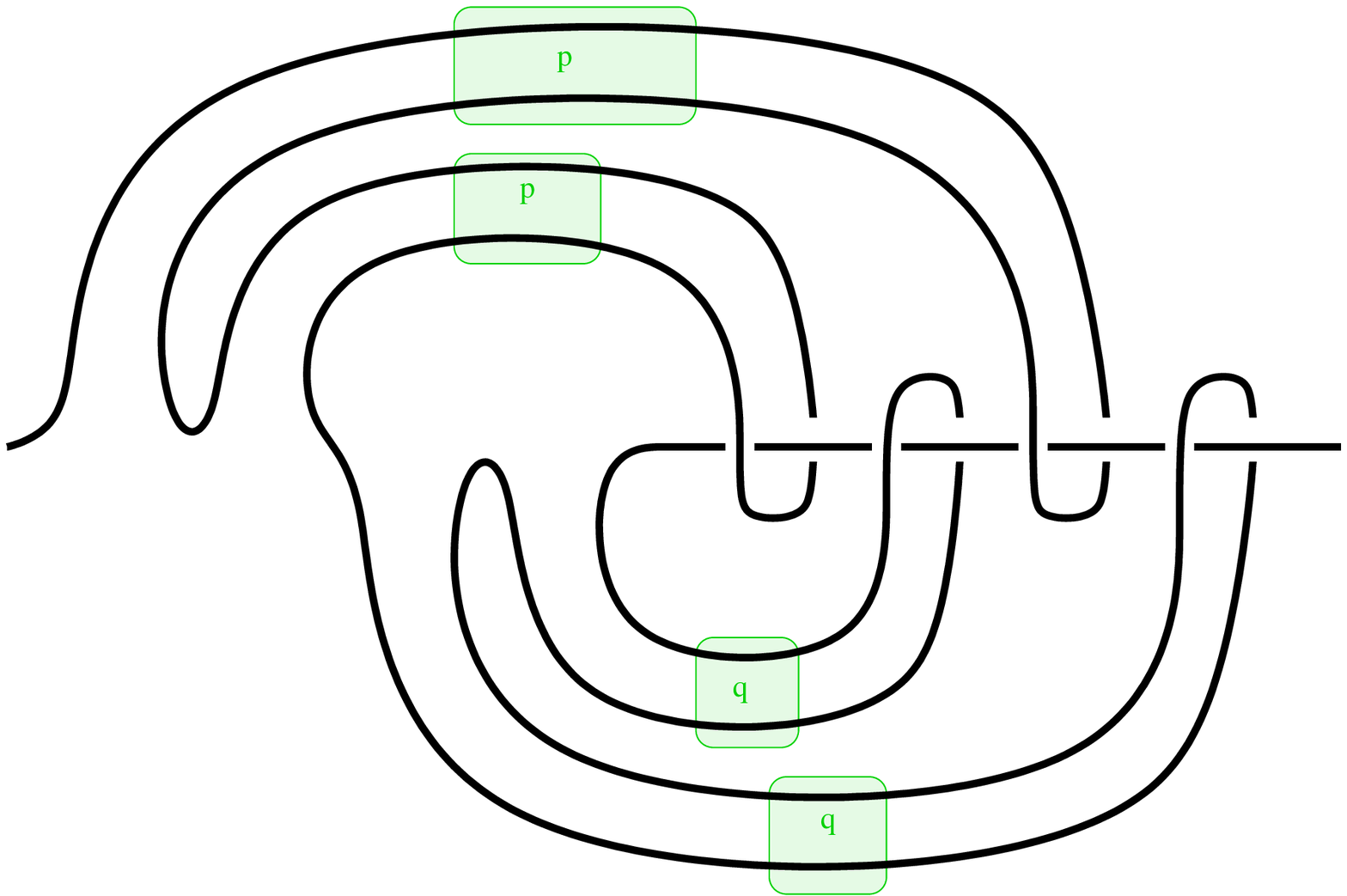}$$
}
\caption{\label{fig:Figimd} Immersion for $G(p,q)$ (left) and one resolution (right)}
\end{figure}

The homotopy-class of a map $S^{2n-2} \to (S^1 \vee S^n) \times (S^n)^2$ 
is determined by the homotopy classes of the projections: (a) $S^{2n-2} \to S^1  \vee S^n$ and (b) (two maps) $S^{2n-2} \to S^n$.  
By the Pontriagin construction, the latter two maps are determined by framed cobordism classes of $(n-2)$-manifolds in 
$\Real^{2n-2}$, taking the pre-image of any point that is not the base-point of the sphere. 
The projection $S^{2n-2} \to S^1 \vee S^n$ is determined by a framed cobordism class of a countable collection of {\it disjoint} 
$(n-2)$-manifolds in $\Real^{2n-2}$. A simple way to construct these manifolds is to take the cohorizontal manifolds, i.e. fix a unit 
direction $\zeta \in B^n$. Define $t^i Co_1^2$ consisting of pairs of points $(p_1,p_2) \in C_2[\Real^1 \times B^n]$ such 
that the displacement vector $t^i.p_2 - p_1$ is a positive multiple of $\zeta$.  Then given $S^{2n-2} \to C_2'[S^1 \times B^n]$
we lift the map to the universal cover of $C_2'[S^1 \times B^n]$, interpreted as the submanifold of $C_2'[\Real^1 \times B^n]$ such that
the points have disjoint $\Zed$-orbits. We take the pre-image of $t^i Co_1^2$.  This manifold family (as a function of $i$) is precisely
what we need to detect the Laurent polynomial associated to the homotopy-class of the projection $S^{2n-4} \to S^1 \vee S^n$. 

\begin{figure}[H]
{
\psfrag{a}[tl][tl][0.7][0]{$\alpha$}
\psfrag{b}[tl][tl][0.7][0]{$\beta$}
\psfrag{tbwkj}[tl][tl][0.7][0]{$t_2^\beta.w_{32}$}
\psfrag{tawij}[tl][tl][0.7][0]{$t_2^\alpha.w_{12}$}
\psfrag{ta}[tl][tl][0.7][0]{}
\psfrag{tb}[tl][tl][0.7][0]{}
\psfrag{pi}[tl][tl][0.7][0]{$p_1$}
\psfrag{pj}[tl][tl][0.7][0]{$p_2$}
\psfrag{pk}[tl][tl][0.7][0]{$p_3$}
\psfrag{0}[tl][tl][0.7][0]{$0$}
\psfrag{R}[tl][tl][0.7][0]{$\Real$}
$$\includegraphics[width=6cm]{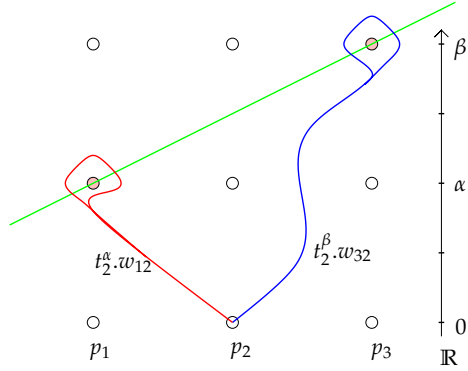}$$
}
\caption{\label{fig:colfig} Collinear manifolds intersecting $\pi_n C_k[S^1 \times B^n]$ generators.}
\end{figure}

Rationally the class $\overline{ev_3}(G(p,q)) \in \pi_{2n-1} C_3'[S^1 \times B^n]$ is a linear 
combination of Whitehead products, by Proposition \ref{htpyck}. 
To determine the linear combination, we use an idea from \cite{BCSS}, where it was shown that certain collinear (trisecant) manifolds
can detect Whitehead products in $\pi_{2n-1} C_3(\Real^{n+1})$.   The key property of the collinear manifolds is they (a) they suffice to
detect the generators of $\pi_{n} C_3[\Real^{n+1}]$, and (b) there is more than one path-component to these manifolds, allowing
one to go further and detect Whitehead products $[w_{12}, w_{23}]$.  This allowed the authors in \cite{BCSS} to
express the type-2 Vassiliev invariant of knots as a linking number of trisecant manifolds, and further as  a count of quadrisecants.  
Consider $Col^1_{\alpha,\beta}$ to be the submanifold of
$C_3[\Real \times B^n]$ where the points $(p_2, t^\alpha.p_1,t^\beta.p_3)$ sit on a straight line in $\Real \times B^n$
in that linear order.  Similarly, define $Col^3_{\alpha,\beta}$ to be the submanifold of $C_3[\Real \times B^n]$ such that
$(t^\alpha.p_1, t^\beta.p_3, p_2)$ sit on a straight line, in that linear order. These two manifolds are disjoint, moreover the former manifold detects
the homotopy-class $t_2^\alpha w_{12}$, and the latter detects $t_2^\beta w_{32}$, thus the preimage of the 
disjoint pair $(Col^1_{\alpha,\beta}, Col^3_{\alpha, \beta})$ by the map $[t_2^\alpha w_{12}, t_2^\beta w_{23}] : S^{2n-1} \to C_3[S^1 \times B^n]$ 
is a $2$-component Hopf link in $S^{2n-1}$. This is a non-trivial framed cobordism class of disjoint manifold pairs, as
the linking number of the two components is $\pm 1$.  Moreover, this linking number is zero if we use the map $[t_2^p w_{12}, t_2^q w_{23}]$ provided
$(p,q) \neq (\alpha,\beta)$. 

The invariant $W_3$ is therefore computable via the linking numbers of the pre-images of the collinear manifolds, i.e. 
one computes the linking numbers of the pre-images of $Col^1_{\alpha,\beta}$ and $Col^3_{\alpha,\beta}$ via the lift of
$\overline{ev_3}(G(p,q)) : S^{2n-1} \to C_3'[S^1 \times B^n]$ to the universal cover of $C_3'[S^1\times B^n]$.  
This is the coefficient of $t_1^\alpha t_3^\beta [w_{12}, w_{23}]$ in $W_3(G(p,q))$. 

Unfortunately the above is a relatively delicate visualisation task. See \cite{BCSS} for examples of how one can directly compute
linking numbers of trisecant manifolds.  We use a variation of an argument of Misha Polyak \cite{MP}.  Polyak gave a direct 
argument showing that the quadrisecant formula of \cite{BCSS}, itself a linking number of trisecant manifolds, can be turned into a 
Polyak-Viro formula,  i.e. a count involving only cohorizontal manifolds.  Interestingly, Polyak's argument is done using maps out
of $4$-point configuration spaces, while ours uses submanifolds of $3$-point configuration spaces.   

There is cobordism of the manifold pair $(Col^1_{\alpha,\beta}, Col^3_{\alpha,\beta})$
and $(t^\alpha Co_2^1 - t^{\alpha -\beta}Co_3^1, t^{\beta-\alpha} Co_1^3 - t^\beta Co_2^3)$.  
A cobordism between these two families is given by the parabolic spline family.  Specifically, fix a direction
vector $\zeta \in \Real^{n+1}$.  Let $V_\zeta$ be the orthogonal compliment to $\zeta$ in $\Real^{n+1}$.  Let
$L \subset V_\zeta$ be a line, and $Q : L \to \Real.\zeta$ be a quadratic function whose second derivative with
respect to arc-length on $L$ is given by $\epsilon$. Thus when $\epsilon=0$ the graph of the function $Q$ can be any
line in $\Real^{n+1}$ except those parallel to the $\zeta$ direction.  Thus as a family parametrized by 
$\epsilon \in [0, \infty]$ we have a cobordism between  $(Col^1_{\alpha,\beta}, Col^3_{\alpha,\beta})$ and 
$(t^\alpha Co_2^1 - t^{\alpha -\beta}Co_3^1, t^{\beta-\alpha} Co_1^3 - t^\beta Co_2^3)$. 

There is an important technical issue here as the family is not disjoint when $\epsilon = \infty$, as it allows
for triple points in the $\zeta$ direction.  This does not pose a problem for us, since generically we can assume in our family 
$S^{2n-4} \times C_3[I] \to C_3[S^1 \times B^n]$
the direction vectors of collinear triples form a closed co-dimension $1$ subset of $S^n$, i.e. the compliment is open and dense, thus generically we
can choose $\zeta$ to be disjoint from this set.  So we can compute the coefficient of $t_1^\alpha t_3^\beta [w_{12}, w_{23}]$ in 
$W_3(G(p,q))$ as the linking number of the pre-image of the above pairs, for $\overline{ev_3}(G(p,q))$.  

Our strategy then, given $f: S^{2n-4} \to \Emb(I, S^1 \times B^n)$ is to first construct the framed cobordism class representing 
$ev_2(f) : S^{2n-4} \times C_2[I] \to C_2'[S^1 \times B^n]$
in the language of cohorizontal manifolds.  This allows us to explicitly construct a null-cobordism corresponding to a null-homotopy of 
$ev_2(f)$.  We then move on to study $ev_3(f) : S^{2n-4} \times C_3[I] \to C_3'[S^1 \times B^n]$.  The null-cobordisms associated to
$ev_2(f)$ can now be attached to the boundary of $ev_3(f)^{-1}(t^k Co_i^j)$, giving us a framed cobordism representation of $\overline{ev_3}(f)$, 
where $f=G(p,q)$. 

\begin{example}\label{framed_cobord_eg}
$W_3(G(p,q)) = t_1^p t_2^q[w_{23},w_{31}] = t_1^{p-q} t_3^{-q} [w_{12},w_{23}].$

\vskip 10pt
By careful choice of the immersion defining $G(p,q)$ we can arrange that there is a unique parameter value in $S^{2n-4}$ where the 
embedding lives in $\Real^3 \times \{0\}$ and the associated planar diagram has eight regular double-points.  Moreover, we can ensure
the locus of parameters in $S^{2n-4}$ such that the associated embedding has double points is the wedge of two embedded copies of $B^{n-2}$ in $S^{2n-4}$, 
parallel to the
coordinate axis, i.e. considering $S^{2n-4}$ to be $B^{n-2} \times B^{n-2}$ with the boundary collapsed.  The resolution with the eight
regular double points is depicted in Figure \ref{fig:Figimd}.  Four of the double points persist along the red coordinate axis (i.e. a copy of $B^{n-2}$) 
and the remaining four persist along the blue coordinate axis.  Considering the evaluation map $ev_2 : S^{2n-4} \times C_2[I] \to C_2[S^1 \times B^n]$, 
the submanifold of $S^{2n-4} \times C_2[I]$ mapping to cohorizontal points is depicted in Figure \ref{fig:Figslices}.  This manifold is
the disjoint union of four embedded spheres, diffeomorphic to $\sqcup_4 S^{n-2}$. There are two such spheres along the blue coordinate axis, 
consisting of the sphere where the $2^{nd}$ point is above the $1^{st}$, and the sphere where the reverse is true; the $1^{st}$ point is above
the $2^{nd}$. Similarly there are two such spheres corresponding to the red coordinate axis.  These sphere are essentially linearly-embedded
in $S^{2n-4} \times C_2[I]$, having  disjoint convex hulls, i.e. they are unlinked.  Moreover, these spheres bound four disjointly-embedded balls, 
$\sqcup_4 B^{n-1} \to S^{2n-4} \times C_2[I]$.  

\begin{figure}[H]
{
\psfrag{1}[tl][tl][0.5][0]{$1$}
\psfrag{4}[tl][tl][0.5][0]{$4$}
\psfrag{8}[tl][tl][0.5][0]{$8$}
\psfrag{12}[tl][tl][0.5][0]{$12$}
\psfrag{16}[tl][tl][0.5][0]{$16$}
\psfrag{c2i}[tl][tl][0.7][0]{$C_2[I]$}
\psfrag{RBn}[tl][tl][0.7][0]{$\textcolor{red}{B^{n-2}}$}
\psfrag{BBn}[tl][tl][0.7][0]{$\textcolor{blue}{B^{n-2}}$}
\psfrag{r12}[tl][tl][0.7][0]{$t^{-q}Co_1^2$}
\psfrag{r21}[tl][tl][0.7][0]{$t^{q}Co_2^1$}
\psfrag{b12}[tl][tl][0.7][0]{$t^{-p}Co_1^2$}
\psfrag{b21}[tl][tl][0.7][0]{$t^{p}Co_2^1$}
$$\includegraphics[width=10cm]{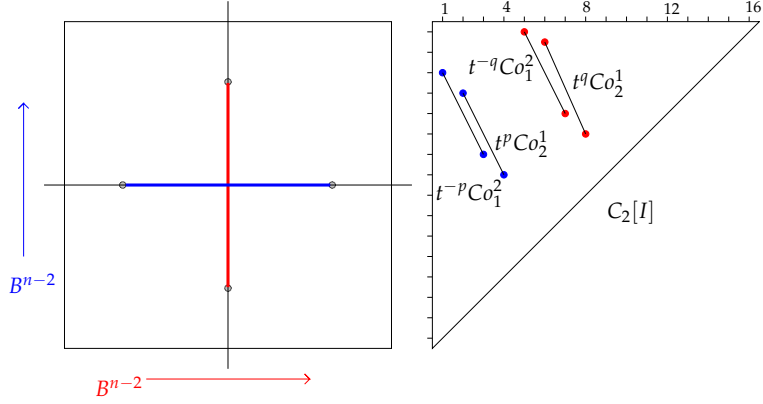}$$
}
\caption{\label{fig:Figslices} Preimages of cohorizontal manifolds $ev_2(G(p,q)) : S^{2n-4} \times C_2[I] \to C_2[S^1\times B^n]$}
\end{figure}

Our `units' for $C_2[I]$ in Figure \ref{fig:Figslices} are the indices $\{1,2,\cdots, 16\}$ for the cohorizontal points, along the parametrization of the 
embedded interval, suitably rescaled.  The four spheres in the preimage are 
trivially framed, thus the disjoint balls they bound gives a null-cobordism.   Thus Figure \ref{fig:Figslices} depicts the framed cobordism classes of
$\widetilde{ev_2}(f)^{-1}(t^iCo_1^2)$ and $\widetilde{ev_2}(f)^{-1}(t^iCo_2^1)$ for all $i$.  The bounding discs can be thought
to be depicted in the black arcs connecting the red and blue cohorizontal points, but these also trace out the cohorizontal points
through the end homotopies.  The bounding discs for the $Co_1^2$ manifolds determine null-homotopies of the maps 
$S^{2n-4} \times C_2[I] \to C_2[S^1 \times B^n]$, while we assert that the corresponding bounding discs for the $Co_2^1$ manifolds determined
by the null-homotopy are as depicted. 

In Figure \ref{fig:Figgpqlk} we depict projections of the manifolds $\overline{ev_3}(G(p,q))^{-1}(t^l Co_i^j) \subset S^{2n-1}$. 
The domain of $\overline{ev_3}(G(p,q))$ is a copy of $S^{2n-1}$ but we think of this sphere as a ball with its boundary collapsed to a point.  
The `ball' being $B^{2n-4} \times C_3[I]$ with four triangular cylinders $B^{2n-4} \times C_2[I] \times I$ attached, due to the null-homotopies.  
The projection in Figure \ref{fig:Figgpqlk} is to the $C_3[I]$ (union triangular cylinders) factor.  
We use $0 \leq t_1 \leq t_2 \leq t_3 \leq 1$ as our coordinates for $C_3[I]$. In the figure 
$t_1$ and $t_3$ are the planar coordinates, with $t_2$ pointing out of the page. One therefore obtains Figure \ref{fig:Figgpqlk} from
Figure \ref{fig:Figslices} by considering how the cohorizontal points from Figure \ref{fig:Figslices} induce cohorizontal points
on the boundary of Figure \ref{fig:Figgpqlk}. One then fills in the interior arcs: for example if there is a $Co_1^2$ point on the
$t_2=t_3$ boundary facet of $C_3[I]$, then there will be an entire straight line parallel to the $t_3$-axis.  Thus the $Co_1^3$ interior
arcs will be orthogonal to the page in this projection.  One then appends the null-cobordisms to the boundary facets of $C_3[I]$. 

For example, the manifold labelled $t^p Co_3^1$ consists  of three parts in the figure.  There are two arcs parallel
to the coordinate $t_2$-axis, these are represented by the short arcs between the nearby pairs of blue points.  In 
$B^{2n-4} \times C_3[I]$ this represents two disjoint $(n-1)$-balls.  There are also two longer diagonal arcs labelled $t^p Co_3^1$, 
one overcrossing $t^{-q}Co_2^3$ and one undercrossing.  The overcrossing indicates the null-cobordism coming from the attachment on 
the $t_2=t_3$ face of $B^{2n-4} \times C_3[I]$, while the undercrossing represents the null-cobordism attached on the $t_1=t_2$ face. 

Pairwise these linking numbers are all zero, with the sole exception of the pair $(Co_3^1, Co_2^3)$, which gives $t_1^{p-q} t_3^{-q}[w_{12},w_{23}]$. 
We have suppressed the diagrams for the linking numbers of the preimages of the pairs $(Co_3^1, Co_1^3)$, $(Co_2^1, Co_1^3)$, and 
$(Co_2^1, Co_2^3)$, as their computations are analogous.

\begin{figure}[H]
{
\psfrag{c2i}[tl][tl][0.7][0]{$C_3[I]$}
\psfrag{RBn}[tl][tl][0.7][0]{$\textcolor{red}{B^{n-2}}$}
\psfrag{BBn}[tl][tl][0.7][0]{$\textcolor{blue}{B^{n-2}}$}
\psfrag{q12}[tl][tl][0.7][0]{$t^{-q} Co_2^3$}
\psfrag{mp21}[tl][tl][0.7][0]{$t^{p}Co_3^1$}
\psfrag{p12}[tl][tl][0.7][0]{$t^{-p} Co_2^3$}
\psfrag{mq21}[tl][tl][0.7][0]{$t^{q} Co_3^1$}
\psfrag{t1}[tl][tl][0.7][0]{$t_1$}
\psfrag{t3}[tl][tl][0.7][0]{$t_3$}
\psfrag{c23c31}[tl][tl][1][0]{$lk(\textcolor{blue}{Co_3^1}, \textcolor{red}{Co_2^3})$}
\psfrag{c31c23}[tl][tl][1][0]{$lk(\textcolor{red}{Co_3^1}, \textcolor{blue}{Co_2^3})$}
$$\includegraphics[width=12cm]{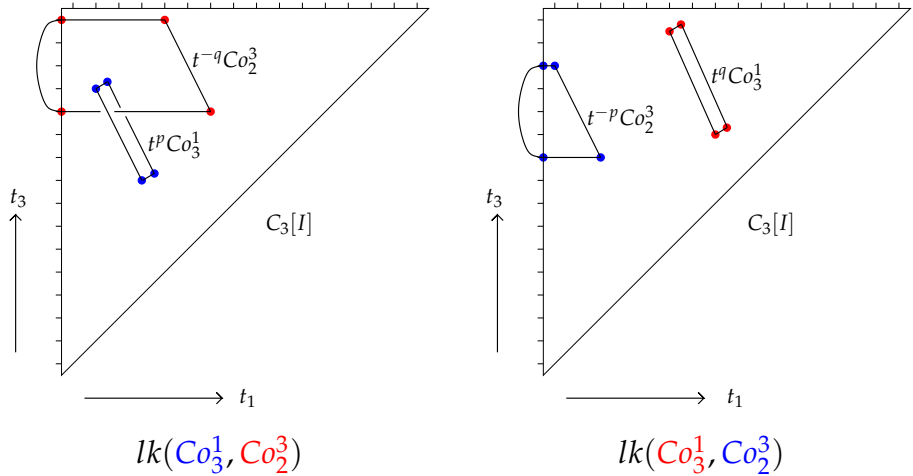}$$
}
\caption{\label{fig:Figgpqlk} Linking cohorizontal manifold preimages for $\overline{ev_3}(G(p,q)) : S^{2n-1} \to C_3'[S^1\times B^n]$}
\end{figure}
\end{example}

\begin{proposition}\label{rankdiff}
The homotopy group $\pi_{n-3} \Diff(S^1 \times B^n \text{ fix } \partial)$ is abelian, for all $n \geq 3$.  
\begin{proof}
A diffeomorphism of $S^1 \times B^n$ can be isotoped canonically so that its support is contained in $S^1 \times \epsilon B^n$ where
$1 \geq \epsilon > 0$, i.e. we can effectively rescale the diffeomorphism radially in the $B^n$ factor. Using conjugation by a translation
in the $B^n$ factor, one can show that up to isotopy, diffeomorphisms of $S^1 \times B^n$ can be assumed to have support in $S^1 \times U$
where $U$ is any open subset of $B^n$. 
\end{proof}
\end{proposition}

Let $\Diff_0(S^1 \times S^n)$ denote the subgroup of $\Diff(S^1\times S^n)$ of elements homotopic to the identity.  This is 
a subgroup of index $8$ in $\Diff(S^1 \times S^n)$ and index two in the subgroup acting trivially on
$H_*(S^1 \times S^n)$.

\begin{theorem} \label{diff abelian} Any two elements of $\Diff_0(S^1\times S^n)$, the group of diffeomorphisms 
homotopic to the identity, commute up to isotopy.
\end{theorem}

\begin{proof} Any element of $\Diff_0(S^1\times S^n)$ is isotopic to one that fixes a neighborhood of
 $S^1\times \{y_0\}$ pointwise.  Thus commutativity follows as in the proof of the first part of 
 Proposition \ref{rankdiff}.\end{proof}

Let $\Emb(HB^i, B^n)$ denote the space of smooth embeddings
$HB^i \to B^n$ that restricts to the standard inclusion $x \longmapsto (x,0)$ on the 
round boundary $\partial^r HB^i = HB^i \cap \partial B^i$. Denote the corresponding framed embedding space by 
$\Emb^{fr}(HB^i, B^n)$.  This space consists of pairs $(f,\nu)$ where $f \in \Emb(HB^i, B^n)$ and 
$\nu$ is a trivialization of the normal bundle to $f$ that restricts to the canonical 
trivialization on $HB^i \cap \partial B^i$.   Both $\Emb(HB^i,B^n)$ and $\Emb^{fr}(HB^i,B^n)$ are
contractible spaces, the proofs are analogous to the homotopy classification of collar neighbourhoods.  
The role these embedding spaces play is as the total spaces of fiber bundles. The half-ball is
defined in Definition \ref{half-ball-def}. 

The first fiber bundle to consider is $\Emb(HB^i, B^n) \to \Emb_u(B^{i-1}, B^n)$ where $u$ 
denotes the unknot component of $\Emb(B^{i-1},B^n)$, i.e. the path-component of the linear
embedding. This bundle is in principle useful, but
the fiber is embeddings of $HB^i$ into $B^n$ which are fixed on their boundary, which is 
not a very familiar space.  By taking the derivative $\frac{\partial}{\partial x_1}$ along the boundary
$B^{i-1}$ we see this space fibers over $\Omega^{i-1} S^{n-i}$ with fiber homotopy-equivalent
to $\Emb(B^i, S^{n-i} \times B^i)$.  This latter space is the space of embeddings $B^i \to S^{n-i} \times B^i$
which restricts to the inclusion $p \longmapsto (*, p)$ on $\partial B^i$, where $* \in S^{n-i}$ is some preferred point.
These fiber bundles were used in an analogous way by Cerf \cite{Ce2} in his appendix, Propositions 5 and 6. 

Alternatively we can form the bundle
$$\Emb(B^i, S^{n-i} \times B^i) \to \Emb(HB^i, B^n) \to \Emb_u^+(B^{i-1}, B^n)$$
where the base space consists of embeddings $B^{i-1} \to B^n$ together with a normal vector
field along the embedding i.e. $\frac{\partial}{\partial x_1}.$  The fiber of this bundle
is technically the embeddings of $HB^i$ into $B^n$ which agree with the standard inclusion (and its
derivative) along $\partial HB^i$.  This fiber has the same homotopy type as $\Emb(B^i, S^{n-i} \times B^i)$.
 
Similarly, we have the corresponding bundles for the framed embedding spaces

$$\Emb^{fr}(B^{i}, S^{n-i} \times B^{i}) \to \Emb^{fr}(HB^i, B^n) \to \Emb_u^{fr}(B^{i-1}, B^n).$$

Given that the total space is contractible, this allows us to describe the unknot component
of these embedding spaces as classifying spaces.  

\begin{lemma}\label{half-discs}
$$B\Emb^{fr}(B^i, S^{n-i} \times B^i) \simeq \Emb^{fr}_u(B^{i-1}, B^n)$$
$$B\Emb(B^i, S^{n-i} \times B^i) \simeq \Emb^+_u(B^{i-1}, B^n)$$
\end{lemma}

We take a moment to unpack some of the underlying geometric ideas involved in the lemma. 

There is an isomorphism of homotopy groups
$$\pi_k \Emb(B^i, S^{n-i} \times B^i) \to \pi_{k+1} \Emb_u^+(B^{i-1}, B^n)$$
moreover, this map has an explicit geometric description.  To do this, we need the 
exact fiber of the bundle $\Emb(HB^i, B^n) \to \Emb_u^+(B^{i-1}, B^n)$.  This is the space
of embeddings of $HB^i$ into $B^n$ which agrees with the standard inclusion $p \longmapsto (p,0)$
and its derivative on the full boundary of $HB^i$. Denote this space by 
$\Emb_\partial(HB^i, B^n)$.  Serre's homotopy-fiber construction tells us that $\Emb_\partial(HB^i, B^n)$
is homotopy-equivalent to
$$HF= \{ \alpha : [0,1] \to \Emb(HB^i, B^n) \text{ s.t. } \alpha(0) = *, \alpha(1) \in \Emb_\partial(HB^i, B^n) \}.$$
In the above, $*$ denotes the basepoint of $\Emb(HB^i, B^n)$, i.e. the standard inclusion $p \longmapsto (p,0)$. 
The homotopy-equivalence between $HF \to \Emb_\partial(HB^i, B^n)$ is given by associating $\alpha(1)$ to $\alpha$. 
The homotopy-equivalence between $HF$ and $\Omega \Emb^+_u(B^{i-1}, B^n)$ is given by associating $\overline{\alpha}$ to
$\alpha$ where $\overline{\alpha}(t) = \alpha(t)_{|B^{i-1}}$. 

For the sake of those not familiar with these methods we describe the homotopy-equivalence
directly. For this we need to adjust our model slightly. We replace
the space $\Emb(HB^i, B^n)$ with the homotopy-equivalent space of embeddings $H^i \to \Real^n$ where
the support is constrained to be in $HB^i$, i.e. the maps are the standard inclusion $p \longmapsto (p,0)$
outside of $HB^i$.  Similarly, $\Emb_u^+(B^{i-1}, B^n)$ would be the space of embeddings $\Real^{i-1} \to \Real^n$
with a normal unit vector field, the embeddings and normal vector required to be standard outside of $B^{i-1}$. 
From this perspective, the fiber $\Emb(HB^i, B^n) \to \Emb_u^+(B^{i-1}, B^n)$ is the space of embeddings 
$H^i \to \Real^n$ where the support is not only contained in $HB^i$, but the embedding and its derivative
in the normal direction is required to be standard on $\partial H^i$. Observe the explicit deformation-retraction of
$\Emb(HB^i, B^n)$ to a point, given by associating to $f \in \Emb(HB^i, B^n)$ the path (as a function of $t$), 
$f_t \in \Emb(HB^i, B^n)$ where 
$$f_t(x_1, x_2, \cdots, x_k) = f(x_1-t, x_2, \cdots, x_k) + (t,0,\cdots,0).$$
Thus the map $\Emb_\partial(HB^i, B^n) \to \Omega \Emb^+_u(B^{i-1}, B^n)$ in this model is the one
that associates to $f \in\Emb_\partial(HB^i, B^n)$ the path $f_t \in \Emb^+_u(B^{i-1}, B^n)$ given by
$$f_t(x_2, \cdots, x_k) = f(1-t, x_2, \cdots, x_k)+(t,0,\cdots,0).$$
The vector field being $\frac{\partial f}{\partial x_1}(1-t, x_2, \cdots, x_k)$.   So it would be
reasonable to call this map {\it slicing the embedding}. 

We mention a few elementary consequences of Lemma \ref{half-discs}. 

$$SO_n \simeq \Emb^{fr}_u(B^0, B^n) \simeq B \Emb^{fr}(B^1, S^{n-1} \times B^1).$$

For embeddings with $1$-dimensional domains we have
$$\Emb^+_u(B^1, B^n) \simeq B \Emb(B^2, S^{n-2} \times B^2).$$

The space $\Emb^+_u(B^1, B^n)$ is a bundle over $\Emb_u(B^1, B^n)$, and this space is equal to
$\Emb(B^1, B^n)$ when $n \geq 4$.  The fiber of this bundle is $\Omega S^{n-2}$, provided 
$n \geq 4$.  This bundle is known to be trivial.  One trivialization can be expressed
as a splitting at the fiber $\Emb^+(B^1, B^n) \to \Omega S^{n-2}$.  To construct it, use
the null-homotopy of the Smale-Hirsch map \cite{cubes} $\Emb^+(B^1, B^n) \to \Omega S^{n-1}$. 
Given that the normal vector field is orthogonal to the Smale-Hirsch map, one can use
the holonomy on $S^{n-1}$ to homotope the normal vector field canonically to a map orthogonal
to the $x_1$-axis, giving the map $\Emb^+(B^1, B^n) \to \Omega S^{n-2}$, and the splitting

$$\Emb^+(B^1, B^n) \simeq \Emb(B^1, B^n) \times \Omega S^{n-2}.$$

Substituting $i=n$ into Lemma \ref{half-discs} we get the theorem of Cerf \cite{Ce1}
\begin{theorem}
$$\Emb_u(B^{n-1}, B^n) \simeq B\Diff(B^n \text{ fix } \partial).$$
\end{theorem}

Raising the codimension by one, and provided $n \geq 2$ we get the identification
$$\Emb_u^+(B^{n-2}, B^n) \simeq B\Emb(B^{n-1}, S^1 \times B^{n-1}).$$

When $n \geq 2$, $\pi_0 \Emb(B^{n-1}, S^1 \times B^{n-1})$ is a monoid under concatenation, 
isomorphic as a monoid to $\pi_1 \Emb_u^+(B^{n-2}, B^n)$, therefore with inverses.
When $n \geq 3$, $\pi_0 \Emb(B^{n-1}, S^1 \times B^{n-1})$ is a commutative monoid
under concatenation, therefore an abelian group, isomorphic to $\pi_1 \Emb_u^+(B^{n-2}, B^n)$.

The space $\Emb(B^{n-1}, S^1 \times B^{n-1})$ has a concatenation operation, 
which could also be thought of as an action of the operad of $(n-1)$-discs. 
The space $\Emb^+_u(B^{n-2}, B^n)$ similarly has a concatenation
operation with one less degree of freedom.  It can be encoded as an action
of the operad of $(n-2)$-discs.  These discs actions turn the two sets
$\pi_0 \Emb(B^{n-1}, S^1 \times B^{n-1})$ and $\pi_1 \Emb_u^+(B^{n-2}, B^n)$ into
commutative monoids with the concatenation operation, provided $n \geq 3$.  Moreover, one can see that
the concatenation operation and concatenation of loops are the same operation on 
$\pi_1 \Emb_u^+(B^{n-2}, B^n)$.  Thus, the isomorphism
$\pi_0 \Emb(B^{n-1}, S^1 \times B^{n-1}) \simeq \pi_1 \Emb_u^+(B^{n-2}, B^n)$
is an isomorphism of groups, which must be abelian.  Lastly, notice that
$\Emb_u^+(B^{n-2}, B^n)$ fibers over $\Emb_u(B^{n-2}, B^n)$ with fiber $\Omega^{n-2} S^1$, 
provided $n \geq 4$.  So for all $n$ we have a homotopy-equivalence

\begin{corollary}\label{cd2-cor}
$$\Emb_u(B^{n-2}, B^n) \simeq B\Emb(B^{n-1}, S^1 \times B^{n-1}).$$
\end{corollary}

Since the concatenation operation on $\Emb(B^{n-1}, S^1 \times B^{n-1})$ 
turns $\pi_0 \Emb(B^{n-1}, S^1 \times B^{n-1})$ into a group, it makes sense to consider
the fiber bundle
$$\Diff(B^{n+1} \text{ fix } \partial) \to \Diff(S^1 \times B^n \text{ fix } \partial) \to \Emb(B^n, S^1 \times B^n).$$
Every embedding $B^n \to S^1 \times B^n$ that is standard $p \longmapsto (1, p)$ on
$\partial B^n$ is the fiber of some trivial smooth fiber bundle $S^1 \times B^n \to S^1$
with fiber $B^n$, by Lemma \ref{half-discs} and isotopy extension.  Thus $\Diff(S^1 \times B^n \text{ fix } \partial)$
acts transitively on $\Emb(B^n, S^1 \times B^n)$.  We record the observation.

\begin{theorem}\label{diff-fibr}
The group $\Diff(S^1 \times B^n \text{ fix } \partial)$ acts transitively on $\Emb(B^n, S^1 \times B^n)$.  
Moreover every reducing ball $B^n \to S^1 \times B^n$
is the fiber of some smooth fiber bundle $S^1 \times B^n \to S^1$. 
\end{theorem}

Alternatively attach a $S^{n-1}\times D^2$ to obtain $S^{n+1}$ where the reducing ball is now an $n$-ball $\Delta_1$ with boundary a standard $(n-1)$-sphere.  By the Cerf - Palais theorem there is a diffeomorphism of this sphere taking $\Delta_1$ to a standard n-ball fixing $\partial \Delta_1$ pointwise. 


\begin{theorem}\label{non-sep-s1sn}
The group $\Diff(S^1 \times S^n)$ acts transitively on the non-separating $n$-spheres in $S^1 \times S^n$.
Moreover, every non-separating $n$-sphere is the fiber of a fiber bundle $S^1 \times S^n \to S^1$. 
\begin{proof}
Provided $n<3$ this is classical.  When $n \geq 3$ observe that complementary to a non-separating sphere there
is an embedding $S^1 \to S^1 \times S^n$ that intersects the sphere precisely once and transversely.  Since
$dim(S^1 \times S^n) \geq 4$, we can isotope our embedding to be equal to $S^1 \times \{*\}$ and similarly 
isotope our non-separating sphere.  If we drill a neighbourhood of $S^1 \times \{*\}$ out of $S^1 \times S^n$
we have constructed $S^1 \times B^n$, and our non-separating sphere is converted to a reducing ball. 
The result follows from Theorem \ref{diff-fibr}.
\end{proof}
\end{theorem}

Let $S^1_0$ denote a fixed $S^1\times \{x_0\} \subset S^1\times S^n$.  
Let $\Emb_0(S^1\times B^n, S^1 \times S^n)$ (resp. $\Emb_0(S^1, S^1 \times S^n))$  denote the component of 
$\Emb(S^1, S^1 \times S^n)$ that contains the standard inclusion to $S^1\times N(x_0)$ (resp. $S^1\times \{x_0\})$, where 
$N(x_0)\subset S^n$ is a closed regular neighborhood of $x_0\in S^n$.  Note that the restriction map  
$\Emb_0(S^1\times B^n, S^1 \times S^n)\to \Emb_0(S^1, S^1 \times S^n)$ is a fiber-bundle with fiber having
the homotopy-type of one path-component of the free loop space of $SO_n$. The map 
$\Diff(S^1\times S^n\fix N(S^1_0))\to \Diff(S^1\times B^n\fix \partial)$ given by deleting the interior of $N(S^1_0)$
is a homotopy-equivalence. 

\begin{lemma} \label{key exact} The locally trivial fiber bundle 
$\Diff(S^1 \times S^n  \text{ fix } N(S^1_0)) \to \Diff_0(S^1 \times S^n) \to \Emb_0(S^1\times B^n, S^1 \times S^n)$ induces 
a long exact homotopy sequence whose final terms are 
$$\cdots\to \pi_1 \Emb_0(S^1, S^1\times S^n; S^1_0)\overset{p}\to \pi_0 \Diff(S^1\times B^n\fix \partial) \overset{\phi}\to \pi_0 \Diff_0(S^1\times S^n)\to 0$$ 
 Here $p$ is induced by isotopy extension and $\phi$ is induced by extension which is the identity on the complementary $S^1\times B^n$.\qed\end{lemma}

Most of the rest of the paper will be spent showing that there is a homomorphism 
$\pi_0 \Diff(S^1\times B^3 \fix \partial) \to \pi_2 \Emb(I, S^1\times B^3)$ which when composed with $W_3$ gives rise to  a 
homomorphism, of the same name, $W_3:\pi_0 \Diff(S^1\times B^3 \fix \partial) \to  (\pi_5(C_3(S^1\times B^3))/\textrm{torsion})/R$ such that 
$W_3\circ p$ is trivial.  Furthermore, there exists an infinite set of elements of 
$\pi_0 \Diff(S^1\times B^3 \fix \partial)$ whose $W_3$ images  are linearly independent.  From this we will obtain the following 
result whose proof will be completed in \S8.  A sharper form of this result is given as Theorem \ref{knotted main}.

\begin{theorem}  \label{main} Both the group $\pi_0(\Diff(S^1\times B^3 \fix \partial)/\Diff(B^4\fix \partial))$ and \\
$\pi_0(\Diff_0(S^1\times S^3))/\Diff(B^4\fix\partial))$ contain an infinite set of linearly independent elements.\end{theorem}

\section{2-Parameter Calculus} \label{2 parameter section}
This section introduces techniques for working with 1 and 2-parameter families of embeddings of the interval into a 4-manifold.  

\subsection{Spinning}  We start by setting conventions for the operation we call \emph{spinning} that other authors call double point resolution.  Spinning about arcs is an operation that generates $\pi^D_1 \Emb(I,M; J_0)$, the \emph{Dax subgroup}, i.e. the subgroup represented by loops that are homotopically trivial in $\Maps(I,M; J_0)$ \cite{Ga2}.  Here $J_0$ is an oriented properly embedded $[0,1]$ in the oriented $4$-manifold $M$ with $1_{J_0}$ a fixed parametrization.  

\begin{definition}  \label{spinning} Let $Q$ be an oriented 2-sphere in $M$ with $Q\cap J_0=\emptyset$.  Given an 
embedding $\lambda:[0,1]\to \inte(M)$ with $\lambda\cap J_0=\lambda(0)$ and $\lambda\cap Q=\lambda(1)$, we obtain a based loop $\alpha_t$ in 
$\Emb(I,M; J_0)$ by using $\lambda$ to drag $J_0$ around $Q$.  Let $\gamma_0\subset J_0$ be a small arc containing $\lambda(0)$.  $
\alpha_t$ is defined so that for $t\in [0,.25]$,  $\gamma_t:=\alpha_t(\gamma_0)$ is a small arc containing $\lambda(4t)$, where $\gamma_{.
25}$ is an embedded arc in $Q$.  During $[.25,.75]$, keeping endpoints fixed, $\gamma_t$ rotates around $Q$.  I.e. view $Q=S^1\times [0,1]$ with $S^1\times 0$ and $S^1\times 1$ identified to points and for $t\in [.25,.75], \gamma_{t}=\theta_t\times [0,1]$ for monotonically increasing $\theta_t$.  Finally during $t\in [.75, 1]$, 
$\gamma_t$ returns to $\gamma_0$ following  the reverse of $\lambda$.   Corners are rounded so that each $\alpha_t$ is smooth.  The local 
picture of spinning about $Q$ is shown in Figure \ref{fig:FigureCalc1}.  The direction of the spinning is determined by the rule that (motion of $\gamma_t$, 
orientation of $\gamma_t$) gives the orientation of $Q$.  Any loop in $\Emb(I,M;J_0)$, isotopic to one constructed as above is called a 
\emph{$\lambda$-spinning of $J_0$ about $Q$}.  The spinning that goes about $Q$ in the opposite direction is called a $-\lambda$-spinning 
of $J_0$ about $Q$. \end{definition}

\begin{figure}[ht]
$$\includegraphics[width=7cm]{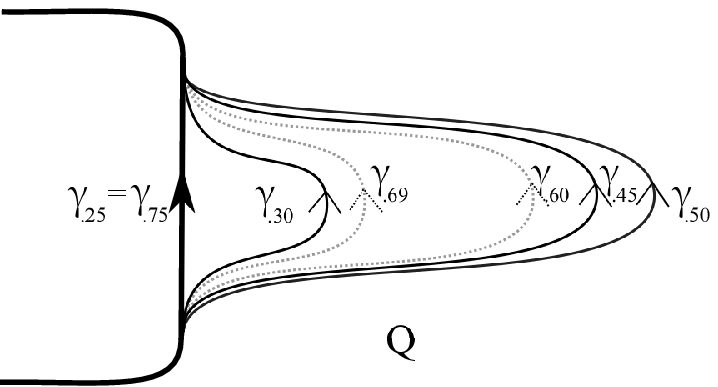}$$
\caption{\label{fig:FigureCalc1}} 
\end{figure}

\begin{lemma} \label{spinning dependence} Spinning depends only on the orientation of $J_0$, the orientation of $Q$ and the relative path homotopy class of $\lambda$, i.e. if $\lambda_v\subset \Maps(I,M)$ is a path homotopy of $\lambda_0 :=\lambda$ to $\lambda_1$, then we require that $\inte(\lambda_v)\cap Q=\emptyset$.  \qed\end{lemma}

If $\lambda$ is an embedded arc from $J_0$ to the oriented arc $\tau \subset M\setminus J_0$, then let $B$ be a 3-ball normal to $\tau$ at $\lambda(1)$ oriented so that (orientation of $\tau$, orientation of $B$)=orientation of $M$.  $B$ is called a normal 3-ball.  Let $Q=\partial B$ oriented with the outward first boundary orientation and $\lambda'=\lambda\setminus \inte(B))$.  Define the \emph{$\lambda$ spinning about $\tau$} to be the $\lambda$'-spinning about $Q$.  \emph{Chord diagram} notation for this spinning and its inverse are shown in Figure \ref{fig:FigureCalc2}a).   The sign denotes whether this is a positive or negative spinning.  \emph{Band/lasso} notation is shown in Figure \ref{fig:FigureCalc2}b).  The \emph{band} $=\cup\{\gamma_t|t\in [0,1/4]\}$ with $\lambda'$ being the core of the band and a \emph{lasso} is a circle in $Q$ containing, up to a small isotopy and rounding corners, $\gamma_{1/4}$. $\alpha_{1/4}$ and $\alpha_{3/4}$ are shown in Figure \ref{fig:FigureCalc2}d) and $\alpha_{1/2}$ in Figure \ref{fig:FigureCalc2}c).  The positive (resp. negative) spinning corresponding to the band $\beta$ and lasso $\kappa$ will be denoted $\sigma(\beta,\kappa)$ (resp. $-\sigma(\beta,\kappa)$).  We call $\beta\cap J_0$ the \emph{base of the band} and $\beta\cap \kappa$ the \emph{top of the band}. We orient the \emph{core of the band} to point from the base to the top.

\begin{figure}[ht]
$$\includegraphics[width=12cm]{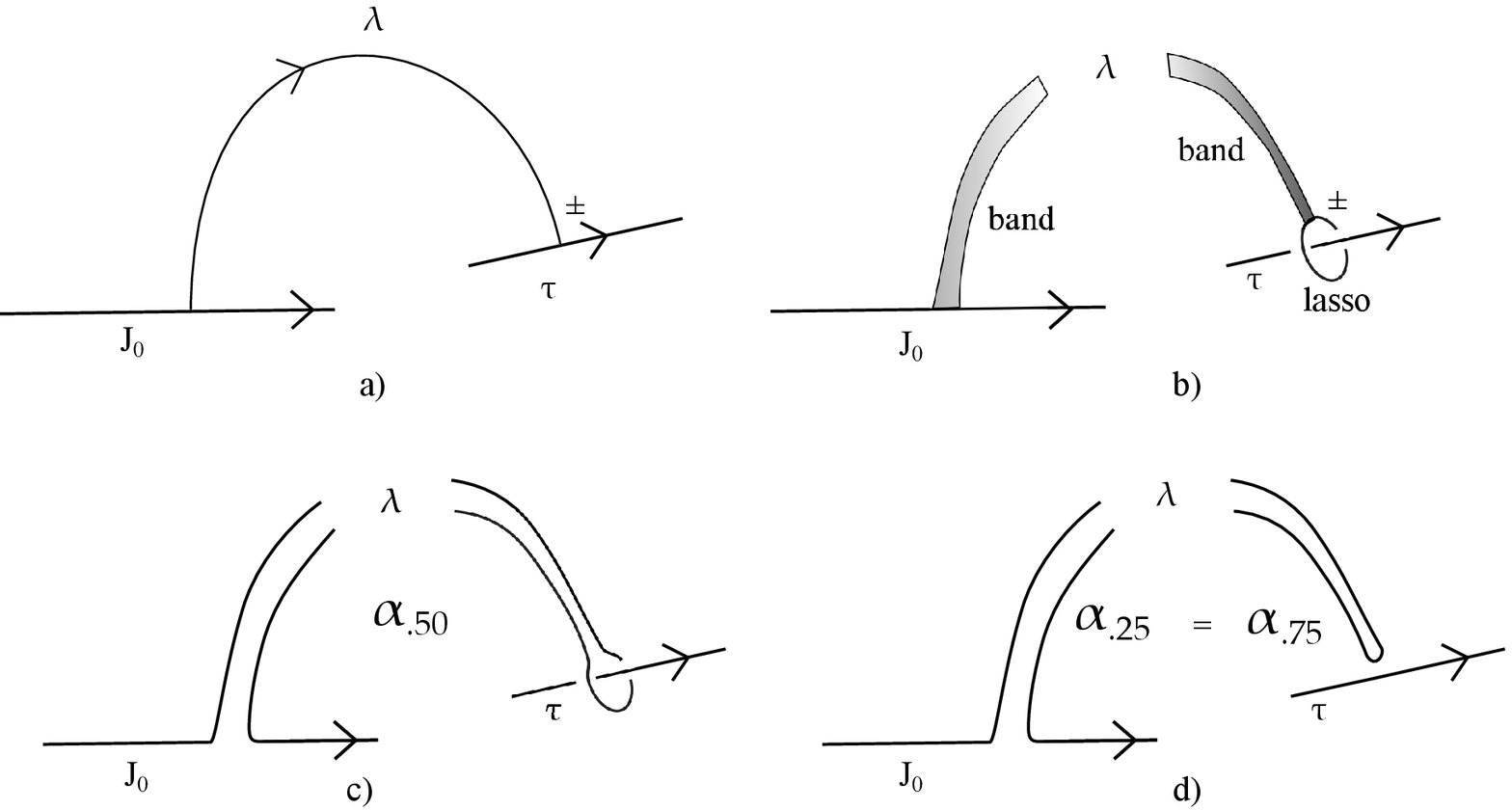}$$
\caption{\label{fig:FigureCalc2}} 
\end{figure}

For this section and the next we view  $V=S^1\times B^3$ as $D^2\times S^1\times [-1,1]$ with the product orientation.   Let $I_0$ be a properly embedded arc in $V$.   When  $\tau\subset I_0$, or very close to it, the $\lambda$ of Figure \ref{fig:FigureCalc2}) will be replaced by $n\in \BZ=\pi_1(S^1\times B^3; I_0)$, where $I_0$ is viewed as the basepoint.    Unless said otherwise all charts of $V$ will be of the form $(D^2\times [0,1])\times [-1,1]$, where the $S^1$-direction is the $[0,1]$-direction.  Spinnings will almost always be about oriented arcs $\tau$ in a $D^2\times I\times 0$ with the band $\subset D^2\times S^1\times 0$.  Here the normal ball $B$ intersects  $D^2\times I\times 0$ in a 2-disc called the \emph{lasso disc} with the \emph{lasso} its boundary. We call $B$ the \emph{lasso 3-ball} and $Q$ the \emph{lasso sphere}.  In this paper, given a lasso, the lasso disc, sphere and 3-ball will be clear from context. The spinning will be denoted L/H (resp. H/L) if the homotopy from $\gamma_{1/4}$ to $\gamma_{3/4}$ first goes into the past (resp. future) and then into the future (resp. past).  

We now give an oriented intersection theoretic way to decide whether or not a spinning of $I_0$ about an oriented arc $\tau$ is positive or negative.  First, orient the band $\beta$ to be coherent with its intersection with $I_0$ as in Figure \ref{fig:FigureCalc3} and then orient the lasso disc to be coherent with $\beta$ as in Figure \ref{fig:FigureCalc3}.     

\begin{figure}[ht]
$$\includegraphics[width=10cm]{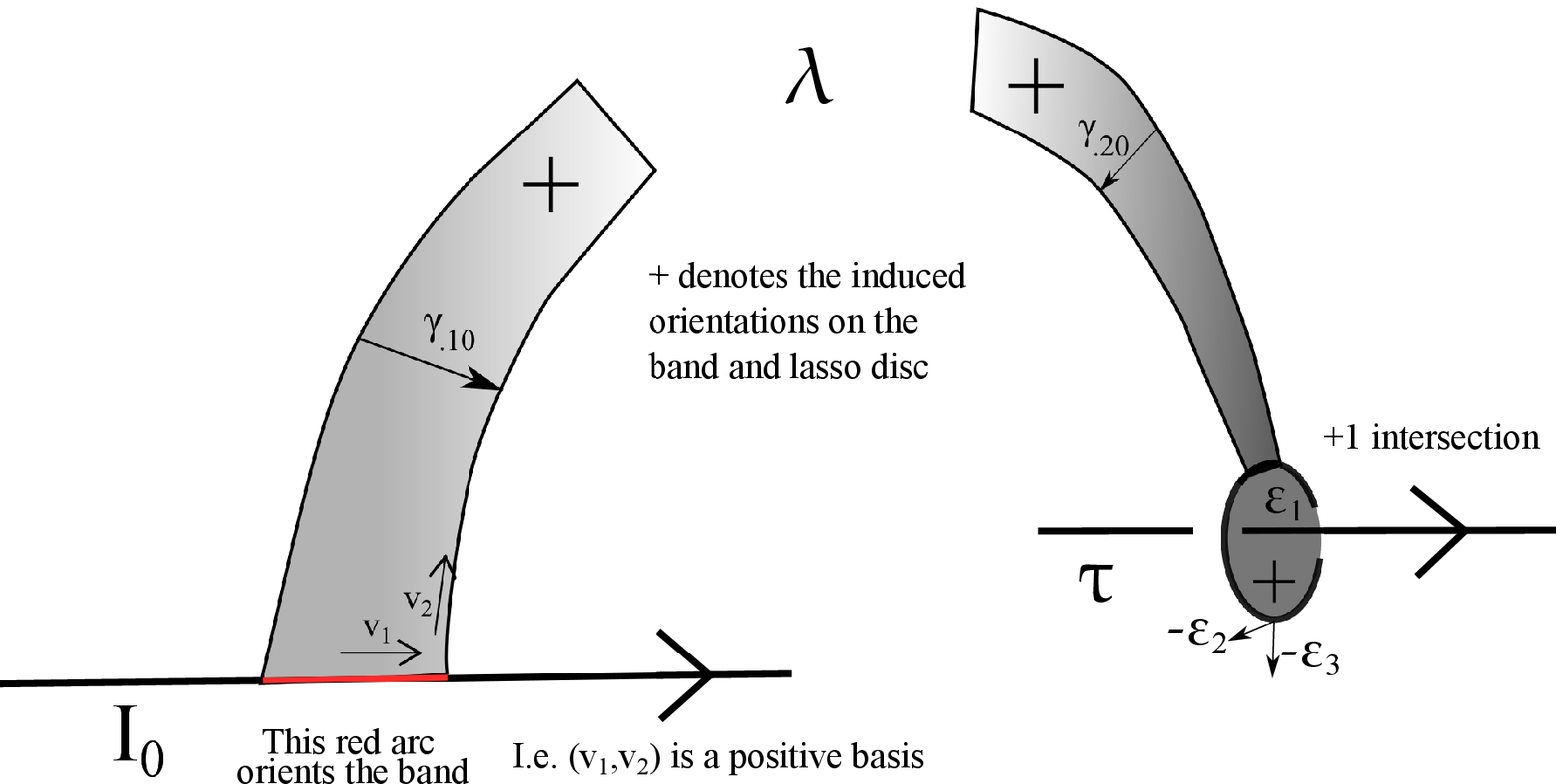}$$
\caption{\label{fig:FigureCalc3}} 
\end{figure}

\begin{lemma} \label{sign intersection} The spinning is positive if and only if the spinning is L/H (resp. H/L) and  the oriented intersection number of $\tau$ with the lasso disc is +1 (resp. -1). \end{lemma}

\begin{proof} Let $(\epsilon_1, \cdots, \epsilon_4)$ denote the standard orientation of $D^2\times I\times [-1,1]$.  We can assume that $\epsilon_1$ defines the orientation of $\tau$ so that the orientation of a normal ball $B$ is given by $(\epsilon_2,\epsilon_3,\epsilon_4)$. We can also assume that the band and lasso appear as in Figure \ref{fig:FigureCalc3}), i.e. so that $-\epsilon_3$ is an outward normal to $B$ and an orienting vector for $\alpha_{.5}$ is $-\epsilon_2$.  Therefore,  $Q=\partial B$ is oriented by $(\epsilon_2,\epsilon_4)$ and if the spinning is $L/H$, then the motion vector for $\alpha_{1/2}$ is $\epsilon_4$ when the tangent vector to $\alpha_{1/2}$ is $-\epsilon_2$.  It follows that the lasso disc has intersection +1 with $\tau$.  The H/L case follows similarly.\end{proof}

\begin{lemma} \label{elementary homotopies} The spinnings of Figure \ref{fig:FigureCalc4} a) - d) and e) - f) are homotopic in $\Omega\Emb(I,M; I_0)$. Furthermore, the homotopy from a) to b) is supported in the union of a small 4-ball that contains the bands and the 3-balls spanning the spinning spheres $Q$ and $Q'$.  The homotopies from b) to d) are supported in a small neighborhood of the union of the bands and the sphere $Q$. The homotopy from e) to f) is supported in a neighborhood of the arc.  Note that going from e) to f) the $p$ changes sign, the arc changes orientation and the +/- changes to -/+.    \end{lemma}

\begin{figure}[ht]
$$\includegraphics[width=11cm]{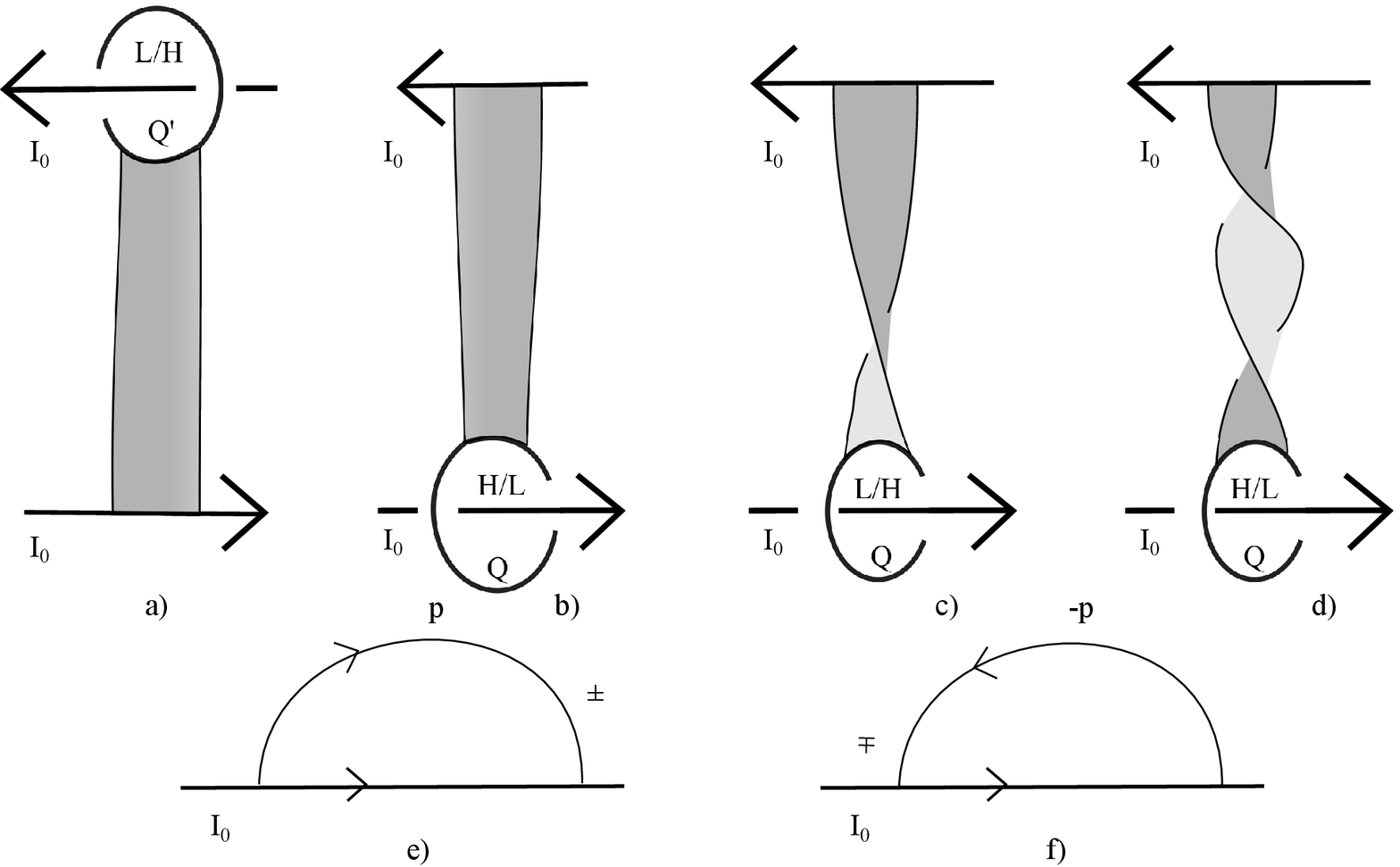}$$
\caption{\label{fig:FigureCalc4}} 
\end{figure}

\begin{proof} The assertions regarding a) to b) are immediate.  To go from b) to c) first choose coordinates so that a neighborhood of the top of the band as well as the lasso are contained in the $x - y$ plane and near the top of the band, the core of the band lies in the $y$-axis.   Next rotate $Q$ to take the lasso to its reflection in the $y$-axis.  This is achieved by a rotation of the $x - t$ plane.  Finally, a rotation in the $z - t$ plane isotopes the band back into $D^2\times I\times 0$.  Using a different rotation we could have obtained the opposite crossing.  Apply this homotopy twice to go from c) to d).  Using Lemma \ref{sign intersection} the assertions regarding e) and f) follows from that of a) and b). \end{proof}

\begin{remark} \label{determined}  This lemma implies that the core of a band determines the isotopy class of the band up to possibly adding a half twist.  Thus the isotopy class of the band is determined by the core arc and the orientation on the lasso disc. \end{remark}

\begin{lemma} \label{undo}  Let $\lambda_1, \lambda_2$ be parallel oriented embedded arcs from $I_0$ to $I_0$ whose positive endpoints intersect in a subarc $\tau$ as in Figure \ref{fig:FigureCalc5} a).  Let $\sigma_1$ (resp. $\sigma_2)$ denote the spinning corresponding to $\lambda_1$ (resp. $\lambda_2)$ with that of $\sigma_2$ being oppositely signed.  Then, the loop $\alpha_t$ in $\Emb(I, S^1\times B^3; I_0)$ which is the simultaneous spinning of $\sigma_1$ and $\sigma_2$ is homotopically trivial via a homotopy whose support lies in a small neighborhood of a parallelizing band between $\lambda_1$ and $\lambda_2 \setminus N(\tau)$.  

More generally, let $\alpha$ be the two spinnings as in Figure \ref{fig:FigureCalc5} b).  Here the spinnings are expressed in band/lasso notation.  Except for neighborhoods of the base arcs, where they differ by a half twist, the bands $\beta_1$ and $\beta_2$ are parallel, connecting to parallel lassos and  parallel lasso spheres $Q_1$ and $Q_2$.  Then, $\alpha$ is homotopic to $1_{I_0}$ via a homotopy supported in the union of a 4-ball $U$ about the non parallel parts of the bands as in Figure \ref{fig:FigureCalc5} c) and a small neighborhood of the parallelism between the spheres and the bands.  \qed\end{lemma}

\begin{figure}[ht]
$$\includegraphics[width=9cm]{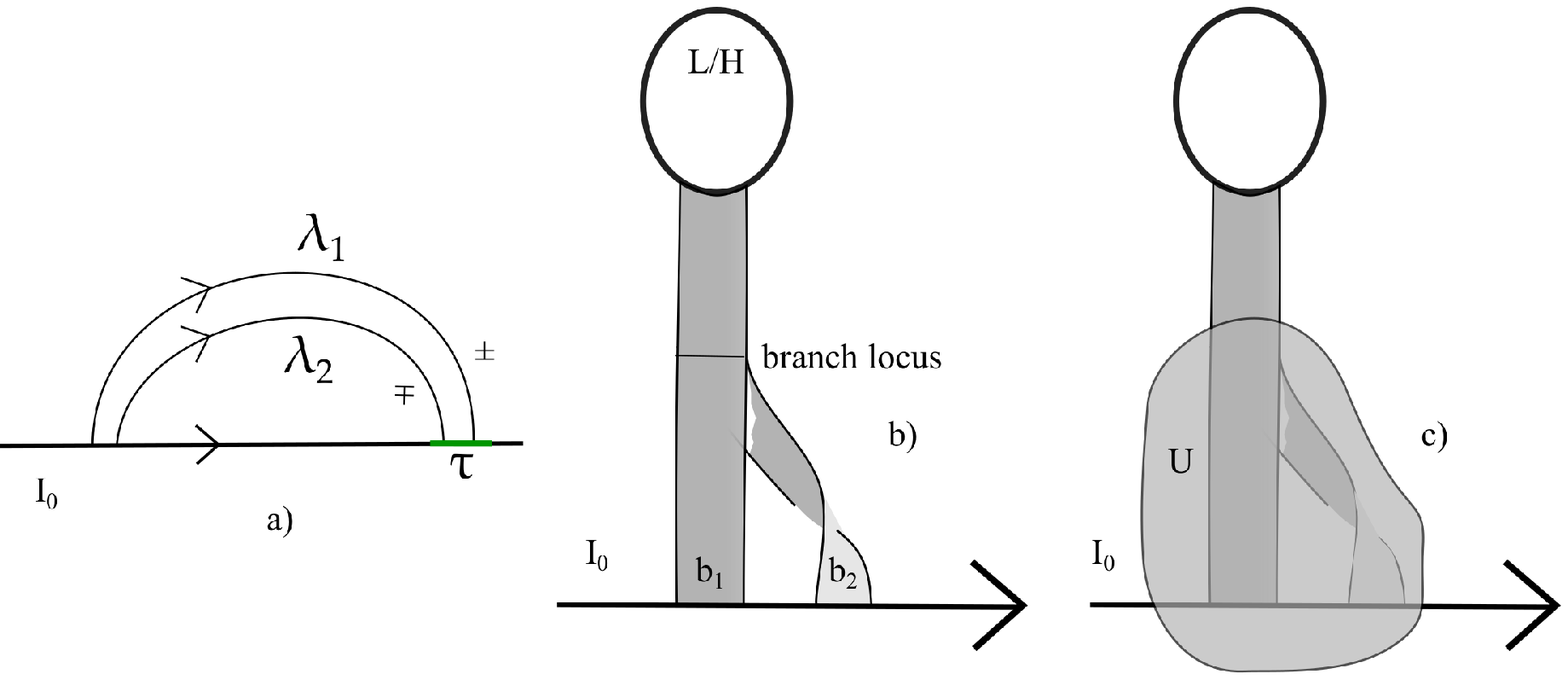}$$
\caption{\label{fig:FigureCalc5}} 
\end{figure}

\begin{definition}  \label{undo definition} The pair of spinnings defining $\alpha$ above is called a \emph{parallel cancel pair}.  Here $\beta_1$ and $\beta_2$ are carried by a \emph{branched band surface}.  The arcs in $\beta_1$ and $\beta_2$ where the bands diverge  are called the \emph{branch loci}. A homotopy as in Lemma \ref{undo} is called the \emph{undo homotopy}.\end{definition}

\begin{definition}  \label{splitting}  Let $\gamma_t$ be a $\lambda$-spinning about the 2-sphere $Q$.  Suppose that $Q=\partial B, B$ a 3-ball and $E\subset B$ a properly embedded 2-disc with $\partial E$  transverse to $\gamma_t, t\in [.25,.75]$ as in Definition \ref{spinning}.  Let $Q_1, Q_2$ be the result of compressing $Q$ along $E$.  Then $\gamma$ is homotopic to a concatenation $\gamma_1*\gamma_2$ of two spinnings along nearby $\lambda$'s, where $\gamma_i$ spins about $Q_i, i=1,2$.  The operation of replacing $\gamma$ by $\gamma_1*\gamma_2$ is called \emph{splitting}.  The reverse operation is called \emph{zipping}. See Figures \ref{fig:FigureCalc5}d) and   e). \end{definition}

\begin{remarks} \label{splitting support}  i)  The  support of the homotopy from $\gamma$ to $\gamma_1*\gamma_2$ can be taken to be a small neighborhood of $\lambda\cup Q\cup E$.

ii)  Suppose that $\gamma$ is presented by the band $\beta$ and lasso $\kappa$.  Suppose that $Q, D, B$ denote the lasso sphere, disc and 3-ball.  Let $\beta_1$ and $\beta_2$  be obtained from $\beta$ by removing a small neighborhood of its core and let $\kappa_1$ and $\kappa_2$ be the components of $\partial D_1$ and $\partial D_2$, where $D_1$ and $D_2$ are obtained from $D$ by removing a small neighborhood of a properly embedded arc $e\subset D$ which intersects $\beta$ at its core.  Then $\sigma(\beta_1,\kappa_1)*\sigma(\beta_2,\kappa_2)$ is a splitting of $\sigma(\beta,\kappa)$ and we call $(\beta_1,\kappa_1), (\beta_2,\kappa_2)$ a splitting of $(\beta,\kappa)$.  They are spinnings about $Q_1$ and $Q_2$,  the components of $\partial (B\setminus \inte(N(E))$, where $E\subset B$ is a properly embedded 2-disc such that $E\cap D=e$.   Here we require that the spinning across $Q_1$ and $Q_2$ be L/H (resp. H/L) if the spinning across $Q$ is L/H (resp. H/L).  Note that $\sigma$ is homotopic to $\sigma(\beta_1,\kappa_1)*\sigma(\beta_2,\kappa_2)$ by a homotopy supported in a small neighborhood of $\beta\cup Q\cup E$.  


iii) As an application, suppose that $\tau_0, \tau_1, \cdots, \tau_n$ are the local edges of a 1-complex $K\subset M$ emanating from the vertex $v$.  If $\sigma_0$ is a $\lambda$-spinning about a subarc of $\tau_0$ close to $v$, then $\sigma_0$ splits to a concatenation $\sigma_1*\cdots*\sigma_n$ via a homotopy supported away from $K$.  Figure \ref{fig:FigureCalc5c} a) and b) demonstrates the argument for $n=2$.  \end{remarks}
\begin{figure}[ht]
$$\includegraphics[width=9cm]{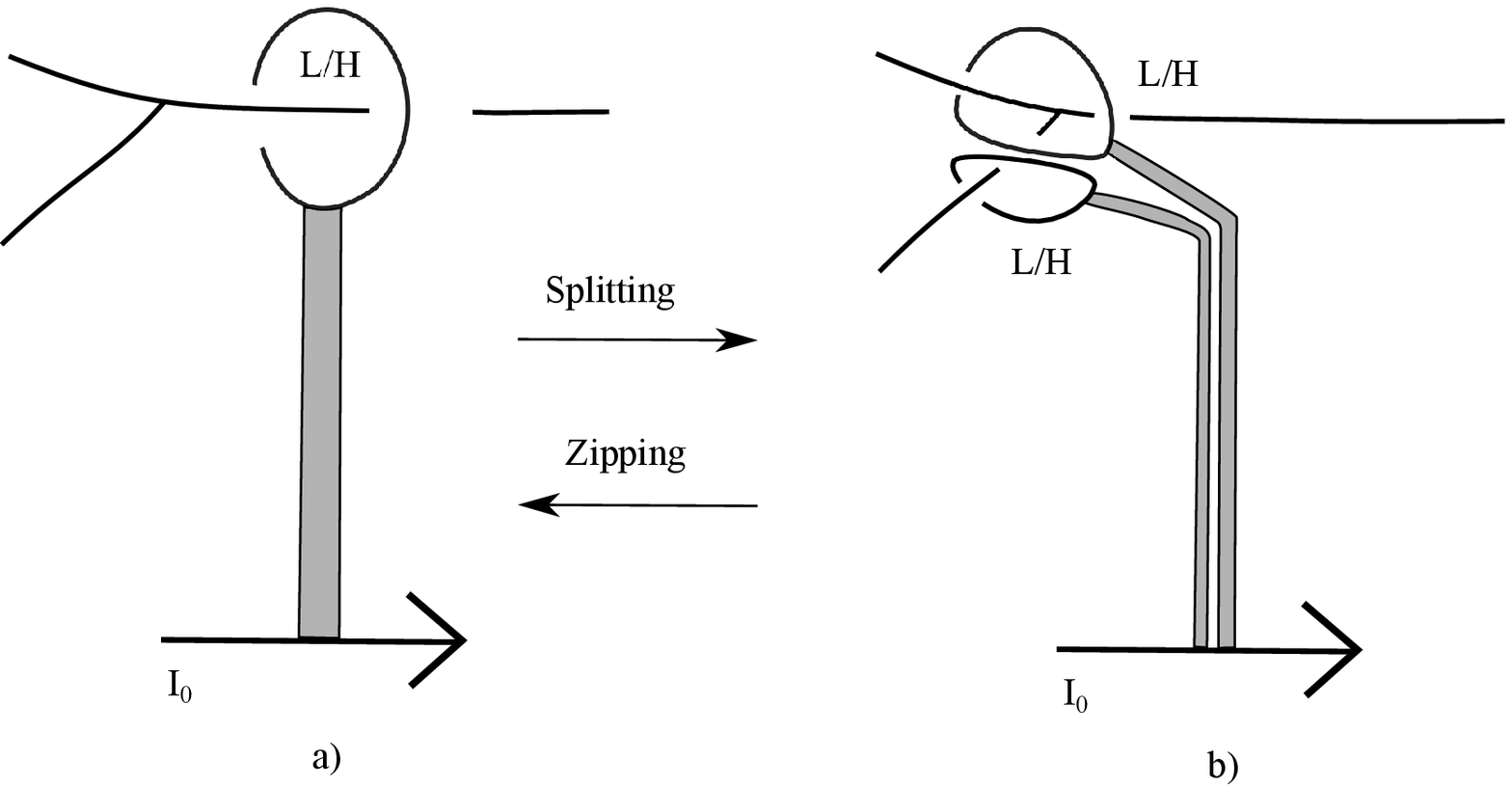}$$
\caption{\label{fig:FigureCalc5c}} 
\end{figure}

\begin{definition}  An  \emph{abstract chord diagram} $C$ in the manifold $M$ consists of 

a) an oriented properly  embedded arc $I_0$

b) finitely many pairwise disjoint ordered pairs of distinct points $p(C)$ in $I_0, (x_1,y_1), \cdots, \\ (x_n,y_n)$ with $p(C)$ linearly ordered by the orientation on $I_0$

c) For each $1\le i\le n, g_i\in \pi_1(M; I_0)$

d) For each $1\le i\le n, \eta_i\in \pm 1$. \end{definition}

\begin{remarks} 1) An abstract chord diagram $C$ in a 4-manifold gives rise to an $\alpha(C)\in \Omega\Emb(I, M; I_0)$, well defined up to homotopy, which is a concatenation of spinnings, by choosing pairwise disjoint embedded paths $\lambda_i$ from $x_i$ to $y_i$ representing $g_i$.  Call such paths \emph{chords}.  An abstract chord diagram together with chords is called a \emph{realization} and sometimes just called a \emph{chord diagram}.   Since spinnings commute as elements of $\pi_1 \Emb(I, M; I_0$), this element is unaffected by modifying the relative location of the pairs of points $(x_i,y_i)$, however most of the abstract chord diagrams of interest in this paper represent the trivial element, up to homotopy and we are not free to move these points in general.  We will be considering the following types of modifications.  \end{remarks}

\begin{definition} We have the following operations on abstract chord diagrams.

i) (reversal) $(x_i, y_i, g_i, \pm)$ is replaced by $(y_i, x_i, g_i^{-1}, \mp$)

ii) (exchange) If $a_i, b_j \in p(C)$ are adjacent in the linear ordering  where $a_i\in\{x_i, y_i\}$ and $b_j\in \{x_j, y_j\}, i\neq j$, then replace $a_i\in\{x_i, y_i\}$ by $b_j$ and  $b_j\in \{x_j, y_j\}$  by $a_i$.

iii) (sliding) As input, $y_i\in p(C)$ is adjacent to an $x_j\in p(C)$ with $i\neq j$.  As output the chord $(x_i, y_i, g_i, \pm)$ is replaced by three chords $(x_i, y'_i, g_i, \pm)$, $(x_r, y_r, g_i*g_j, \sigma_r)$, $(x_s, y_s, g_i*g_j, \sigma_s)$.  Here $y'_i$ is moved to the other side of $x_j$,  Also $x_r, x_s, x_i$ are order adjacent as are $y_s, y_j, y_r$. \emph{Sign Rule}:  $\sigma_r=\pm$ if (resp. $\sigma_s=\pm$)  the interval between $y_i$ and $y_r$ (resp. $y_i$ and $y_s$) contains none or both of $x_j, y_j$, otherwise and $\sigma_r$ (resp. $\sigma_s$) $=\mp$. See Figure \ref{fig:FigureCalc10c}.

iv) (isotopy) Points  are moved isotopically in $I_0$ without any collisions.\end{definition}
\begin{figure}[ht]
$$\includegraphics[width=12cm]{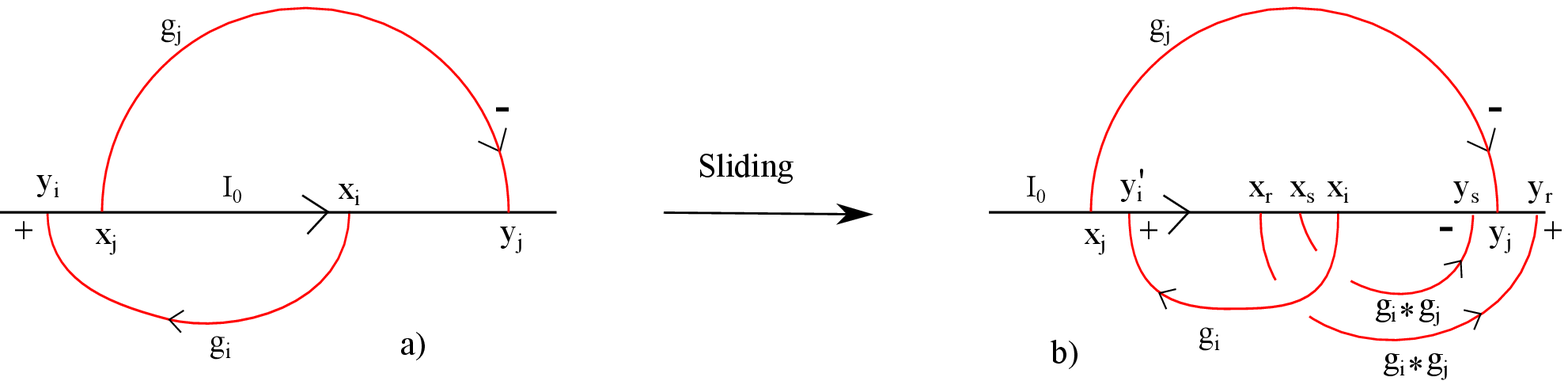}$$
\caption{\label{fig:FigureCalc10c}} 
\end{figure}
\begin{remark} Let $\tau_j $ denote a chord from $x_j$ to $y_j$.  A sliding is the result of splitting the $(x_i, y_i, g_i, \pm)$ spinning near $x_j$, then isotoping the resulting lasso that links $\tau_j$ along $\tau_j$ and finally do a second splitting near $y_j$.  Note that the support of this isotopy is disjoint from $\tau_j$.\end{remark}  

\begin{definition} We say the abstract chord diagrams $C $ on $I_0$ and  $C'$  on $I_1$ are \emph{combinatorially equivalent}  if their data  agrees up to reversals and isotopy. \end{definition}

\begin{lemma} \label{combinatorial} If the abstract chord diagrams $C$ and $C'$  defined on $I_0$ are combinatorially equivalent, then $\alpha_C$ is homotopic to $\alpha_{C'}$.\qed\end{lemma}

\begin{lemma} (Change of Basepoint)  \label{basepoint} Let  $I_0$ and $I_1$ be path homotopic embedded arcs in the 4-manifold $M$.  An abstract chord diagram $C_0$ on $I_0$ induces a combinatorially equivalent chord diagram $C_1$ on $I_1$.  If $\alpha_j:=\alpha_{C_j}\in \Omega \Emb(I,M, I_j), j=0,1$ then $\alpha_1 = \gamma*\alpha_0*\gamma^{-1}$ is well defined up to homotopy and in particular is independent of the path isotopy $\gamma$ from $I_0$ to $I_1$.  \end{lemma}

\begin{proof} Since $I_0$ and $I_1$ are path homotopic, there is a canonical isomorphism between $\pi_1(M,I_0)$ to $\pi_1(M,I_1)$ and hence $C_0$ induces a chord diagram $C_1$ on $I_1$ well defined up to combinatorial equivalence and hence $\alpha_1$ is well defined up to homotopy.  Alternatively, we can assume that $I_0$ and $I_1$ agree in a small neighborhood $U$ of their initial point and $p(I_0)\subset U$.  It follows that the homotopy class of $\alpha_1$ is independent of the conjugating $\gamma$.\end{proof} 

\subsection{Factorization}

In this subsection we introduce techniques for constructing and working with 2-parameter families of embeddings.  We will define the \emph{bracket} of two null homotopic loops in $\Emb(I,M;J_0)$ whose domain and range supports are disjoint and will show that such brackets are well defined in $\pi_2 \Emb(I,M;J_0)$, anticommute and are bilinear.    Our fundamental example is the $G(p,q)$ family in $\E$ depicted in Figure \ref{fig:FigureCalc6} a) which is the \emph{bracket} of the 1-parameter families $B_p$ and $R_q$ defined by the blue and red cords.  The $p$ or $q$ means that the indicated cord goes $p$ or $q$ times about the $S^1$.  

\begin{definition}  \label{support} Let $M$ be an oriented 4-manifold and $1_{J_0}:[0,1]\to M$ a proper embedding with image $J_0$.  Let $\alpha:[0,1]\to M$ be an embedding.  We define the \emph{domain support of $\alpha$} $:= \Dsupp (\alpha)=\cl\{s\in I|\alpha(s)\neq 1_{J_0}(s)\}$ and the \emph{range support of $\alpha$} $:=\Rsupp(\alpha) =\cl\{\alpha(I)\setminus J_0\}$.  If $f:X\to \Emb(I,M;J_0)$, then define the \emph{domain support of $f$}$:= \Dsupp (f)=\cl(\cup_{x\in X}\Dsupp(f_x))$ and the \emph{range support of $f$}$:= \Rsupp(f)=\cl(\cup_{x\in X}\Rsupp(f_x))$.  Define the \emph{parameter support of $f$} to be $\cl\{x\in X|f_x\neq\Jzero \}$.  \end{definition}

\begin{figure}[ht]
$$\includegraphics[width=11cm]{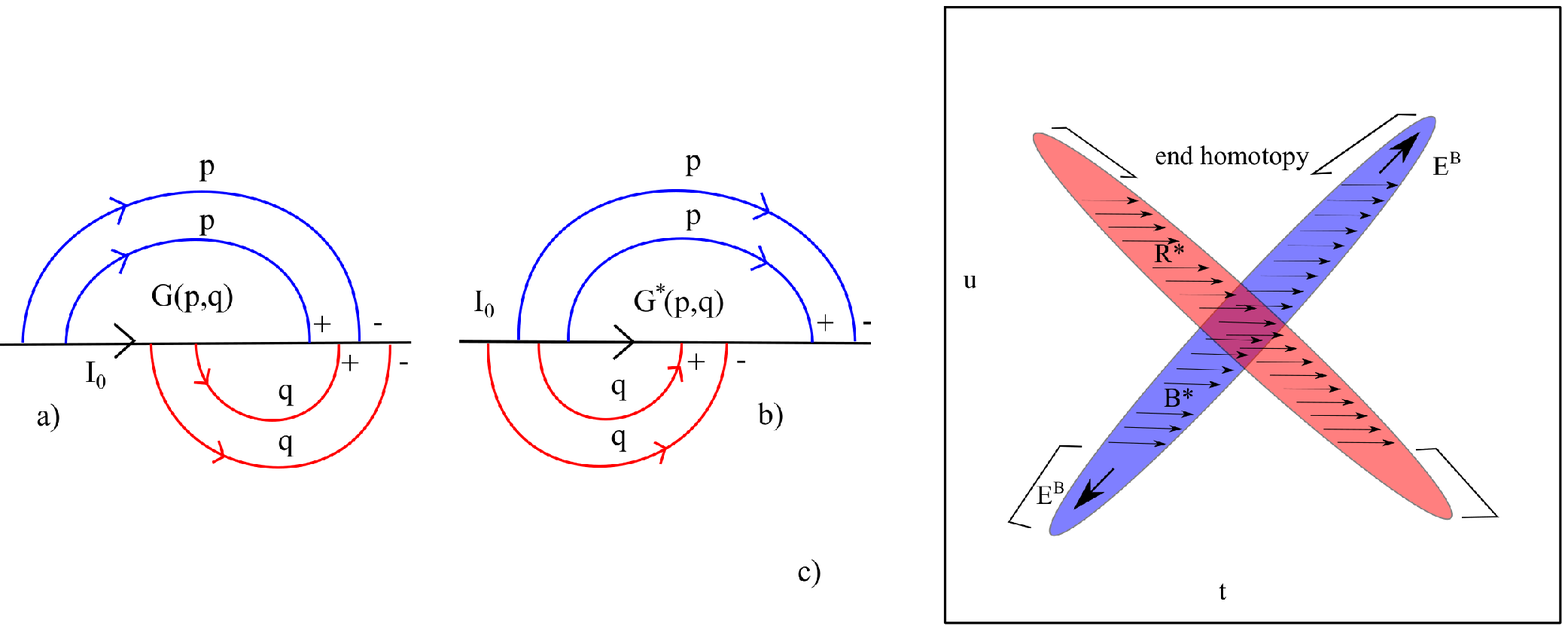}$$
\caption{\label{fig:FigureCalc6}} 
\end{figure}

\begin{example} The domain support of $B_p$ (resp. $R_q$) are four small intervals that contain the endpoints of the blue (resp. red) cords while the range support of $B_p$ (resp. $R_q)$ is contained in a small neighborhood of the blue (resp. red) cords.  Note that both $B_p$ and $R_q$ are null homotopic and have disjoint domain and range supports.  $(B,R)$ is an example of a \emph{separable pair} which we now define.\end{example}

\begin{definition} Let $B, R\in\Omega\Emb(I,M;J_0)$, the based loop space.  We say that the ordered pair $(B,R)$ is \emph{separable} if 

i) $B$ and $R$ have disjoint domain supports

ii) $B$ and $R$ have disjoint range supports

iii) $B$ and $R$ are null homotopic\end{definition}

\begin{definition} \label{adjunction} If  $F^B, F^R:X\to \Emb(I,M; J_0)$ are such that for all $x\in X, F^B_x, F^R_x$ have both disjoint domain and range supports, then define $F=F^B\circ F^R:X\to \Emb(I,M; J_0)$ by $F_x(s)=F_x^B(s)$ (resp. $F_x^R(s))$ if $s\in \Dsupp( F_x^B$) (resp. $\Dsupp(F_x^R))$ and $F_x(s)=\Jzero(s)$ if $s \notin \Dsupp( F_x^B)\cup \Dsupp(F_x^R))$. $F$ is called the \emph{adjunction} of $F^B$ and $F^R$.  \end{definition}

\begin{definition} \label{bracket} Let $(B,R)$ be a separable pair with parameter supports within $[3/8,5/8]$. Define the \emph{bracket} of $(B,R)$ to be $F\in\Omega\Omega \Emb(I,M;J_0)$, the adjunction of the elements $F^B, F^R\in\Omega\Omega \Emb(I,M;J_0)$  defined as follows.

i) For $u\in [.25,.75], F_{t,u}^B(s)=B_{t+u-.5}(s)$ when $s\in\Dsupp( B)$ and $\Jzero(s)$ otherwise.  $F_{t,u}^R(s)=R_{t+.5-u}(s)$ when $s \in\Dsupp(R)$ and $\Jzero(s)$ otherwise.

ii) For $u\in [0,1/4], F_{u}^B\in \Omega \Emb(I,M;J_0)$, is a null homotopy of $F_{.25}^B$ such that $F^B_{t,u}=\Jzero$ if $(t,u)\notin[1/8,3/8]\times [1/8,1/4]$.  $F_u^R$, is a null homotopy of $F_{.25}^R$ such that $F^R_{t,u}=\Jzero$ if $(t,u)\notin[5/8,7/8]\times [1/8,1/4]$.  

iii) for $u\in [3/4,1], F_{t,u}^B=F_{t-.5,1-u}^B$ when $(t,u)\in [5/8,7/8]\times [3/4,7/8]$ and $\Jzero$ otherwise.  $F_{t,u}^R= F_{t+.5, 1-u}^R$ when $(t,u)\in [1/8,3/8]\times [3/4,7/8]$ and $\Jzero$ otherwise.  

We let $[F]$ denote the class in $\pi_2(\Emb(I,M; J_0))$ represented by $F$.  The homotopies $F_u^B, F_u^R,\ u\in [1/8,1/4]$ and $u\in [3/4,7/8]$ are called \emph{end homotopies}.  $B$ and $R$ are called the \emph{midlevel loops} of the adjunction $F$.  \end{definition}

\begin{example}  The blue (resp. red) region of Figure \ref{fig:FigureCalc6} c) contains the parameter support of the 2-parameter family $F^{B_p}$ (resp. $F^{R_q}$)  arising from $B_p$ (resp. $R_q$). The arrows are meant to suggest that as we horizontally traverse the blue or red region, for $u\in [1/4, 3/4]$, we see a conjugate of $B$ or $R$.    \end{example}

\begin{remark} If the range supports of $B$ and $R$ are disjoint, then $F$ is homotopically trivial.  \end{remark}

\begin{lemma} (Independence of End Homotopies) \label{independence}  If $F, G \in \Omega\Omega\Emb(I,M; I_0)$ are separable such that $F_u=G_u$ for all $u\in [1/4,3/4]$, then $[F]=[G]\in \pi_2(\Emb(I,M; I_0))$.\end{lemma}

\begin{proof}  It suffices to consider the case that $F^R=G^R$. The map $F$, schematically shown in Figure \ref{fig:FigureCalc6}, is homotopic to the one in Figure \ref{fig:FigureCalc7} a), which is homotopic to the one in Figure \ref{fig:FigureCalc7} b), which is homotopic to the map $G$.   \end{proof}

\begin{figure}[ht]
$$\includegraphics[width=11cm]{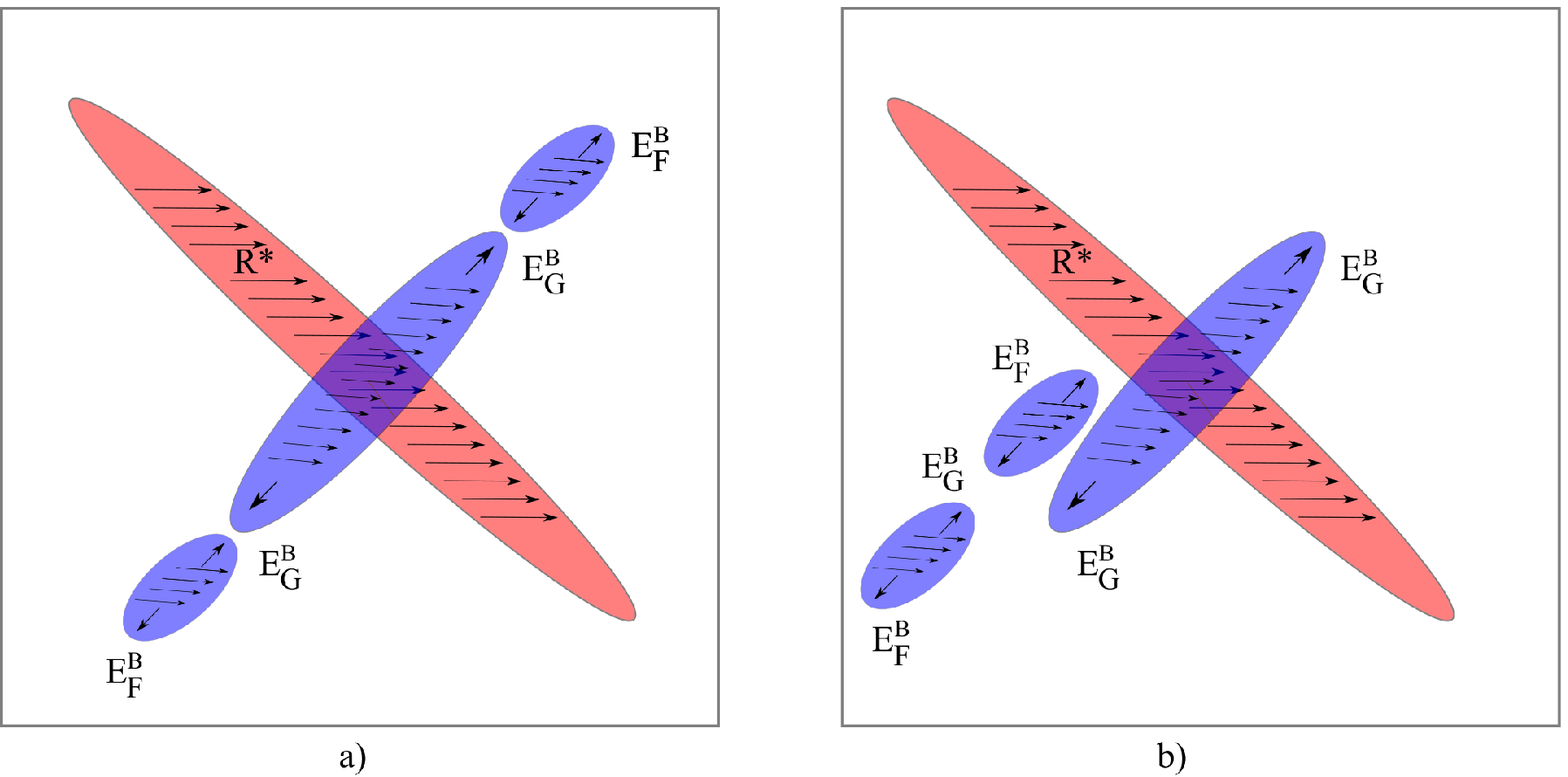}$$
\caption{\label{fig:FigureCalc7}} 
\end{figure}

\begin{proposition} \label{pair} The class $[F]\in \pi_2 \Emb(I,M; I_0)$ is determined by $(B,R)$.  \end{proposition}

\begin{proof}  Let $F$ and $G$ have the common $(B,R)$.  Since the domain and range supports of $B$ and $R$ are disjoint, we can independently reparametrize $F^B$ and $F^R$ so that $F_u$ and $G_u$ coincide for $u\in [1/4+\epsilon, 3/4-\epsilon]$.  The Independence lemma implies that they represent the same element of $\pi_2(\Emb(I,M; I_0))$.\end{proof}

\begin{definition}  \label{separable pair} We denote by $[(B,R)]$ the class in $\pi_2(\Emb(I,M; I_0))$ induced from $(B,R)$.  We say that the elements $B, R \in \Omega(\Emb(I,M; I_0))$ are \emph{separable} if their domain and range supports, as in Definition \ref{support} are disjoint.  We call $(B,R)$ a \emph{separable pair}.  We say that $(B_0, R_0)$ and $(B_1, R_1)$ are \emph{separably homotopic} if there exist homotopies $B_s, R_s s\in [0,1]$, such that for each $s, B_s$ and $R_s$ are separable. \end{definition}

\begin{lemma} If $B_0$ and $R_0$ are homotopically trivial and $(B_0, R_0)$ and$ (B_1, R_1)$ are separably homotopic, then $[(B_0,R_0)]=[B_1,R_1]$.\qed\end{lemma}

\begin{proposition}(bilinearity) \label{bilinearity} If $(B,R)$ is a separable pair, separably homotopic to ($B_1* B_2* \cdots * B_m, R_1* \cdots * R_n$) where each $B_i$ and $R_j$ is homotopically trivial and $*$ denotes concatenation, then $[B,R]=\sum_{i,j} [(B_i, R_j)]$.\end{proposition}

\begin{proof}  It suffices to consider the case that n$=2$ and $m=1$.  The result will then follow from induction.  First notice that the support of a concatenation is the union of the supports of the terms so that each of $(B,R_1)$ and $(B, R_2)$ is a separable pair.  By the independence of end homotopies we can assume that the end homotopy for $R$ is a concatenation of those of $R_1$ and $R_2$.  Since $R$ is a concatenation, we can assume that the usual support of $R_1 \subset [0,.4]\subset [0,1]$, while support $R_2\subset [.6,1]\subset [0,1]$.  Thus we can assume that $(B,R)$ appears as in Figure \ref{fig:FigureCalc8} a) and we can split B into two copies of itself as in Figure \ref{fig:FigureCalc37} b) and so $[(B,R)]=[(B,R_1)]+[(B,R_2)]$.  \end{proof} 

\begin{figure}[ht]
$$\includegraphics[width=11cm]{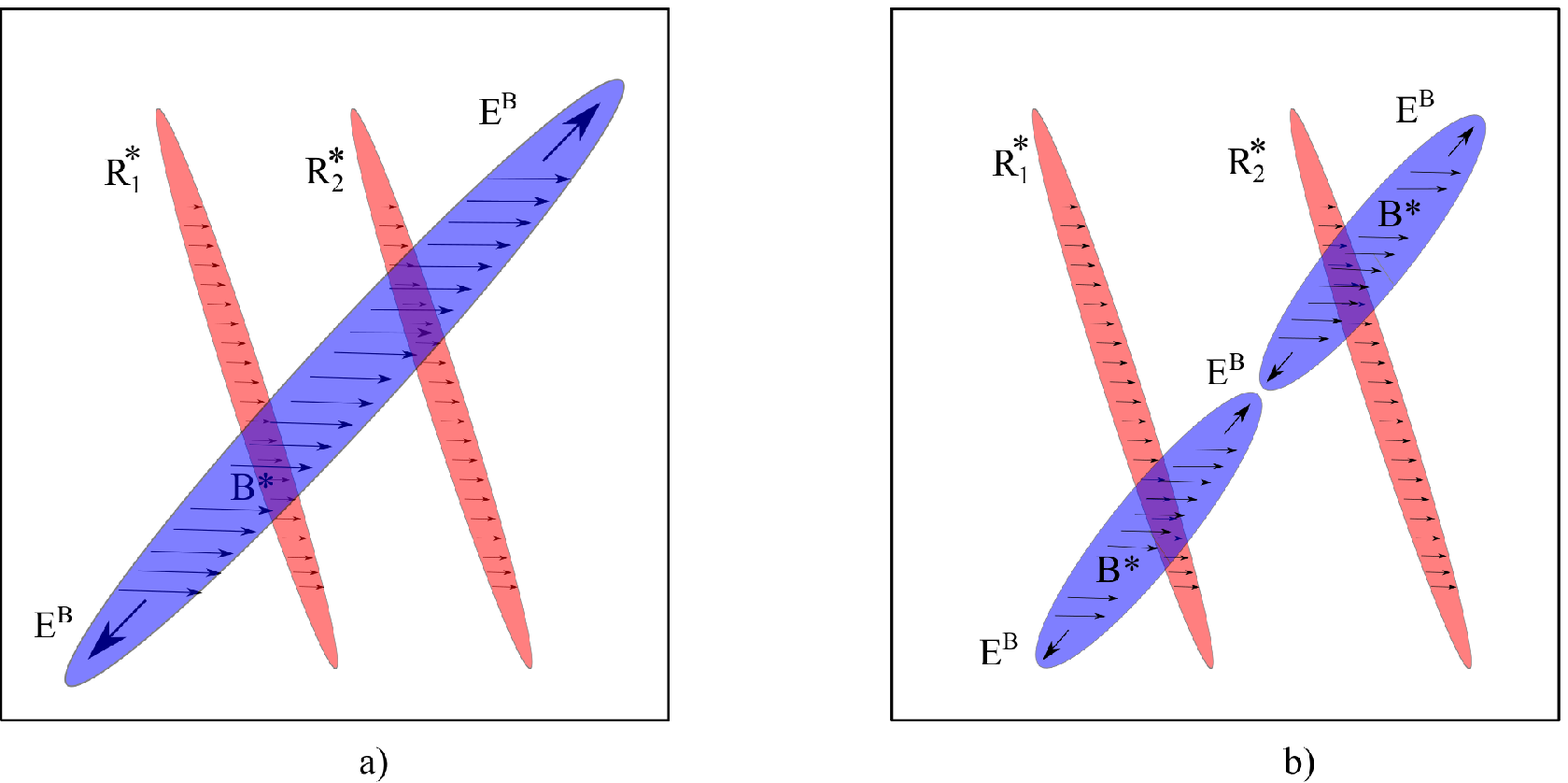}$$
\caption{\label{fig:FigureCalc8}} 
\end{figure}

\begin{lemma} \label{color reversal}  We have the following equalities in $\pi_2 \Emb(I,M; I_0)$ where $\bar B$ denotes the reversal of $B$.  
$$[(B,R)]=[(R,\bar B)]=[(\bar R,B)]=-[(R,B)]=[(\bar B, \bar R)]$$\end{lemma}

\begin{proof}  Start with $(B,R)$ as in Figure \ref{fig:FigureCalc9} a) and then rotate clockwise to obtain Figure \ref{fig:FigureCalc9} b) which represents the same element of $\pi_2 \Emb(I,M; I_0)$.  Finally, homotopically reparametrize the rotated $F^B$ and $F^R$ to obtain Figure \ref{fig:FigureCalc9} c) which depicts $ (R, \bar B)$.  Since $B$ and $R$ have disjoint supports in both domain and range, this reparameterization can be done independently on $F^B$ and $F^R$ subject to the condition that during the homotopy all points that are both blue and red lay in the original $u\in [1/4, 3/4] $ region.  For this reason the arrows can be made horizontal in only one direction.  This argument applied to a counterclockwise rotation shows that $[(B,R)]=[\bar R, B]$.  Next note that the standard homotopy from $1_{I_0}$ to $R*\bar R$ is through loops that have domain and range supports disjoint from that of $B$ and hence $0=[(1_{I_0},B)]=[(R*\bar R, B)]=[(R,B)]+[(\bar R, B)]$ thereby establishing the third inequality.  The final equality follows from the earlier ones.\end{proof}

\begin{figure}[ht]
$$\includegraphics[width=13cm]{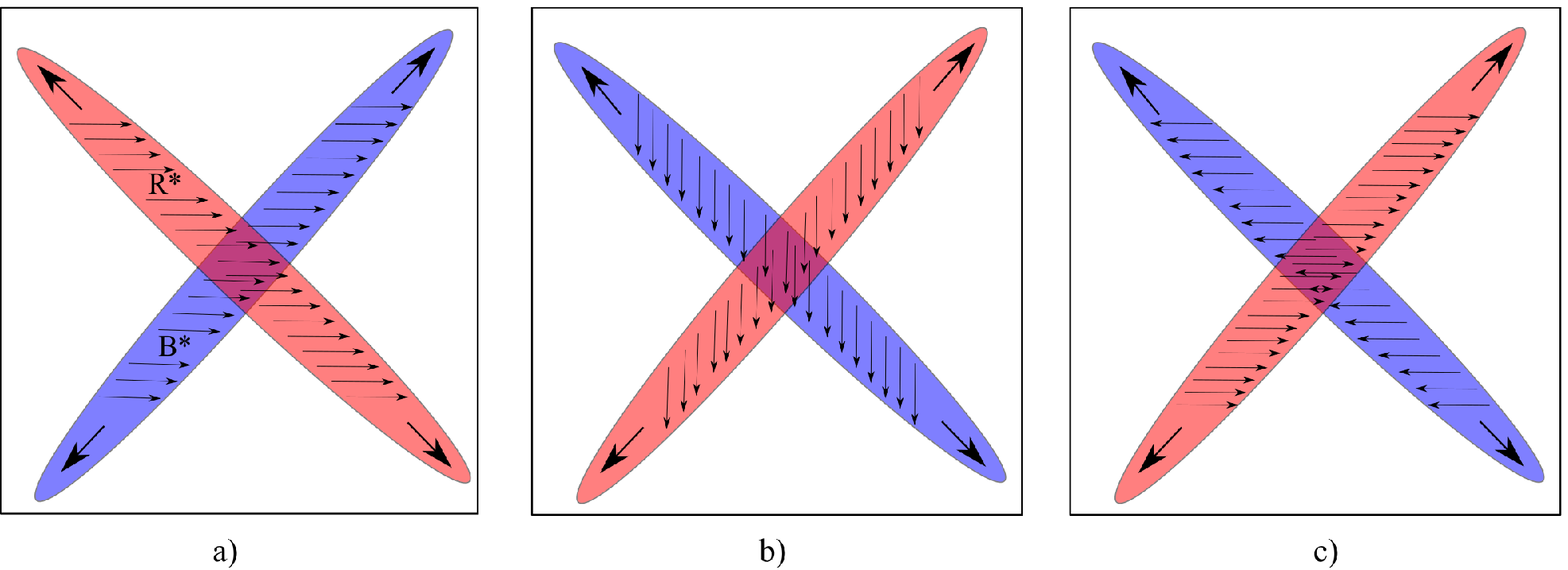}$$
\caption{\label{fig:FigureCalc9}} 
\end{figure}

\begin{example}  If $B, R \in \Omega \Emb(I,M; I_0)$ are defined by abstract chord diagrams $C_B, C_R$ such that $p(C_B)\cap p(C_R)=\emptyset$, then $(B,R)$ is separable.  That is because the domain support of $B$ is contained in $p(C_B)$ and the range support is contained in a small neighborhood of the chords and we can assume that the  chords  of $C_B$ and $C_R$ are disjoint. Our fundamental example is the separable pair denoted $G(p,q)$ defined by the chord diagram shown in Figure \ref{fig:FigureCalc6} a), where $C_B$ (resp. $C_R$) is defined by the blue (resp. red) chords.  It's class in $\pi_2 \Emb(I,M; I_0)$ is called the $(p,q)$-\emph{primitive class}.  A second basic example is the separable pair denoted $G^*(p,q)$ shown in Figure \ref{fig:FigureCalc6} b) and called the \emph{symmetric $G(p,q)$}.\end{example}


\begin{definition} We say that the abstract chord diagrams $C_1$ and $C_2$ are \emph{disjoint} if $p(C_1)\cap p(C_2)=\emptyset$.  The \emph{sum} of disjoint chord diagrams is the union of the data of the diagrams. \end{definition}

\begin{lemma}  \label{chord moves} Let $B_i$ and $R_i$ be represented by pairwise disjoint abstract chord diagrams $C_{B_i}$ and $C_{R_i}$ for $i=0,1$.  

i) (reversal) If $C_{B_1}$  is obtained from $C_{B_0}$ by a reversal, then $(B_0, R_0)$ is separably homotopic to $(B_1, R_0)$.  If $B_0$ and $R_0$ are homotopically trivial, then  $[(B_0, R_0)]=[(B_1, R_0)]$.

ii)  (exchange) If $C_{B_1}$ is obtained from $C_{B_0}$ by an exchange along an interval whose interior is disjoint from $p(C_R)$, and both B$_0, R_0$ are homotopically trivial, then $[(B_0, R_0)]=[(B_1,R_0)]$.

iii) (sliding) If $C_{B_1}$ (resp. $C_{R_1}$) is obtained from $C_{B_0}$ (resp. $C_{R_0}$) by sliding a chord over either a blue or red chord, then $[(B_1, R_0)]=[(B_0, R_0)]$ (resp. $[(B_0, R_1)]=[(B_0, R_0)]$.

iv) (addition/cancellation)  If $C_{B_1}$, $C_{B_0}$, $C_{R_0}$ are disjoint chord diagrams with $C_{B_2}$ the sum of $C_{B_1}$ and $C_{B_0}$ and $[(B_1,R_0)]=0$, then $[(B_2, R_0)]=[(B_0, R_0)]$. The similar statement holds with $R$ and $B$ switched.

v) (isotopy) If $C_{B_t}$ (resp. $C_{R_t}$) is an isotopy of $C_{B_0}$ to $C_{B_1}$ (resp. $C_{R_0}$ to $C_{R_1})$ and for every $t$, $p(C_{B_t})\cap p(C_{R_t})=\emptyset$, then $[(B_1, R_1)]=[(B_0, R_0)]$.\end{lemma}

\begin{proof}  i) and v) are immediate.  iv) follows from Proposition \ref{bilinearity}.  Being the composite of an isotopy and two splittings, it follows from Remark \ref{splitting support} that the support of a sliding is disjoint from the chord being slid over. Therefore the  homotopy from $(B_0, R_0)$ to $(B_1, R_0)$ is separable.  This proves iii).  To prove ii) consider the representative Figure \ref{fig:FigureCalc11} a) which shows a chord diagram with the exchange interval indicated.  The chord diagram of Figure \ref{fig:FigureCalc11} b) has the same $C_R$, but the blue chord diagram $C$ is given by a parallel pair of oppositely signed arcs representing the same group element.  Furthermore, the intervals between their initial points and between their final points contains no points of $p(C_R)$.  Apply the undo homotopy, Lemma \ref{undo} to $\alpha_C$ to reduce $C$ to the trivial diagram.  Since the domain and range support of this undo homotopy is disjoint from that of $\alpha_{C_R}$, it follows that $[(\alpha_C, \alpha_{C_R})]=0$.     Using bilinearity we obtain Figure \ref{fig:FigureCalc11} c).  We again use the undo homotopy to cancel a pair of parallel oppositely signed blue chords to obtain the diagram of Figure \ref{fig:FigureCalc11} d) which has a pair of chord diagrams isotopic to the exchanged pair. \end{proof} 

\begin{remark} \label{undo cancel} A special case of iv) is when $C_{B_0}$ contains a pair of parallel oppositely signed chords with the same group elements,  $p(C_{R_0})$ is disjoint from the parallelism between their endpoints and $C_{B_1}$ is the chord diagram with the parallel pair deleted, in which case $[(B_0,R_0)]=[(B_1,R_0)]$.\end{remark}

\begin{figure}[ht]
$$\includegraphics[width=9cm]{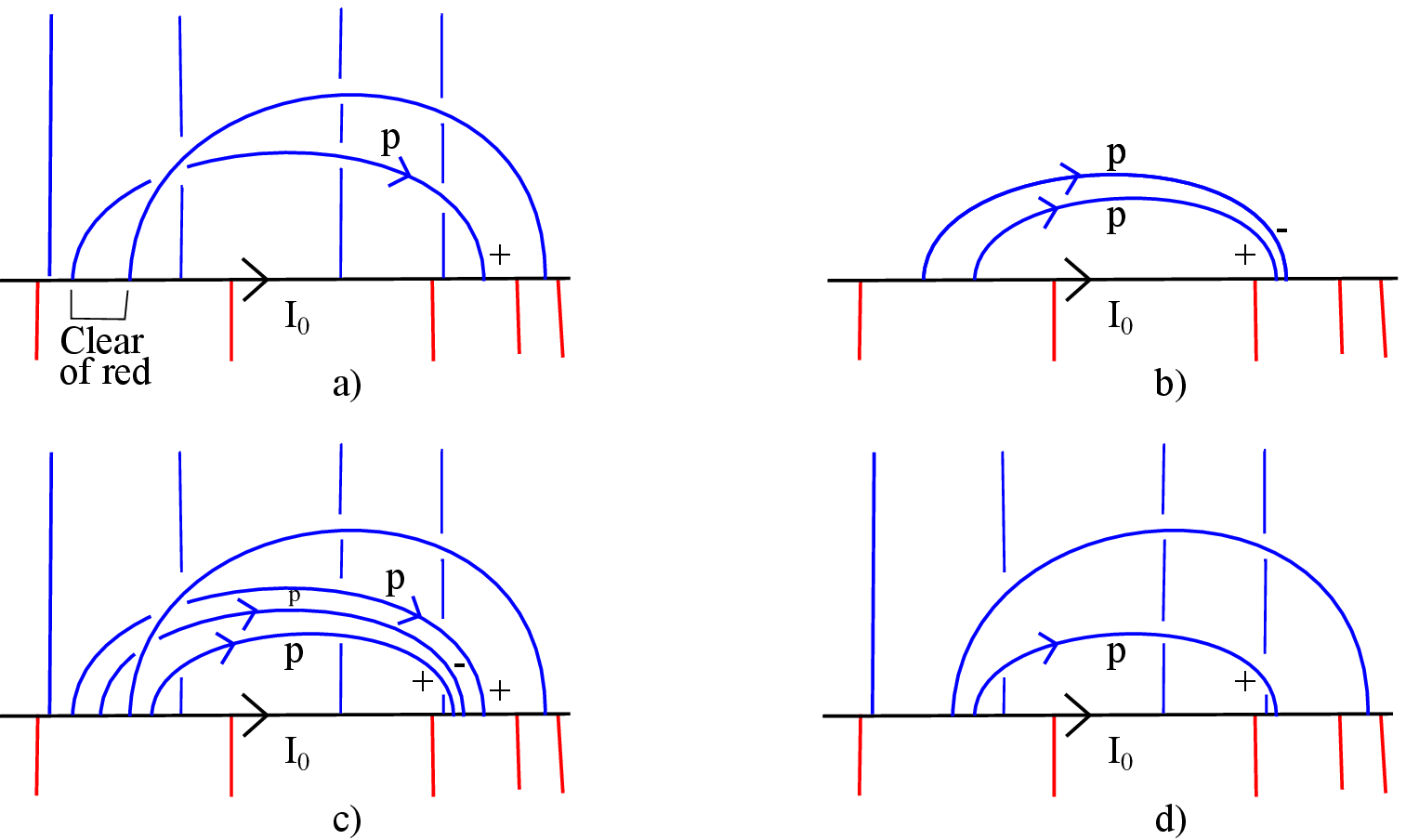}$$
\caption{\label{fig:FigureCalc11}} 
\end{figure}

\begin{definition}  Let $C_B$ and $C_R$ be disjoint chord diagrams on $I_0$.  We say that $(C_B, C_R)$ is \emph{combinatorially equivalent} to $(C'_B, C'_R)$ if one can be transformed to the other through disjoint chord diagrams that change by reversals, isotopies and allowable exchange moves.\end{definition}

We have the following analogues of Lemmas \ref{combinatorial}, \ref{basepoint} which, in combination with Lemma \ref{basepoint} have similar proofs.

\begin{lemma} \label{combinatorial two} If $(C_B, C_R)$ and $(C'_B, C'_R)$ are combinatorially equivalent, then $[(C_B, C_R)]=[ (C'_B, C'_R)] \in \pi_2(\Emb(I,M; I_0))$.\qed\end{lemma}

\begin{lemma} (Change of Basepoint)  \label{basepoint two} Let  $I_0$ and $I_1$ be path homotopic embedded arcs in the 4-manifold $M$.  A pair of disjoint chord diagrams $(C^0_B, C^0_R)$ on $I_0$ induces a pair $(C^1_B, C^1_R)$ on $I_1$ unique up to combinatorial equivalence. Furthermore, $[(C^1_B, C^1_R)]\in \pi_2 \Emb(I,M;I_0)$ is obtained from $[(C^0_B, C^0_R)]\in \pi_2 \Emb(I,M;I_1)$ by a change of basepoint isomorphism which is independent of the path isotopy from $I_0$ to $I_1$. \qed\end{lemma}

\begin{lemma} \label{symmetric g} $[G^*(p,q)]=-[G(p,p-q)]$.\end{lemma}
\begin{proof}  The argument is given in Figure \ref{fig:FigureCalc11c}.  Figure \ref{fig:FigureCalc11c} a) shows $G^*(p,q)$.  One of the red chords is reversed to obtain Figure \ref{fig:FigureCalc11c} b).  That chord is slid over the adjacent blue chord to obtain Figure \ref{fig:FigureCalc11c} c).  A red chord is then reversed and the resulting pair of oppositely signed red chords with group element $p$ are cancelled using the undo homotopy.   \end{proof}

\begin{figure}[ht]
$$\includegraphics[width=13cm]{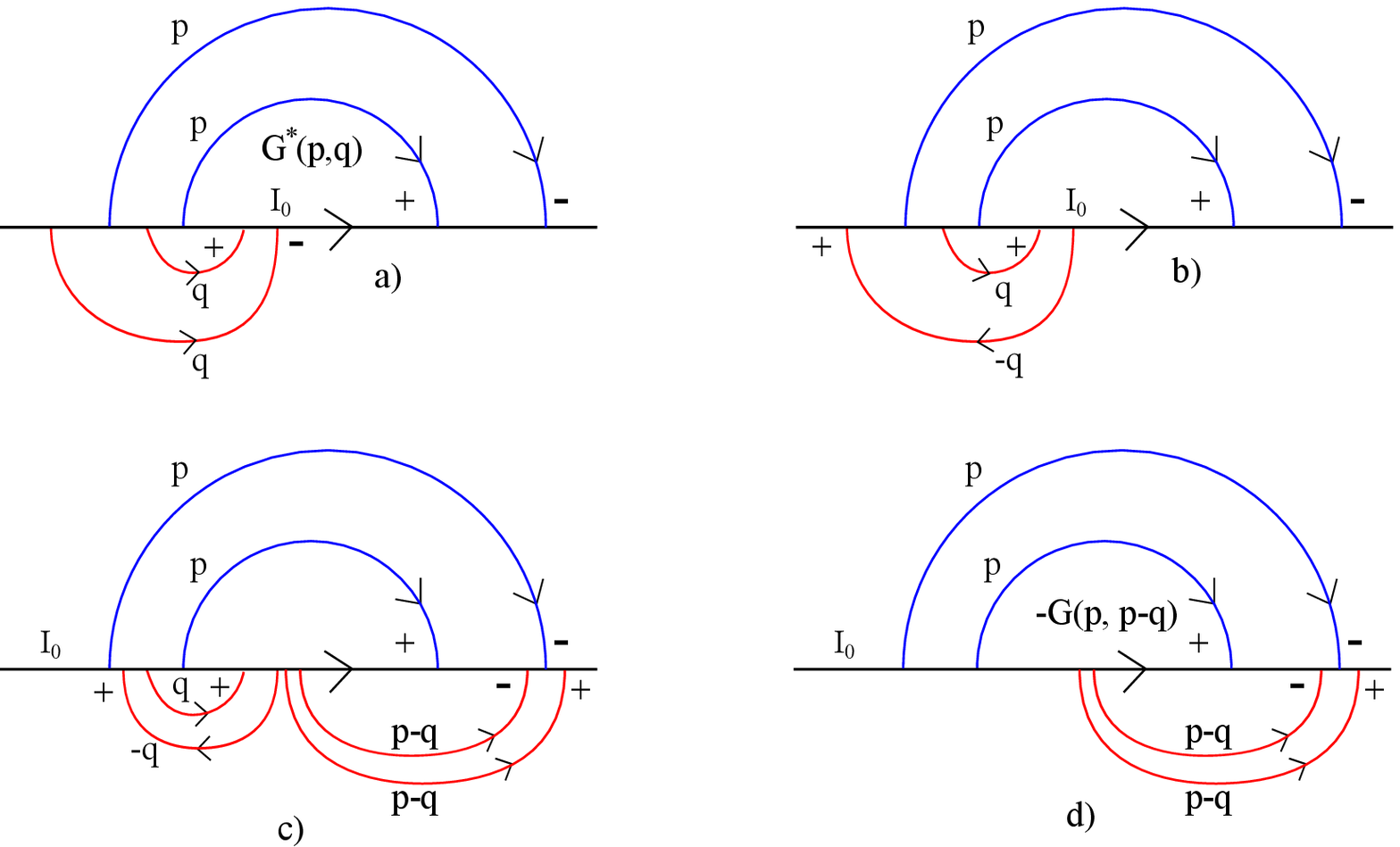}$$
\caption{\label{fig:FigureCalc11c}} 
\end{figure}

\subsection{Fundamental Classes} We have already introduced the primitive family $G(p,q)$.  We now introduce the \emph{elementary classes} $E(p,q)$ and \emph{double classes} $D(p,q)$ and compute $[E(p,q)]$ in terms of $[G(p,q)]$'s and $[D(p,q)]$ in terms of $[E(p,q)]$'s.  

\begin{figure}[ht]
$$\includegraphics[width=9cm]{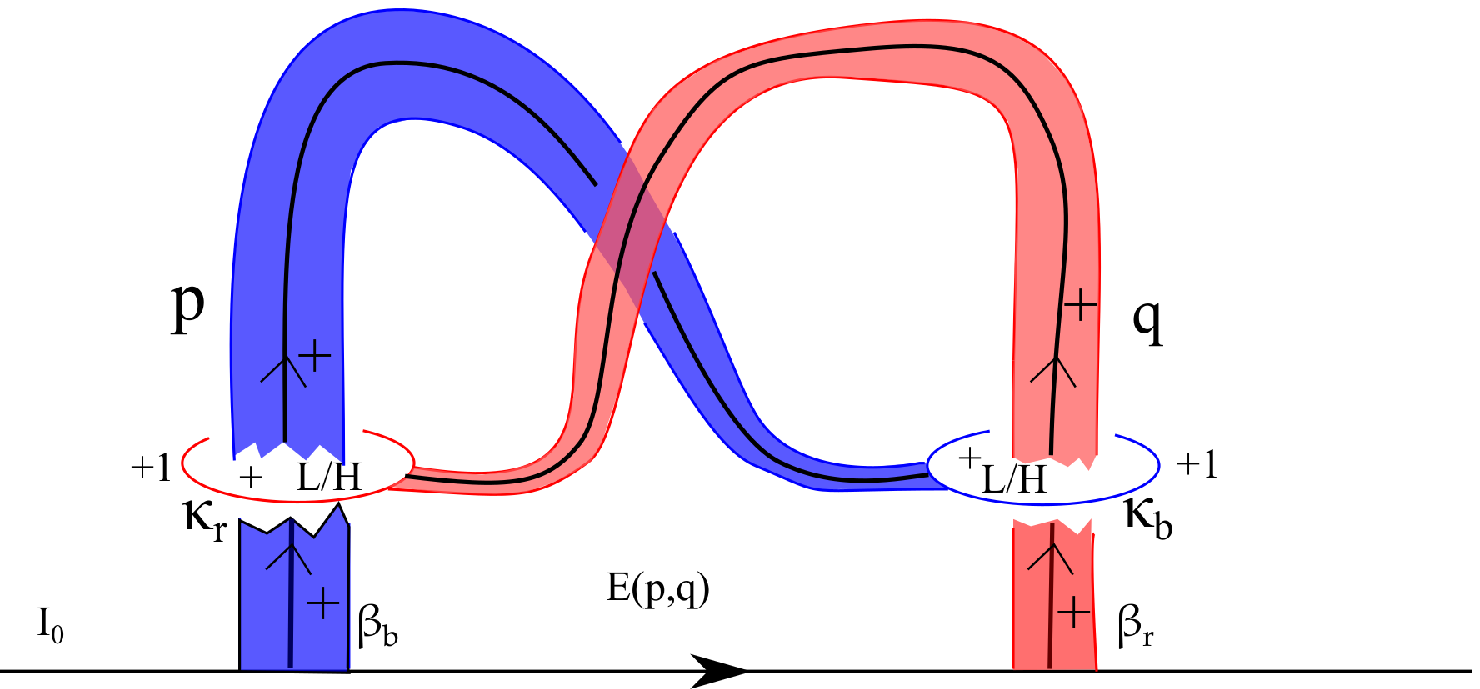}$$
\caption{\label{fig:FigureCalc12}} 
\end{figure}

\begin{definition} \label{eqp definition} Define the \emph{standard elementary family} $E(p,q) $by the band/lasso diagram as in Figure \ref{fig:FigureCalc12}.  In words $E(p,q) = (B,R)$ where $B$ (resp. $R$) is defined by the band $\beta_b$ (resp. $\beta_r$) and lasso $\kappa_b$ (resp. $\kappa_r$), where the  base of $\beta_b$ appears on $I_0$ before the base of $\beta_r$.  Also $\kappa_b$ (resp. $\kappa_r$) links the core of $\beta_r$ (resp. $\kappa_b)$ very close to its base and both spinnings are positive.   When the top of the core of $\beta_b$ (resp. $ \beta_r$) is pushed slightly to lie on $I_0$, then the core represents $p$ (resp. $q$) $\in \pi_1(S^1\times B^3; I_0)$.  

More generally, an \emph{$E(p,q)$ family} is one of the form $(B,R)$ where each of  $B$ and $R$ are represented by a single band and lasso where the lasso of one links the core of the other very close to its base.  \end{definition}

\begin{remarks} \label{epq signs}  i)  Up to the ordering of their bases on $I_0$, twisting of the bands and $L/H, H/L$ designation an $E(p,q)$ family is isotopic to the standard $E(p,q)$ for appropriate $p$ and $q$, hence up to sign and permutation of $p$ and $q$ such a family represents $[E(p,q)]$.  

ii) Recall that the sign of a spinning is determined by a rotation about an oriented sphere, which links an oriented arc $\tau$.  Here, the oriented arcs are the oriented cores of the bands $\beta_b$ and $\beta_r$.  

iii) To over emphasize, the order of the base of the bands on $I_0$ and the homotopy classes of their cores determine $p$ and $q$.  Given such bands, the sign of $[E(p,q)]$ is determined by the color of the first band, and the signs of the spinnings.  The sign of a given spinning is determined by the  $L/H, H/L$ data and the $<$band core, lasso$>$ algebraic intersection number from Lemma \ref{sign intersection}. Thus the sign is given by the parity of the five possible differences from the standard model.  In the standard model, shown in Figure \ref{fig:FigureCalc12}, we have $L/H$ with +1 intersection number in both cases.  That figure also shows the oriented cores of the bands as well as the orientations of the bands and the lasso discs. \end{remarks}

\begin{lemma}\label{epq} $[E(p,q)]=-[G(-q,p)]+[G(p,-q)]=-[G^*(-q,p)]+[G^*(p,-q)]$.\end{lemma}

\begin{proof}  We first  isotope $E(p,q)$ as in Figure \ref{fig:FigureCalc13} a), then homotope it to  a chord diagram as in Figure \ref{fig:FigureCalc13}  b), where we use Lemma \ref{sign intersection} to determine the signs of the chords.   Use Lemma \ref{chord moves} i) to reverse the red chords and then \ref{chord moves} iii) to isotope the chords to obtain Figure \ref{fig:FigureCalc13} c).    Figure \ref{fig:FigureCalc13} d) is the result of applying two exchanges, Lemma \ref{chord moves} ii).  A representative of $[G(-q,p)]$ is shown in Figure \ref{fig:FigureCalc13} e).  Note that it differs from the standard $G(-q,p)$ by a color change and replacing $R$ by $\bar R$, i.e. two sign changes at the $\pi_2$ level.  We use bilinearity  to add \ref{fig:FigureCalc13} d) and e) and then use the undo homotopy  Lemma \ref{undo} to cancel a pair of parallel red chords to obtain \ref{fig:FigureCalc13} f).  Adding Figure \ref{fig:FigureCalc13} g) which is a $[-G(p,-q)]$ and then doing an undo homotopy to the red chords we obtain Figure \ref{fig:FigureCalc13} h)  which is homotopically trivial, since another undo homotopy eliminates the blue chords.  It follows that $[E(p,q)]+[G(-q,p)]-[G(p,-q)]=0$.

The proof that the left hand side equals the right hand side follows similarly.  Here we add a $-[G^*(p,-q)]$ as shown in Figure \ref{fig:FigureCalc13} i) to Figure \ref{fig:FigureCalc13} d) which  gives Figure \ref{fig:FigureCalc13} j).  Adding to that a $[G^*(-q,p)]$ as shown in Figure \ref{fig:FigureCalc13} k) gives a chord diagram representing the trivial element.  It follows that $[E(p,q)]=-[G^*(-q,p)]+[G^*(p,-q)]$.  \end{proof}

\begin{figure}[ht]
$$\includegraphics[width=11cm]{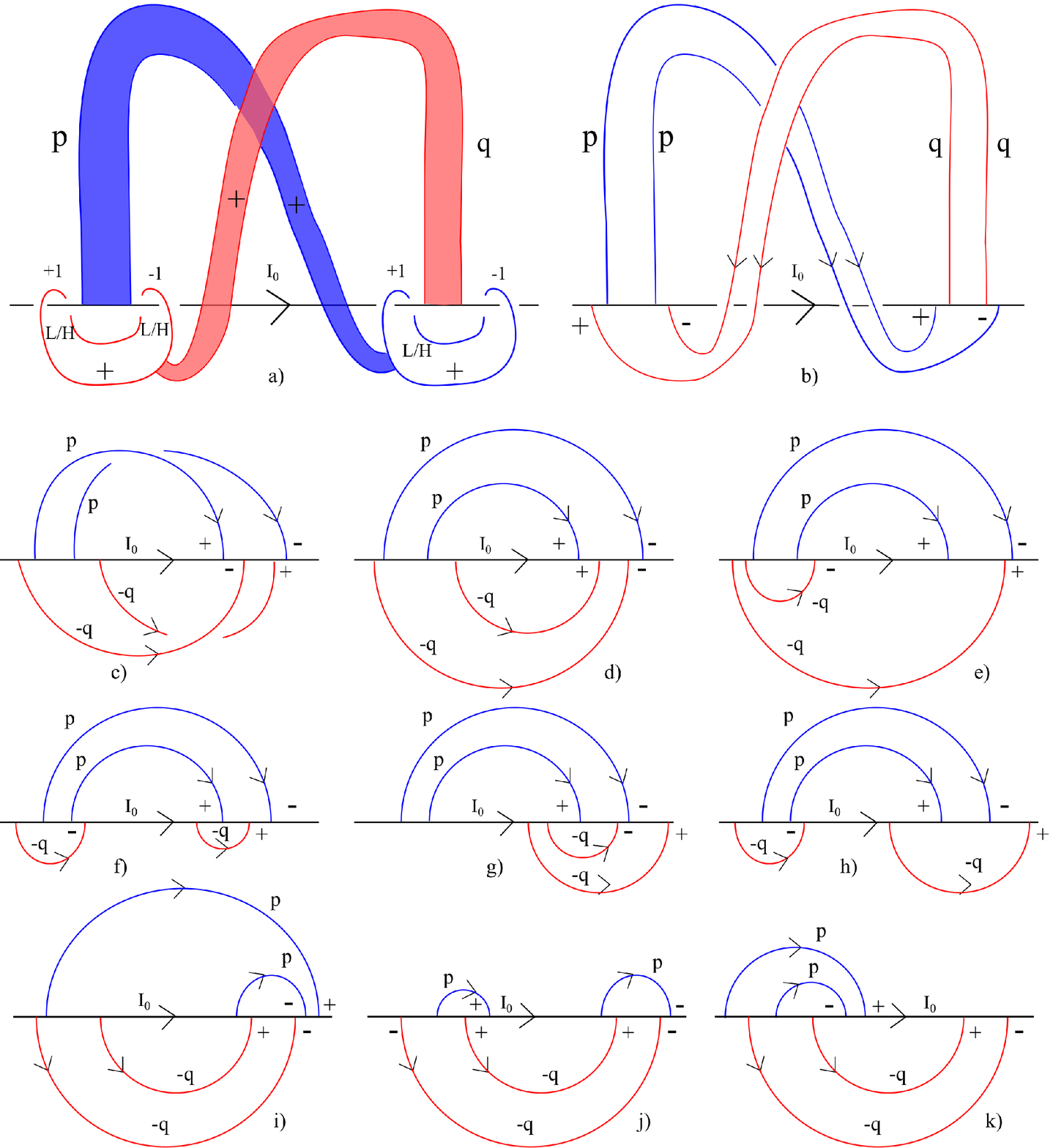}$$
\caption{\label{fig:FigureCalc13}} 
\end{figure}

\begin{corollary} \label{epq relations} i) $[E(p,q)]=0$ if $q=-p$

ii) $ [E(p,q)]=-[E(-q,-p)]$

iii) (Hexagon Relation) $[G(p,q)]+[G(p,p-q)]=[G(q,p)]+[G(q,q-p)]$. \end{corollary}

\begin{proof}  The first two conclusions are immediate.  The third follows from the second equality of Lemma \ref{epq}, Lemma \ref{symmetric g} and a change of variables.\end{proof}

\begin{definition} \label{double} Define the \emph{standard double family} $D(p,q)$ in band/lasso notation as in Figure \ref{fig:FigureCalc14}.  In words, it is of the form $(B,R)$ where $B$ consists of the band $\beta_b$ and lasso $\kappa_r$.  $R$ consists of a parallel cancel pair defined by the bands $\beta_r, \beta'_r$ and lassos $\kappa_r, \kappa'_r$.  Base$(\beta_r)$, base$(\beta_b)$, base$(\beta'_r)$ appear in order along $I_0$.  $\kappa_r,\kappa_r'$ link band$(\beta_b)$ just above its base and $\kappa_b$ links band$(\beta_r)$, band$(\beta'_r)$ at the branch loci.  When the top of the bands are pushed to lie on $I_0$ then $\beta_b$ represents $p$ and both $\beta_r, \beta'_r$ represent $q$.  Also, the spinning of $\sigma(\beta_b,\kappa_b)$ about the oriented cores of $\beta_r, \beta_r'$ is positive.  The spinning of $\sigma(\beta_r, \kappa_r$) (resp. $\sigma(\beta_r',\kappa_r')$ about the oriented core of $\beta_b$ is positive (resp. negative).  

More generally, we call a family $(B,R)$ a \emph{$D(p,q)$ family} if up to isotopy it is represented by bands and lassos as a standard $D(p,q)$ family, up to switching the roles of $B$ and $R, L/H H/L$ designations and twisting of the bands.    \end{definition}

\begin{figure}[ht]
$$\includegraphics[width=7cm]{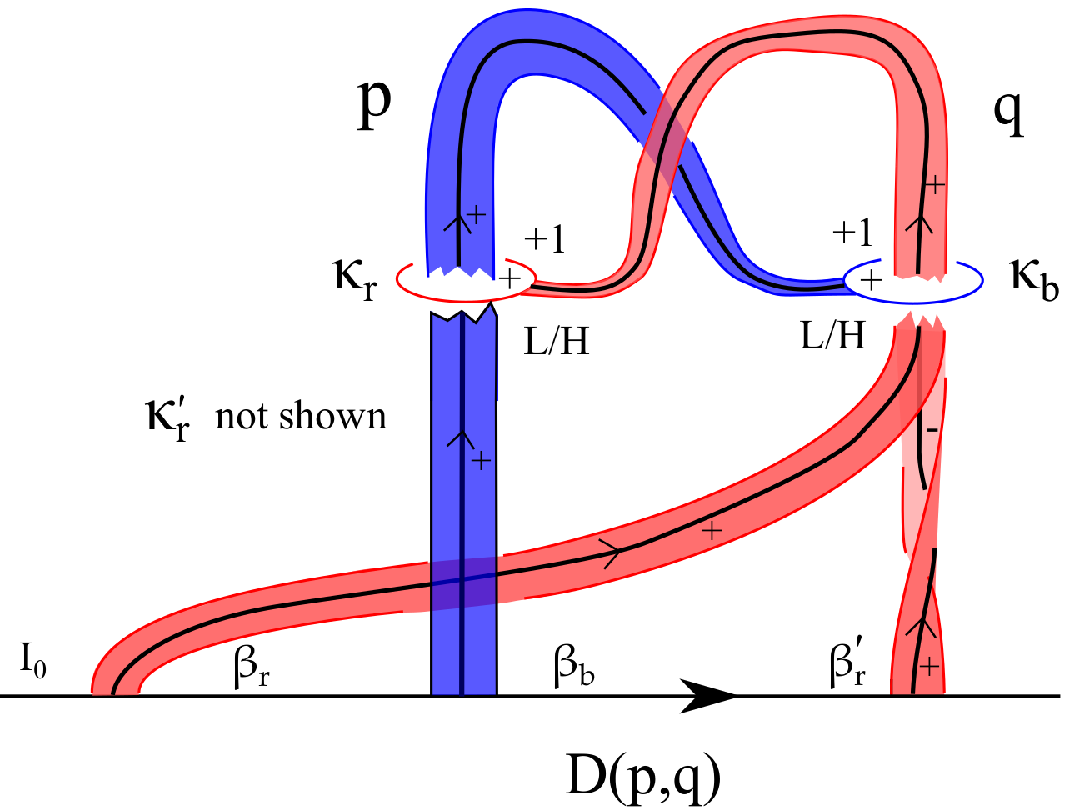}$$
\caption{\label{fig:FigureCalc14}} 
\end{figure}

\begin{remarks} \label{dpq signs} i)  Up to sign and permutation of $p$ and $q$ a $D(p,q)$ family represents the class $[D(p,q)]$. 

ii) The order of the base of the bands and the homotopy classes of their cores determine $p$ and $q$.  Given such bands, the sign of $[D(p,q)]$ is determined by the color of the first band on $I_0$, and the signs of the spinnings.  Here, we need not consider the spinning from $\beta_r', \kappa_r' $since it is always opposite that of $\beta_r, \kappa_r$.  If we use $L/H, H/L$ data and intersection numbers to determine the spinnings, then as before the final sign is given by the parity of the five possible differences from the standard model.  \end{remarks}

\begin{lemma} \label{dpq} $[D(p,q)]=-[E(q,p)]-[E(p,q)]$.\end{lemma}

\begin{proof}  First note that $D(p,q)$ is a separable family.  By construction $R$ is a concatenation of $\sigma(\beta_r, \kappa_r)$ and $\sigma(\beta'_r, \kappa'_r)$ and each is homotopically trivial.  Applying bilinearity to Figure \ref{fig:FigureCalc14} we obtain $(B_1, R_1)$ and $(B_2, R_2)$ of Figures \ref{fig:FigureCalc15} a), b).  $(B_1, R_1)$ is homotopic to the standard $E(q,p)$ after a blue/red switch, thus $[(B_1,R_1)]=-[E(q,p)]$.  Similarly, the family in Figure \ref{fig:FigureCalc15} b) gives $-E(p,q)$.  Here there is a single sign change from the standard $E(p,q)$ arising from $< \text{band}(\beta_b), \kappa'_r>=-1$.\end{proof}

\begin{figure}[ht]
$$\includegraphics[width=12cm]{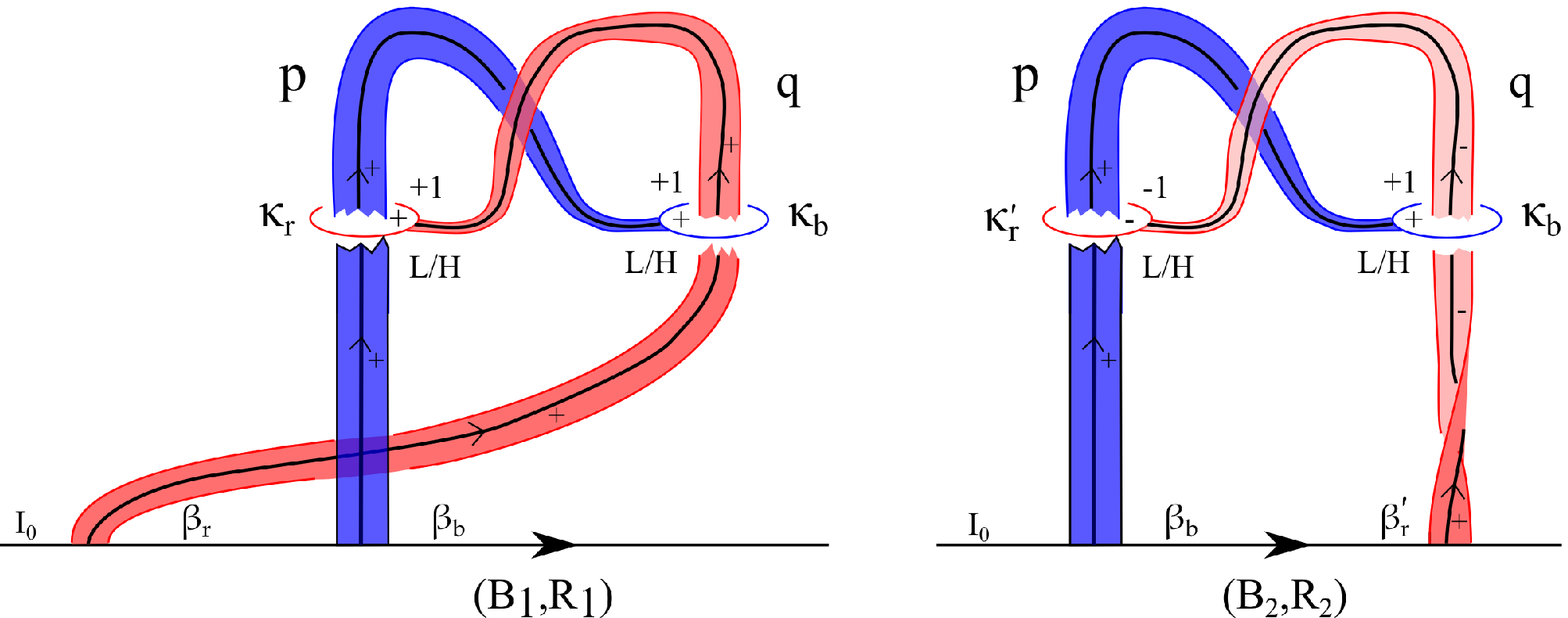}$$
\caption{\label{fig:FigureCalc15}} 
\end{figure}

\begin{corollary} \label{dpq gpq}$[D(p,q)]=-[G(q,-p)]+[G(-q,p)]-[G(p,-q)] + [G(-p,q)].$\end{corollary}

\begin{corollary}\label{dpq relations} i) $[D(p,-p)]=0$

ii) $[D(p,q)]=[D(q,p)]$

iii) $[D(-p,-q)]=-[D(p,q)]$\end{corollary}


\section{Barbell Diffeomorphisms} \label{barbell section}

\begin{figure}[ht]
$$\includegraphics[width=7cm]{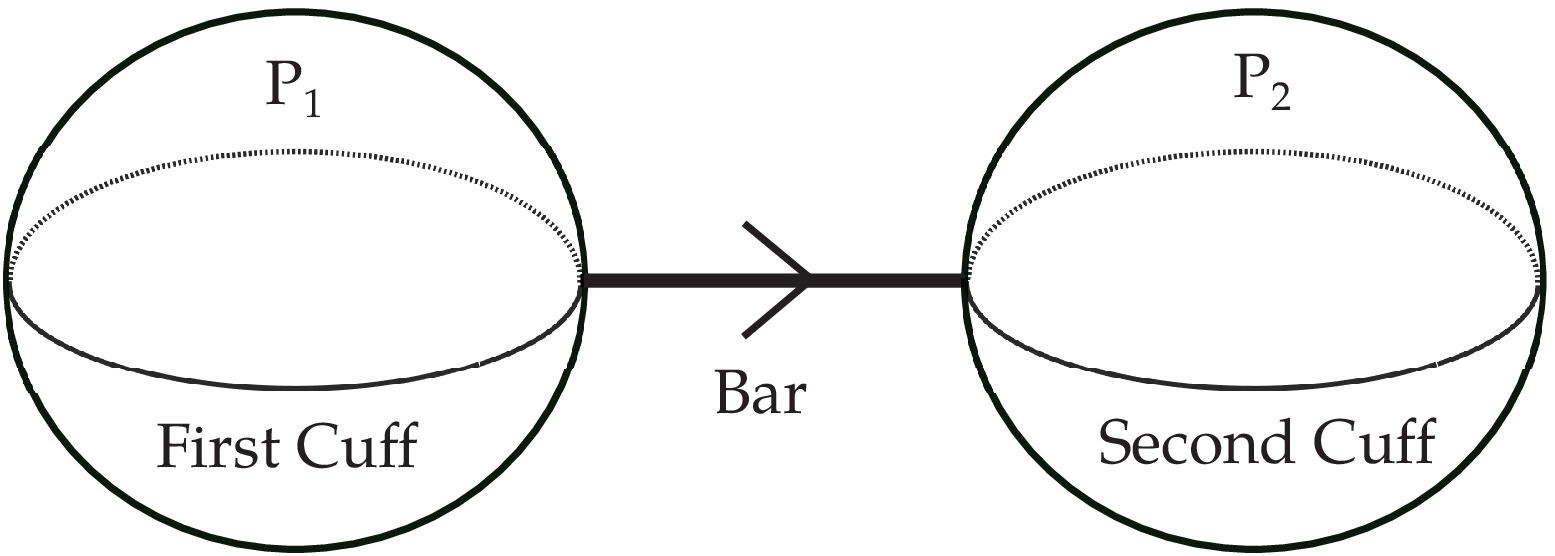} \hskip 1cm \includegraphics[width=7cm]{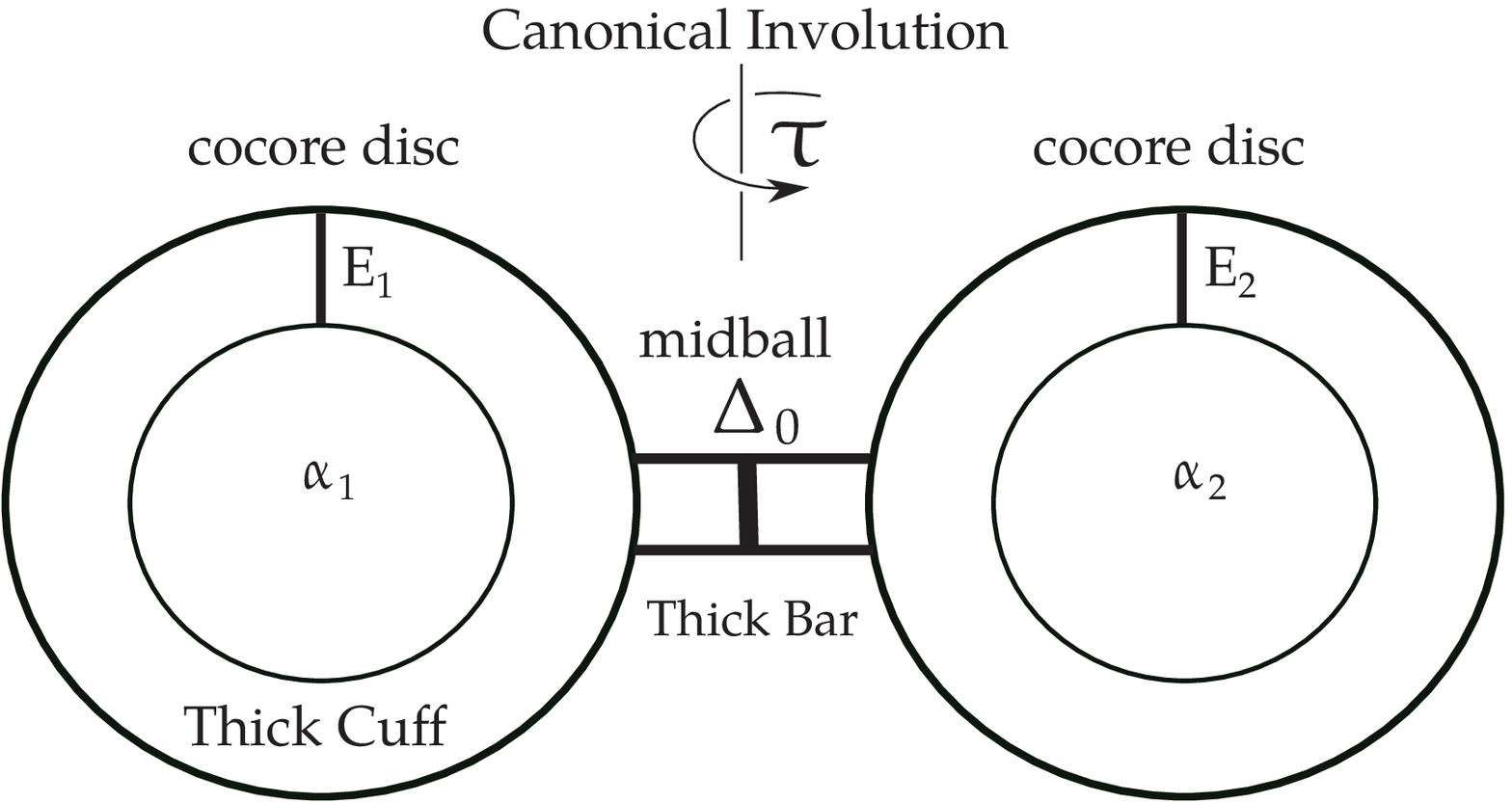}$$ 
\centerline{(a) \hskip 7.5cm (b)}
\caption{\label{fig:Fig.B1} (a) A Barbell.
(b) The $t=0, y=0$ slice of the model thickened barbell.}
\end{figure}

In  \S 2 we defined the families $\{\alpha_k\}$ and $\{\theta_k\}$ each of which generated $\pi_1 \Emb(S^1, S^1\times S^3; S^1_0)$, where $S^1_0$ is the standard vertical $S^1 \subset S^3$.  In \S3 we noted that isotopy extension applied to an $\alpha_k, k\ge 1$ or $\theta_k, k\ge 2$ gives rise to an element of $\Diff(S^1\times B^3 \fix \partial)$ which is conjecturally isotopically nontrivial.  In this section we explicitly identify these diffeomorphisms, the foundational observation being that each is supported in a single $S^2\times D^2\natural S^2\times D^2\subset S^1\times B^3$.  We call such a space a \emph{thickened barbell} $\mN\mB$, the \emph{barbell} $\mB$ itself, a spine, being the disjoint union of two oriented 
2-spheres together with an oriented arc that joins them.  See Figure \ref{fig:Fig.B1}(a).  We will see that 
$\pi_0 \left( \Diff(\mN\mB \fix \partial)/\Diff(B^4\fix \partial) \right)= \BZ$ and is generated by the \emph{barbell map}.  The barbell map is the result of applying isotopy extension to a loop of a pair of properly embedded arcs in $B^4$, where a subarc of the first arc is standardly spun around the second and then restricting to the closed complement of these arcs.   

By embedding the barbell in another space $M$ and pushing 
forward the barbell map we obtain elements of $\Diff(M\fix \partial M)$ an operation we call 
\emph{barbell implantation}.  Implantations yielding elements of $\Diff(S^1\times B^3\fix \partial)$ induced from the $\theta_{k}$'s 
and $\alpha_{k}$'s will be described. Other implantations that we will prove to be isotopically nontrivial in both  $\Diff(S^1\times B^3\fix \partial)$ and $\Diff_0(S^1\times S^3)$ modulo $\Diff(B^4 \text{ fix } \partial)$ will be given in the next section.    For $M=S^4$ we offer explicit candidates including the ones arising from the $\theta_k$'s via filling $S^1\times B^3$ with an $S^2\times D^2$.  We will show 
how to construct the image of $\{x_0\}\times B^3 \subset S^1\times B^3$ and  $\{x_0\}\times S^3 \subset S^1\times S^3$ under implantation,  
thereby exhibiting possibly knotted nonseparating 3-balls (resp. 3-spheres) in $S^1\times B^3$ (resp. in $S^1\times S^3$), inducing knotted 3-balls in $S^4$.  Some of these will be shown to be non trivial in later sections.  For $S^4$ we will show that 
some barbell maps are order at most 2.  Finally, we note that barbell maps extend and 
generalize the graph surgery diffeomorphisms of Watanabe \cite{Wa1}.

We start by defining the \emph{model barbell} and its thickening in $\BR^4$, 
along the way establishing conventions and identifying various subspaces.

\begin{definition}  The \emph{model barbell} $\mB$ is the union two 2-spheres in $\BR^3$ of radius 
$\frac{1}{4}$ respectively centered $(1,0,0)$ and $(2,0,0)$ together with the arc 
$[1.25, 1.75]\subset \BR\times \{0\}^2$ called the \emph{bar} that points from the first 
sphere $P_1$ to the second called $P_2$.   These spheres are called the \emph{cuffs} with 
$P_1$ (resp. $P_2$) being the \emph{first} (resp. \emph{second}) cuff.  Construct the underlying 
space of a \emph{model thickened barbell} $\subset \BR^4$ by first taking an 
$\epsilon$-neighborhood $N_\epsilon^3(\mB)\subset\BR^3$ and then taking the product with 
$[-\epsilon,\epsilon]$.  Here $\epsilon>0$ and is small.  The 
space $N(P_i):=N_\epsilon^3(P_i)\times [-\epsilon,\epsilon]$  is called a \emph{thick cuff} and the 
closed complement of the thick cuffs is called the \emph{thick bar}.  The properly embedded 2-disc 
$E_1\subset \mN\mB$ normal to $P_1$ and centered at $(1,0,\frac{1}{4},0)$ is called the 
\emph{cocore disc} to $P_1$.  $E_2$ centered at $(2,0,\frac{1}{4},0)$ is defined and named analogously.  
Let $\tau$ denote the \emph{canonical involution} of $\mN\mB$ induced from the involution on $\BR^4$ 
corresponding to rotating the $x,y$ plane by $\pi$ at the point $(1.5,0,0,0)$ and fixing the $z,t$ coordinates.  
Define $\Delta_0\subset\mN\mB$ to be the properly embedded transverse 3-ball  to the dual arc at 
$(1.5,0,0,0)$ and call it the \emph{midball}.  See Figure \ref{fig:Fig.B1}(b).  
Let $\alpha_i = (i,0,0,t), t\in [-\epsilon, \epsilon], i=1,2$.  
Let $B_i = B^3\times [-\epsilon,\epsilon]$ where $B^3$ is the closed complementary 3-ball 
region of $N_\epsilon^3(P_i)$.  View $\alpha_i$ as the cocore of the 3-handle $B_i$.

We give $\mN\mB$ the orientation induced from $\BR^4$.  We orient $P_1$ so that at 
$(1,0,\frac{1}{4},0),\ (\epsilon_2, -\epsilon_1)$ is a positive basis and orient $E_1$ so that at 
$(1,0,\frac{1}{4},0)$,\ $(\epsilon_3,\epsilon_4)$ is a positive basis.  Here 
$(\epsilon_1, \epsilon_2, \epsilon_3, \epsilon_4)$ is the standard positive basis for 
$\BR^4$.  Note that $\langle P_1,E_1\rangle =1$.  Use $\tau$ to push forward the orientations on $P_1, E_1$ 
to ones on $P_2, E_2$.  Use $\epsilon_1$ to positively orient the bar and 
$(\epsilon_2,\epsilon_3, \epsilon_4)$ to orient $\Delta_0$. 

A \emph{barbell} in a 4-manifold is a subspace diffeomorphic to the model barbell with 
oriented cuffs and bar such that the cuffs have trivial normal bundles.  A \emph{thick barbell} 
in a 4-manifold is an embedded $S^2\times D^2\natural S^2\times D^2$ together with a barbell 
spine.  
\end{definition} 

The following is standard.

\begin{lemma}  
\begin{itemize}
\item $H_2(\mN\mB)\simeq\BZ^2$ with generators $[P_1], [P_2]$.
\item $H_2(\mN\mB,\partial \mN\mB)\simeq\BZ^2$ with generators $[E_1], [E_2]$.
\item $H_2(\mN\mB, \partial E_1\cup\partial E_2)\simeq\BZ^4$ with generators $[P_1], [E_1], [P_2], [E_2]$.\qed
\end{itemize}
\end{lemma}

\begin{construction} \label{barbell definition} We now define the \emph{barbell} map $\beta:\mN\mB\to \mN\mB$.  
Let $W= \mN\mB\cup B_1\cup B_2$ and hence $\mN\mB$ is the closed complement of  
$N(\alpha_1)\cup N(\alpha_2)\subset W$ where $N(\alpha_i)=B_i$.     
A loop $\lambda_s$ in $\Emb(N(\alpha_1), W)$ based at $N(\alpha_1)$ that is fixed near 
$(\partial \alpha_1)\times B^3$ and avoids $B_2$ induces a diffeomorphism of 
$\mN\mB \fix \partial \mN\mB$ by isotopy extension whose proper isotopy class depends only on the based homotopy class of that loop.  To define the path $\lambda$ we construct two discs 
$F_0, F_1$, where $F_1$ is a parallel copy of $E_1$ having been pushed off in the $\epsilon_1$ 
direction and $F_0$ is a disc that coincides with $F_1$ near its boundary and defined below.   
$\lambda_s$ is a loop that first sweeps across $F_0$ and then sweeps back using $F_1$.  
In what follows for $X\subset \mB$, let $X_{t_0}$ denote $X\cap \{t=t_0\}$.

\begin{figure}[ht]
$$\includegraphics[width=8cm]{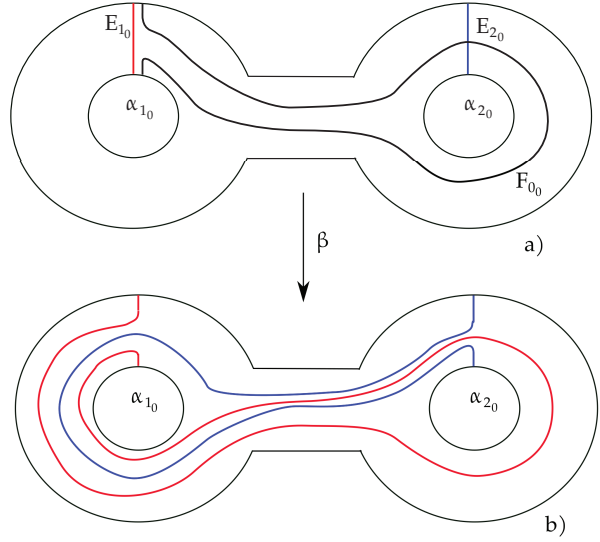}$$
\caption{\label{fig:Fig.B3} A slice of the barbell map.}
\end{figure}

We describe $F_0$ as follows.  It will intersect each $\mN\mB_t$ in an arc $F_{0_t}$ whose ends 
coincide with that of $F_{1_t}$.  Here $F_{0_0}$ lies in the $y=0$ plane and is shown in Figure \ref{fig:Fig.B3} a). 
 As $t$ increases (resp. decreases) $F_{0_t}$ slips to the $y>0$-side (resp. $y<0$-side) of 
 $\partial B_{2_t}$.  For $t$ close to $\pm \epsilon$, $F_{0_t}$ coincides with $F_{1_t}$.   

The images of $\beta(E_1)_0, \beta(E_2)_0$ are shown in Figure \ref{fig:Fig.B3} b). As t increases 
$\beta(E_1)_t$ (resp. $\beta(E_2)_t$) slips to the $y>0$-side (resp. $y<0$-side) of $\partial B_{2_t}$ (resp. $\partial B_{1_t}$). 
  Note that $\beta(E_i)$ coincides with $E_i$ near its boundary.  \end{construction}

\begin{remarks} \label{framing} i) An orientation preserving diffeomorphism of $\mN\mB$ supported 
in the thick bar that is a full right hand twist of the $x,z$-plane as one traverses the bar does 
not change the isotopy class of the barbell map, since it fixes $F_0\cup F_1$ setwise, hence does 
not change the based isotopy class of the loop $\lambda_s$. 

ii) The barbell map is the result of applying isotopy extension to a positive $\lambda$-spinning of $\alpha_1$ about  $\alpha_2$ where 
$\lambda$ is a straight arc from $\alpha_1$ to $\alpha_2$ and each $\alpha_i$ is oriented by $\epsilon_4$.    It follows from Lemma \ref{elementary homotopies} that the barbell map is also the result of a negative spinning of $\alpha_2$ about $\alpha_1$. \end{remarks}

\begin{lemma} In $H_2(\mN\mB, \partial (E_1\cup E_2))$,

i) $\beta_*[E_1]=[E_1]+[P_2]$

ii) $\beta_*[E_2]=[E_2]-[P_1]$\qed\end{lemma}

The next result follows by construction.

\begin{theorem} 
\label{involution inverse}$\tau\circ\beta\circ \tau$ is isotopic to $\beta^{-1}$.
\end{theorem}
 

\begin{theorem} $\pi_0 \Diff(\mN\mB \fix \partial)/ \pi_0 \Diff(B^4 \fix \partial)\simeq \BZ$ and is
 generated by the barbell map.\end{theorem}

\begin{proof}  Let $\psi \in \Diff(\mN\mB \fix \partial)$.  Since $\psi$ fixes 
$\partial \mN\mB$ pointwise it is orientation preserving.   Its effect on 
$H_2(\mN\mB,\partial (E_1\cup E_2))$ is as follows.

a) $\psi_*([P_i])=[P_i],\ i=1,2$

b) $\psi_*([E_1])=[E_1]+n[P_2]$ 

c) $\psi_*([E_2])=[E_2]-n[P_1]$ for some $n\in \BZ$.   

Since $\psi$ fixes $\partial \mN\mB$ 
a) follows.  Apriori, $\psi_*([E_1])=n_1[E_1]+n_2[P_1]+n_3[E_2]+n_4[P_2]$.  
Since $\psi$ fixes $\partial \mN\mB$,   $n_1=1$ and $n_3=0$.  Now double $\mN\mB$ to 
obtain $S^2\times S^2\sharp S^2\times S^2$ and extend $\psi$ to $\hat\psi$ so that 
$\hat\psi$ is the identity outside of $\mN\mB$.  In the double $P_1$ becomes a 
$S^2\times \{*\}$ and the doubled $E_1$ becomes the 2-sphere $\hat E_1=\{*\} \times S^2$, 
both in the first $S^2\times S^2$ factor.  Since $\hat E_1$ has trivial normal bundle 
as does $P_1$ and $P_2$ it follows that  $n_2=0$.  In a similar manner 
$\psi_*([E_2])=[E_2]+m[P_1]$.  Therefore if $n_4=n$, then 
$0=\langle E_1,E_2 \rangle = \langle \psi(E_1),\psi(E_2)\rangle$ 
and hence $m=-n$.  

Therefore, $f:= \beta^{-n}\circ\psi$ acts trivially on 
$H_2(\mN\mB,\partial(E_1\cup E_2))$ and hence for $i=1,2, f(E_i)$ is homotopic to 
$E_i $ rel $\partial E_i$.  After an application of Theorem 0.6 i) \cite{Ga2}
$f$ can be isotoped so that $f|(E_1\cup\partial \mN\mB)=\id$ and then
 by Theorem 10.4 \cite{Ga1}, $f$ can be further isotoped so that 
 $f|(E_1\cup E_2\cup\partial \mN\mB)=\id$ and hence $f=\id$ modulo $\Diff(B^4 \fix \partial)$.\end{proof}

\begin{definition} Let $M$ be a properly embedded 3-manifold in the 4-manifold $V$.  We say that $Y$ is 
obtained from $M$ by \emph{embedded surgery} if there is a sequence $M=M_0, M_1, M_2, \cdots, \\ M_n=Y$ 
such that $M_i$ is obtained from the regular neighborhood $N(M_{i-1})$ by first attaching a single 
4-dimensional handle embedded in $V$ and then restricting to some relative boundary components.  \end{definition}

\begin{construction} \label{delta image}We construct $\beta(\Delta_0):= \Delta_1$ by embedded surgery.  
To  start with, $\Delta_0$ is isotopic to the 3-ball $\Delta_b$ obtained from $\Delta_0$ 
by two embedded 2-handle surgeries where the second 2-handle attachment is given by the cocore of the first as in
 Figure \ref{fig:FigureB4} a).  More precisely, the first 0-framed 
 4-dimensional 2-handle $\sigma$ is attached to link $B_2$ and intersect $E_2$ in a cocore $C$.  Further
 if $\Delta_a$ is the result of this embedded surgery, then 
  $\lambda_s\cap \Delta_a=\emptyset$ all $s$, where $\lambda_s$ is as in Construction \ref{barbell definition}.   The second 2-handle $\tau$ is a normal
   neighborhood of $C$ and hence $\Delta_b$  is disjoint from $E_2$.  Since $\beta(E_2)$ is obtained by removing 
   a small subdisc $C'\subset C\subset E_2$ and replacing it by a disc $C''$, it follows that 
   $\beta(\Delta_b):=\Delta_c$ is obtained by removing a solid torus, the unit normal bundle 
   of $C'$, and reimbedding it as the unit normal bundle of $C''$.  Finally, $\Delta_1$ is 
   obtained by isotoping the attaching zone of $\beta(\tau)$ to lie in $\Delta_0$.   
   At the 3-dimensional level $\Delta_1$ is obtained from $\Delta_0$ by 0-surgery along a 
   Hopf link.  
\end{construction}

\begin{figure}[ht]
$$\includegraphics[width=6cm]{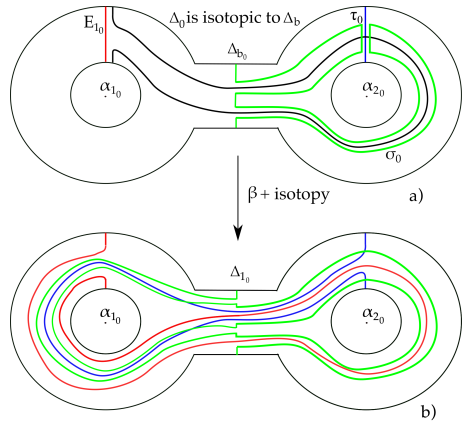}$$
\caption{\label{fig:FigureB4} Construction of $\Delta_1$ as seen from the $t=0, y=0$ slice.}
\end{figure}

\begin{remark} Note that if $\Delta^1, \cdots, \Delta^n$ are $n$ parallel copies of $\Delta_0$ with $\Delta^1$ closest to $B_2$, 
then $\beta(\cup \Delta^i)$ are $n$ parallel copies of $\beta(\Delta_1)$.  The corresponding 2-handles 
$\beta(\sigma^1), \beta(\sigma^{2}), \cdots, \beta(\sigma^n)$ \\
(resp. $\beta(\tau^n), \beta(\tau^{n-1}), \cdots, \beta(\tau^1))$ 
\emph{nest}, one inside the previous.  
\end{remark}

\begin{definition}  Let N be a 4-manifold and $\mB_0$ a smoothly embedded barbel with framed cuffs.  
Let $N(\mB_0)$ be an $\epsilon$-regular neighborhood.  There is an orientation 
preserving diffeomorphism $f:N(\mB_0)\to \mN\mB$ that restricts to one from $\mB_0$ to the 
model barbell $\mB$, which essentially preserves normal fibers of the thickenings.  Pulling 
back the barbell map induces a map $\beta_f:N\to N$, well defined up to isotopy.  The operation that starts with a framed cuffed barbell $\mB_0\subset N$ and produces $\beta_{f}\in \Diff(N)$ 
up to isotopy, is called \emph{implantation}.  Here $\beta_f$ can be viewed as an element of 
$\Diff(N)$ or $\Diff(N \fix \partial)$.  
\end{definition}

\begin{remarks} \label{implantation well defined} If $N$ is oriented we can assume that the framing of the cuffs together with the orientation of the cuffs agrees with the orientation of $N$.  Different framings of the bar will give rise to  maps $f$ and $g$ that differ by a twist 
as in Remarks \ref{framing} i).  That remark also shows that $\beta_f$ is 
isotopic to $\beta_g$.  Thus when $N$ is oriented the isotopy class of $\beta_f$ depends only on the embedding of the barbell and the framing of the cuffs. When the framing is implicit, then $\beta_f$ depends only on the embedding. \end{remarks} 

By chasing orientations we obtain:

\begin{lemma} If $\mB$ is a barbell in $N$, then reversing the orientation of $N$ 
without changing the orientations on the barbell changes the implantion to its inverse.
\end{lemma}

We now give a bundle theoretic interpretation of the barbell map.  Consider the fiber bundle $\Diff(D^n \fix \partial) \to \Emb(\sqcup_2 B^n, B^n)$ from the group of diffeomorphisms of $B^n$ to
the space of embeddings of two copies of $B^n$ into itself.  To define this bundle, we fix two disjoint $n$-discs
embedded in the interior of a fixed $B^n$.  The bundle comes from restricting a diffeomorphism to the sub-discs.
Not only it is a fiber-bundle, but it is null-homotopic, Proposition 5.3 \cite{Bu}.  Thus there is an injection
$\pi_k \Emb(\sqcup_2 B^n, B^n) \to \pi_{k-1} \Diff(B^n \text{ fix } \sqcup_2 B^n)$, moreover the co-kernel of
this injection is isomorphic to $\pi_{k-1} \Diff(B^n \text{ fix } \partial)$. 

$\Emb(\sqcup_2 B^n, B^n)$ has the homotopy-type of $S^{n-1} \times SO_n^2$.  Thus we have
$$\pi_{n-1} \Emb(\sqcup_2 B^n, B^n) \simeq \BZ \oplus \bigoplus_2 \pi_{n-1} SO_n.$$
The $\BZ$ factor corresponding to $\pi_{n-1} S^{n-1}$ is what generates the barbell diffeomorphism.
We have a short-exact sequence
$$ 0 \to \pi_{k-1} S^{n-1} \oplus \bigoplus_2 \pi_{k-1} SO_n \to \pi_{k-2} \Diff(B^n \text{ fix } \sqcup_2 B^n) \to \pi_{k-2} \Diff(B^n \fix \partial) \to 0$$

We consider $S^{n-1} \times B^2 \natural S^{n-1} \times B^2$ to be $(S^{n-1} \times I \natural S^{n-1} \times I) \times I$, i.e. 
as a trivial $I$-bundle over $S^{n-1} \times I \natural S^{n-1} \times I$.  Then as a bundle over $I$, denote the group of
fiber-preserving diffeomorphisms of $S^{n-1} \times B^2 \natural S^{n-1} \times B^2$ that are the identity on the boundary
by $\Diff^I(S^{n-1} \times B^2 \natural S^{n-1} \times B^2)$.  By design, this group is the loop-space on
$\Diff(B^n \text{ fix } \sqcup_2 B^n)$, thus we have shown there is a short exact sequence

\begin{proposition}\label{barbell-fib}
$$0 \to \pi_{k-1} S^{n-1} \oplus \bigoplus_2 \pi_{k-1} SO_n \to 
 \pi_{k-3} \Diff^I(S^{n-1} \times B^2 \natural S^{n-1} \times B^2 \text{ fix } \partial) \to 
 \pi_{k-2} \Diff(B^n \fix \partial) \to 0$$

valid for all $n \geq 3$ and all $k \geq 3$.
\end{proposition}

The barbell diffeomorphism (family) being the image of the $\pi_{n-1} S^{n-1}$ generator in $\pi_{n-3} \Diff^I(S^{n-1} \times B^2 \natural S^{n-1} \times B^2 \text{ fix } \partial)$, 
when $k=n$. 
%
%

We now consider \emph{$S^4$ implantation} where the cuffs $P_1, P_2$ of $\mB$ are fixed unknotted 
and unlinked oriented 2-spheres that are permuted under an elliptic involution $\tau$ of $S^4$.  We assume that $P_1\cup P_2\subset B$, a fixed oriented 3-ball, thereby determining the framings of $P_1\cup P_2$ as in the model barbell.  
The bar is then determined up to isotopy by its relative homotopy class and hence by a word $w(x,y)$ 
in the rank-2 free group, the fundamental group of $S^4\setminus (P_1\cup P_2)$.  
See Figure \ref{fig:Fig.B5}.  With generators as in the figure we will assume that w is reduced 
and if $w\neq 1$, then the first (resp. last) letter of $w$ is $y^{\pm 1}$ (resp. $x^{\pm 1}$).  
For example, if $x$ was the first letter, then it can be removed by rotating $N(P_1)$ fixing 
$P_1$ pointwise.  We denote such a barbell by $\mB_w$.  It is well defined up to isotopy fixing 
$P_1\cup P_2$ pointwise.  We denote by $\beta_w$ the corresponding diffeomorphism of $S^4$ 
which is also well defined up to isotopy. 

\begin{figure}[ht]
$$\includegraphics[width=8cm]{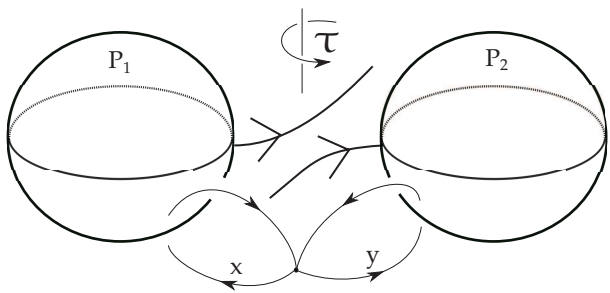}$$
\caption{\label{fig:Fig.B5} Generators for the fundamental group of a barbell complement in $S^4$}
\end{figure}


\begin{definition} 
We say that $w$ is \emph{inverse-$xy$-palendromic}  if $w$ is obtained by first taking its inverse 
and then everywhere switching $x$ and $y$.  
\end{definition}

\begin{example} 
$(yx^{-1})^n$ is inverse-$xy$-palendromic. 
\end{example}

\begin{proposition}  
If $w$ is inverse-$xy$-palendromic, then $\beta_w^2$ is isotopic to $\id$.
\end{proposition}

\begin{proof} Since $w$ is inverse-$xy$-palendromic, $\tau(\beta_w)$ is setwise isotopic 
to $\beta_w$ via an isotopy that is supported in the bar.  The time one map of the isotopy, 
$g:\tau(\mB_w)\to \mB_w$, is orientation preserving on the cuffs and reversing on the bar.  
It follows from Theorem \ref{involution inverse} that $\beta_{w}$ is isotopic $\beta_w^{-1}$.  
\end{proof}

\begin{conjecture} \label{special case}
$\pi_0 \Diff_0(S^4) \neq 1$.  
In particular, $\beta_{yx}$ and $\beta_{yx^{-1}}$ are not isotopic to $\id$.  
See Figure \ref{fig:Fig.B6}. 
\end{conjecture}

\begin{figure}[ht]
$$\includegraphics[width=8cm]{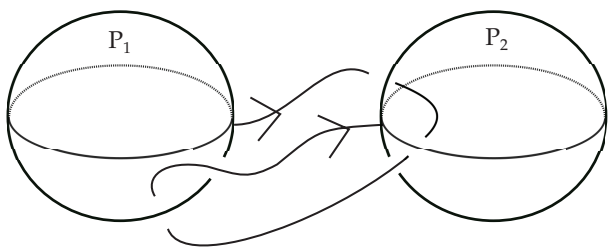}$$
\caption{\label{fig:Fig.B6} Is this a non trivial element of $\pi_0 \Diff_0(S^4)$?}  
\end{figure}

\begin{questions} 
i)  What are the relations in the subgroup of $\pi_0 \Diff_0(S^4)$ generated by the $\beta_w$'s?

ii) We can also consider implantations in $S^4$ arising from knotted or linked cuffs. 
 When do these implantations give isotopically trivial diffeomorphisms?  What if any are 
 relations amongst these implantations?
\end{questions}

\begin{definition}  
Let $\mB$ be a barbell in $M$.  Define $a(P_1)$, the \emph{associated sphere to $P_1$}, be the 
sphere obtained by orientation preserving tubing $P_1$ along the bar to a parallel copy of $P_2$. 
 \end{definition}

\begin{remark} Let $\alpha$ be an arc in a normal disc of $P_2$ from $a(P_1)$ to $P_2$. 
Then $N(a(P_1)\cup \alpha\cup P_2)$ is a barbell whose thickening is isotopic to $N(\mB)$.  
The next result shows that if $w\neq 1$, $\mB_w$ has thickening with a knotted barbell spine. 
\end{remark}

\begin{proposition}  Let $\mB_w$ be a barbell in $S^4$. Then

i) $a(P_1)$ is knotted if and only if $w\neq 1$ and hence if $w\neq 1$, then $\mB_w$ is not isotopic to $\mB_1$.

ii) If $w\neq 1$, then $\mB_w$'s thickening has no barbell spine isotopic to $\mB_1$.
 \end{proposition}

\begin{proof} i) If $w=1$, then $a(P_1)$ is obviously unknotted.  In general $a(P_1)$ is a 
ribbon 2-knot $K$ whose fundamental group is $\langle x,y|xwyw^{-1} \rangle$.  Since we can assume that w 
is a nontrivial reduced word that does not start (resp. end) with a $x^{\pm 1}$ nor end with 
a $y^{\pm 1}$ it follows that the relator is non trivial and hence the group is non cyclic. 

 ii) Every non separating embedded 2-sphere in $\partial \mN\mB_1$ is isotopically trivial.    
\end{proof}



 We now consider $S^1\times X$ implantations where $X=S^3$ or $B^3$.  

\begin{construction}  \label{implant delta} 
Let $\mB_0\subset S^1\times X, X\in\{B^3, S^3\}$ a barbell with $\beta_0$ the implanted map. 
If $\Delta_0:=x_0\times X$ is disjoint from the cuffs and transverse to the bar, then 
$\Delta_1:=\beta_0(\Delta_0)$ is constructed as follows.  Let 
$h:(N(\mB_0),\mB_0)\to (\mN\mB,\mB)$ an orientation preserving diffeomorphism, where $\mN\mB$ 
is the thickened model barbell.  We can assume that $f(\Delta_0\cap N(\mB_0))$ is $m\ge 0$ 
parallel 3-balls normal to the model bar.  Construction \ref{delta image} shows how to 
construct the image of these balls under the barbell map.  Now pull back by $h$.  
\end{construction}

\begin{remark}  
$\Delta_1$ is obtained by embedded surgery to $\Delta_0$ where the attaching zone of the 
2-handles is the split union of $m$ Hopf links.  At the 3-dimensional level, $\Delta_1$ is obtained 
by 0-surgery to the components of this link, thus its topology is unchanged. 
\end{remark}

\begin{construction} (Barbells from Spinning) Let $J$ be an embedded oriented $S^1$ in the oriented 4-manifold $M$.  
Construction \ref{barbell definition} and Remark \ref{framing} show that the diffeomorphism of $M$ induced by applying 
isotopy extension to a $\lambda$-spinning of a subarc $J_1$ of $J$ about another subarc $J_2$ is the barbell map 
corresponding to a barbell $\mB\subset M$ with $P_1$ linking $J_1, P_2$ linking $J_2$ and the bar given by $\lambda$.  
 Applying this to $\theta_k$ and $\alpha_i$, denoted in \S 2 as $\theta_{1,k}$ and $\alpha_{1,i}$, we obtain the barbells shown in  Figure \ref{fig:Fig.B7}. 
  To see this for $\theta_3$, first consider Figure \ref{fig:Fig.B7} c) which shows the basic construction 
  and then send $f$ to ``infinity".  We now view these barbells as subsets of $S^1\times B^3$ and denote 
  them by $\mB(\theta_k) $ and $\mB(\alpha_i)$. 
   Let $\beta_{\theta_k}, \beta_{\alpha_k} \in \Diff_0(S^1\times B^3 \fix \partial)$ denote the corresponding implantations.  \end{construction}

\begin{figure}[ht]
{
\psfrag{cd2}[tl][tl][0.7][0]{}
\psfrag{s1}[tl][tl][0.7][0]{$S^1$}
\psfrag{sn}[tl][tl][0.7][0]{$S^3$}
\psfrag{f}[tl][tl][0.7][0]{}
\psfrag{C}[tl][tl][0.7][0]{c)}
\psfrag{D}[tl][tl][0.7][0]{d)}
$$\includegraphics[width=8cm]{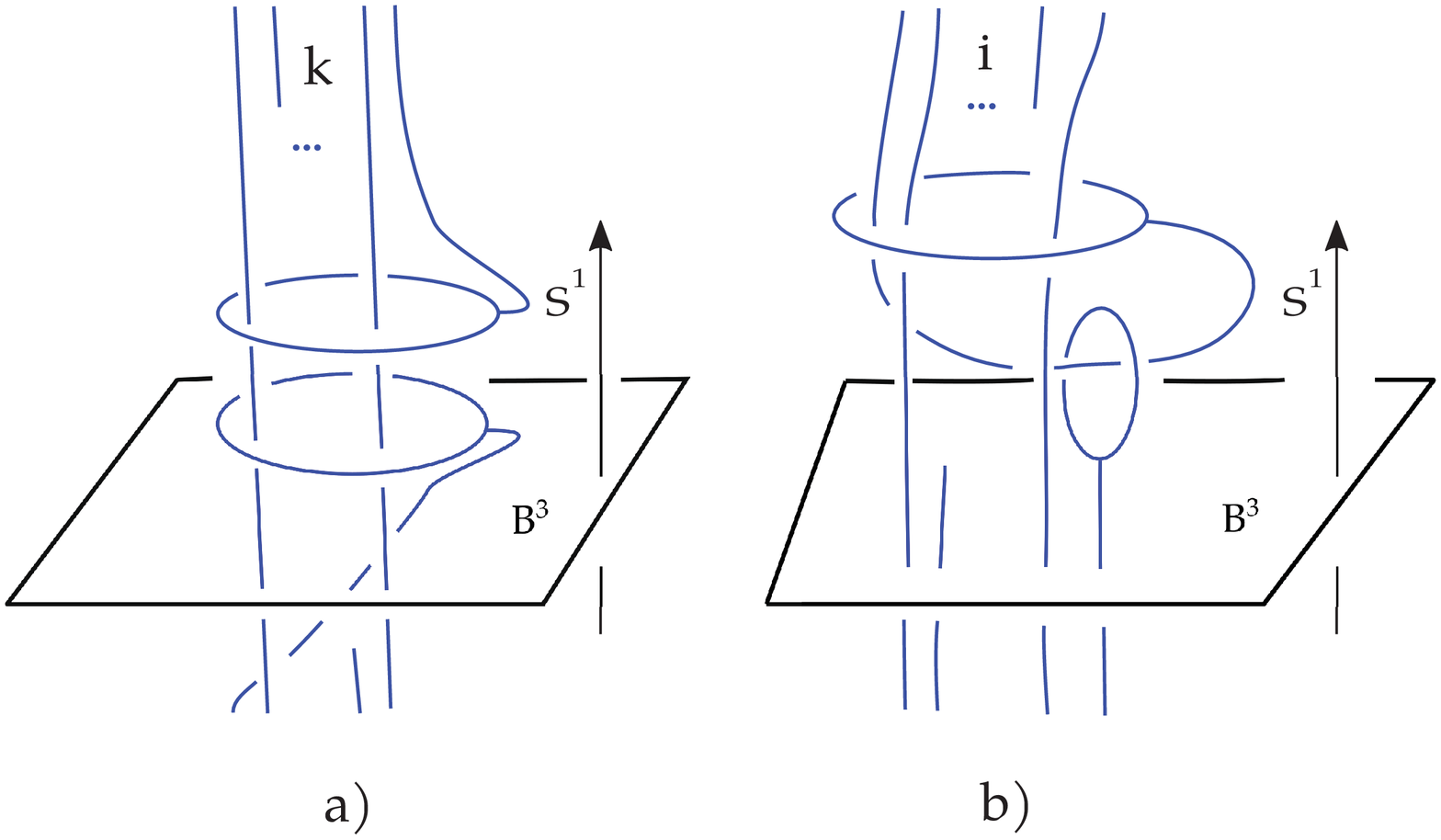} \hskip 0.5cm \includegraphics[width=8cm]{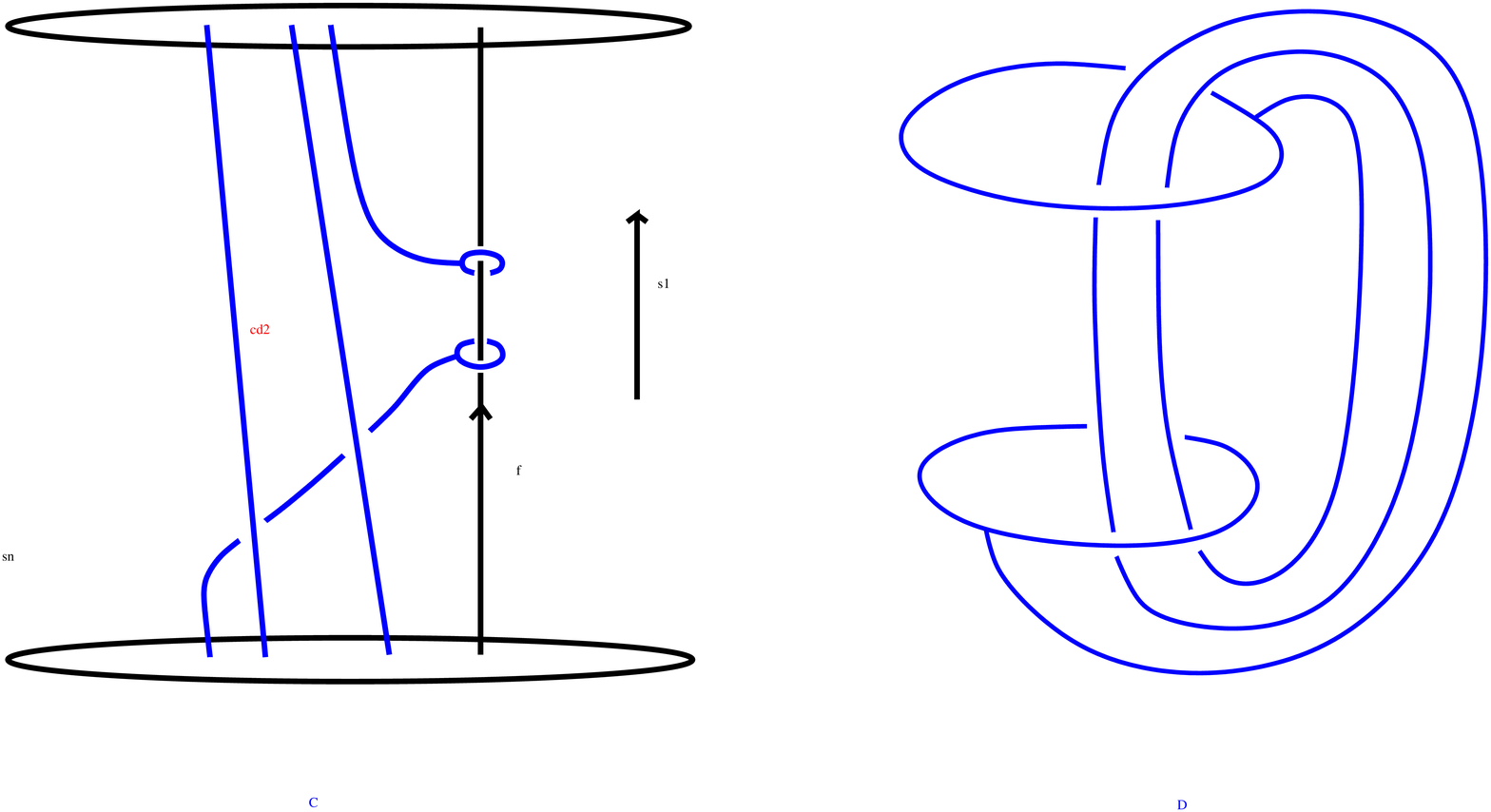}$$
}
\caption[]{\label{fig:Fig.B7} 
\begin{tabular}[t]{ @{} r @{\ } l @{}}
 (a) & The barbell $\mB(\theta_k)$ induced from $\theta_{k}$ \\  
 (b) & The barbell $\mB(\alpha_i)$ induced from $\alpha_{i}$ \\
 (c) & The initial  barbell induced from $\theta_{3}$ \\  
 (d) &  $\mB(\theta_3)$ viewed as a subset of $S^4$ 
\end{tabular} 
}
\end{figure}


Since $S^4$ is obtained by gluing an $S^2\times D^2$ to $S^1\times B^3$ elements of $\Diff(S^1\times B^3 \fix \partial)$ 
extend to elements of $\Diff_0(S^4)$. For example, viewing  $\mB(\theta_3)$ in $S^4$ we obtain the barbell 
$\mB_{S^4}(\theta_3)$ shown in Figure \ref{fig:Fig.B7} d).  
Since this barbell is inverse-$xy$-palendromic its implantation is order at most 2.


\begin{remark} 
\label{watanabe} Watanabe constructs \cite{Wa1} various $B^4$-bundles over $n$-spheres to show, 
among other things, that $\pi_1 \Diff(B^4\fix \partial)$ is non trivial.  Specializing his 
construction to $n=1$ and considering a bundle's gluing map, give possible nontrivial elements 
of $\Diff(B^4\fix \partial)$.  He notes that  this construction applied to a 4-manifold $M$ might yield nontrivial elements 
of $\pi_0\Diff(M \fix \partial)$.  Here we observe that these diffeomorphisms are compositions of 
barbell maps.  We very briefly outline this in the simplest case when $M=S^4$.

Watanabe starts with an embedding $\alpha$ of a two vertex arrow graph into $S^4$ and then 
constructs the corresponding $Y$-link in $S^4$, see \cite{Wa1} \S4.  We abuse notation by 
letting $\alpha$ denote this embedding when extended to a regular neighborhood.  The 
$Y$-link has two components, one Type 1 and one Type 2.  Call their respective regular 
neighborhoods $V_1$ and $V_2$.  Watanabe's diffeomorphism of $S^4$, which we denote by 
$\phi_\alpha$, arises from \emph{parametrized Borromean surgery} of Type 1 on $\alpha(V_1)$ 
and Type 2 on $\alpha(V_2)$.  Roughly, $V_1=B^4\setminus \inte(N(D_1)\cup N(D_2)\cup N(D_3))$ 
where $D_1$ and $D_2$ are 2-discs and $D_3$ a 1-disc, the union standardly embedded.  
$V_2 = B^4\setminus  \inte(N (E_1)\cup N( E_2)\cup N( E_3))$ where $E_1$ is a 2-disc and $E_2, E_3$ 
are 1-discs and the union is standardly embedded, thus $V_2 = S^2\times D^2\natural S^2\times D^2$ 
plus a 1-handle.  Being a $Y$-link,  the 1-handle of $\alpha(V_2)$ passes parallel to $\alpha(D_3)$ 
through $\alpha(V_1)$.  Type 1 surgery effectively replaces $D_3$ by $D_3'$ where $D_1\cup D_2\cup D_3'$ 
is the 4-dimensional Borromean rings viewed as a three component tangle in the $B^4$ defining $V_1$. 
Thus, Type 1 surgery reimbeds the 1-handle of $\alpha(V_2)$, hence reimbeds $\alpha(V_2)\subset S^4$. 
Let $\alpha':V_2\to S^4$ be this new embedding.    Type 2 surgery on $V_2$ induces 
$\psi\in \Diff(V_2 \fix \partial)$ by an arc pushing diffeomorphism and $\phi_\alpha$ is obtained 
by pushing forward $\psi$ by $\alpha'$.  The track of this arc pushing is more or less a 
disc $E_2'$ where $E_1\cup E_2'\cup E_3$ is the 4-dimensional Borromean rings.  Our assertion 
follows from the fact that $\psi$ is the composition of two barbell maps.  To see this expressed 
from our point of view, consider $V_2$ as $T\times [-1,1]$ where $T$ is a $D^2\times S^1$ with 
two open 3-balls $B_1$ and $B_2$ removed. Here, we have two properly embedded discs $F_0\subset V_2$ and $F_1\subset V_2$ that coincide
near their boundaries and $\psi$ is obtained by first pushing in along $F_0$ and then out 
along $F_1$.  The $F_1$ is a $\sigma\times [-1,1]$, where $\sigma\subset T$ is a properly embedded arc 
and $F_0$ intersects each $T\times t$ in an arc $F_{0_t}$.  $F_{0_0}$ and 
$\sigma\times \{0\}$ are shown in Figure \ref{fig:Fig.B9}.  
$F_{0_0}$ consists of two subarcs $\tau_0, \tau'_0$ that go around the the second 3-ball $B_2\times \{0\}$. 
As $t$ increases (resp. decreases) $\tau_t $ slips over (resp. under) $B_2\times \{t\}$ and $\tau'_t$ 
slips under (resp. over) $B_2\times \{t\}$, both returning to coincide with a segment in 
$\sigma\times \{t\}$.  This can be done so that for all $t,\ \inte(\tau_t)\cap \inte(\tau'_t)=\emptyset$ and 
hence $F_0$ is embedded.  With this description it is evident that $\psi$ and hence 
$\phi_\alpha$, is isotopic to the composition of two barbell maps.\end{remark}

\begin{figure}[ht]
$$\includegraphics[width=10cm]{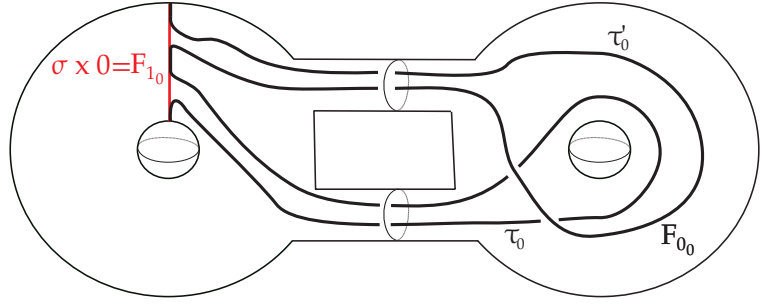}$$
\caption{\label{fig:Fig.B9} $t=0$ slice of Watanabe's Type 2 Borromean surgery.}  
\end{figure}


\section{1- and 2-parameter families arising from implantations} \label{family section}

In the last section we introduced the barbell map $\beta:S^2\times D^2\natural S^2\times D^2\to S^2\times D^2\natural S^2\times D^2$ and constructed $\Delta_1=\beta(\Delta_0)$ where $\Delta_0$ is the standard separating 3-ball.  In this section we slice $\Delta_0$ into 1- and 2-parameter families of discs and arcs.  We will show how under barbell implantation into a 4-manifold $M$, where the bar is transverse to a 3-ball, these families produced loops and spheres in $\Emb(D^2,M)$ and $\Emb(I,M)$ respectively.

Recall that $\mN\mB=V\times [-1,1]$ where $V=D^2\setminus \inte(B_1\cup B_2)$ where $B_1$ and $B_2$ are  3-balls.  By scaling $[-1,1]$ to $[0,1]$, we henceforth identify $\mN\mB$ with $V\times [0,1]$.  Let $D\subset V$ denote the separating 2-disc such that $\Delta_0=D\times [0,1]$.  Recall that $\beta(V\times u)=V\times u$ for all $u\in [0,1]$ it follows that $\Delta_1=\cup_{u\in [0,1]} \beta (D\times u)$.  The next result follows by slicing the construction of $\Delta_1$.

\begin{proposition}   \label{one parameter} $ \{\beta(D\times u)\}_{u\in [0,1]}$ is isotopic to the family $\{D_u\}$ obtained as follows.

i)  For $u\in [0,.125]$, $D_u=D\times u$.  See Figure \ref{fig:FigureImp1} a).

ii)  For $u\in [.125,.25]$ squeeze $D_{.125}$ to create a neck as in Figure \ref{fig:FigureImp1} b).

iii) For $u\in [.25,.50]$ isotope the neck halfway around $B_1$, going from the $y>0$ side to the $y<0$ side. See Figure \ref{fig:FigureImp1} c).

iv) For $u\in [.50, .75]$ isotope the neck back to its original position shown in Figure \ref{fig:FigureImp1}  b), continuing from the $y>0$ side to the $y<0$ side.  

v) For $u\in [.75, 1]$ reverse the isotopy from $[0,.25]$.  \qed\end{proposition}

\begin{figure}[ht]
$$\includegraphics[width=13cm]{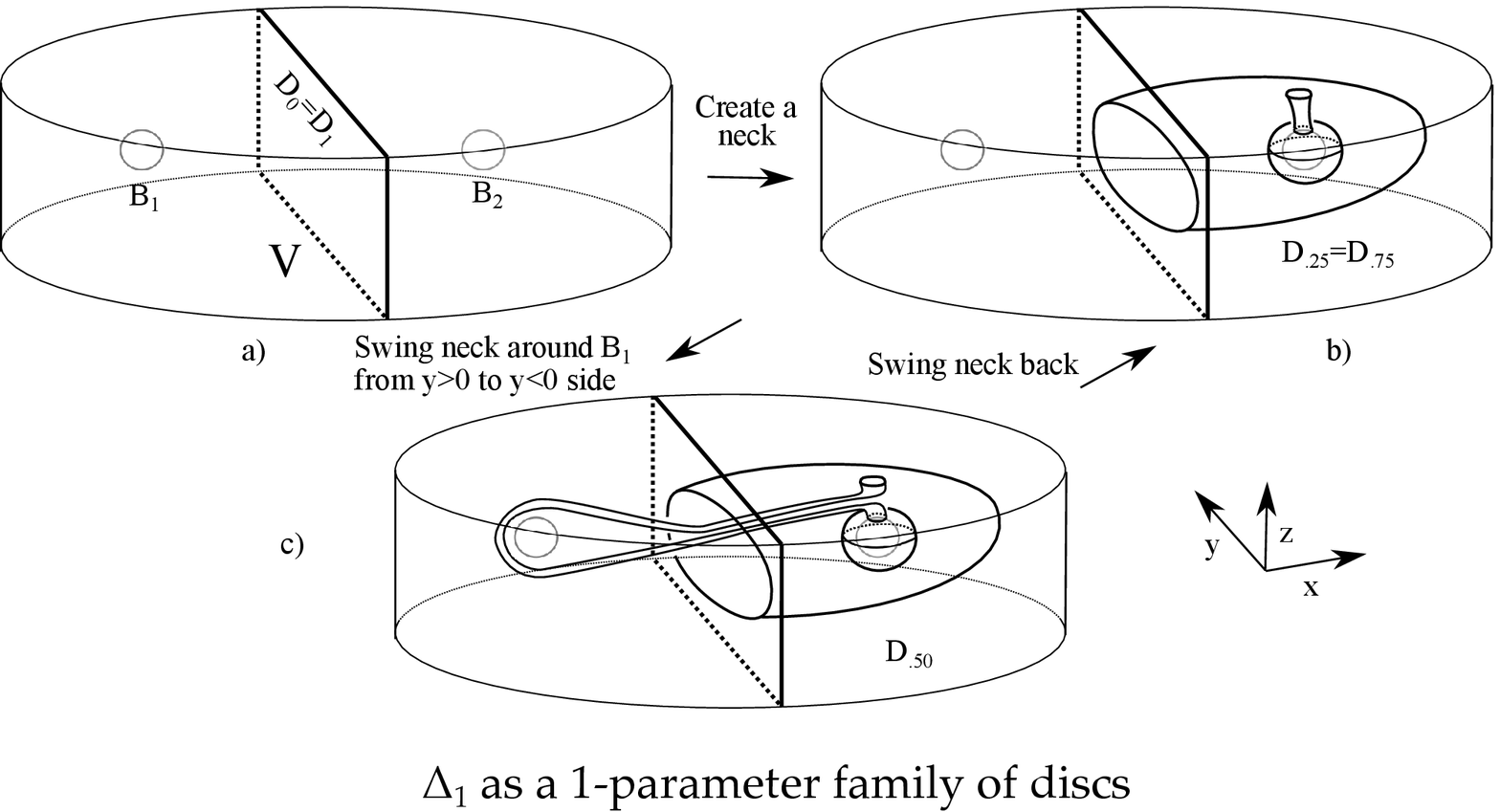}$$
\caption{\label{fig:FigureImp1}} 
\end{figure}

\begin{remark} \label{b1 centric} $\{\beta(D\times u)\}_{u\in [0,1]}$ is similarly isotopic to the family $\{E_u\}$ where we first create a neck connecting discs that surround $B_1$ instead of $B_2$.  For that isotopy the neck swings around $B_2$ going from the $y<0$ side to the $y>0$ side.  The neck of $E_{.50}$ is essentially seen in Figure \ref{fig:Fig.B3} b) where the red arc is replaced by a cylinder.  Notice that replacing the blue arc by a cylinder essentially gives the neck of $D_{.50}$. \end{remark}

 We now describe another isotopic family $\{H_u\}$ that is neither $B_1$ nor $B_2$ centric.  For that isotopy it is convenient to change the base disc $D$ to the disc $H\subset V$ shown in Figure \ref{fig:FigureImp2} a), via an isotopy $\phi_{D,H}(t)$ of $V$ supported in a neighborhood of $D$.

First consider the disc $H_{.50}$ shown in  Figure \ref{fig:FigureImp2} c). This disc is obtained by removing four open discs from $H_0$, shown in Figure \ref{fig:FigureImp2} a) and replacing them with two blue discs that surround $B_2$ and two red ones that surround $B_1$.  The red (resp. blue) discs emanate from the first two sheets of the folded $H_0$ that are closest to $B_1$ (resp. $B_2$).  To go from Figure \ref{fig:FigureImp2} c) to Figure \ref{fig:FigureImp2} b)(resp. d)) isotope the blue necks down (resp. up) and the red necks up (resp. down).  The region between the blue discs of $H_{.25}$ is a product which can be boundary compressed using a fold of $H_0$.  A similar statement holds regarding the red discs, where the boundary compression starts at the other fold.  These boundary compressions are done simultaneously and are supported away from the arc $J_0$.  The result of these  boundary compression is shown in Figure \ref{fig:FigureImp2} e).

\begin{proposition}  Under the change of basepoint map $\phi_{D,H}(1)$, the 1-parameter family $D_u$ is isotopic to the following 1-parameter family $H_u$.

i) For $u\in[0,.125]$,\ $H_u=H_0$.

ii) For $u\in [.125, .25]$,\ $H_{.125}$ is isotoped to the surface $H_{.25}$ shown in Figure \ref{fig:FigureImp2} b) via the intermediary surface $H_{.18}$ shown in Figure \ref{fig:FigureImp2} e).  This isotopy fixes $J_0$ pointwise.  The isotopy for $u\in [.125,.18]$ is supported near $H_0$.  For $u\in [.18,.25] $ the isotopy restricted to the disc below (resp. above) $J_0$ is supported in the union of a regular neighborhood of $H_0$ and the right (resp. left) hand side of $H_0$.  

iii) For $u\in[.25,.75]$ the blue necks are isotoped up and the red necks down to obtain Figure \ref{fig:FigureImp2} d) with the intermediary $H_{.50 }$ shown in Figure \ref{fig:FigureImp2} c).  

iv) For $u\in [.75, 1]$, the isotopy is analogous to the reverse of the  $H_0$ to $H_{.25}$ isotopy. Two boundary compressions supported away from $J_0$ produce $H_{.82}$ which is the mirror in the $[0,1]$ coordinate of $H_{.18}$.  \end{proposition}

\begin{figure}[ht]
$$\includegraphics[width=13cm]{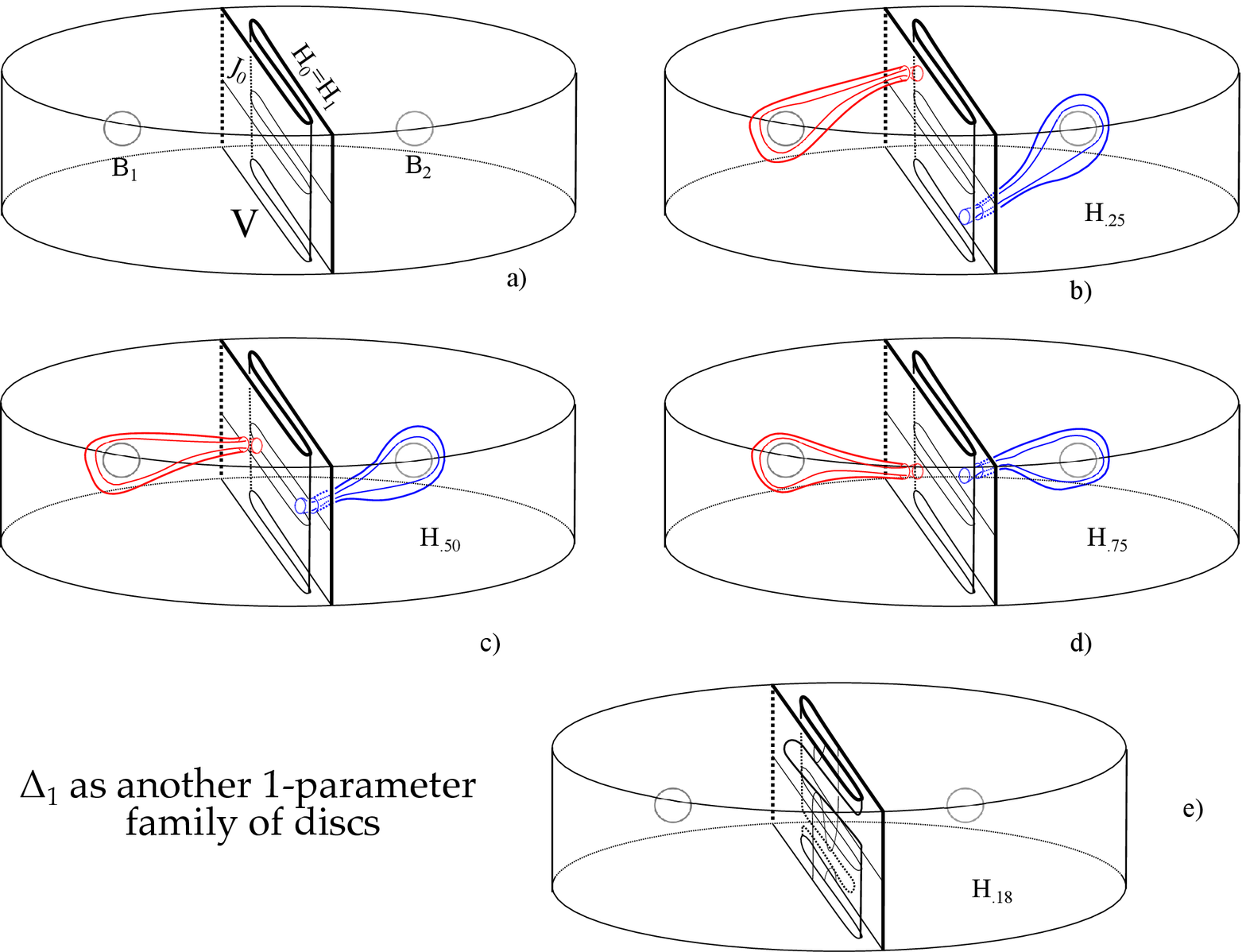}$$
\caption{\label{fig:FigureImp2}} 
\end{figure}

\begin{proof} A direct calculation gives an isotopy between $\{\phi_{D,H}(1)(D_u)\}$, which by abuse of notation we call $\{D_u\}$ and 
$H_u$ that preserves each $V\times u$ slice.  Alternatively, let $e_1, e_2$ be properly embedded arcs in $V$ such that for $i=1,2$, 
$e_i\cap B_i\neq\emptyset$ and $e_i\cap (H_0\cup D_0)=\emptyset$.  The loops in $\Emb(N(e_1)\cup N(e_2),V)$ induced from $D_u$ and 
$H_u$ by isotopy extension are homotopic.  It follows that $\{D_u\}$ and $\{H_u\}$ differ by an element of 
$\pi_1 \Emb(D^2, B^3 \fix \partial B^3)$, but that group is trivial by Cerf \cite{Ce2} and Hatcher \cite{Ha1}.\end{proof}

Since $V=D^2\times [0,1]\setminus \inte(B_1\cup B_2)$, the $[0,1]$ factor induces a $[0,1]$-fibering on each $H\times u$  and hence expresses $\Delta_0$ as a 2-parameter family of intervals.    Pushing forward to $H_u$ expresses $\Delta_1$ as a 2-parameter family $\gamma_{t,u}$.  Because $\beta(V\times u)=V\times u$ for all $ u$, it follows that we can assume that $\gamma_{t,u}$ is supported in $V\times .50$.  

The family $\gamma_{t,u} \in \Omega\Omega ([0,1],V\times I; J_0)$ is a bracket as in Definition \ref{bracket}.  To see this define $H^B_u$ as the 1-parameter family which only involves the blue discs, so $H^B_{.50}$ is $H_0$ with two blue discs attached and the passage from $H^B_{.25}$ to $H^B_{.18}$ only involves the boundary compression of the blue discs, etc.  In an analogous manner we define $H^R_u$.  The associated 2-parameter families $\gamma^B_{t,u}, \gamma^R_{t,u}$ satisfy the conditions of Definition \ref{bracket}.  The midlevel loops $B,R\in \Omega\E$ are defined in band lasso notation as in Figure \ref{fig:FigureImp3} a).  For $B$ (resp. $R$) only the blue (red) band and lasso are used.  We call $\gamma_{t,u}$ the \emph{barbell family} and all its defining information its \emph{data}.  We have the following result.

\begin{figure}[ht]
$$\includegraphics[width=13cm]{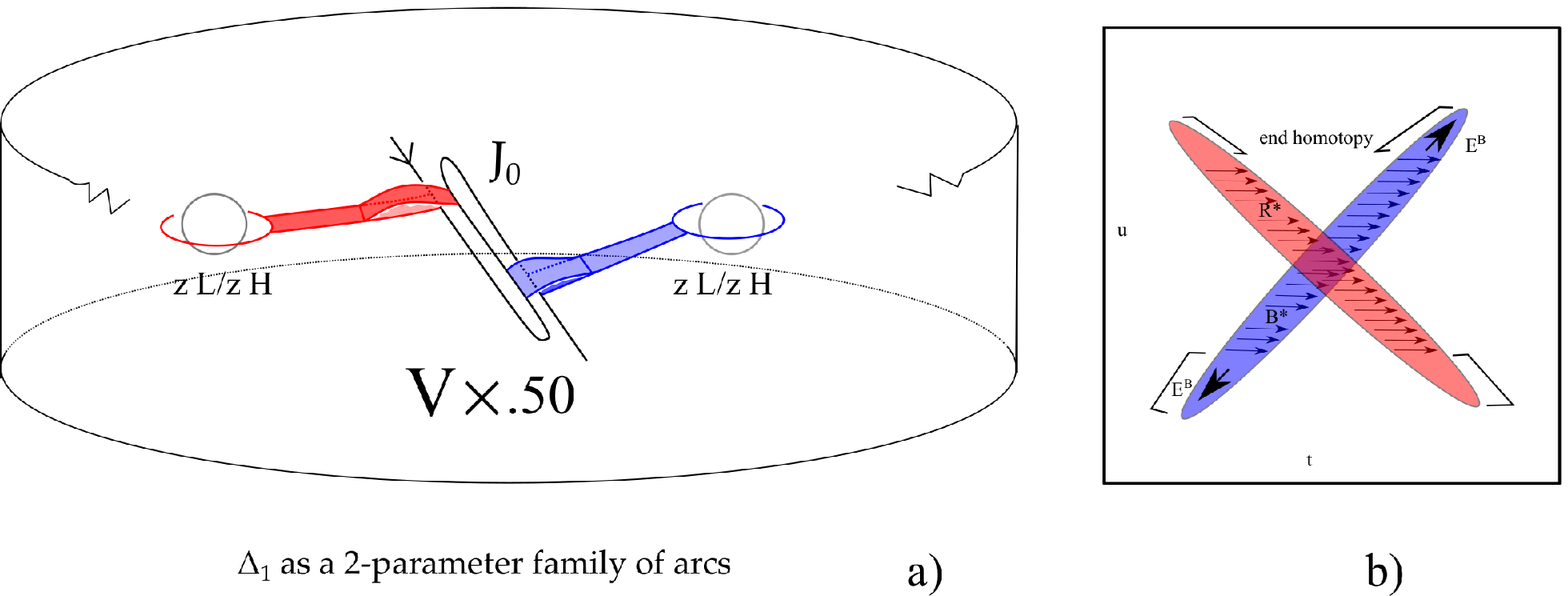}$$
\caption{\label{fig:FigureImp3}} 
\end{figure}

\begin{proposition}\label{barbell family} The barbell family $\gamma_{t,u}$ is supported in $V\times .50$ and is homotopic to the separable family $(B,R)$ of Figure \ref{fig:FigureImp3} a) described in band lasso notation.  The notation $z L/ z H$ means that the spinnings go about the spheres first by going down in the $z$ direction and then up.  The support of $\gamma^B_{t,u}$ (resp. $\gamma^R_{t,u}$) in the parameter space is the blue (resp. red) region shown in Figure \ref{fig:FigureImp3} b)\qed\end{proposition}

\begin{remark} Note that the end homotopy is homotopic to the undo homotopy. \end{remark}

We now describe how barbell implantation induces an $ [\alpha_{t,u}]\in \pi_2\E$.  View $S^1\times B^3$ as $D^2\times S^1\times [-1,1]$ with $D^2$ viewed as $[0,1]\times I$.  Fixing $x_0\in 
S^1$ and $I_0=([0,1]\times .50)\times x_0\times 0$, then $U:=D^2\times x_0\times [-1,1]$ corresponds to the trivial 2-paraxmeter family by translating within $U$ each $[0,1]\times s\times x_0\times w$ to $I_0$.  If $\psi\in \Diff(S^1\times B^3 \fix \partial)$, then $\psi$ induces a 2-parameter family $\alpha_\psi\in \Omega\Omega\E$ by naturally translating each 
$\psi([0,1]\times s\times x_0\times w)$ so that its end points coincide with that of $I_0$.  Isotopic $\psi$'s give rise to homotopic $\alpha_\psi$'s.  

\begin{construction}\label{implanted family} Let $f:\mN\mB\to S^1\times B^3$ be an embedding of the barbell neighborhood such that the bar is transverse to $U$ and the cuffs are disjoint from $U$.  We can assume that $\finv(U) =\Delta_0^1\cup \cdots\cup \Delta_0^n$, parallel copies of the midball and the $[0,1]$-fibering of $U$ pulls back to the standard $[0,1]$-fibering of each $\Delta_0^i$.  Further, each $f(\Delta_i)$ is of the form $D_i\times x_0\times [-.5,.5]$ where $D_i$ is a subdisc of $D^2$ and that the projections of $D_i$ to the $I$ factor are pairwise disjoint.  It follows that $\alpha_{\beta_f}$ is represented by the sum of n 2-parameter families $\alpha^i_{\beta_f}$, where $\beta_f\in \Diff_0(S^1\times B^3 \fix \partial)$ is the implantation of $f$.  Finally, to construct $\alpha^i_{\beta_f}$ push forward the data of the barbell family.\end{construction}

\begin{definition} \label{hat thetak} For $k\ge 2$ define the 2-parameter families $\hat\theta_k^1, \cdots, \hat\theta_k^{k-1}$ analogously to the family $\hat\theta_k^L$ shown in Figure \ref{fig:FigureImp4} b) where $k=5$ and $L=3$.  Here the red bands go around the $S^1$ just under $L$ times while the blue bands go around the $S^1$ just under $k-L$ times.   Here the color coding and support in the parameter space are as for the $\gamma_{t,u}$ family.  Define $\hat\theta_k = \sum_{L=1}^{k-1} \hat\theta_k^L$.  $\hat\theta_k$ is called the \emph{symmetric} $\theta_k$.  \end{definition}

\begin{lemma} $[\alpha_{\beta_{\theta_k}}]=\pm[\hat\theta_k]\in \pi_2 \E$.\end{lemma}

\begin{proof}  $\alpha_{\beta_{\theta_k}}$ is obtained by applying Construction \ref{implanted family} to the barbell $\mB(\theta_k)$.  Each $\alpha^L$ appears qualitatively as $\hat\theta_k^L$  up to twisting of the bands near the lasso and whether the individual spinnings are $L/H$ or $H/L$ with differences uniform in $L$.  Since any such difference changes the sign of the class in $\pi_2$ by Lemmas \ref{elementary homotopies} and \ref{color reversal} the result follows.\end{proof}


\begin{figure}[ht]
$$\includegraphics[width=13cm]{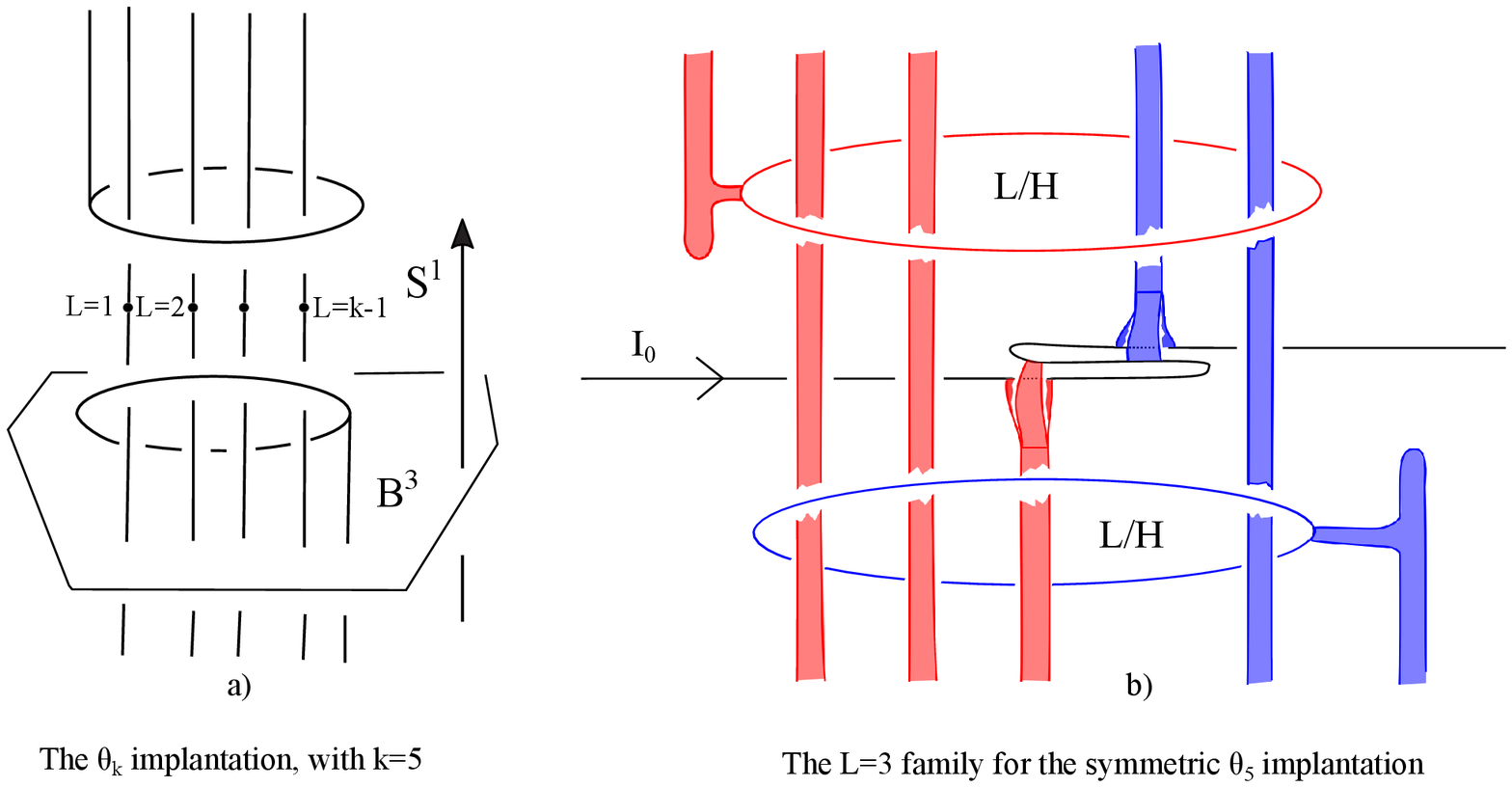}$$
\caption{\label{fig:FigureImp4}} 
\end{figure}

\begin{definition} \label{twisting implantation} Let $f:\mB\to M$ be an embedding of the barbell such that the cuffs bound disjoint 3-balls $V_1$ and $V_2$ and the bar is transverse to $V_1\cup V_2$.  A reembedding of the bar supported  within a small neighborhood $N(V_1\cup V_2)$ of $V_1\cup V_2$ to obtain $f_1$  is called a \emph{twisting} of $f$.  Here the components of the bar that intersect both $N(V_1\cup V_2)$ and the cuffs remain unchanged.   $\beta_{f_1}$ is called a \emph{twisting} of the implantation $\beta_f$.\end{definition}

\begin{remark}  Up to isotopy, a twisting corresponds to replacing arcs $\sigma_i$ that locally link a cuff once to ones that link $n_i\in \BZ$ times.  \end{remark}

\begin{definition}  \label{thetak twisting} (Twisting the $\hat\theta_k$ implantation)  Ordering the points through the cuffs $v$ and $w$ as in Figure \ref{fig:FigureImp5} a) we  construct the $v(m_1, \cdots, m_{k-1}), w(n_1, \cdots, n_{k-1})$ implantation.  This means that, the $ i$'th (resp. $j$'th) arc through the $P$ (resp. $Q$) cuff now locally links it $m_i$ (resp. $n_j$) times.  For example see Figure \ref{fig:FigureImp5} b).  The \emph{$\hat\alpha_{k-1}$ implantation} is the $\hat\theta_k $ implantation twisted by $v(0,\cdots, 0, 1), w(1, 1, \cdots, 1)$.  Define $\delta_k$ the $\theta_k(0,0,\cdots, 0,1)(0,0,\cdots, 0,1,0)$ implantation.  See Figure \ref{fig:FigureImp5} c) for $\delta_5$. \end{definition}

\begin{figure}[ht]
$$\includegraphics[width=13cm]{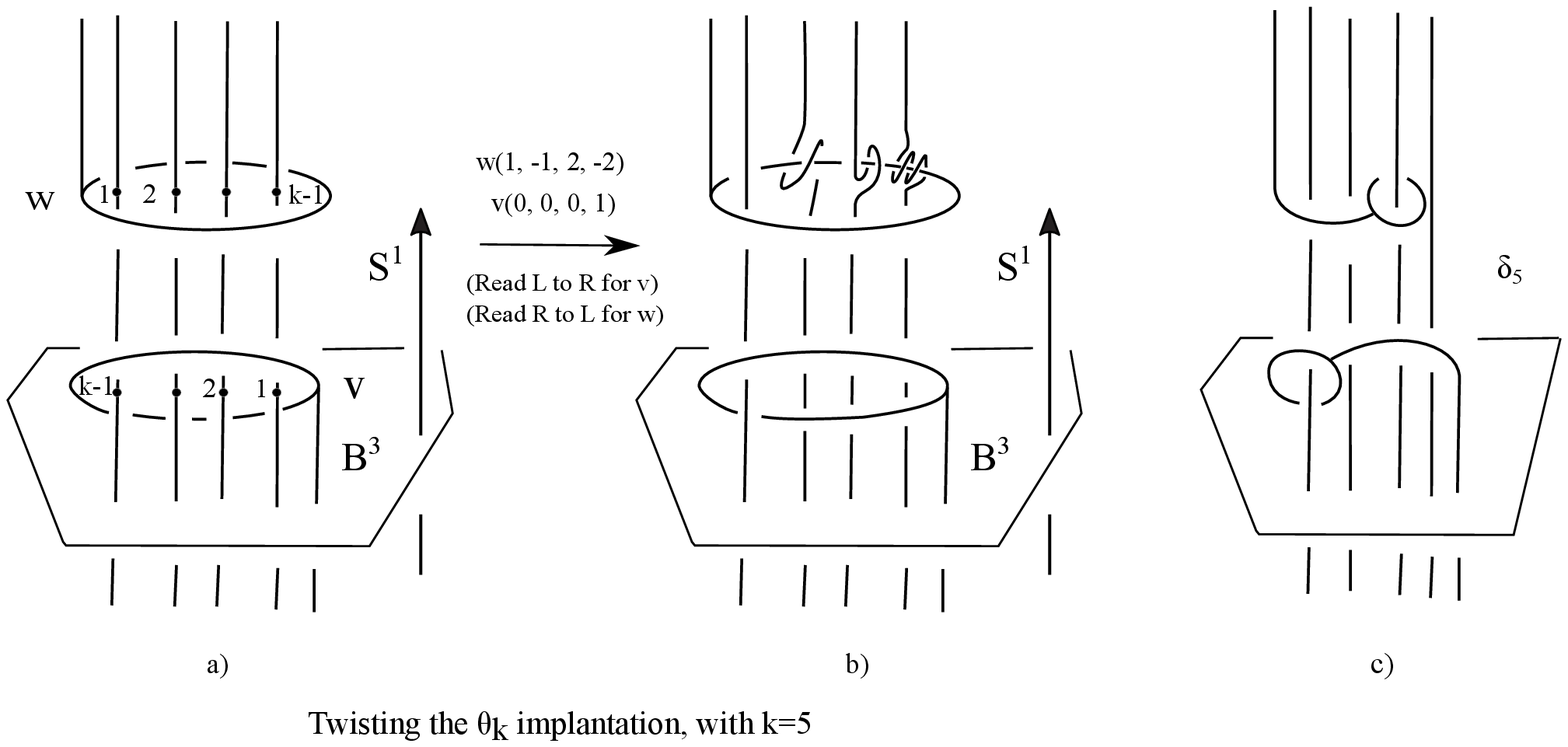}$$
\caption{\label{fig:FigureImp5}} 
\end{figure}


\begin{construction}\label{2parameter twisting}  Figure \ref{fig:FigureImp6} shows how to modify a 2-parameter family according to twisting an implantation.  Note that modifying a band by a $2\pi$-twist does not change the homotopy class of the 2-parameter family.  Therefore, what matters is how many times the modified band goes through the cuff rather than how exactly it is reembedded.  \end{construction}

\begin{figure}[ht]
$$\includegraphics[width=13cm]{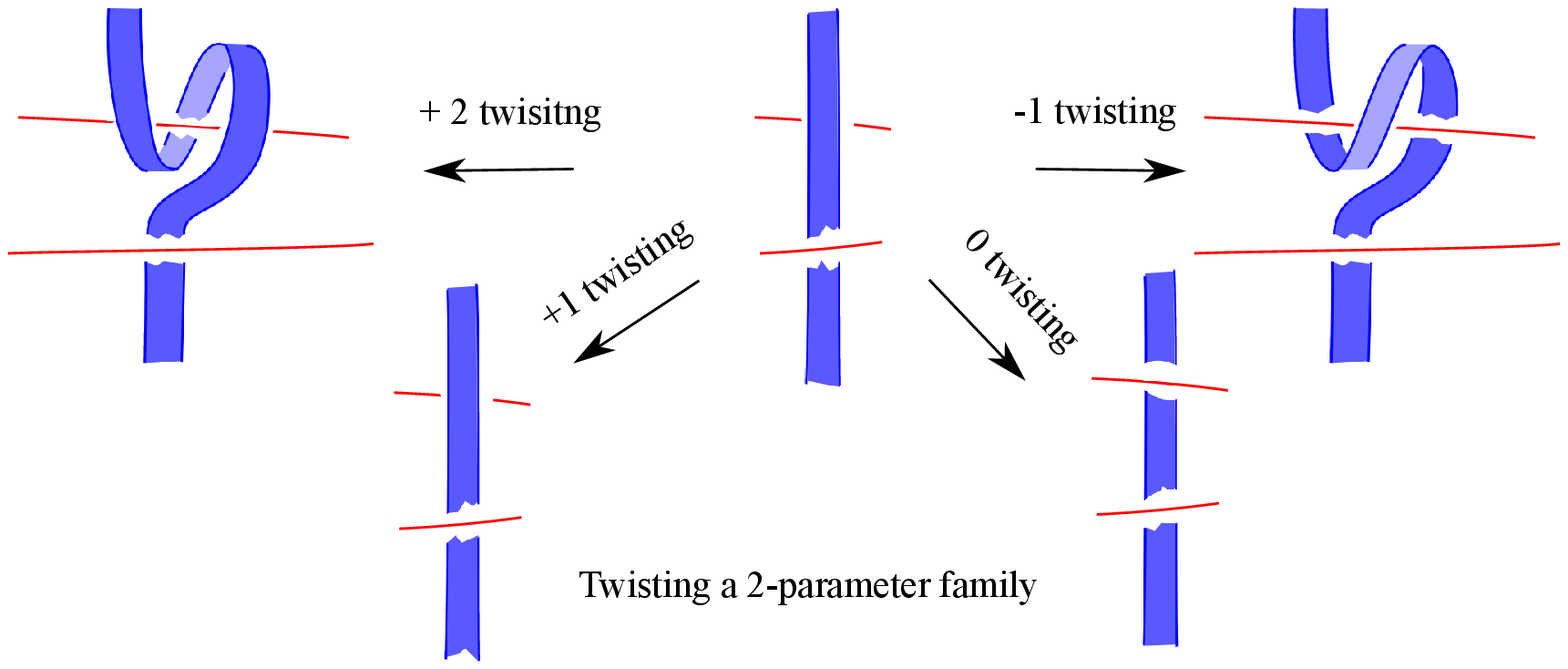}$$
\caption{\label{fig:FigureImp6}} 
\end{figure}

\begin{remark}  In the next section we shall see that $\hat\theta_k$ gives rise to a $(k-1)\times(k-1)$ matrix $F_k$ with values in $\pi_2 \E$ such that the sum of the entries equals $[\hat\theta_k]\in \pi_2 \E$.  Further, at the $\pi_2$ level, a $v(m_1, \cdots, m_{k-1}), w(n_1, \cdots, n_{k-1})$ twisting  has the effect of multiplying the $i$'th row of $F_k$ by $m_i$ and the $j$'th column by $n_j$.  \end{remark}


\section{Factoring the symmetric $\theta_k$} \label{factoring section}

In this section we investigate the 2-parameter family $\hat\theta_k:I^2 \to \Emb(I, S^1\times B^3, I_0$) called the \emph{symmetric} $\theta_k$ defined in Definition \ref{hat thetak}.  We will show $\hat\theta^L_k$ factors into $(k-1)^2$ 2-parameter families, giving rise to a $(k-1) \times (k-1)$ matrix $F^L_k(p,q)$ with values in $\pi_2 \E$   equal to those represented by the $(p,q)$'th family.  Summing over $L$ we obtain $F_k$ which will be shown to be skew symmetric.  It will follow that $[\hat\theta_k]=0\in \pi_2 \E$ and hence $W_3(\ker(\Diff_0(S^1\times B^3\fix f)\to \Diff_0(S^1\times S^3)))=0$.

As stated in  previous sections we view $S^1\times B^3$ as $(D^2\times S^1)\times [-1,1]$ with the product orientation and $I_0$ denotes a fixed geodesic arc through the origin of $D^2\times x_0\times 0$.

\begin{definition}\label{setup} Fix $k\ge 2$.  For $1\le L\le k-1$ we define properly embedded arcs $J_k^L\subset D^2\times S^1\times 0$ path homotopic to $I_0$ whose ends also coincide with those of $I_0$.  We first do a small initial isotopy of $I_0$ to obtain the arc $J_0^0$ of Figure \ref{fig:FigureCalc16} b).  To obtain $J^L_k\subset D^2\times S^1\times 0$, send the two local top strands just under $k-L$ times positively around the $S^1$ and send  the two local bottom strands just under $L$ times negatively around the $S^1$ to obtain an embedded arc, as exemplified by $J^3_5$ or $J^2_5$ shown in Figure \ref{fig:FigureCalc16} c), d).   The 20 strands at the bottom of those figures go around the $S^1$ to attach to the strands at the top.  

For $1\le L\le k-1$ homotope each $\hat\theta^L_k$ to one based at $J^L_k$ by shrinking the bands without modifying the lassos.  For example, when $k=5$ and $L=3$, starting with the family shown in  Figure \ref{fig:FigureImp4} b) we obtain the one shown in Figure \ref{fig:FigureCalc17}.   Denote this 2-parameter as the separable pair $(B^L_k, R^L_k)$, where $B^L_k$ (resp. $R^L_k$) is defined by the blue (resp. red) bands and lassos.  \end{definition}

\begin{figure}[ht]
$$\includegraphics[width=11cm]{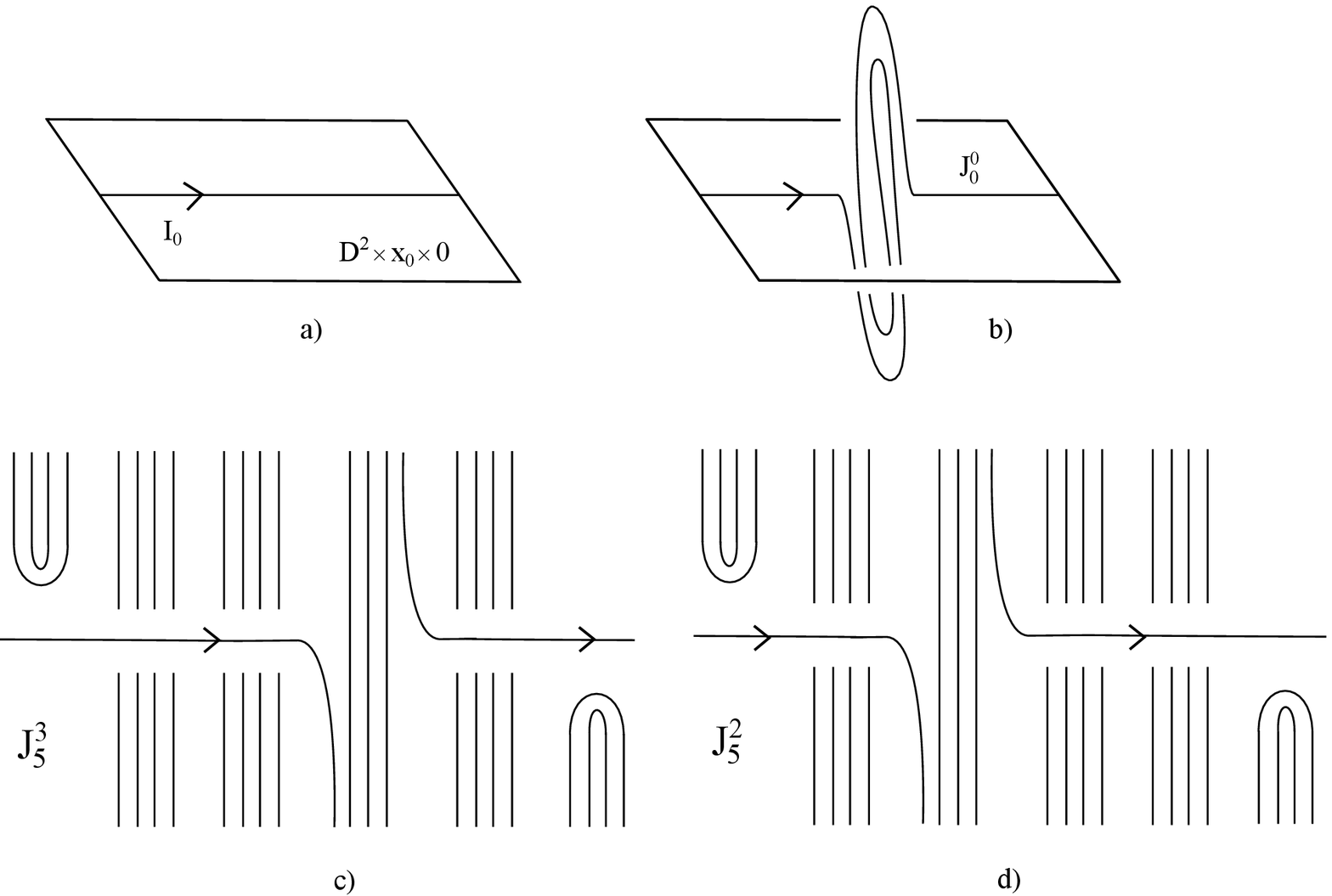}$$
\caption{\label{fig:FigureCalc16}} 
\end{figure}

\begin{figure}[ht]
$$\includegraphics[width=10cm]{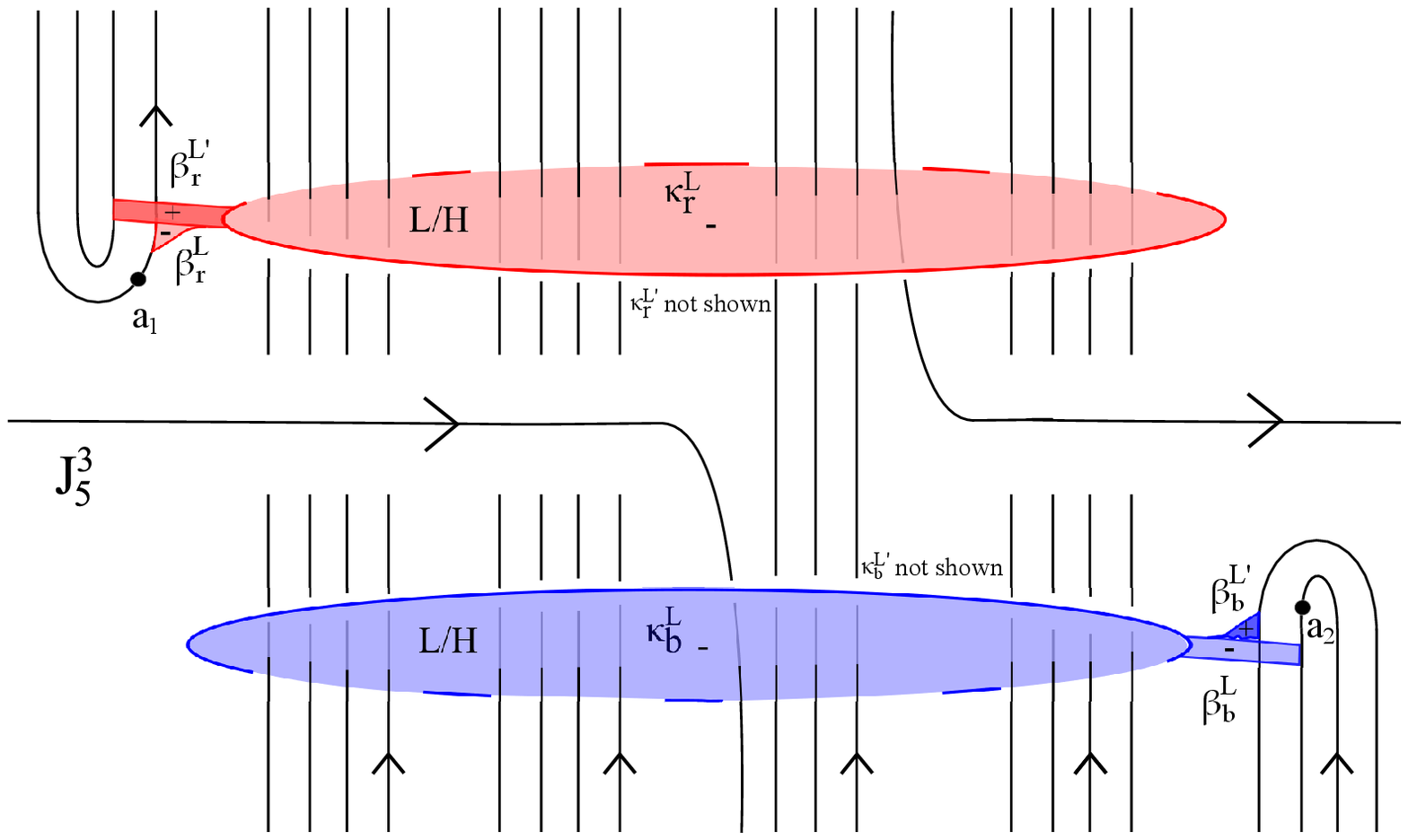}$$
\caption{\label{fig:FigureCalc17}}
\end{figure}

\begin{remarks} i) By Lemma \ref{basepoint two}, since $J^L_k$ is path isotopic to $I_0$  and as we shall see $(B^L_k, R^L_k)$ can be represented by an abstract chord diagram, it gives a well-defined element of $\pi_2 \Emb(I, S^1\times B^3, I_0)$, and so $[\hat\theta_k]=\sum_{L=1}^{k-1} [(B^L_k, R^L_k)]$.

ii) Fix $k\ge 2$.  To minimize notation we suppress most of our $k$-subscripts and abuse notation by introducing others.  \end{remarks}

\noindent\textbf{Introduction to the rest of the chapter}  We will express $B^L$ (resp. $R^L)$ as a concatenation $B^L_1, \cdots , B^L_{k-1}$ (resp. $R^L_1, \cdots, R^L_{k-1})$ and then use bilinearity to produce the $(k-1)^2$ 2-parameter families representing the classes $F^L(p,q), 1\le p, q\le k-1$,  and so $[\hat\theta_k^L]= \sum F^L(p,q)$.  The terms $F^L(p,q)$ define a $(k-1) \times (k-1)$ matrix.  Summing over $L$ we obtain the $(k-1)\times (k-1)$ matrix $F_k$ whose sum equals $[\hat\theta_k]$.  We shall see that $F^L(p,q)=-F^{k-L}(q,p)$ and hence $F_k$ is skew-symmetric.

\begin{definition}By the \emph{red coloring} (resp. \emph{blue coloring} of $J^L$ we mean a red (resp. blue) half disc $D_r$ (resp. $D_b$) in $D^2\times S^1\times 0$ such that $\partial D_r$ consists of an arc in $J^L$ and a complementary arc as in Figure \ref{fig:FigureCalc18} a) and similarly for $D_b$.  These discs should be viewed as long and thin.  With notation as in those figures, these discs color the arc $[a_1, a_3]\subset J^L$ red and the arc $[a_2, a_4]\subset J^L$ blue as in Figure \ref{fig:FigureCalc18} c). We call the \emph{bottom} (resp. \emph{top}) of $D_r$ to be that part of the half disc near $a_1$ (resp. $a_2$).   Similarly the \emph{bottom} of $D_b$ is near $a_2$ and the \emph{top} of $D_b$ is near $a_3$.   \end{definition}

\begin{figure}[ht]
$$\includegraphics[width=13cm]{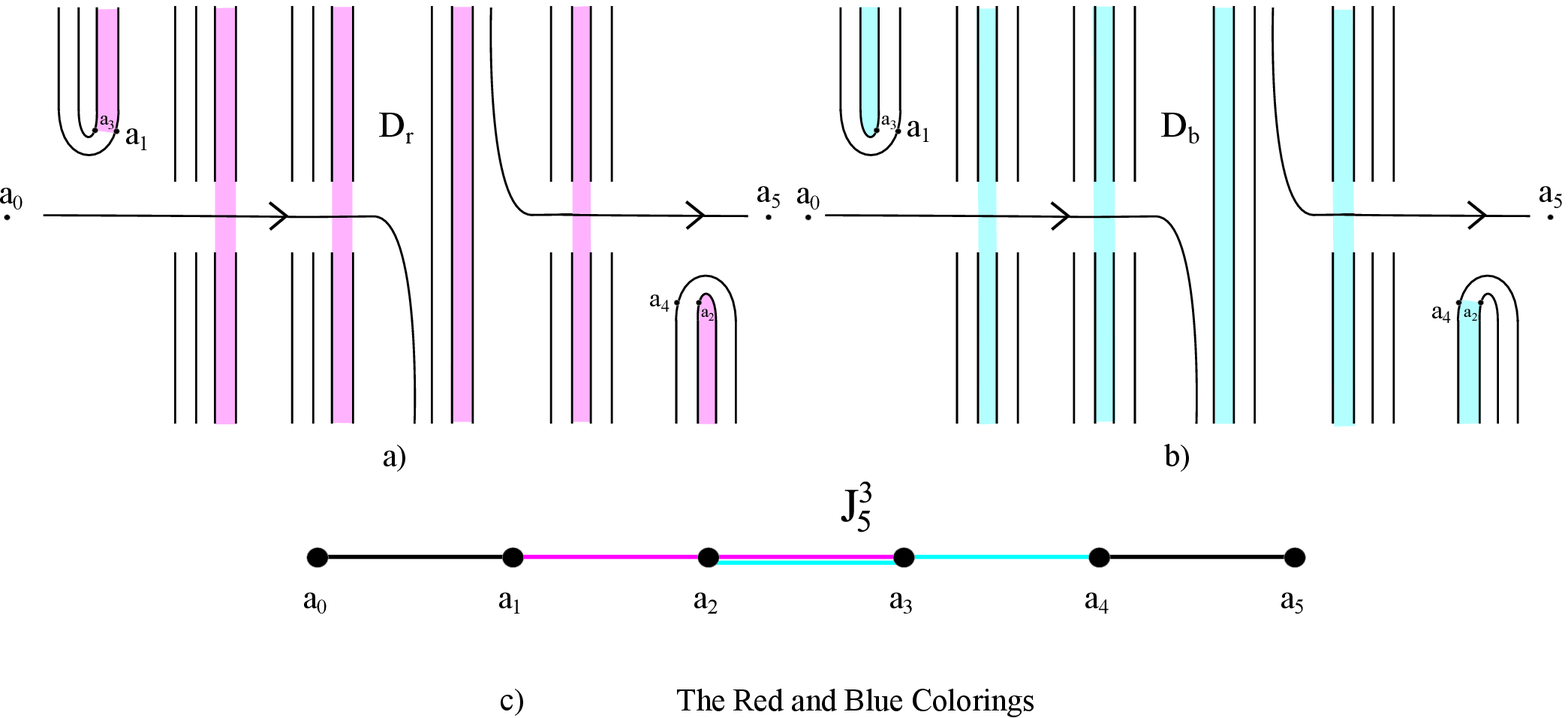}$$
\caption{\label{fig:FigureCalc18}} 
\end{figure}
\begin{remark} These colorings are very helpful for keeping track of both location and orientation.   Here is a first application.\end{remark}

\begin{lemma} \label{big undo} $B^L, R^L\in\Omega \Emb(I, S^1\times B^3, J_L)$ are homotopically trivial.  \end{lemma}

\begin{proof}  While we already know this result it is useful to see it from the coloring point of view.  We prove this for $R^L$ with the $B^L$ argument being similar.  Following the red coloring, slide the base of the red bands until they reach $D_r$'s top.  From there it is apparent that these bands form a parallel cancel pair. See Figure \ref{fig:FigureCalc19}.  By Lemma \ref{undo} the support of the undo homotopy is within the union of a neighborhood of the bases and a neighborhood of the parallelism between the bands and the parallelism between the lasso spheres, thus there is no problem with the band linking the lassos.\end{proof}

\begin{figure}[ht]
$$\includegraphics[width=8cm]{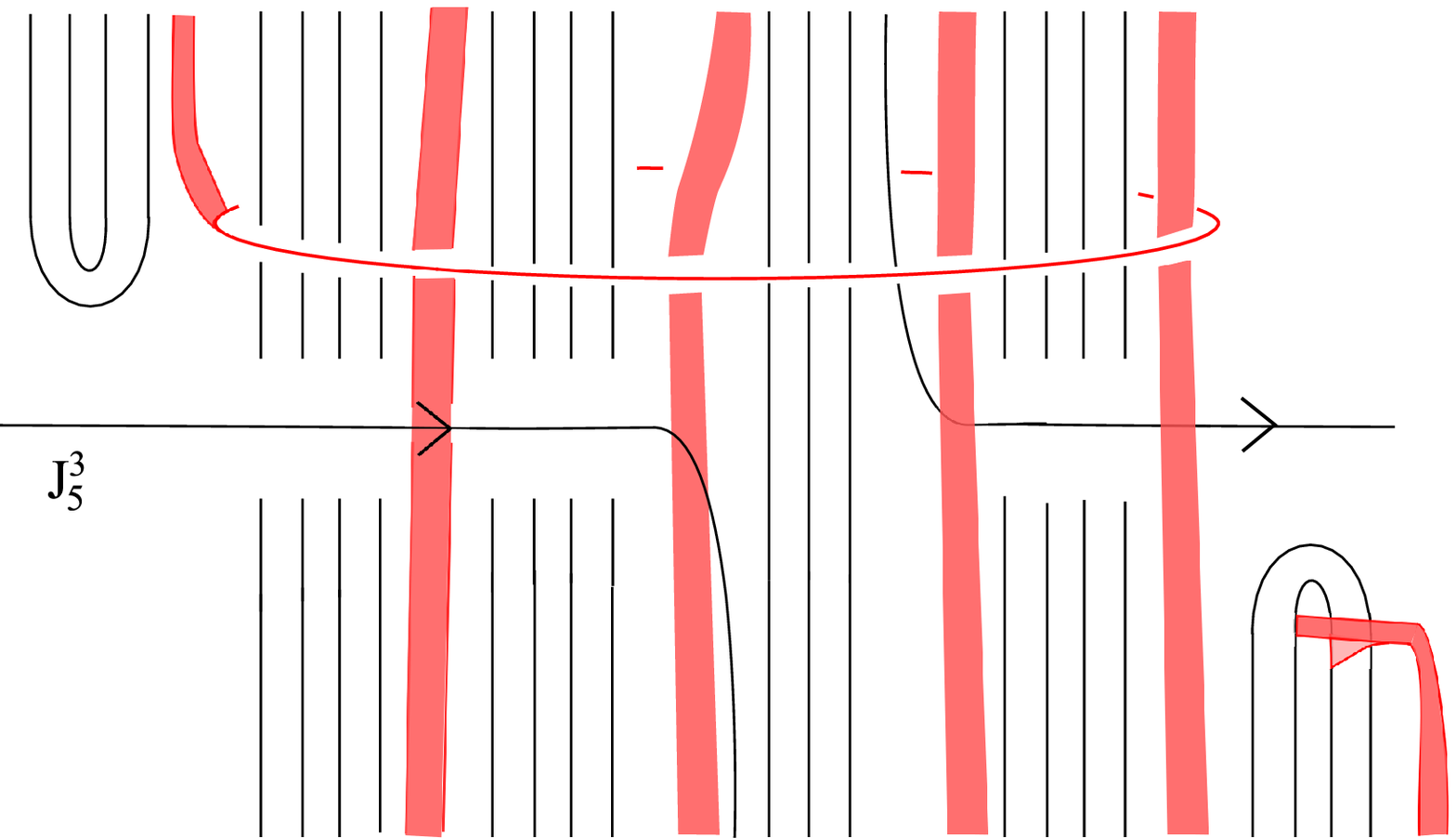}$$
\caption{\label{fig:FigureCalc19}} 
\end{figure}

\begin{convention} We let $\beta^L_b, \beta^{L'}_b, \kappa^L_b, \kappa^{L'}_b$ denote the band and lassos defining $B^L$, where base$(\beta^L_b)$ appears before base$(\beta^{L'}_b$ along $J^L$.  $\beta^L_r, \beta^{L'}_r, \kappa^L_r, \kappa^{L'}_r$ are defined analogously.  The red coloring and blue colorings allow us to readily deduce which base comes first and to keep track of how a given one is oriented.  This in turn enables us to orient the red and blue bands near their bases and then orient their lasso discs as indicated in Figure \ref{fig:FigureCalc17}.  By Remark \ref{determined} it suffices to draw only the core of the band provided we also know the orientation of the lasso disc.    In what follows usually only the core of the band will be shown in the figures.  \end{convention}

\begin{lemma}$B^L$ (resp. $R^L$) is homotopic to the concatenation of the loops $B^L_1, \cdots, B^L_{k-1}$ (resp. $R_{L_1}, \cdots, R^L_{k-1})$ represented in band lasso form in Figure \ref{fig:FigureCalc20}.   Each of these loops is homotopically trivial. \end{lemma}

\begin{figure}[ht]
$$\includegraphics[width=9cm]{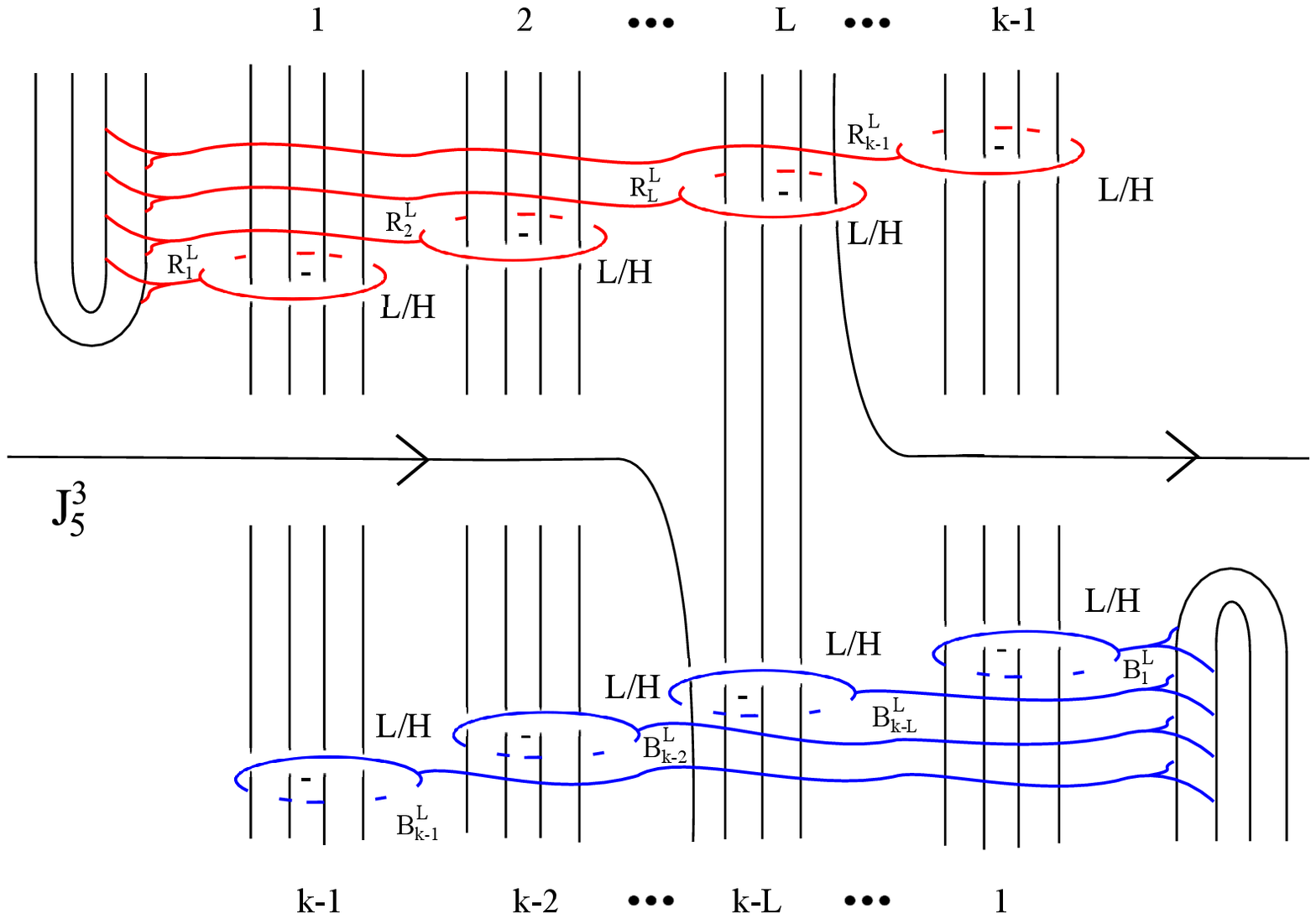}$$
\caption{\label{fig:FigureCalc20}} 
\end{figure}

\begin{proof}  The principle is that if a lasso encircles two arcs $\tau_1$, $\tau_2$, then the corresponding spinning about the union can be homotoped into a concatenation of spinnings of the same type (e.g. L/H or positive) about $\tau_1$ and $\tau_2$.  The proof that the $B^L_i$'s and $R^L_j$'s are homotopically trivial loops follows as the proof of Lemma \ref{big undo}.\end{proof}

\begin{notation} $B^L_i$ (resp. $R^L_j$) denotes loop represented by the pair of blue (resp. red) bands and lassos corresponding to linking about the i'th (resp. j'th) local pair of four strands, counted from the right (resp. left).  \end{notation}

\begin{remarks} i) This notation is chosen, since up to signs/orientations there is a symmetry $\psi$ which takes $J^L_k$ to $J^{k-L}_k$ such that the bands and lassos defining $B^L_i$ are taken to those defining $R^{k-L}_i$.  This will induce the skew symmetry mentioned above.\end{remarks}

\begin{lemma} $[\theta^L_k] = \sum_{1\le i,j\le k-1} [(B^L_i, R^L_j)]$.\end{lemma}

\begin{proof} The homotopy from $(B^L, R^L)$ to $(B^L_1*\cdots*B^L_{k-1}, R^L_1*\cdots*R^L_{k-1})$ is separable, thus the result follows from bilinearity.\end{proof}

\begin{definition}  Define $F^L(p,q) =[(B^L_p, R^L_q)]$.\end{definition}

\begin{proposition} \label{pq big}  If  $p\ge k-L$ and $q\ge L$, then $F^L(p,q)=D(p,-q$)\end{proposition}

\begin{proof} To minimize notation we suppress the $L$ superscripts.  Figure \ref{fig:FigureCalc21} shows a representative case.    Observe that  $B_p$ is the concatenation of $B^0_p$ and $B^1_p$ where $B^0_p$ (resp. $B^1_p$) is given by a band lasso pair we call ($\beta^0_b, \kappa^0_b$) (resp. ($\beta^1_b, \kappa^1_b)$) where base($\beta^0_b)$ precedes base($\beta^1_b)$ on $J$ and that $(B_p,R_p)$ is separably homotopic to $(B^0_p*B^1_p, R_q)$.  The hypotheses on p and q are the conditions such that both the blue and red lassos can slide off of $J$ following four of its local strands, so in particular each of $B^0_p$ and $B^1_p$ is homotopically trivial.  Thus, by bilinearity $[(B_p, R_q)]=[(B^0_p, R_q)]+[(B^1_p, R_q)]$.  Since the support of the undo homotopy of $R_q$ is disjoint from both $\beta^1_p$ and its lasso sphere it follows that $[(B^1_p, R_q)]=0$ and hence $F(p,q)=[(B^0_p, R_q)]$. 

\begin{figure}[ht]
$$\includegraphics[width=9cm]{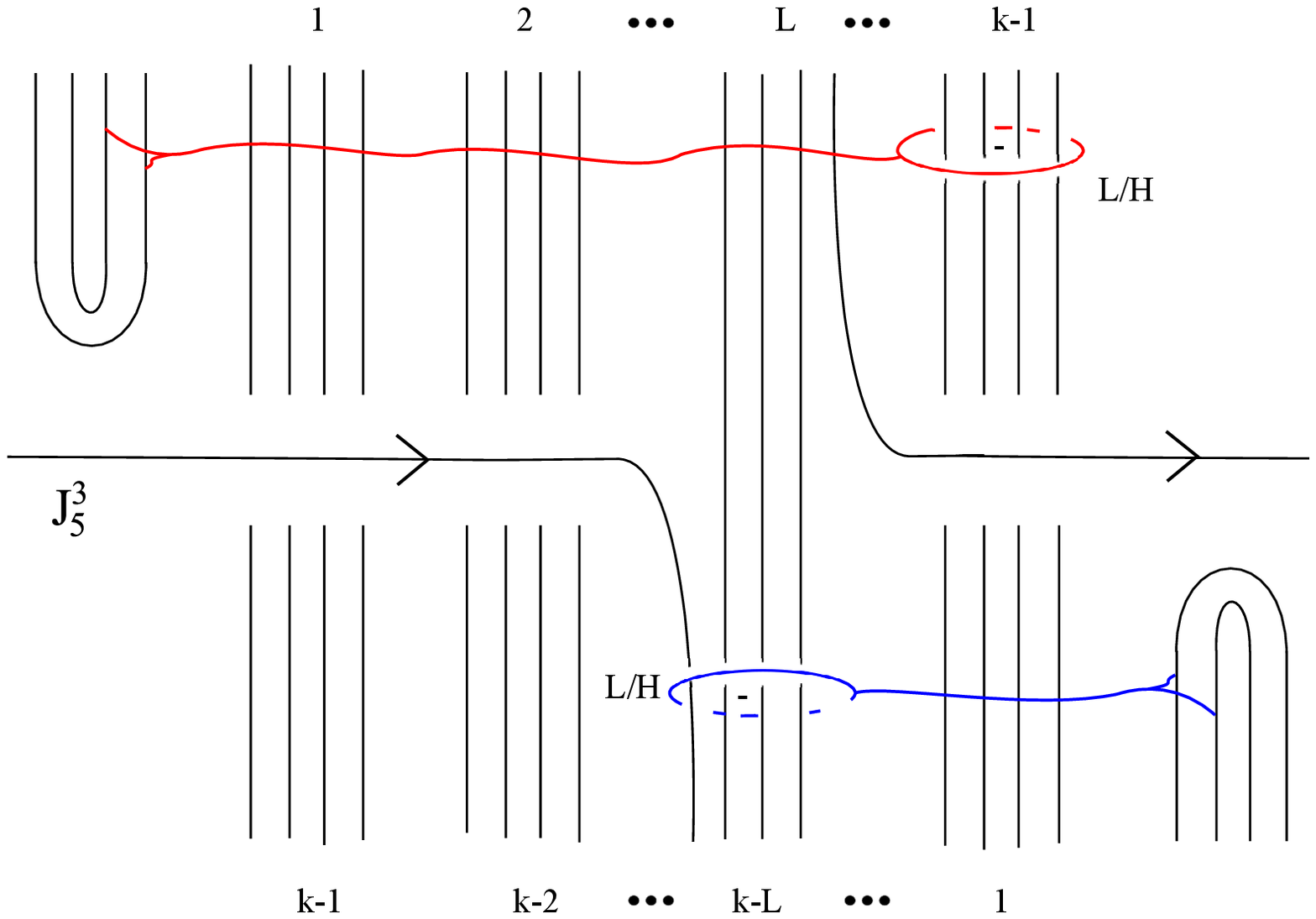}$$
\caption{\label{fig:FigureCalc21}} 
\end{figure}

To compute $F(p,q)$ slide the red bands so that their bases are near the top of the red coloring $D_r$ and the red lasso discs link $\beta^0_b$ near its base.  See Figure \ref{fig:FigureCalc22} a).  Next, isotope $\beta^0_b$ and $\kappa^0_b$ so that $\kappa^0_b$ links the cores of the red bands at their branch loci.  This is done in two steps.  First, use the 4-strand isotopy to isotope $\kappa^0_b$ off of $J$ and link the red bands.  Then, isotope $\kappa^0_b$ to link the bands at their branch loci.  Figure \ref{fig:FigureCalc22} b) shows $\kappa^0_b$ midway through the second step.

\begin{figure}[ht]
$$\includegraphics[width=13cm]{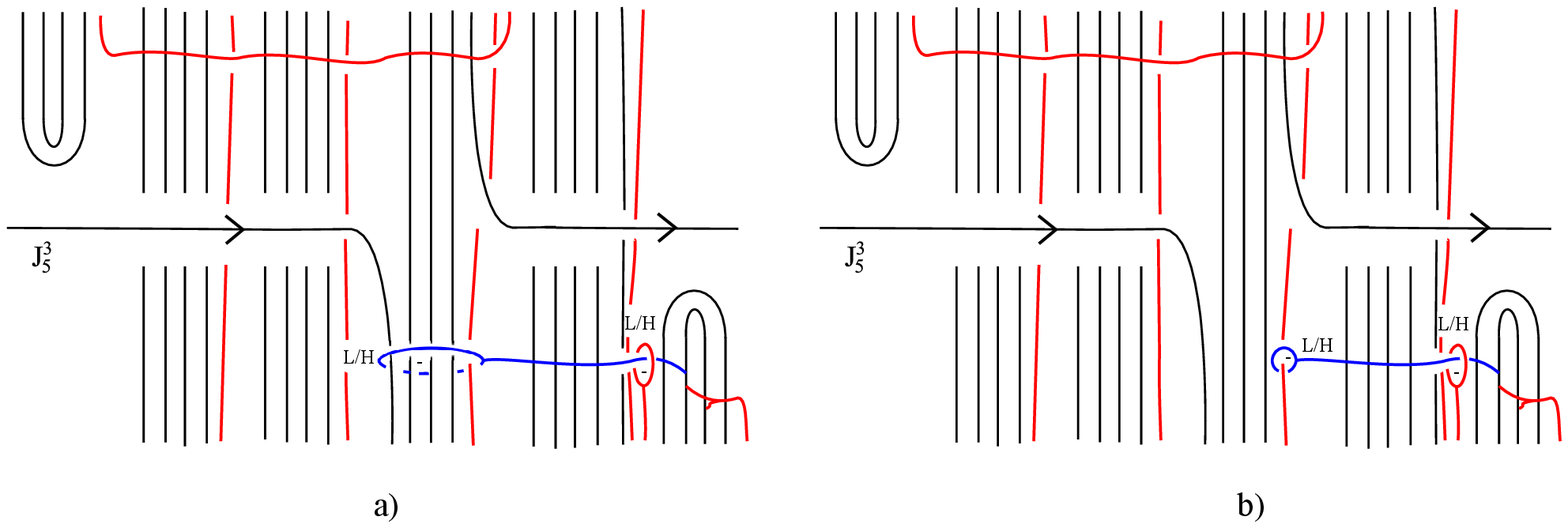}$$
\caption{\label{fig:FigureCalc22}} 
\end{figure}

We have therefore obtained a $\pm D(m,n)$ as in Definition \ref{double}.  Since the red bases have been moved k times around the $S^1$ and its lasso $k-q$ times around the $S^1$, following the convention of Definition \ref{double}, $n=(k-q)-k=-q$.  The blue lasso moves $-(k-p)$ times around $S^1$ to slide off of $J$ and then moves $k$ times around $S^1$ to link the red bands at their branch loci.  Therefore, $m=-(k-p)+k=p$.  

We now compute the sign.  Following Remarks \ref{dpq signs} the sign is determined the parity of how many of the following things deviate from the standard $D(m,n)$.  
\vskip 8pt
i) It is of standard $(B,R)$ type

ii) The blue spinning is $L/H$

iii) The red spinning is $L/H$

iv) $<$red band core, blue lasso disc$>= -1$

v) $<$blue band core, red lasso disc$>= -1$
\vskip 8pt
Since there are two changes the sign is $+1$.\end{proof} 

\begin{proposition} \label{pq middle} If $q<L$ and $p\ge k-L$ \emph{or} $q\ge L$ and $p<k-L$, then $F^L(p,q) = F^{L,s}(p,q)+F^{L,r}(p,q)$ where
\vskip 8pt

$F^{L,s}(p,q)=$
\vskip 8pt
$D(p,-q)$ if $q<L$ and $p+q\ge k$

$D(p,-q) - D(p+q,-q)$ if $q<L$ and $p+q<k$

$D(-q,p)$ if $p<k-L$ and $p+q\ge k$

$D(-q,p)-D(-p-q,p)$ if $p<k-L$ and $p+q<k$.
\vskip 8pt
$F^{L,r}(p,q)=$
\vskip 8pt
$-D(p-q,q)$ if $q<L$ and $p+q\ge k$

$-D(p-q,q)+D(p,q)$ if $q<L$ and $p+q<k$

$-D(p-q,-p)$ if $p<k-L$ and $p+q\ge k$

$-D(p-q,-p)+D(-q,-p)$ if $p<k-L$ and $p+q<k$. \end{proposition}

\begin{proof}  To minimize notation we delete the $L$ superscript.
\vskip 8pt
\noindent\emph{Case 1:} $q<L$
\vskip 8pt
\noindent\emph{Idea of Proof:}   As before we reduce to $B_p$ being defined by a single band and lasso.  Under the hypothesis only the blue lasso can slide off of $J$ by following four strands.  To address this, we homotope $R_q$ into a concatenation of $R^s_q$ and $R^r_q$ each of whose lassos slides off  $J$ by following the red coloring.  We define $F^s(p,q)=[(B_p, R^s_q)]$ and $F^r(p,q)=[(B_p, R^r_q)]$.  Each of these has two subcases; $p+q\ge k$ and $p+q<k$.  In the first case the blue and red lassos can slide off of $J$ without crashing into each other.  In the second case we factor $B_p$ into $B^0_p$ and $B^1_p$ each of which can now slide off of $J$ without crashing into the red bands and lassos.  An argument similar to that of the previous proposition shows that each of $F^{s}(p,q), F^r(p,q)$ is one or two $\pm D(m,n)$'s depending on whether subcase 1 or 2 applies.
\vskip 10pt
\begin{figure}[ht]
$$\includegraphics[width=13cm]{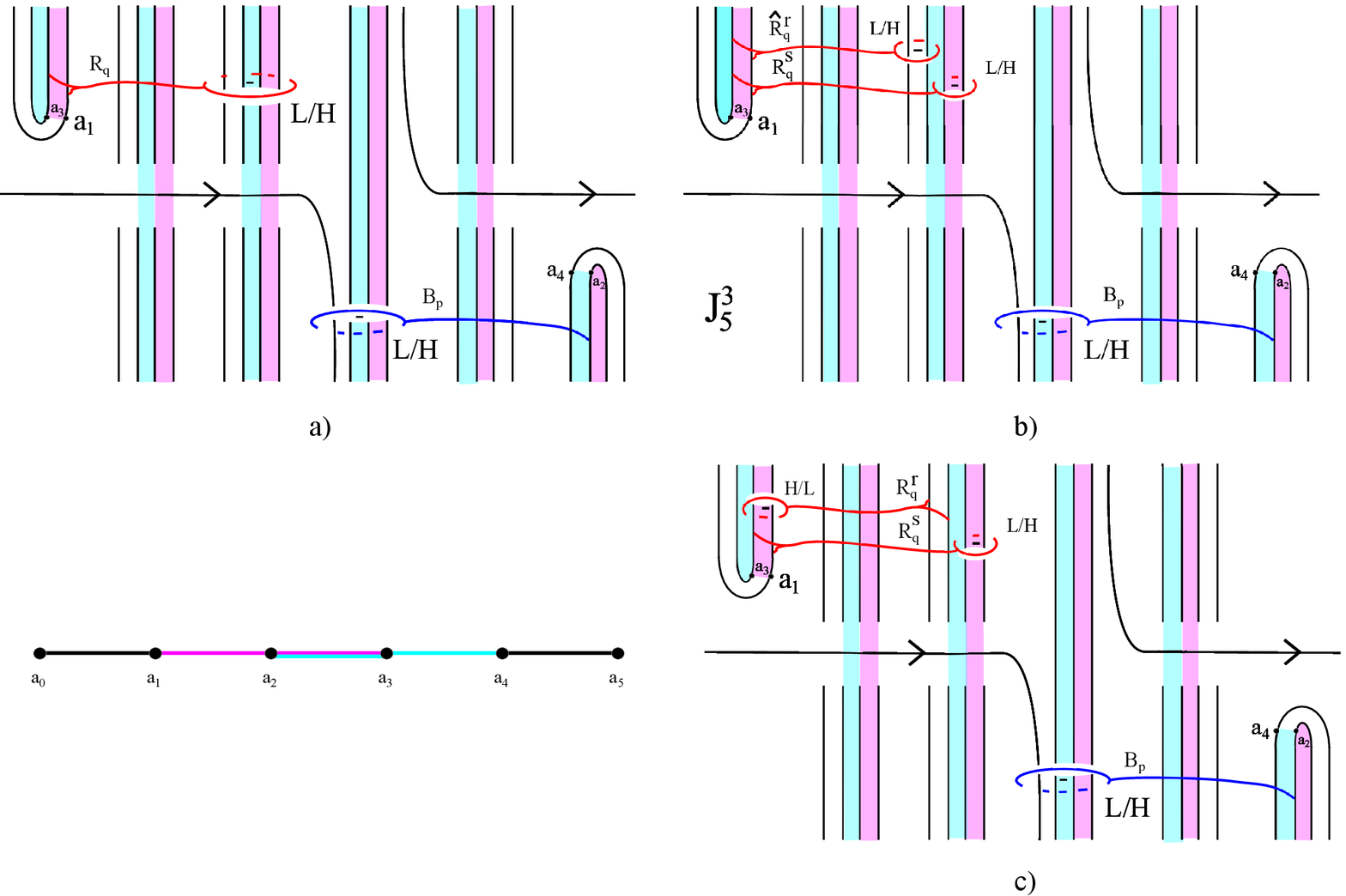}$$
\caption{\label{fig:FigureCalc23}} 
\end{figure}
\noindent\emph{Proof of Case 1:} $B_p$ is represented by $(\beta_b, \kappa_b)$ and $(\beta'_b, \kappa'_b)$ where the base of $\beta_b$ appears before that of $\beta'_b $ on $J$.  We abuse notation by calling $B_p$ the loop represented by just $(beta_b, \kappa_b)$, since as in the proof of the previous proposition we can use this $B_p$ to compute $F(p,q)$.   Figure  \ref{fig:FigureCalc23} a) shows a representative example of a $(B_p, R_q)$ for Case 1.  Now homotope $R_q$ into the concatenation of $R^s_q$  and $\hat R^r_q$ as in Figure \ref{fig:FigureCalc23} b) and use Lemma \ref{elementary homotopies} to reverse the direction of the band and lassos representing $\hat R^r_q$ to obtain $R^r_q$.  See Figure \ref{fig:FigureCalc23} c) We use the superscript $r$ (resp. $s$) to denote \emph{reversed} (resp. \emph{standard}). Several loops in Figure \ref{fig:FigureCalc23}, such as $R^r_q$, are represented by two bands and lassos whose bands are branched.  For such a pair, the orientation of the lasso disc shown in Figure \ref{fig:FigureCalc23} is for the lasso whose base appears first on $J$.  The homotopies of $R_q$ to $R^s_q*\hat R^r_q$ and $\hat R^r_q$ to $R^r_q$ are supported away from both the domain and range support of $B_p$, thus $(B_p, R_q)$ (resp. ($B_p, \hat R^r_q$)) is separably homotopic to $(B_p, R^s_q*R^r_q)$ (resp. ($B_p, R^r_q$)).  Each  of the loops $R^s_q$ and $R^r_q$ are homotopically trivial for their lassos can be slid off J by following the red coloring.    Therefore,  $$F(p,q)=[(B_p, R_q)]=[(B_p, R^s_q*\hat R^r_q)]=[(B_p, R^s_q)]+[(B_p, \hat R^r_q)]=[(B_p, R^s_q)]+ [(B_p, R^r,q)].$$  Define $F^s(p,q)=[(B_p,R^s_q)] $ and $F^r(p,q)=[B_p, R^r_q)]$.

\vskip 8pt

\noindent\emph{Computation of $F^s(p,q)$}  This is very similar to that of the previous proposition.  If $p+q\ge k$, then following the red coloring, slide the red lassos to link the core of the blue band near its base and then slide the red bands so that their bases are at the top of the red coloring.  Next first slide the blue lasso off of $J$ to link the red band cores and then isotope the blue lasso to link the red band cores at their branch loci.  The same calculation as in Proposition \ref{pq big} shows that $F^s(p,q)=D(p,-q)$.

If $p+q<k$, then slide the blue lasso to just before it crashes into the red lassos and bands.  See Figure \ref{fig:FigureCalc24} a).  Next homotope $B_p$ into a concatenation of $B^0_p$ and $B^1_p$ as in Figure \ref{fig:FigureCalc24} b).  The separable homotopy and bilinearity arguments imply that $[(B_p, R^s_q)]=[(B^0_p, R^s_q)]+[(B^1_p, R^s_q)]$.  That $[(B^1_p, R^s_q)]=D(p,-q)$ follows as in the $p+q\ge k$ case.  Now $[(B^0_p, R^s_q)]$ is also a $\pm D(m,n)$ with the following changes in comparison with that of $[(B^1_p, R^s_q)]$.  Here we have $<$red band core, blue lasso$>=+1$ vs the $-1$ before, while the other four values determining the sign are the same as those for $[(B^1_p, R^s_q)]$.  Therefore, we have a sign change.  The previous argument also shows that $n=-q$.  To compute $m$, note that the lasso of $B^0_p$ is already linking the core of the red bands and a horizontal isotopy moves it to their branch loci while the $B^1_p$ lasso had to travel an extra $-q$ times around the $S^1$ to get to the same point.  It follows that $m=p+q$ and hence $[(B^0_p, R^s_q)]=-D(p+q,-q)$ and so $F^s(p,q)=D(p,-q)-D(p+q,-q)$.  \qed

\begin{figure}[ht]
$$\includegraphics[width=13cm]{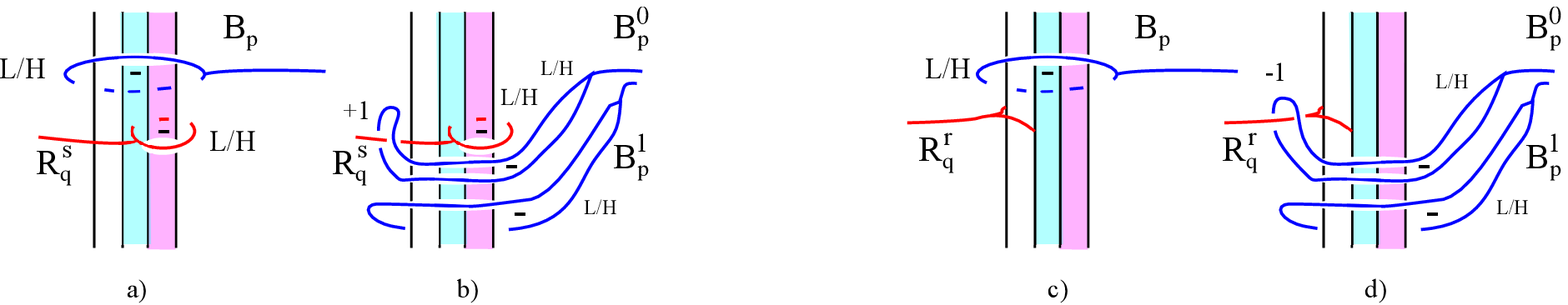}$$
\caption{\label{fig:FigureCalc24}} 
\end{figure}\vskip 10pt

\noindent\emph{Computation of $F^r(p,q)$:}  If $p+q\ge k$, then following the local four strands, isotope the blue lasso off of $J$ to link the bands of the red lassos and then further isotope the blue lasso to link these bands at their branch loci.  Next, follow the red coloring to isotope the red lassos off of $J$ and link the blue band near its base.  Finally, isotope the red bands, so that their bases are near the top of the red coloring.  This is done in two steps.  First isotope so that the bases contain the points $a_1, a_3$ and then follow the red coloring from its bottom to near its top.  Simultaneously isotope the linking blue lasso to follow along. We now compute the values of this $\pm D(m,n)$.  The blue lasso goes $-(k-p)$ times around the $S^1$ to slide off of $J$ and then, in the aggregate, $k-q$ times about the $S^1$ while following the red bands.  Therefore, $m=p-q$.  The red lassos go $k$ times around the $S^1$ while in the aggregate the bases go $k-q$ times about the $S^1$ so that $n=k - (k-q)=q$.  We now compute the sign.
\vskip 8pt
i) It is of standard $(B,R)$ type

ii) The blue spinning is $L/H$

iii) The red spinning is $H/L$

iv) $<$red band core, blue lasso disc$>= +1$

v) $<$blue band core, red lasso disc$>= +1$
\vskip 8pt
Since there is one sign change from the standard $D(m,n)$ the sign is -.  Therefore $F^r(p,q)=-D(p-q,q)$.
\vskip 8pt
If $p+q<k$, then slide the blue lasso to just before it crashes into the red bands.  See Figure \ref{fig:FigureCalc24} c).  Next homotope $B_p$ into a concatenation of $B^0_p$ and $B^1_p$ as in Figure \ref{fig:FigureCalc24} d).  The separable homotopy and bilinearity arguments imply that $[(B_p, R^r_q)]=[(B^0_p, R^r_q)]+[(B^1_p, R^r_q)]$.  That $[(B^1_p, R^r_q)]=-D(p-q,q)$ follows as in the $p+q\ge k$ case.

Now $[(B^0_p, R^r_q)]$ is also a $\pm D(m,n)$.   Here $n = q$ as before. To compute $m$, note that the lasso of $B^0_p$ is already linking the core of the red bands while the $B^1_p$ lasso had to travel an extra $-q$ times around the $S^1$ to get to essentially the same point.  It follows that $m=(p-q)+q=p$.  For the sign, note that there is a single change in comparison with the five values computed for $[(B^1_p, R^r_q)]$.  Here $<$red band core, blue lasso disc$>= -1$.  It follows that the sign is positive. Therefore,  $[(B^0_p, R^s_q)]=D(p,q)$ and hence  $F^r(p,q)=-D(p-q,q)+D(p,q)$.\qed
\vskip10pt

\noindent\emph{Case 2:} $p<k-L$
\vskip 8pt
\noindent\emph{Proof of Case 2:}  $R_q$ is represented by $(\beta_r, \kappa_r), (\beta'_r,\kappa'_r)$ where the base of $\beta_r$ appears first.  The argument of Case 1 shows that $[(B_p, \sigma(\beta_r, \kappa_r)]=0$ and so we again abuse notation by letting $R_q$ be represented by just $(\beta'_r, \kappa'_r)$.  Figure \ref{fig:FigureCalc25} a) shows a representative example of a $(B_p, R_q)$ for Case 2.  Note the sign of the $\kappa'_r$ lasso disc.   In this case $B_p$ is homotopic to  a concatenation of $B_p^s$ and $\hat B_p^r$ as shown in Figure \ref{fig:FigureCalc25} b).  We reverse the direction of the band and lasso representing $\hat B^r_p$ as in Figure \ref{fig:FigureCalc25} c) to obtain $B^r_p$.  As before we obtain $F(p,q)=[(B^s_p, R_q)] +[(B^r_p, R_q)]$ and define $F^s(p,q)=[(B^s_p, R_q)]$ and $F^r(p,q)=[(B^r_p, R_q)]$.
\begin{figure}[ht]
$$\includegraphics[width=13cm]{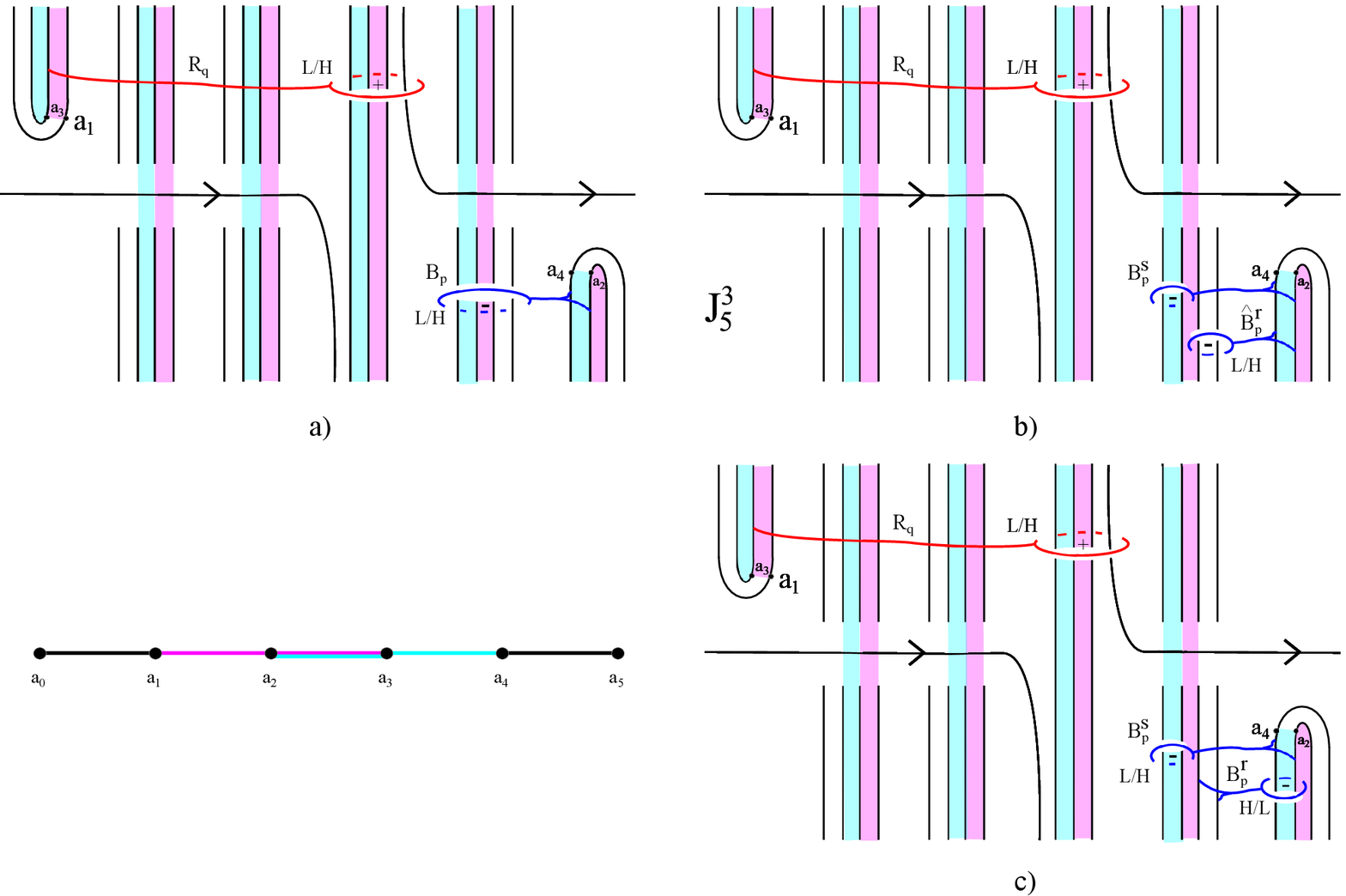}$$
\caption{\label{fig:FigureCalc25}} 
\end{figure}
\vskip 10 pt
\noindent\emph{Computation of $F^s(p,q)$:}  As in Case 1 this will be one or two $\pm D(m,n)$'s depending on whether or not $p+q\ge k$.  In both subcases the branched band will be blue, so following Definition \ref{double} $m$ will be computed using the red band and $n$ from the blue bands.  With that in mind the calculation follows the method of Case 1.  When $p+q\ge k$ we have $m=-q$ and $n=p$.  To compute the sign note that there are two cancelling changes in comparison with Case 1.  Here there is a red/blue switch and also $<$blue band core, red lasso$>=+1$, in comparison with the earlier $-1$.  Therefore the sign is $+$ and $F^s(p,q)=D(-q,p)$.

When $p+q<k$, then slide the red lasso to just before it crashes into the blue lassos and bands. See Figure \ref{fig:FigureCalc26} a).  Next homotope $R_q$ to a concatenation of  $R^0_q$ and  $R^1_q$ as in  Figure \ref{fig:FigureCalc26} b).  
Here $[(B^s_p, R_q)]=[(B^s_p, R^0_q)]+[(B^s_p, R^1_q)]$.  That $[(B^s_p, R^1_q)]=D(-q,p)$ follows as in the $p+q\ge k$ case.  Using that, the method of Case 1 shows that $[(B^s_p, R^0_q)]=-D(-p-q, p)$. Therefore,  $F^s(p,q) = D(-q,p) -D(-p-q, p)$. \qed
\vskip 10pt
\noindent\emph{Computation of $F^r(p,q)$:}  Again there will be one or two $\pm D(m,n)$'s depending on whether or not $p+q\ge k$, the branched band will be blue and so $m$ will be computed using the red band and $n$ from the blue bands.  With that in mind the calculation follows the method of Case 1.  When $p+q\ge k$ we have $m=p-q$ and $n=-p$.  We now compute the sign.
\vskip 8pt

i) It is of $(R, B)$ type

ii) The blue spinning is $H/L$

iii) The red spinning is $L/H$

iv) $<$red band core, blue lasso disc$>= +1$

v) $<$blue band core, red lasso disc$>= -1$
\vskip 8pt
There are three changes so that the sign is - and hence $F^r(p,q)=-D(p-q, -p)$.

When $p+q<k$, then slide the red lasso to just before it crashes into the blue bands. See Figure \ref{fig:FigureCalc26} c).  Next homotope $R_q$ to a concatenation of  $R^0_q$ and  $R^1_q$ as in  Figure \ref{fig:FigureCalc26} d).  
Again $[(B^r_p, R_q)]=[(B^r_p, R^0_q)]+[(B^r_p, R^1_q)]$.  
\begin{figure}[ht]
$$\includegraphics[width=13cm]{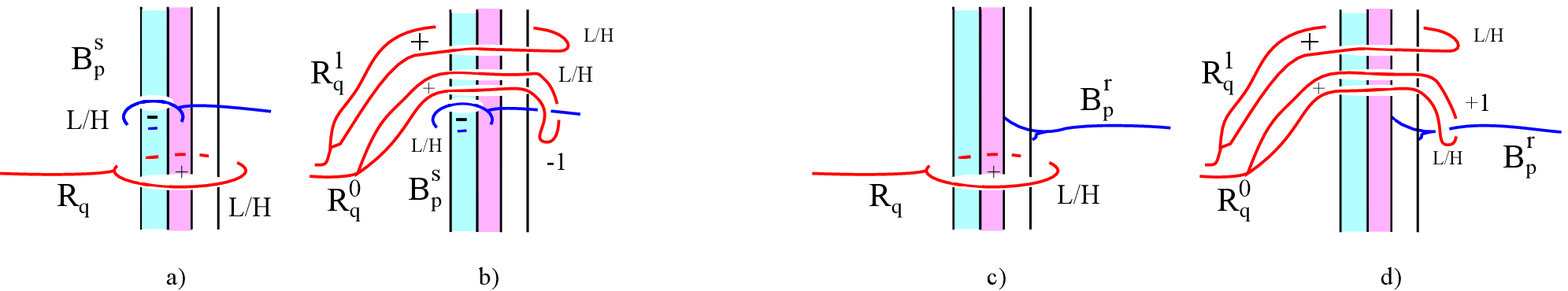}$$
\caption{\label{fig:FigureCalc26}} 
\end{figure}
That $[(B^r_p, R^1_q)]=-D(p-q,-p)$ follows as in the $p+q\ge k$ case.  Using that, the method of Case 1 shows that $[(B^r_p, R^0_q)]= D(-q, -p)$. Therefore,  $F^s(p,q) = -D(p-q,-p) +D(-q, -p)$. \end{proof}

\begin{proposition}\label{pq small}  For $p<k-L$ \emph{and} $q<L, \ F^L(p,q) = 0$.\end{proposition}

\begin{figure}[ht]
$$\includegraphics[width=13cm]{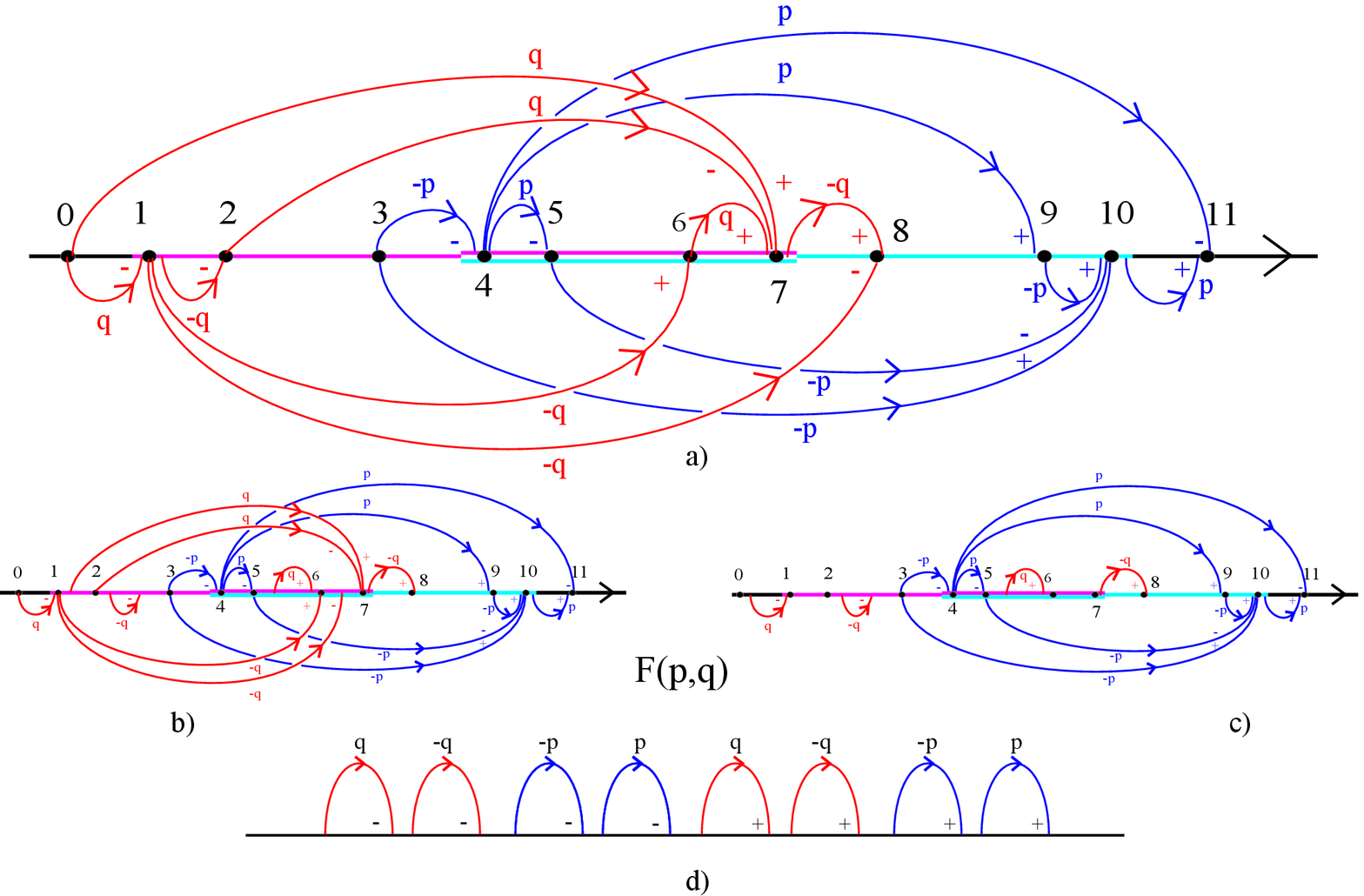}$$
\caption{\label{fig:FigureCalc27}} 
\end{figure}

\begin{proof}  Again we supress the L's. 
\vskip 8pt
 \emph{Step 1:} $F(p,q)$ is represented by the abstract chord diagram pairs of Figure \ref{fig:FigureCalc27} a), b), c) and d).

\vskip 8pt

\noindent\emph{Proof:}  A representative $(B_p, R_q)$ is shown in Figure \ref{fig:FigureCalc28} a).  Figure \ref{fig:FigureCalc28} b) shows details of neighborhoods of the band bases and lasso discs.  Each of the spinnings determined from the four bands and lassos is homotopic to a concatenation of four spinnings of arcs about arcs, via a homotopy that is supported in a small neighborhood of the relevant band and lasso ball.  For example, ($\beta_r, \kappa_r$) corresponds to spinning an arc near the point labeled 1 about arcs near the points 0,8,6 and 2.  The labeling of the eleven points in the diagram is induced by their ordering in $J$.   See Figure \ref{fig:FigureCalc28} c).  This data gives rise to the abstract chord diagram of Figure \ref{fig:FigureCalc27} a), where the chords are oriented from the low label to the high one.  For each chord we need to compute group elements and signs.   For example, $\beta_r, \kappa_r$ gives rise to a $\lambda$ spinning where $\lambda$ is approximately a horizontal arc in $S^1\times B^3$ from 1 to 0.  If $D_r $denotes the oriented lasso disc bounded by $\kappa_r$, then near 0, $<J, D_r>=+1$, so by Lemma \ref{determined} the oriented chord from 1 to 0 has positive sign.  By Lemma \ref{elementary homotopies} reversing the direction changes the sign as indicated in Figure \ref{fig:FigureCalc27} a).  To compute the group element we concatenate $\lambda$ with the arc $[0,1]\subset J$ to obtain an oriented closed loop representing $-q\in \pi_1(S^1\times B^3; J)$.  Since we reversed the direction, the chord in Figure \ref{fig:FigureCalc27} a) is labeled $q$.  Note that the colorings inform us that the short arcs between 9,5,3 and 11 as well as the one between 1 and 7 can be homotoped into $J$.  It follows that all eight of the chords between these points are labeled $\pm q$.  

Applying multiple applications of the Exchange Lemma \ref{chord moves} to the red chords gives Figure \ref{fig:FigureCalc27} b).  Two applications of the undo homotopy eliminates two pairs of parallel oppositely signed red chords with the same group element, giving Figure \ref{fig:FigureCalc27} c).  Similarly, blue exchanges followed by two applications of the undo homotopy gives Figure \ref{fig:FigureCalc27} d).\qed
\vskip 8pt
\indent\emph{Step 2:}  F(p,q)=0
\vskip 8pt
\noindent\emph{Proof:}  Using bilinearity, the class of Figure \ref{fig:FigureCalc27} a) factors into four classes represented by the four chord diagrams $H(\pm p, \pm q)$, where $H(p,q)$ is the subdiagram consisting of the chords labeled $p$ and $q$.  The other diagrams are defined similarly.  Notice that any diagram can be turned into another by an appropriate sign change of the group elements and exchange moves.  By bilinearity $[H(p,q)]$ factors into four classes represented by the chord diagrams $D_a, D_b, D_c, D_d$ respectively shown in Figures \ref{fig:FigureCalc29} a), b), c), d).  The chord diagrams $D_a, D_b, D_c$ are readily shown to respectively represent $G(p+q,p), -G(q,p)$ and $G^*(p,q)$.  The chord diagram $D_e$ is obtained from $D_d$ by a chord sliding move.  Applying bilinearity $[D_e]$ factors into $[D^0_e]+[D^1_e]$ where $D^0_e$ (resp. $D^1_e$) consists of the red chords and the blue chords labeled $p$ (resp. $p+q$).   Now $[D^0_e]=-E(p,-q)$ and $[D^1_e]$ is readily shown to be equal to $G(p+q,q)$.  Therefore, $[D_d]=[D_e]=G(p+q,q)-E(p,-q)$ and hence  $[H(p,q)]$

$=G(p+q,p)-G(q,p)+G^*(p,q)+G(p+q,q)-E(p,-q)$

$=G(p+q,p)+G(p+q,q)-G(q,p)-G(p,p-q)+G(q,p)-G(p,q)$

$=G(p, p+q)+G(p,-q)-G(p,p-q)-G(p,q).$

Here the first equality followed from Lemmas \ref{symmetric g} and \ref{epq} and the second from the symmetry relation, Lemma \ref{epq relations}  iii).  It follows that $[H(p,q)]+[H(p,-q)]=[H(-p,q)]+[H(-p,-q)]=0$ and hence $F(p,q)=0$.\end{proof}

\begin{figure}[ht]
$$\includegraphics[width=13cm]{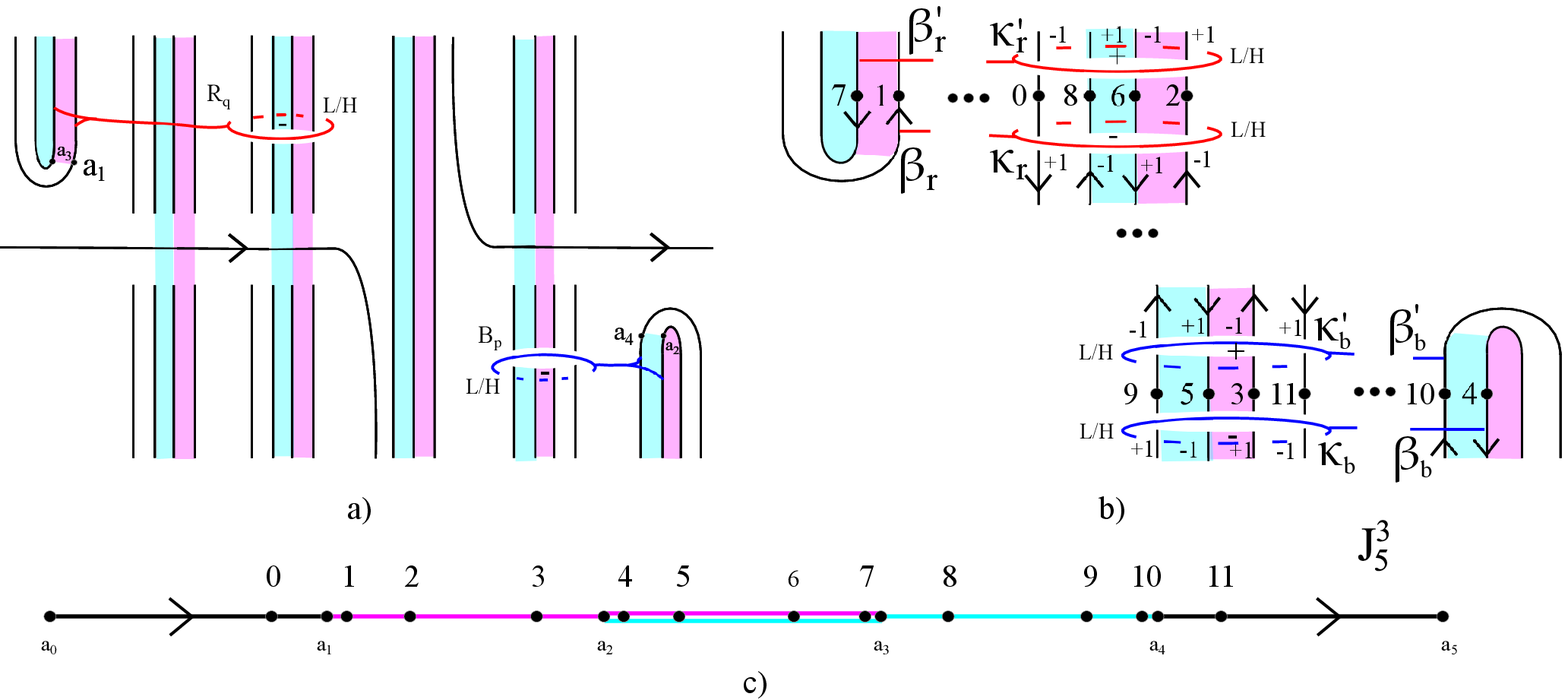}$$
\caption{\label{fig:FigureCalc28}} 
\end{figure}

\begin{figure}[ht]
$$\includegraphics[width=11cm]{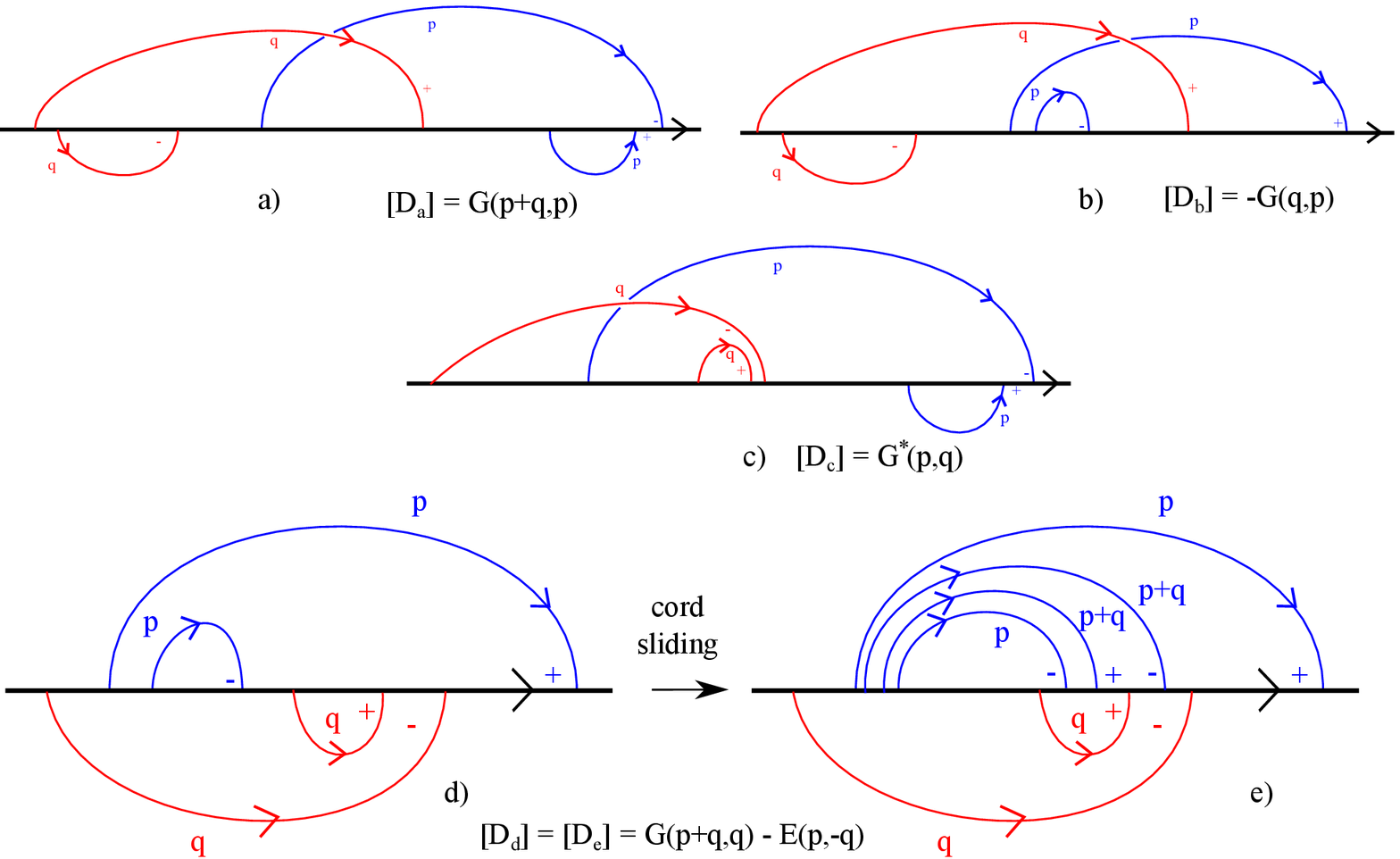}$$
\caption{\label{fig:FigureCalc29}} 
\end{figure}

\begin{definition} \label{roman defs} Define

$\textrm{I}(p,q)= D(p,-q)$

$\textrm{II}_b(p,q)= D(-q,p) -D(p-q,-p)$

$\textrm{II}_{be}(p,q)= D(-q,p)-D(-p-q,p) -D(p-q,-p)+D(-q,-p)$

$\textrm{II}_r(p,q)= D(p,-q) -D(p-q,q)$

$\textrm{II}_{re}(p,q)= D(p,-q) - D(p+q,-q) -D(p-q,q)+D(p,q)$.
\vskip 8pt

Here the $b$ (resp. $r$, resp. $e$) denotes blue (resp. red, resp. extra). \end{definition}

\begin{definition}  Define $F_k(p,q)=\sum_{L=1}^{k-1} F_k^L(p,q)$.  Define $A_k^L$ (resp. $A_k$) the $(k-1)\times (k-1)$ matrix with entries $F_k^L(p,q)$ (resp. $F_k(p,q))$.\end{definition}

\begin{theorem}  \label{fpq} 1) If $p+q<k$, then $$F(p,q)= p \textrm{II}_{re}(p,q)+q \textrm{II}_{be}(p,q).$$

2) If $p+q\ge k$, then $$ F(p,q)=(k-p-1) \textrm{II}_b(p,q)+(k-q-1)\textrm{II}_r(p,q)+(p+q+1-k) I(p,q).$$\end{theorem}

\begin{proof} 1)  Fix $p$ and $q$.  Figure \ref{fig:FigureCalc30} a) shows the points $1,2,\cdots k-1\subset \BR$ with the red point, denoting $q$, the $q$'th  on the left and the blue point the $p$'th on  the right.  Since $p+q<k$, the red point is strictly to the left.  If $L\le q$ (resp. $p \ge k-L$), then Proposition \ref{pq middle} applies for the case $p<k-L, p+q<k$ (resp. $q<L$ and $p+q<k$).  When $q<L<k-p$, then Proposition \ref{pq small} applies.\qed

\begin{figure}[ht]
$$\includegraphics[width=9cm]{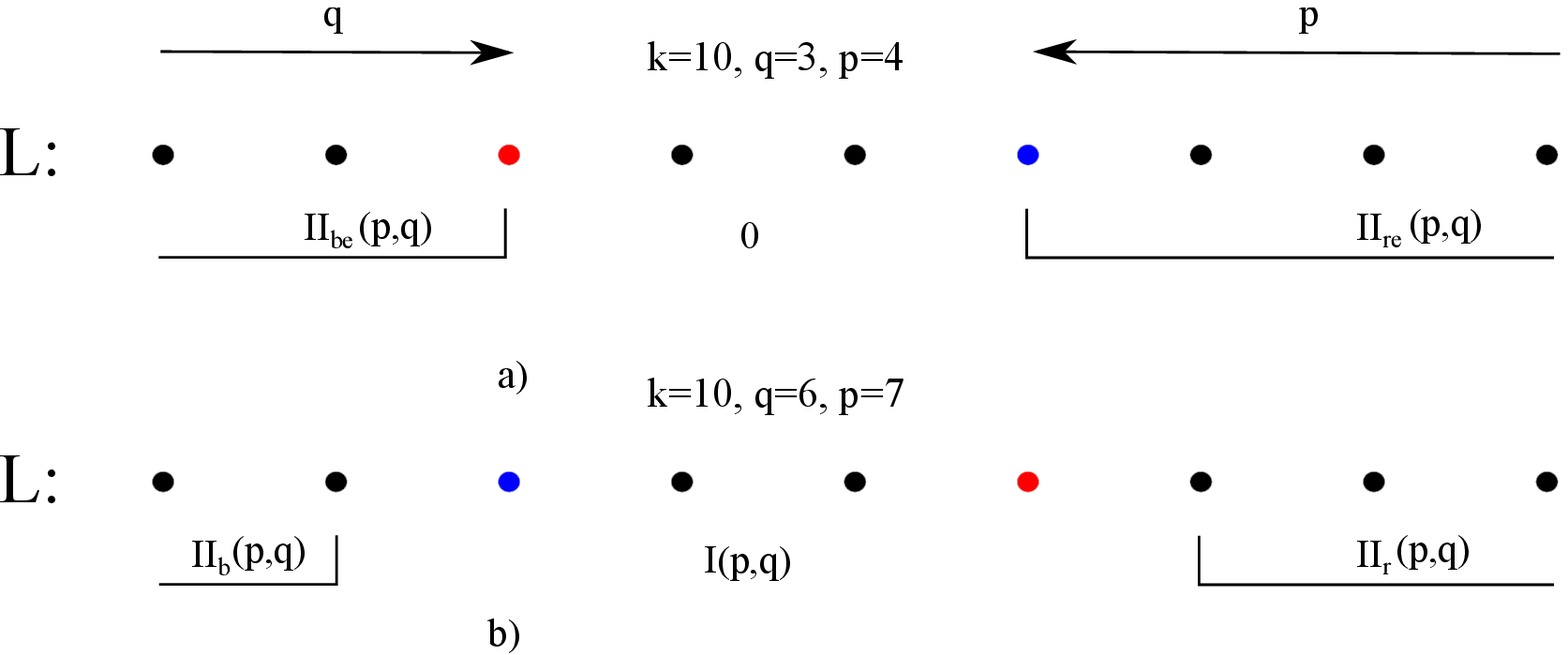}$$
\caption{\label{fig:FigureCalc30}} 
\end{figure}
2) Figure \ref{fig:FigureCalc30} b) shows the $p+q\ge k$ case.  Here the blue point either coincides with or is to the left of the red one.  If $L<k-p$ (resp. $L>q$), then Proposition \ref{pq middle} applies for the case $p<k-L$ (resp. $q<L$) where $p+q\ge k$.  If $k-p\le L\le q$, then Proposition \ref{pq big} applies.\end{proof}

\begin{corollary}\label{thetak trivial} $F_k$ is skew symmetric and hence $[\hat\theta_k]=0\in \pi_2 \E$.\end{corollary}
\begin{proof} Apply Corollary \ref{dpq relations} to Theorem \ref{fpq} to conclude that $F_k(p,q)=-F_k(q,p)$.\end{proof}


\section{Twisting the symmetric $\theta_k$} \label{twisting section}

In Definition \ref{thetak twisting} and Construction \ref{2parameter twisting} we defined how to twist an implantation and construct its associated 2-parameter family.  In particular, we defined the twisted implantations $\theta_k(v,w)$ where $v,w\in \BZ^{k-1}$ with the $\delta_k$ implantations as a special case.  In this section  we first compute $[\hat\theta_k(v,w)]\in \pi_2 \E$ and use that result to show that for $k\ge 4$ the 2-parameter families $\hat\delta_k$ have linearly independent $W_3$ invariants, thereby completing the proofs of Theorems \ref{main} and \ref{knotted main}.

\begin{theorem} \label{pi2} $[\hat\theta_k(v,w)]=\sum_{1\le p,q\le k-1}v_p w_q F_k(p,q)\in \pi_2 \E$, where $v=(v_1,\cdots, v_{k-1})$ and $w=(w_1,\cdots, w_{k-1})$.\end{theorem}

\begin{proof}  Construction \ref{implanted family} showed how to construct the 2-parameter family $\hat\theta_k$ corresponding to the implantation $\theta_k$ as a sum of $k-1$ families in band lasso form.  Construction \ref{2parameter twisting} shows how to modify these families according to the twisting data.  In particular, the modification corresponds to locally twisting the bands near where they pass through the lasso discs.  Now fix $k\ge 2$.  
Recall that we shrunk the bands of these $L$ families to obtain families the families $\hat\theta^L = (B^L, R^L),  L=1,2,\cdots, k-1$ and then separably homotoped $B^L$ (resp. $R^L$) to a concatenation of $B^L_1, \cdots, B^L_{k-1}$.  Since the homotopies to the concatenations were supported away from where the strands passed through the blue and red lasso discs it follows that the effect of twisting is to modify the various $B^L_i, R^L_j$'s as in Figure \ref{fig:FigureTwist2}, where the case $w_j=-2$ is shown.  Note that the four parallel strands undergo a $4\pi$ twist in Figure \ref{fig:FigureTwist2} d).  We can assume that there is no twisting since these twists can be transferred to the blue bands if $j\ge L$ and the red bands if $j<L$ and twisting bands by multiples of $2\pi$ does not change the homotopy class of the 2-parameter family. It follows that the $L$'th family of $\hat\theta(v,w)$ is homotopic to the family as described in Figure \ref{fig:FigureTwist1}.  The argument of Lemma \ref{big undo} applies without change to these families and hence the result follows by Proposition \ref{bilinearity}.\end{proof}

\begin{figure}[ht]
$$\includegraphics[width=9cm]{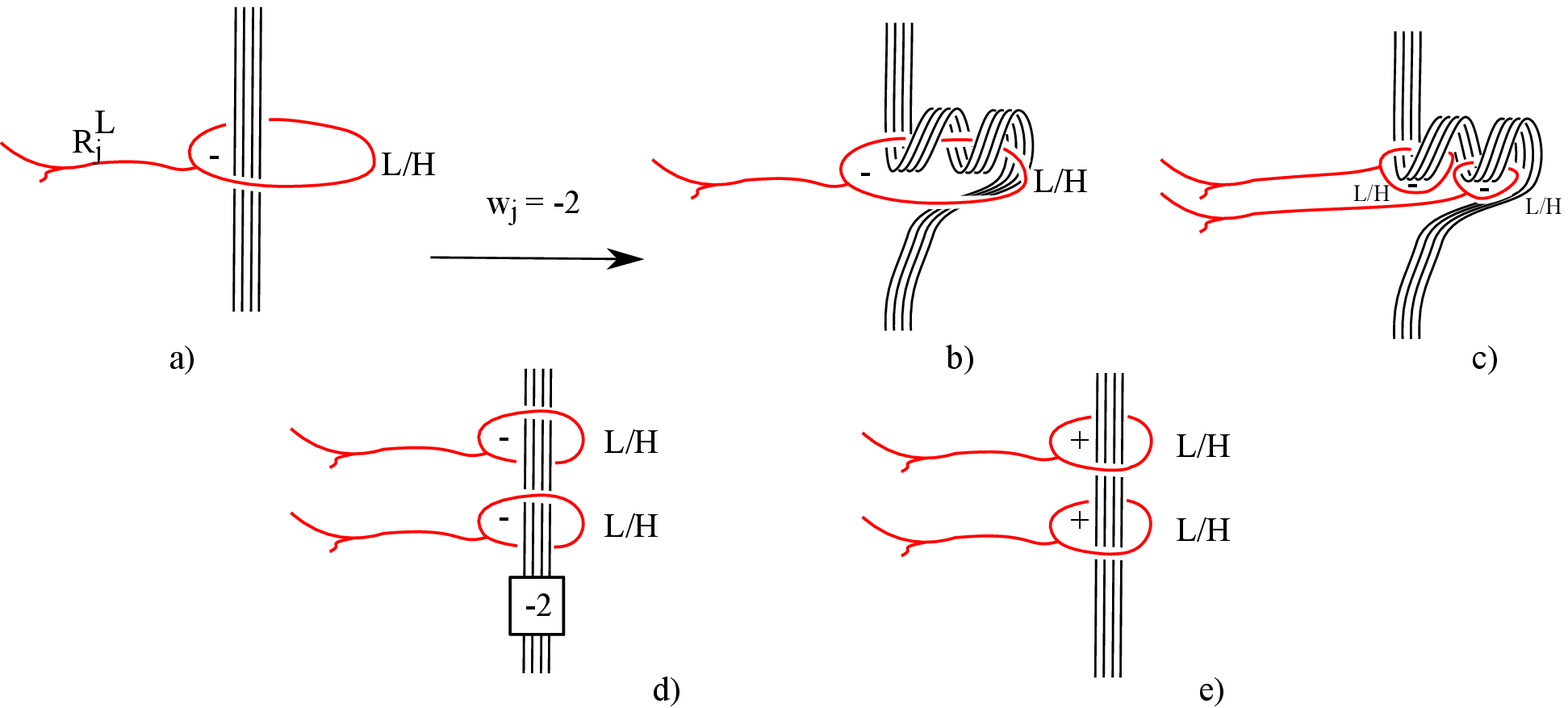}$$
\caption{\label{fig:FigureTwist2}} 
\end{figure}
\begin{figure}[ht]
$$\includegraphics[width=9cm]{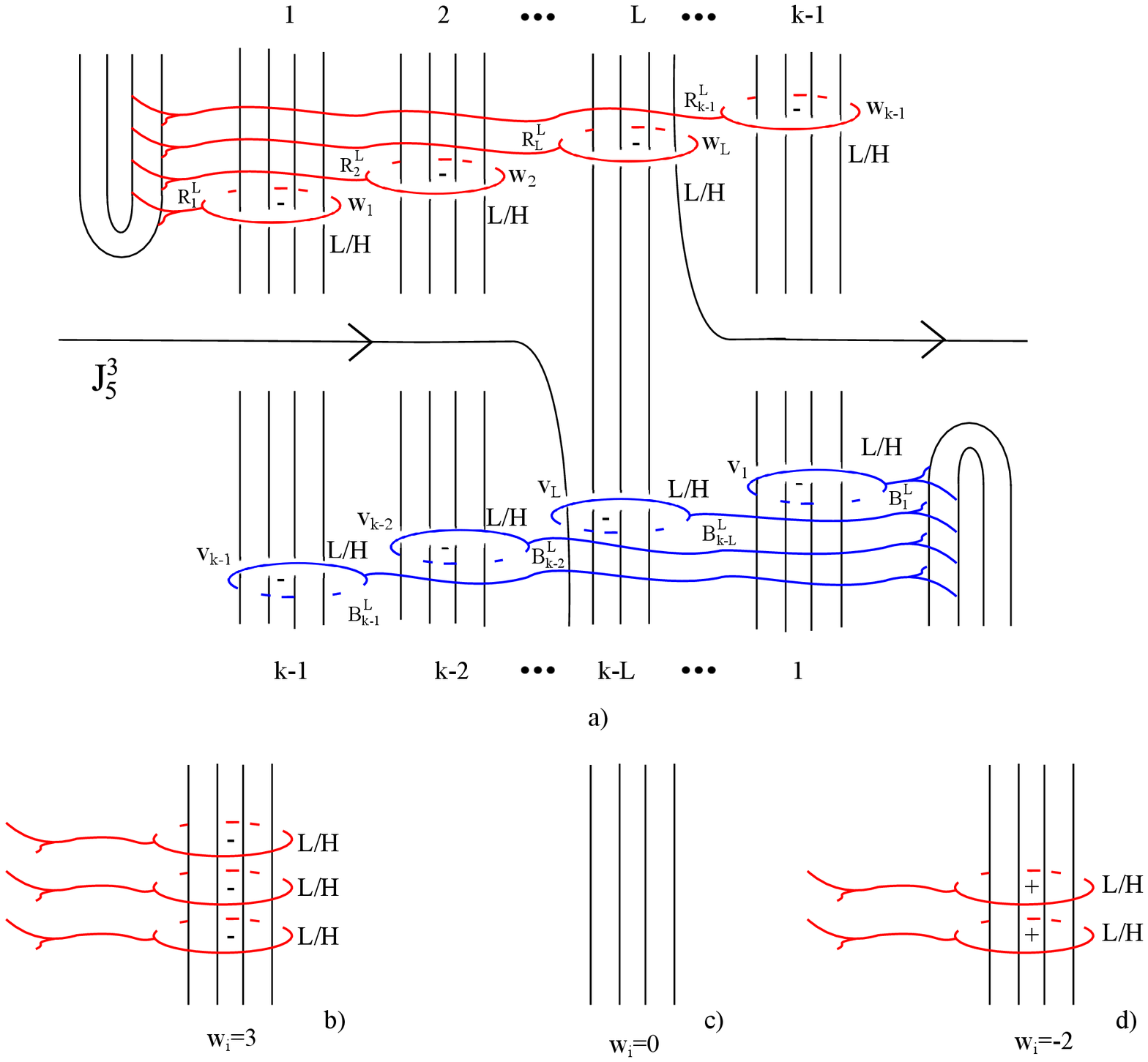}$$
\caption{\label{fig:FigureTwist1}} 
\end{figure}

Applying to the case of $v=(0,...,0,1)$ and $w=(0,0,\cdots,0,1,0)$, i.e. zeros everywhere except for the $k-1$'th entry of $v$ and $k-2$'nd entry of $w$ we obtain:

\begin{corollary} \label{deltak} $[\hat\delta_k]$

$=(k-1)(G(k-2,k-1)-G(k-1,k-2)+G(1-k,2-k)-G(2-k,1-k))$

$-(-G(k-2,-1)+G(-(k-2),1)-G(1,-(k-2))+G(-1,k-2))$

 $\in \pi_2 \E$.\end{corollary}

\begin{proof}  By Theorem \ref{pi2}, Theorem \ref{fpq} and Definition \ref{roman defs},  $[\hat\delta_k]= F_k(k-1,k-2)$

$= 0+1\textrm{II}_r(k-1,k-2)+(k-2) \textrm{I}(k-1,k-2)$ 

$= D(k-1,-(k-2))-D(1,k-2)+(k-2)D(k-1,-(k-2))$

$=(k-1) D(k-1,-(k-2)) - D(1,k-2)$

$=(k-1)(G(k-2,k-1)-G(k-1,k-2)+G(1-k,2-k)-G(2-k,1-k))$

$-(-G(k-2,-1)+G(-(k-2),1)-G(1,-(k-2))+G(-1,k-2))$  

by Corollary \ref{dpq gpq}.\end{proof}

\begin{theorem}  \label{linearly independent} $W_3(\hat\delta_4), W_3(\hat\delta_5), \cdots$ are linearly independent.\end{theorem}

\begin{proof}  Fix $n\ge 4$.  By passing to the quotient of the target of $W_3$ by declaring everything of the form $G(p,q)=0$ if $p\neq n-1$, then we see that $W_3([\hat\delta_n])\neq 0$, but in that quotient $[\hat\delta_m]=0$ for $m<n$.  This implies that if we have a finite linear combination of $\hat\delta_m$'s with $n$ the largest $m$, then that sum equals zero implies that the coefficient of $\hat\delta_n=0$.\end{proof}

\begin{remark} $[\hat\delta_3]=0\in \pi_2 \E$.  This can be deduced from Corollary \ref{deltak} and Corollary \ref{epq relations} iii).\end{remark}

\begin{theorem} \label{knotted main} Both $\pi_0(\Diff(S^1\times B^3\fix\partial)/\Diff(B^4\fix\partial))$ and $\pi_0(\Diff(S^1\times S^3)/\Diff(B^4\fix\partial))$ are infinitely generated.  In particular,  the implantations $\beta_{\delta_4}, \beta_{\delta_5}, \cdots \subset \Diff(S^1\times B^3\fix\partial)$  and their extensions to $\Diff_0(S^1\times S^3)$ represent linearly independent elements.  \end{theorem}

\begin{proof}  That these implantations represent linearly independent elements of $\pi_0(\Diff(S^1\times B^3\fix\partial))$ follows from Corollary \ref{deltak}.  Since a diffeomorphism of $S^1\times B^3$ supported in a $B^4$ can be supported away of $\Delta_0 $ it follows that these elements are linearly independent in $\pi_0(\Diff(S^1\times B^3\fix\partial)/\Diff(B^4\fix \partial))$.  Let $p$ and $\phi$ be as in Lemma \ref{key exact}. Since the $\beta_{\theta_k}$'s represent generators for the image of $p$ and by Theorem \ref{thetak trivial}, all the $\hat\theta_k$'s are homotopically trivial, it follows that $W_3\circ p\equiv 0$ and hence there is an induced homomophism $W_3: \pi_0(\Diff_0(S^1\times S^3))\to (\pi_5(C_3(S^1\times B^3))/\textrm{torsion})/R$ and under this homomorphism the extensions $\beta_{\delta_4}, \beta_{\delta_5},\cdots$ are linearly independent.  A similar argument shows that these elements are linearly independent in $\pi_0(\Diff(S^1\times S^3)/\Diff(B^4\fix\partial))$.\end{proof}

\begin{corollary} \label{knotted 3sphere}There exist infinitely many distinct knotted non separating 3-spheres in $S^1\times S^3$, i.e. there exists infinitely many isotopy classes of 3-spheres that are homotopic to $x_0\times S^3$.\qed\end{corollary}
\begin{corollary} There exist infinitely many isotopically distinct fiberings of $S^1\times S^3$ homotopic to the standard fibering.\end{corollary}

Applying Construction \ref{delta image} we obtain the knotted 3-ball $\Delta_{\delta_4}$ by six embedded 0-framed surgeries on $\Delta_0$, the standard 3-ball in $S^1\times B^3$.  See Figure \ref{fig:FigureKnotted1} b).  That figure also shows the conjecturally knotted 3-ball arising from $\mB(\alpha_1)$.

\begin{figure}[ht]
$$\includegraphics[width=11cm]{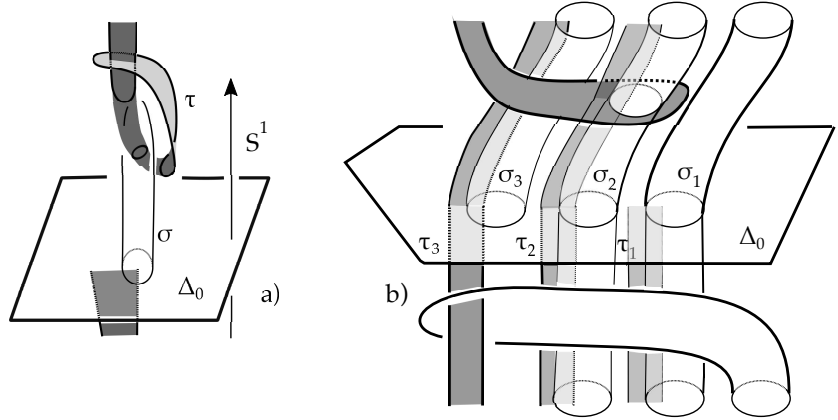}$$
\caption[]{\label{fig:FigureKnotted1}
\begin{tabular}[t]{ @{} r @{\ } l @{}}
(a) Conjecturally knotted $3$-ball arising from $\alpha_{1}$.
(b) Knotted $3$-ball arising from $\delta_4$
 \end{tabular} 
}

\end{figure}

\section{Knotted 3-Balls and the Sch\"onflies Conjecture}\label{knotball}

In this section we  detail the close connection between knotted 3-balls in $S^4$, and knotted 
3-balls in $S^1\times B^3$ and the relation between the Sch\"onflies conjecture and virtual unknotting of 3-balls.

\begin{definition} Let $\Delta_0$ be a 3-ball in $S^4$.  We say that the 3-ball $\Delta$ is \emph{knotted} 
relative to $\Delta_0$ if $\partial \Delta=\partial \Delta_0$ and $\Delta$ is not isotopic 
rel $\partial \Delta$ to $\Delta_0$.  \end{definition}

In what follows $\Delta_0$ will denote  a \emph{standard} or \emph{linear} 3-ball as defined in the introduction. 
   It is to be fixed once and for all.  Its boundary, the 2-unknot, will be denoted by $U$.
     Unless said otherwise, 3-balls in 
$S^4$ will all have boundary equal to $U$ and knottedness is relative to $\Delta_0$.  
Relative knottedness is essential for by \cite{Ce1} p. 231, \cite{Pa} any two smooth embedded $3$-balls 
in the interior of a connected $4$-manifold are smoothly ambiently isotopic.  

We abuse notation by letting $\Delta_0$ also denote $\{x_0\} \times B^3\subset S^1\times B^3$.  
Its use should be clear from context. 

\begin{definition}   The properly embedded non-separating 3-ball 
$\Delta\subset S^1\times B^3$ is \emph{knotted} if it is not properly isotopic to 
$\Delta_0$. \end{definition}

\begin{remark} \label{boundary standard} Equivalently $\Delta$ is knotted if it is properly homotopic to but not 
properly isotopic to $\Delta_0$.  Since any two non-separating 2-spheres in $S^1\times S^2$ are isotopic 
\cite{Gl}, \cite{Ha2}, we can assume that $\partial \Delta = \partial \Delta_0$.  By uniqueness of regular 
neighborhoods we can further assume that $\Delta$ coincides with $\Delta_0$ near partial $\Delta$ and since 
$\Diff_0(S^2)$, the diffeomorphisms homotopic to id,  is connected they have the same parametrization there.  
Finally, since $\pi_1(\Emb(S^2, S^1\times S^2))=1$ up to paths in $SO(3)$ and translations in $S^1\times S^2$, 
\cite{Ha2} it follows that if $ \Delta_2$ and $\Delta_3$ are isotopic and already coincide near $\partial \Delta_2$, 
then there is an isotopy fixing a neighborhood of the boundary pointwise.  By \cite{Ce2}, since 
$\Diff(B^3 \text{ fix }\partial)$ is connected, $\Delta_0$ has a unique parametrization up to isotopy.

Since  $S^1\times B^3$ is diffeomorphic to the closed complement of the unknot $U$ in $S^4$ it follows that we can 
assume, up to isotopy fixing the boundary pointwise, that all 3-balls $\Delta$ with boundary $\Delta_0$ coincide 
with $\Delta_0$ near $\partial \Delta_0$.  Also, that there exists a uniform neighborhood $N(U)$ of $U$ such that 
$\Delta\cap N(U)=\Delta_0$.  \end{remark}

\begin{notation}  Let $N(U)$ be a fixed regular neighborhood of the unknot $U$ in $S^4$, with $\Delta_0$ isotoped to be 
properly embedded in $S^4\setminus \inte(N(U))$.  Fix a diffeomorphism $\psi:S^4-\inte(N(U))\to S^1\times B^3$ such 
that $\psi(\Delta_0)=\Delta_0$.\end{notation} 

The following is immediate.

\begin{proposition}  \label{one one} $\psi$ induces a 1-1 correspondence between isotopy classes of knotted 
3-balls in $S^4$ and knotted 3-balls in $S^1\times B^3$.\qed\end{proposition}

\begin{remark} It is important to remember that our correspondence is given by $\psi$, since by Theorem 
\ref{diff-fibr} $\Diff(S^1 \times B^3 \text{ fix } \partial)$ acts transitively on properly embedded 3-balls of $\soneb$.
 \end{remark}

\begin{theorem} \label{diff vs isotopy} If $\Delta_0$ and $\Delta_1$ are properly embedded 3-balls in 
$S^4\setminus \inte(N(U))$ coinciding near their boundaries, then there exists an orientation preserving 
diffeomorphism $\phi:(S^4,\Delta_1)\to (S^4, \Delta_0)$ fixing $N(U)$ pointwise.  Any 3-ball $\Delta_1$ 
with boundary $U$ restricts to a fiber of a fibration of $S^4\setminus \inte(N(U))$.\end{theorem}

\begin{proof} i) This is \cite{Ce1}, \cite{Pa} applied to 3-balls in $S^4$, together with uniqueness of 
regular neighborhoods and the fact that $\Diff_0(S^2)$ is connected.  \end{proof}

\begin{theorem} \label{isomorphism} $\psi$ induces  isomorphisms between the following abelian groups.  

i) Isotopy classes of 3-balls in $S^4$ with boundary $\Delta_0$

ii) isotopy classes of 3-balls in $S^1\times B^3$ with boundary $ \Delta_0$

iii) $\pi_0(\Diff(S^1\times B^3 \fix \partial)/\Diff(B^4 \fix \partial))$\end{theorem}

\begin{proof}   Recall that by $\pi_0(\Diff(S^1\times B^3 \fix \partial)/\Diff(B^4 \fix \partial))$ 
we mean isotopy classes of diffeomorphisms of $S^1\times B^3$ fixing a neighborhood of $S^1\times S^2$
 pointwise modulo diffeomorphisms that are supported in a compact 4-ball. 

To an element $[\phi]\in  \pi_0(\Diff(S^1\times B^3 \fix \partial)/\Diff(B^4 \fix \partial))$ we 
associate the 3-ball $\Delta=\phi(\Delta_0)$.  It's isotopy class is well defined since 
$\Diff(B^4 \fix \partial)(\Delta_0)=\Delta_0$ up to isotopy.  If $\Delta\subset S^1\times B^3$ is a 
3-ball with boundary $\Delta_0$, then by Theorem \ref{diff vs isotopy} there exists a 
$\phi\in \Diff(\soneb\fix\partial)$ with $\phi(\Delta_0)=\Delta$.  If $\phi'$ is another such 
diffeormorphism and  $\phi_0=\phi\circ\phi'^{-1}$, then $\phi_0(\Delta_0)=\Delta_0$.  By \cite{Ce2} 
we can assume that after isotopy $\phi_0(\Delta_0)=\phi_0(\Delta_0)$ pointwise, and then also 
$\phi_0(N(\Delta_0))=N(\Delta_0)$ pointwise.  Thus $\phi$ is equivalent to $\phi'$ modulo 
$\Diff(B^4\fix \partial)$.  It follows that there is a 1-1 correspondence between ii) and iii). 

Recall that the group structure on  3-balls in $S^1\times B^3$ to be the one induced from 
$\pi_0(\Diff(S^1\times B^3\fix \partial)/\Diff(B^4\fix\partial))$, so $\Delta_0$ is the id and 
multiplication is given by concatenation. Use $\psi$ to induce the group structure on isotopy 
classes of 3-balls in $S^4$ with boundary $\Delta_0$ and hence the isomorphism between i) and ii).  

Theorem \ref{diff abelian} implies that these groups are abelian.  \end{proof}

\begin{remark} Theorem 1.10 of \cite{Ga1} proves the uniqueness of spanning 2-discs in $S^4$, i.e. two 
discs with the same boundary are isotopic rel boundary.  The existence of knotted 3-balls in $S^4$, i.e. 
the non uniqueness of spanning 3-discs in $S^4$, negatively answers Question 10.13 of \cite{Ga1} for $k=3$.  
The second author also conjectured that knotted 3-balls exist \cite{Ga3}. \end{remark}

\begin{corollary} There exist fiberings of the unknot $U\subset S^4$ not isotopic to the linear fibering.  The set
 of isotopy classes of fiberings forms an infinitely generated abelian group.\qed\end{corollary}

The four dimensional smooth Sch\"onflies conjecture is that smooth 3-spheres in the 4-sphere are smoothly 
isotopically standard.  In 1959 Mazur \cite{Ma1} showed that such spheres are topologically standard.  More generally 
Brown \cite{Br} and Morse \cite{Mo} showed that locally flat 3-spheres in $S^4$ are topologically standard.  
The corresponding conjecture is known to be true smoothly for all dimensions not equal to four.  

The following is essentially stated in \cite{Ga1} as Remark 10.14.

\begin{theorem}  \label{schoenflies} The following are equivalent:

i) The Sch\"onflies conjecture is true.

ii) For every 3-ball $\Delta$ in $S^4$ with $ \partial \Delta=\partial \Delta_0$, there exists $n\in \BN$ such 
that the lift of $\Delta$ to the $n$-fold cyclic branched cover of $S^4$, branched over $\partial \Delta$
 is isotopic to $\Delta_0$ rel $\partial \Delta$.  

iii) For every non-separating properly embedded 3-ball $\Delta$ in $\soneb$, there exists $n\in \BN$ such that the 
lift of $\Delta$ to the n-fold cyclic cover of $\soneb$ is isotopically standard.\end{theorem}

\begin{proof}   Let $ \Sigma_0\subset S^4$ be an unknotted $3$-sphere.  If $\Sigma_1\subset S^4$ is an embedded 
3-sphere that coincides with $\Sigma_0$ in a neighborhood of a $3$-disc, then $ \Sigma_1$ is isotopic to $\Sigma_0$ 
if and only if there exists an isotopy that also fixes a neighborhood of the 3-disc pointwise.  

We now show that i) implies ii).  Let $\Delta_0'\subset S^4$ be the linear 3-ball such that
$\Delta_0\cap\Delta_0'=\partial \Delta_0$ and $\Sigma_0=\Delta_0\cup \Delta_0'$ is a smooth 3-sphere.  
Let $\Delta$ be a  3-ball in $S^4$ that coincides with $\Delta_0$ near $\partial \Delta_0$.  By passing 
to a sufficiently high odd degree branched cover over $\partial \Delta_0 $ we can assume that there 
are preimages $\tilde\Delta_0, \tilde\Delta,\tilde\Delta_0'$ of $\Delta_0, \Delta, \Delta_0'$ such 
that $\tilde\Delta_0$ coincides with $\tilde\Delta$ near $\partial \tilde \Delta_0$, 
$\inte(\tilde \Delta)\cap \tilde\Delta_0'=\emptyset$ and $\tilde\Sigma_0=\tilde\Delta_0'\cup \tilde\Delta_0$
is a smooth, necessarily unknotted,  3-sphere in $S^4$. Now $\tilde\Sigma_1=\tilde\Delta_0'\cup\tilde\Delta_1$
 is another 3-sphere in $S^4$.   It follows from the previous paragraph that $\tilde \Sigma_1$ is 
isotopic rel $\tilde\Delta_0'$ to $\tilde\Sigma_0$ and so $\tilde\Delta$ is isotopic to 
$\tilde\Delta_0$ rel $\partial \tilde\Delta_0$ and hence is unknotted.  

Now we show that ii) implies i).  Let $\Sigma_1$ be a 3-sphere in $S^4$.  We can assume that it coincides
with the unknotted 3-sphere $\Sigma_0$, constructed above, near the 3-disc $\Delta_0'$.  Let $\Delta$ 
be the closed 3-disc  in $\Sigma_1$ complementary to $\Delta_0'$.  By hypothesis, $\Delta$ becomes 
unknotted in a $n$-fold branched cover of $S^4$ branched over $\partial \Delta_0$, given by $p:S^4\to S^4$.
In this cover let $\tilde\Delta$ be a preimage of $\Delta$,\ $\tilde\Delta_0$ be the preimage of $\Delta_0$ 
that coincides with $\tilde\Delta$ near their common boundary, and $E_1,E_2, \cdots, E_n$ be the preimages 
of $\Delta_0'$ cyclically ordered about $\partial \Delta_0$ with $\tilde\Delta_0\cup\tilde\Delta$ lying in 
the region $W$ bounded by $E_n \cup E_1$ that contains no other $E_i$'s. If $\tilde\Delta$ is isotopic to 
$\tilde\Delta_0$ via an isotopy supported in $W$, then by composing with $p$ we obtain an isotopy between 
$\Delta$ and $\Delta_0$ supported away from $\Delta_0'$ and hence one between $\Sigma_1$ and $\Sigma_0$. 
Since the isotopy from $\tilde\Delta$ to $\tilde\Delta_0$ is supported away from $\partial \Delta_0$ it 
follows that when lifted to the infinite cyclic branched cover the support of the isotopy hits only 
finitely many preimages of $\tilde\Delta_0'$.  Thus the original n could have been chosen so that the 
support of the isotopy is disjoint from some $E_i$, \ $i\neq 1,n$.  Since the region between $E_1$ and 
$E_i$ (resp. $E_i$ and $E_n$) is a relative product,  the isotopy can be modified to be supported in $W$.

The equivalence of ii) and iii) follows from Theorem  \ref{diff vs isotopy}.\end{proof}

\begin{remark} An unpublished consequence of \cite{Ma2} due to Barry Mazur and rediscovered in conversations between the second author and Toby Colding is the following.  If the Sch\"onflies conjecture is false, then there exists an $f\in \Diff_0(S^1\times S^3)$ such that $f(x_0\times S^3)$ is isotopically non standard even after lifting to any finite sheeted covering of $S^1\times S^3$.   \end{remark}

\begin{definition} We  say that the 3-ball $\Delta\subset \soneb$ is \emph{virtually unknotted} if it becomes 
unknotted after lifting to some finite cover of $\soneb$. \end{definition}

\begin{proposition}  If $f:S^1\times B^3\to S^1\times B^3$ is the result of finitely 
many pairwise disjoint barbell implantations each possibly raised to some power in 
$ \BZ$, then for $n\in \BN$ sufficiently large the lift $f_n$ to the n-fold cyclic 
cover is isotopic to $\id$ modulo $\Diff(B^4\fix \partial)$.
\end{proposition}

\begin{proof} For $n$ sufficiently large $f_n$ is homotopic to the composition of finitely many maps each supported in a 4-ball.
\end{proof}

\begin{corollary} Every knotted 3-ball $\Delta$ arising from finitely many barbell implantations is  virtually  unknotted.  \qed\end{corollary}

\begin{proposition} \label{finite cover} For each $k\in \BN$ there exist a barbell $\mB_k\subset S^1\times B^3$ whose preimage in the $k$-fold cover is the disjoint union of $k$ copies of $\mB(\delta_4)$ and hence $\Delta_{\mB_k}$ lifts to a knotted 3-ball in the $k$-fold cover.\end{proposition}

\begin{proof}  One readily constructs $\mB_k$.  The implantation of the preimage is $(\beta_{\delta_4})^k$.  Since $W_3((\beta_{\delta_4})^k)=kW_3(\delta_4)\neq 0$ by Theorem \ref{linearly independent}, the result follows.\end{proof}

\section{More Applications}\label{AppSec}

In the next proposition we list some consequences of Proposition \ref{rankdiff}
and Lemma \ref{half-discs}.  We list the consequences in dimension four, although as we 
see in the proof, all these statements have high-dimensional analogues.

\begin{theorem}\label{appsecthm}
\begin{enumerate}
\item $\pi_0 \Emb(B^2, S^2 \times B^2) \simeq \BZ$, and this is an isomorphism under
the concatenation operation.  The group $\pi_1 \Emb(B^2, S^2 \times B^2)$ is free abelian group
of rank two.   More generally, 
$$\pi_k \Emb(B^2, S^2 \times B^2) \simeq \pi_{k+1} \Emb(B^1, B^4) \times \pi_k \Omega^2 S^2.$$  

\item $\pi_0 \Emb(B^3, S^1 \times B^3)$ is an abelian group with the concatenation operation.  Moreover
it contains an infinitely generated free subgroup.

\item $\Emb_u(B^2, B^4)$ is connected, with
$\pi_1 \Emb_u(B^2,B^4)$ containing an infinitely generated free subgroup.   

\item $\Emb_u(S^2,S^4)$ is connected, with 
$\pi_1 \Emb_u(S^2, S^4)$ containing an infinitely generated free subgroup.
\end{enumerate}

\begin{proof}
{\bf (1)} As we have seen, when $n \geq 4$,  
$$\Emb(B^1, B^n) \times \Omega S^{n-2} \simeq B \Emb(B^2, S^{n-2} \times B^2).$$
The first non-trivial homotopy group of $\Emb(B^1, B^n)$ is known to be 
$\pi_{2n-6} \Emb(B^1, B^n) \simeq \BZ$, generated by the {\it Haefliger trefoil}
\cite{Bu}.  In \cite{Bu} the space $\Emb(B^1, B^n)$ is denoted $\mathcal K_{n,1}$.
The first non-trivial homotopy group of $\Omega S^{n-2}$ is $\pi_{n-3} \Omega S^{n-2} \equiv
\pi_{n-2} S^{n-2} \simeq \BZ$. 

{\bf (2)} Our technique for showing the (families) of diffeomorphisms of $S^1 \times B^n$ 
are non-trivial factors through the fibration $\Diff(S^1 \times B^n \text{ fix } \partial) \to \Emb(B^n, S^1 \times B^n)$. 

{\bf (3)} See Corollary \ref{cd2-cor}, when $n \geq 4$ we have
$\Emb_u(B^{n-2},B^n) \simeq B\Emb(B^{n-1}, S^1 \times B^{n-1})$.  


{\bf (4)}  There is a homotopy-equivalence \cite{Bu}
$$\Emb(S^j, S^n) \simeq SO_{n+1} \times_{SO_{n-j}} \Emb(B^j, B^n).$$
The simplest way to think of this is to consider elements of $\Emb(B^j,B^n)$ as smooth
embeddings $\BR^j \to \BR^n$ that restricts to the standard inclusion $x \to (x,0)$ outside
of the ball $B_j$. One can conjugate such embeddings via a stereographic projection map, 
converting the embeddings $\BR^j \to \BR^n$ to embeddings $S^j \to S^n$ that are standard on
a hemisphere.  One can then post-compose such an embedding with an isometry of $S^n$.  We are
in the fortunate circumstance where the homotopy-fiber of the map $SO_{n+1} \times \Emb(B^j,B^n) \to \Emb(S^j,S^n)$
can be identified, and it is the orbits of the $SO_{n-j}$-action, acting diagonally on the product. 

Thus $\Emb_u(S^2, S^4)$ is a bundle over $SO_5/SO_2$ with fiber $\Emb_u(B^2,B^4)$.  $\pi_2 SO_5/SO_2 \simeq \BZ$
and this group maps isomorphically to the subgroup of index two in $\pi_1 SO_2$, which maps to zero
in $\pi_1 \Emb_u(B^2,B^4)$, so our map $\pi_1 \Emb_u(B^2, B^4) \to \pi_1 \Emb_u(S^2, S^4)$ is injective.
\end{proof}
\end{theorem}

\begin{remark} The first sentence of Conclusion (1) is Theorem 10.4 of \cite{Ga1}. 
 The proof here is different, generalizable and arguably more direct.  \end{remark}

Allen Hatcher's proof of the Smale Conjecture \cite{Ha1} together with his and Ivanov's
work on spaces of incompressible surfaces \cite{HI} has as the consequence that the 
component of the unknot in the embedding space $\Emb(S^1, S^3)$ has the
homotopy type of the subspace of great circles, i.e. the unit tangent
bundle $UTS^3 \simeq S^3 \times S^2$.  From the perspective of the homotopy-equivalence
$\Emb(S^1, S^3) \simeq SO_4 \times_{SO_2} \Emb(B^1, B^3)$ this is equivalent to saying the
unknot component of $\Emb(B^1, B^3)$ is contractible, $\Emb_u(B^1, B^3) \simeq \{*\}$. 

In dimension four we do not know the full homotopy type of $\Diff(S^4)$, although
there is the recent progress of Watanabe \cite{Wa1} where he shows the rational homotopy
groups of $\Diff(S^4)$ do not agree with those of $O_5$.  In this regard, this paper
asserts the the analogy to Hatcher and Ivanov's spaces of incompressible
surfaces results \cite{HI} are also false in dimension $4$, in particular contrast
with the theorem $\Emb(B^2, S^1 \times B^2) \simeq \{*\}$. 
 
Hatcher and Wagoner \cite{HW} (see Cor. 5.5) have 
computed the mapping class group of $S^1 \times B^n$ 
for a range of $n$. Specifically
$$\pi_0 \Diff(S^1 \times B^n \text{ fix } \partial) \simeq \Gamma^{n+1} \oplus \Gamma^{n+2} \oplus 
\left(\bigoplus_{\infty} \BZ_2\right) $$
provided $n \geq 6$.  The Hatcher-Wagoner diffeomorphisms survive the map
$\pi_0 \Diff(S^1 \times B^n \text{ fix } \partial) \to \pi_0 \Diff(S^1 \times S^n)$, as do the diffeomorphisms
we prove are isotopically nontrivial.

\section{Conjectures and Questions}\label{conjsec}

\begin{conjecture}\label{main conjecture}  The map $p:\pi_1 \Emb(S^1, S^1 \times S^3; S^1_0) \to \pi_0( \Diff(S^1\times B^3\fix \partial)/\Diff(B^4 \fix\partial))$ induced by isotopy extension has kernal the subgroup with $W_2=0$.  In particular the implantations $\beta_{\theta_k}, k\ge 2$ and $\beta_{\alpha_k}, k\ge 1$ are isotopically nontrivial.\end{conjecture}

We thank Maggie Miller for bringing to our attention the following question.

\begin{question} If $\Delta_1$ is a knotted 3-ball in $S^4$ and $T_0$, $T_1$ are respectively obtained from $\Delta_0$, $\Delta_1$ by 
attaching a small 3-dimensional 1-handle $h$, then is $T_1$ knotted, i.e. is not isotopic rel 
$\partial T_1$ to $T_0$?  Does $\Delta_1$ become unknotted after finitely many such stabilizations?\end{question}

\begin{conjecture}  \label{knotted handlebodies} For each $g\ge 0$ there exists 3-dimensional genus-$g$ handlebodies 
$V_0, V_1\subset  S^4$ such that $\partial V_0=\partial V_1$ and the set of $V_0$-compressible simple closed 
curves in $\partial V_0$ coincides with that of $V_1$, but  $V_1$ is not isotopic to $V_0$ via an isotopy that 
fixes $\partial V_1$.\end{conjecture}



\begin{question} 
i)  Determine 
$\pi_0\left(\Diff(S^1\times B^n \text{ fix } \partial)/\Diff(B^{n+1} \text{ fix } \partial)\right)$ 
for $n = 4,5$.  ii) Do all elements become trivial after 
 lifting to the 2-fold cover of $S^1\times B^n$.\end{question}

\begin{remark} As already noted Hatcher has determined such groups for 
$n\ge 6$ and all its elements become trivial passing to 2-fold covers.  
On the other hand by Proposition \ref{finite cover} some elements remain nontrivial when passing to such covers, 
when $n=3$.  \end{remark}




The following is a restatement of a special case of Lemma \ref{half-discs}, 
where as explained there, the isomorphism is given by \emph{slicing the embedding}.  
Here notation is as in Theorem \ref{appsecthm}.

\begin{theorem} 
$\pi_0\Emb(B^3, S^1\times B^3)\simeq \pi_1\Emb_u(B^2, B^4)$.
\end{theorem}

\begin{problem}  \emph{Use this to prove or disprove the Sch\"onflies conjecture, 
e.g. to prove the conjecture show that loops in $\Emb_u(B^2,B^4)$ are \emph{homotopically} 
trivial where at discrete times of the homotopy one can do  geometric moves corresponding 
to passing to finite sheeted coverings.  Such a move might change the homotopy class and 
may arise from certain \emph{reimbeddings} of the loop.}
\end{problem}

\begin{question} Does there exists a barbell $\mB\subset S^4$ whose implantation represents a 
nontrivial element of $\pi_0 \Diff_0(S^4)$?\end{question}


\section{Appendix:  G(p,q) is a Whitehead product} \label{appendix section}

The goal of this section is to prove the following result.

\begin{theorem}  \label{gpq whitehead} Let $I_0$ denote the standard oriented properly embedded arc in $S^1\times B^3$.  Under the homomorphism $W_3:\pi_2\Emb(I,S^1\times B^3; I_0) \to \pi_5(C_3(S^1\times B^3, I_0))/R$, then up to sign independent of $p$ and $q$,  $W_3(G(p,q))=t_1^p t_2^q [w_{13}, w_{23}]$.\end{theorem}


\begin{remark}  While stated for $S^1\times B^3$ the analogous statement and proof holds for any orientable 4-manifold $M$.  In the general case $J_0$ is an oriented properly embedded arc in $M$ with $\{a_1, a_2, a_3\}\subset J_0$, $a_1<a_2<a_3$, the basepoint for $C_3(M)$.  In the statement below $\lambda_1, \lambda_2$ are paths from $a_1, a_2$ to $a_3$ respectively and $G(\lambda_1,\lambda_2)$ is defined exactly like $G(p,q)$ except that $\lambda_1, \lambda_2$ represent elements of $\pi_1(M; J_0).$\end{remark}

\begin{theorem} \label{gpq theorem} Let $J_0$ denote an oriented properly embedded arc in the oriented 4-manifold $M$.  Under the homomorphism $W_3:\pi_2\Emb(I,M; I_0) \to \pi_5(M, J_0)/R$, then up to sign independent of $\lambda_1$ and $\lambda_2$,\ $W_3(G(\lambda_1,\lambda_2))=t_1^{\lambda_1} t_2^{\lambda_2} [w_{13}, w_{23}]$.\end{theorem}

\noindent\emph{Idea of Proof of Theorem \ref{gpq whitehead}:}  We are motivated by the mapping space model of Sinha \cite{Si1} that was derived from the work of Goodwillie, Klein and Weiss, e.g. see \cite{GKW}. $G(p,q)$ induces a map $C_3\langle I\rangle\times I^2\to C_3\langle S^1\times B^3\rangle$ which is a map of the 5-ball into $C_3\langle S^1\times B^3\rangle$.  By closing off the boundary facets we will essentially modify this map to an element of $\pi_5(C_3\langle S^1\times B^3\rangle)$ which is rationally generated by Whitehead products and under our map, up to sign, $G(p,q)$ is taken to $t_1^p t_2^q[w_{13},w_{23}]$.  We close off the boundary facets essentially subject to the constraints of the mapping space model, thus our element of $\pi_5$ is well defined up to other elements that might be obtained by closing off with the same constraints.  An element $z$ of $\pi_5 (C_3(S^1\times B^3))$ modulo those elements is by definition $W_3(z)$.  

\begin{notation}  $C_n(M)$ denotes the \emph{configuration space} of distinct ordered $n$-tuples of $M$ and  $C_n\langle M \rangle$ denotes the quotient of the Fulton - McPherson compactification as defined in Definition 4.1 of \cite{Si1}, though  there it is denoted by $C_n\langle[M]\rangle$.  In this section $n=3$ and $M$ is one of $I $ or $S^1\times B^3$.  For our purposes it suffices to consider the connected component of $C_3\langle I \rangle$  which is the simplex $\{(x,y,z)|0\le x\le y\le z\le 1\}$, so from now on  $C_3\langle I\rangle$ will denote  that simplex.\end{notation}

\begin{definition}  We continue to view $S^1\times B^3$ as $D^2\times S^1\times [-1,1]$ with $I_0$ a geodesic arc through the origin of  $D^2\times \{x_0\}\times \{0\}$.  We fix an identification of $[0,1]$ with $I_0$ by $1_{I_0}:I \to I_0 $ and will frequently implicitly identify one with the other.  In particular, we abuse notation by having $(a_1, a_2, a_3)=(1/4, 1/2, 3/4)\in I_0$ denote the \emph{basepoint} for both $C_3(I)$ and $C_3(S^1\times B^3)$.  Let $U_1=[j_0,j_1], U_2=[k_0, k_1], U_3=[g_0,g_1]$ be disjoint closed intervals in $(0,1)$ containing $a_1, a_2, a_3$ in their interiors.  We color $U_1, U_2$ and $U_3$  \emph{blue, red} and \emph{green} respectively.  See Figure \ref{fig:FigureCalc31} a).    \end{definition}

\begin{figure}[ht]
$$\includegraphics[width=13cm]{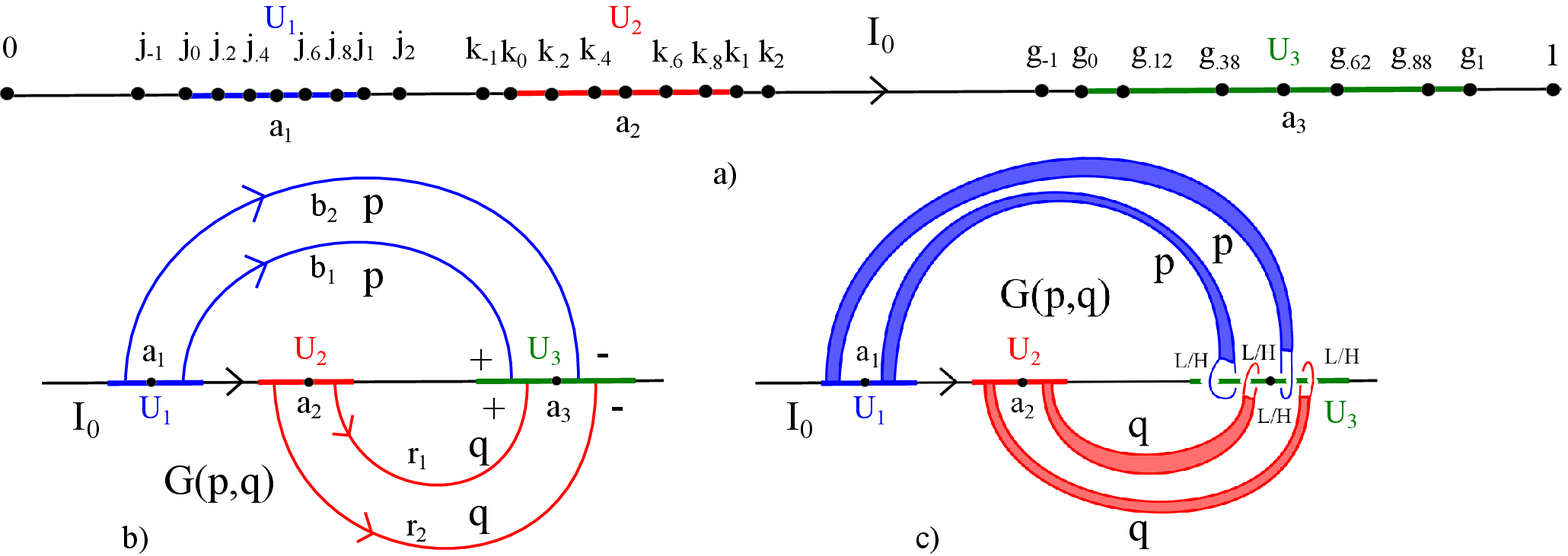}$$
\caption{\label{fig:FigureCalc31}} 
\end{figure}

We introduced $G(p,q)$ in Figure \ref{fig:FigureCalc6} a) as the bracket $(B_p,R_q)$.  To minimize notation we drop  the subscripts.  In the next definition we specify $B$, $R$ and their end homotopies more precisely and denote the resulting bracket by $\alpha_{t,u} \in \Omega\Omega\Emb(I,S^1\times B^3; I_0)$.

\begin{definition}  We represent $(B,R)$ by the chord diagram pair shown in Figure \ref{fig:FigureCalc31} b) and equivalently in band lasso notation in Figure \ref{fig:FigureCalc31} c). The positive chords (resp. negative chords) spin around $[g_{.12}, g_{.20}]\subset U_3$ (resp. $[g_{.80}, g_{.88}]\subset U_3$); the corresponding lasso discs intersect $U_3$ at $g_{.15}$ and $g_{.18}$ (resp. $g_{.82}$ and $g_{.85}$).  The blue (resp. red) chords are called $b_1, b_2$ (resp. $r_1, r_2$) and the spinnings about these chords are respectively supported in the domain in $[j_{.6}, j_{.8}]$ and $[j_{.2}, j_{.4}]$  (resp. $[k_{.6}, k_{.8}]$ and $[k_{.2}, k_{.4}]$).   The end homotopies for $B$ and $R$ are the \emph{undo homotopies} $\nu_b, \nu_r$ of Definition \ref{br homotopy}.    We represent the resulting bracket $(B,R)$   by $\alpha_{t,u} \in \Omega\Omega\Emb(I,S^1\times B^3; I_0)$.

We define the \emph{left shifted} (resp. right shifted) mappings $B^L, R^L$ (resp. $B^R, R^R$) as in Figure \ref{fig:FigureCalc32}.  Here the ends of $b_2$ and $r_2$ (resp. $b_1$, $r_1$) have been isotoped so that they spin about $U_3$ to the left (resp. right) of $a_3$ close to the subinterval $[g_{.12}, g_{.38}]$ (resp. $[g_{.62}, g_{.88}])$ around which $b_1,r_1$  (resp. $b_2, r_2$) already spin.\end{definition}

\begin{figure}[ht]
$$\includegraphics[width=9cm]{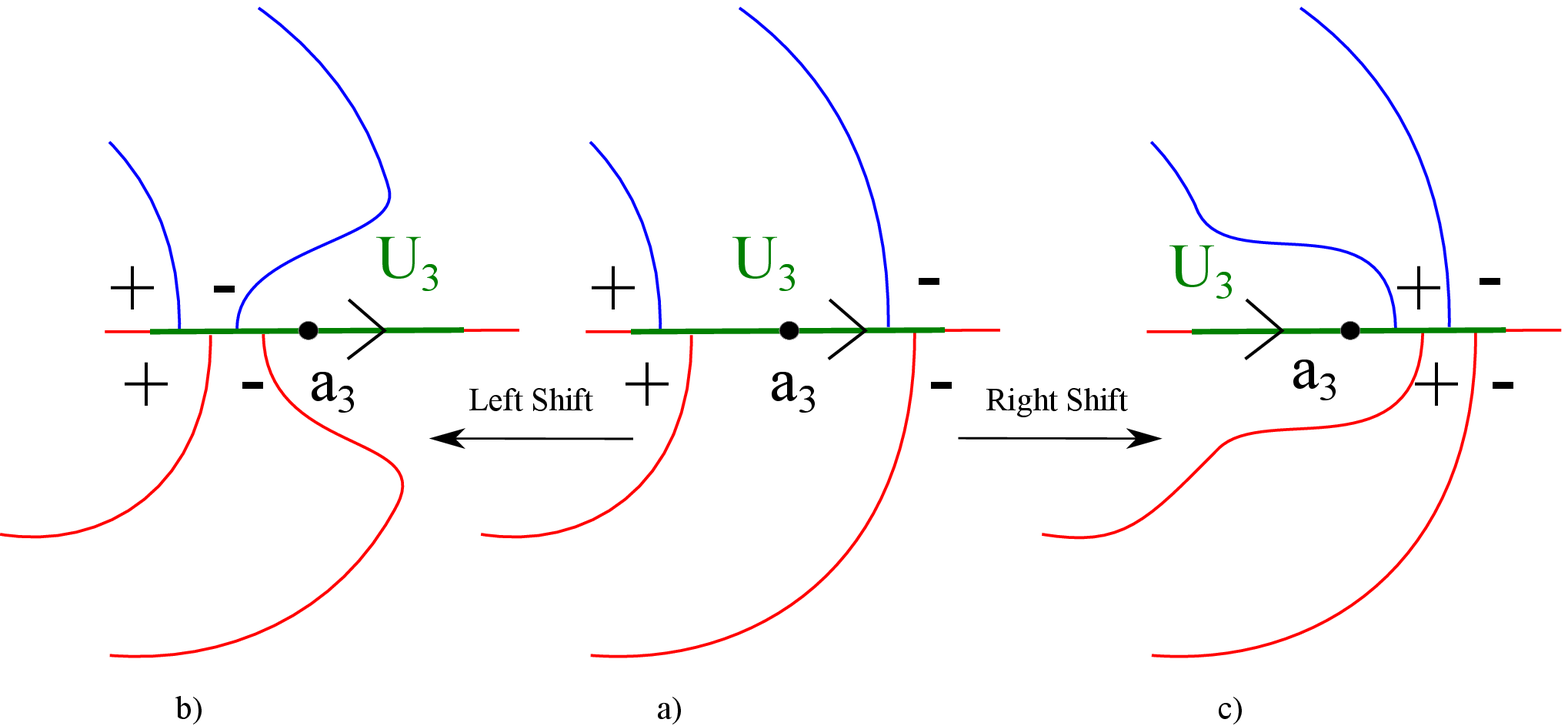}$$
\caption{\label{fig:FigureCalc32}} 
\end{figure}

\begin{remark} The range support of a shift is contained in a small neighborhood of $U_3$.  \end{remark}

We now describe homotopies in $\Omega\Maps(I,S^1\times B^3; I_0)$ of $B$ and $R$ and their $L$ and $R$ shifted versions to the constant map to $1_{I_0}$.

\begin{definition} \label{br homotopy} The \emph{back track} homotopy is the  homotopy in $\Omega\Maps(I,M;I_0)$ of a spinning to $1_{I_0}$ corresponding to withdrawing the band and lasso and the \emph{undo} homotopy is defined in Definition \ref{undo definition}.  We now elaborate on these and their variants in our current context.

 The \emph{blue back track homotopy} is the homotopy $\beta^B$ of $B$ to the constant $1_{I_0}$ as shown in Figure \ref{fig:FigureCalc33}.  Each element of the homotopy is a loop of immersed intervals that are embeddings when restricted to $U_1$, though there may be intersections with $U_3$.  The first part of the homotopy contracts the lasso 3-ball to the arc at the top of the band, so at that moment $\beta_{t}^B$ is a loop that sends an arc to the top of the band and then withdraws it.  In a similar manner we define the \emph{red back track homotopy} $\beta^R$ as well as the back track homtopies for $B^L, B^R$ etc.  \

\begin{figure}[ht]
$$\includegraphics[width=13cm]{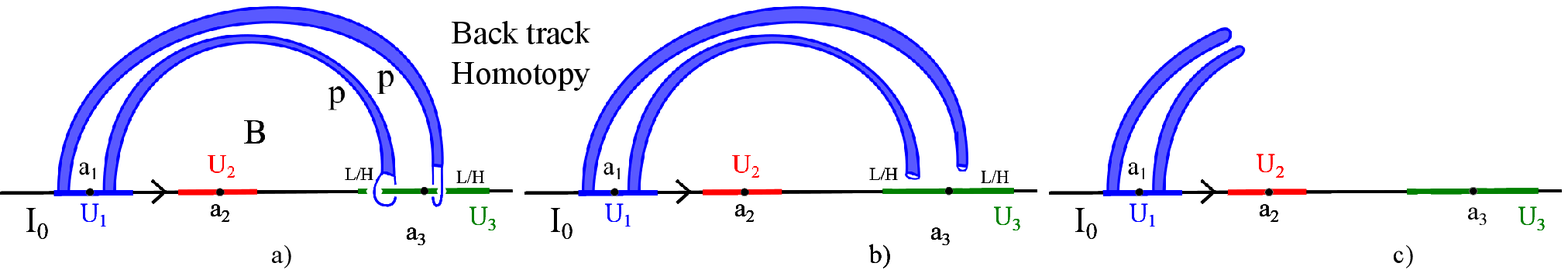}$$
\caption{\label{fig:FigureCalc33}} 
\end{figure}

The  \emph{blue undo} homotopy $\nu^B$, a homotopy in $\Omega\Emb(I, S^1\times B^3;I_0)$, is defined in Figures \ref{fig:FigureCalc34} a) - g) via continuous transformations of bands, lassos and lasso discs.   Figure \ref{fig:FigureCalc34} a) shows $\nu^B_0=B$.  To go from $\nu^B_0$ to $\nu^B_{.25}$ we zip up the bands to a single one from which two lassos $\kappa_0, \kappa_1$ simultaneously emanate as in Figure \ref{fig:FigureCalc34} b).  Figure \ref{fig:FigureCalc34} c) is a detail of 34 b).  Note that the original lasso discs $D_0, D_1$ which are also the lasso discs for $\kappa_0$ and $\kappa_1$ each intersect $U_3$ in one point.  To go from $\nu^B_{.25}$ to $\nu^B_{.4}$ we zip up the lasso discs to obtain a single lasso disc $D$ and lasso $\kappa$.  Here $D$ contains $D_0$ and $D_1$ as subdiscs, the zipping is disjoint from $U_3$ and $D$ intersects $U_3$ in two points.  See Figure \ref{fig:FigureCalc34} d).   To obtain $\nu^B_{.5}$ we properly isotope $D$ to be disjoint from $U_3$.  See Figure \ref{fig:FigureCalc34} e).  This isotopy induces an isotopy of the lasso sphere.  The key point is that the lasso spheres remain disjoint from $U_3$ throughout the isotopy,  since at each moment the interior of one hemisphere is in the future while the other is in the past.  Finally, we withdraw the lasso and bands as in the back track homotopy.    In a similar manner we define the \emph{red undo} homotopy as well as the undo homotopies for $B^L$, $B^R$ etc.  

\begin{figure}[ht]
$$\includegraphics[width=13cm]{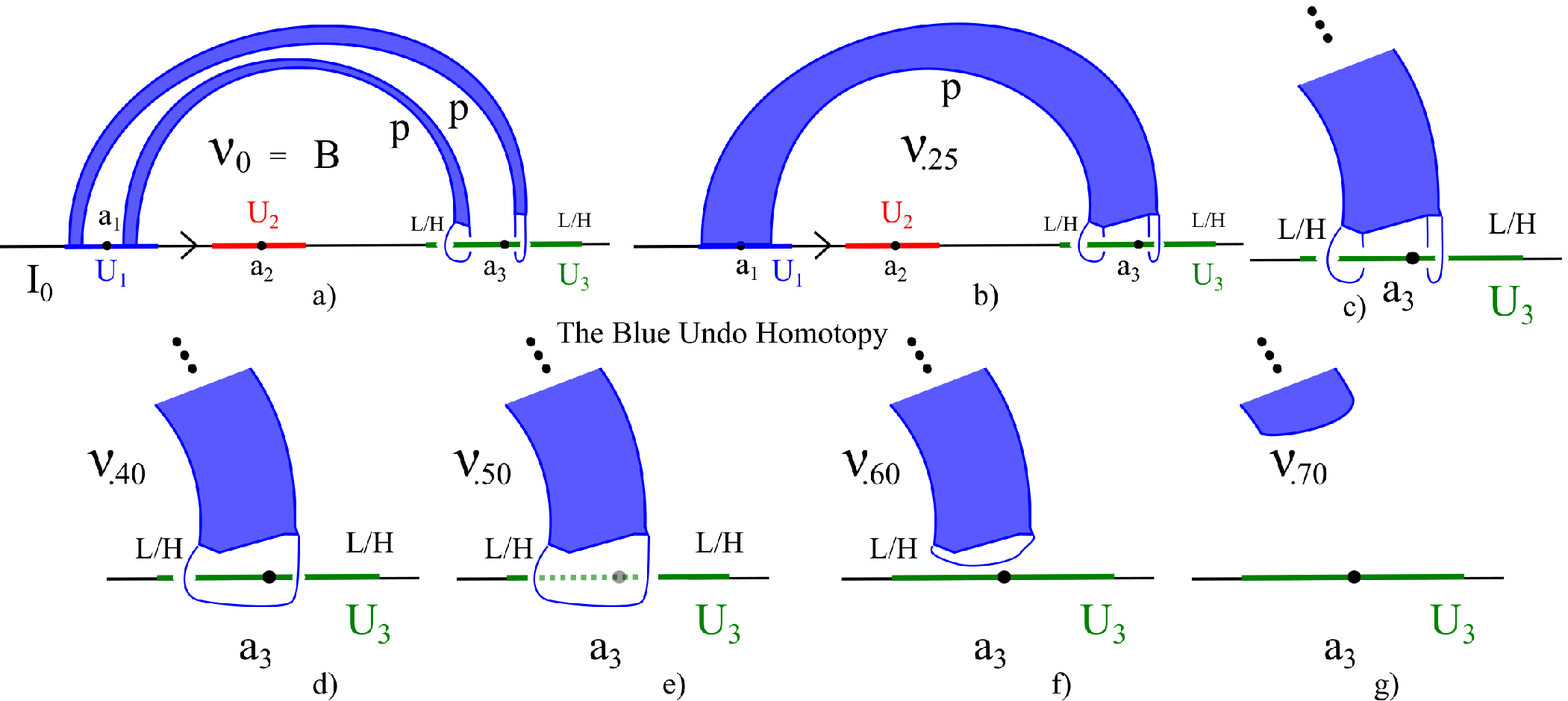}$$
\caption{\label{fig:FigureCalc34}} 
\end{figure}

In a natural way define the \emph{transition homtopy} $T$ which is a homotopy between the back track and undo homotopies.  It's support is in a small neighborhood of the supports of the undo and back track homotopies.  We denote  by  $T^B$ (resp. $T^R$)  the transition homotopy for $B$ (resp. $R$).  \end{definition}

\noindent\textbf{Summary of Intersections:}\label{summary} We now catalogue the intersections and self-intersections of the immersions $I\to S^1\times B^3$ that arise from evaluating at parameter points  the homotopies  $\beta^B, \beta^R, \nu^B, \nu^B, T^B, T^R$, as well as their $L$ and $R$ shifted versions.  

1) In the domain $B$ is supported in $U_1$ as are all the homotopies of $B$.  All homotopies of $B$ are homotopies of loops of embeddings when restricted to $U_1$.  $R$ is supported in $U_2$ as are all the homotopies of $R$.  All homotopies of $R$ are homotopies of loops of embeddings when restricted to $U_2$.

Call a double point of a given map or between two maps of $I\to S^1\times B^3$ to be of \emph{$U_i/U_j$ type} if it involves the image of a point of $U_i$ intersecting the image of one from $U_j$.

2) Double points between $\nu^B(e,f)$ and $\nu^R(e',f')$ are only of type $U_1/U_2$, similarly for their shifted versions. $\nu^B$ and $\nu^R$ and their shifted versions are homotopies of loops of embeddings.

3) Double points of a $\beta^B(e,f)$ are only of type $U_1/U_3$, similarly for its shifted versions. There are no $U_1/U_2$ double points between a $\beta^{b} (e,f)$ and a $\beta^R(e',f')$.   Double points of a $\beta^R(e,f)$ are only of type $U_2/U_3$, similarly for its shifted versions

4) Double points involving transitional homotopies   or their shifted versions are only of type $U_1/U_2, U_1/U_3$ or $U_2/U_3$.  

5) Any $U_3$ double point involving an $L$ shifted map (resp. $R$ shifted map) occurs in $[g_{.12}, g_{.38}]$ (resp. $[g_{.62}, g_{.88}]$.

\begin{definition}  Define two open covers $\{U_L, U_R\}, \{U_\beta, U_\nu\}$ of $\partial C_3\langle I\rangle $ called the \emph{standard decompositions} as indicated in Figure \ref{fig:FigureCalc35} a).    

$U_L:=\{(\omega_1, \omega _2, \omega _3)\in \partial C_3\langle I\rangle |\omega _3> g_{-1}\}$

$U_R:=\{( \omega _1, \omega _2, \omega _3)\in \partial C_3\langle I\rangle| \omega _3<g_0\}$

$U_\beta:= \{(\omega _1, \omega _2, \omega _3)\in \partial C_3\langle I\rangle|\omega _3<k_2\}\cup (j_{-1},j_2)\times(k_{-1},k_2)\times 1$

$U_\nu:= \{(\omega _1, \omega _2, \omega _3)\in \partial C_3\langle I\rangle|\omega _3>k_1\}\setminus [j_0,j_1]\times [k_0, k_1]\times 1$.   

For $(t,u)\in I^2$ we define two open covers of $\partial C_3\langle I\rangle\times (t,u)$ called the L-R and $\beta$-$\nu$ open covers.  For the L-R cover, use the  standard $\{U_L, U_R\}$ cover, independent of $(t,u)$.  Let $\epsilon>0$ be very small and $P$ as in Figure \ref{fig:FigureCalc34} b).  For $t,u\in P\setminus N_\epsilon(\partial P)$  use the standard $\{U_\beta, U_\nu\}$ decomposition.  For $(t,u)\in N_\epsilon(\partial P)$ define $U_\beta$ as above and  $U_\nu=\partial C_3\langle I\rangle$.  For $(t,u) \notin P\cup N_\epsilon(\partial P)$, define $U_\nu=\partial C_3\langle I\rangle$ and $U_\beta=\emptyset$.\end{definition}

\begin{figure}[ht]
$$\includegraphics[width=13cm]{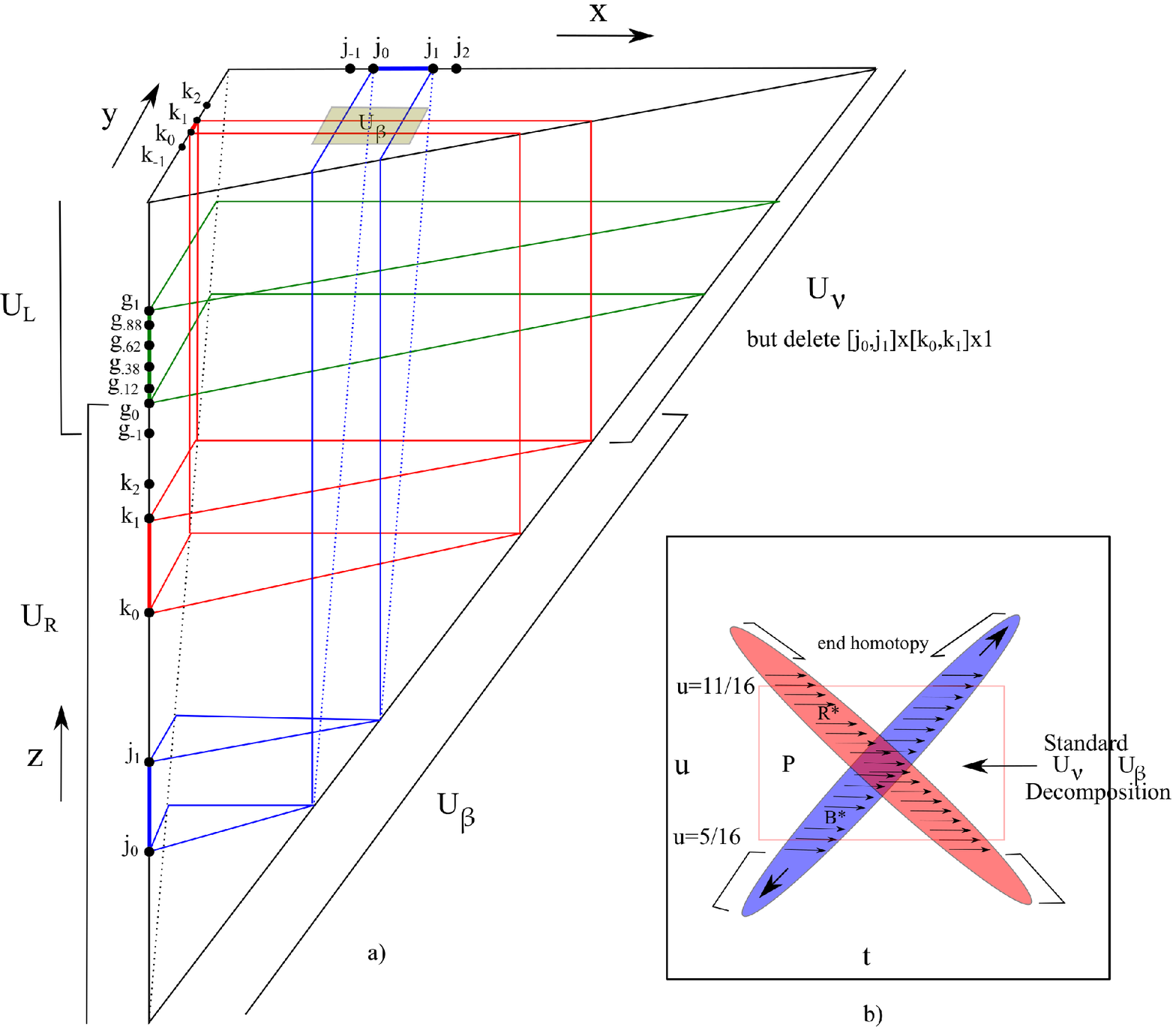}$$
\caption{\label{fig:FigureCalc35}} 
\end{figure}

\begin{remark}  $P\setminus N_\epsilon(\partial P)$ contains all the parameters that are both blue and red and $I\times I\setminus P\cup N_\epsilon(P)$ contains the parameters corresponding to the end homtopies.\end{remark}

\begin{definition}   Let $E^3$ denote the 3-ball $C_3\langle I\rangle\cup (\partial C_3\langle I \rangle\times [0,1])  $ and $E^5:=E^3\times I^2$. Here points are denoted $(\omega, v,t,u)$ where $(t,u)\in I^2$,  $\omega\in C_3\langle I\rangle$ and $v\in [0,1]$.  Also, $v=0$ unless $\omega\in \partial C_3\langle I\rangle$.

We define $\alpha:E^5\to \Maps(I,S^1\times B^3; I_0)$.  First, define $\alpha:C_3\langle I\rangle\times I^2\to \Maps(I, S^1\times B^3;I_0)$ by $\alpha^\omega_{t,u}=\alpha_{t,u}$.  Extend $\alpha$ to ($\partial C_3\langle I\rangle\times [0,1])\times I^2$ as follows.  For $v\in[0,1/4]$, according to $\omega\in U_L$ or $U_R$, monotonically shift $\alpha_{t,u}$ to the left or right.  So, if say $\omega \in U_L\setminus U_R$, then $\alpha^\omega_{.25,t,u}$ is totally shifted to the left.  For $\omega \in  U_L\cap U_R$ it is somewhere between the right and the left i.e. is \emph{partially shifted}.  Next, for $v\in [1/4,1]$ homotope $\alpha^\omega_{.25,t,u}$ to $1_{I_0}$ according to whether it is in $U_\beta$ or $U_\nu$.  I.e. if $\omega\in U_\beta\setminus U_\nu$ (resp. $U_\nu\setminus U_\beta)$ the back track (resp. undo) homotopy is used.  In $U_\nu\cap U_\beta$ the transition homotopy is used.  \end{definition}


\begin{definition}  For $i=1,2$ let $p_i:I\to U_i $ denote the retraction which fixes $U_i$ pointwise and let $p_3$ the retraction to $a_3$. Letting $p=(p_1,p_2, p_3)$ we obtain $p:C_3\langle I\rangle\to C_3(I)$.  Let $\gamma:I\to S^1\times B^3$ be a proper immersion   supported in the domain on $U_1 \cup U_2$ such that  $\gamma|U_1$  and $\gamma|U_2$ are embeddings.  Let $0\le x_0\le y_0\le z_0\le 1$ and $(x_1,y_1,z_1):=(p_1(x_0), p_2(y_0), a_3)$. Let $x_s$ be the constant speed path in I from $x_0$ to $x_1$ with $y_s$ and $z_s$  similarly defined. We say that $(x_0, y_0, z_0)$ is $\gamma$-\emph{ready} if for all $s>0, \gamma(x_s), \gamma(y_s), \gamma(z_s)$ are pairwise disjoint.\end{definition}

\begin{lemma} \label{embedded ready} If $\gamma$ is embedded, then any $(x,y,z)\in C_3\langle I\rangle$ is $\gamma$-ready.\qed\end{lemma}

\begin{notation}  In what follows if $a,b\in I$, then $[a,b]$ denotes the closed interval, possibly a point, between $a$ and $b$ or $b$ and $a$.  \end{notation}

\begin{proposition}  If $\omega\in \partial C_3\langle I\rangle$, then $\omega$ is $\alpha^\omega_{v,t,u}$ ready for all $v,t,u$.\end{proposition}

\begin{proof}  \emph{Case 1:} $(t,u)\in P\setminus\inte(N_\epsilon(\partial P))$.

\vskip 8pt
\noindent{Proof of Case 1:}   Let $\omega=(x_0,y_0, z_0)$ and let $\alpha$ denote $\alpha^\omega_{v,t,u}$.  If for some $s>0$,  $\alpha(x_s)=\alpha(y_s)$, then $x_s\neq y_s$ since these are unit speed paths with $x_0\le y_0$ and $x_1<y_1$.  Therefore, either $x_s\in U_1$ and $y_s\in U_2\cup U_3$ or $x_s\in U_2$ and $y_s\in U_3$.  Similar statements hold for $x_s, z_s$ and $y_s, z_s$.  Thus it suffices to show that there are no $U_i/U_j$ intersections involving two of $x_s, y_s$ or $z_s$.

We first show that there is no such $U_i/U_j$ intersection involving a $ z_s$.  Recall that $ z_1=a_3$.  If $z_0\ge a_3$, then no $z_s\in [z_0, a_3]$ is involved with a $U_1/U_2$ intersection.  Since $\omega\in U_L\setminus U_R$ then $z_s$ cannot be involved with a $U_i/U_3$ intersection.  If $z_0\in [k_1, a_3]$, then again $z_s\in[z_0, a_3]$ cannot be involved with a $U_1/U_2$ intersection.  If $a_3\ge z_0\ge k_2$, then $\omega\in U_\nu\setminus U_\beta$, so $z_s\in[z_0, a_3]$ cannot be involved with a $U_i/U_3$ intersection.  If $z_0\in[k_1,k_2]$, then $\omega\in U_R\setminus U_L$  so $z_s\in[z_0, a_3]$  cannot be involved with a $U_i/U_3$ intersection.  If $z_0\le k_1$, then $\omega\in U_\beta\setminus U_\nu$ so there are no $U_1/U_2$ intersections and since $\omega\in U_R\setminus U_L$, $z_s$ cannot be involved with a $U_i/U_3$ intersection.

It remains to show that there is no $U_i/U_j$ intersection involving both a $y_s$ and an $x_s$.  
If $y_0\ge k_1$, then $y_1=k_1$ and hence $y_s\in[y_0, y_1]$ is not involved with  a $U_1/U_2$ intersection.  If $y_0\ge k_2$, then $\omega\in U_\nu\setminus U_\beta$ so $y_s\in [y_0, y_1]$ is not involved with $U_i/U_3$ intersections.  If $y_0\le k_2$, then so are $x_0,x_1$ and $y_1$ and hence $([y_0, y_1]\cup[x_0,x_1])\cap U_3=\emptyset$ so $y_s\in[y_0, y_1]$ cannot be involved with a $U_i/U_3$ intersection.  If $y_0\in ([0, k_0]\cup [k_1,k_2])$, then $[y_0,y_1]\cap \inte U_2=\emptyset$, so $y_s\in[y_0, y_1]$ cannot be involved with a $U_1/U_2$ intersection.  If $x_0\notin \inte(U_1)$, then neither is $x_1$.  Therefore if $y_0\le k_2$ and $ x_0\notin\inte(U_1)$, then $y_s\in[y_0, y_1]$ cannot be involved with a $U_1/U_2$ intersection.  If $y_0\in [k_0, k_1]$ and  $x_0\in [j_0, j_1]$, then $y_0=y_1, x_0=x_1$.  Being in $\partial C_3\langle I\rangle$ it follows that $z_0=1$ and hence $\omega\in U_\beta\setminus U_\mu$ which implies that they cannot be involved in a $U_1/U_2$ intersection.  \qed

\vskip 8pt
\noindent\emph{Case 2:}  $(t,u)\notin P \setminus\inte(N_\epsilon(\partial P))$.
\vskip 8pt
\noindent\emph{Proof of Case 2}  Here $\alpha|U_i = \id$ for $i=1$ or $ 2$.  We will assume the case of $i=2$, the other case being similar.  Hence it suffices to show that there are no $U_1/U_3$ intersections.  Such intersections are ruled out by the argument of Case 1.\end{proof} 

\begin{definition}  Define $F_s:E^5\to C_3\langle S^1\times B^3\rangle$ by $$F_s(\omega,v,t,u)=\alpha^\omega_{v,t,u}((1-s)\omega+s(p(\omega))).$$\end{definition}


\begin{definition}  Define $J_{U_3}:=U_3\times [-1,0]\subset D^2\times S^1\times [-1,1]$, the \emph{$U_3$-cohorizontal space} and for $s\in U_3, J_s:=s\times [-1,0]$, the \emph{s-cohorizontal space}.  Let $f_1, f_2, f_3$ be the coordinate functions for $F_1:E^5\to C_3(S^1\times B^3)$.  Call a point of $f_1^{-1}(J_{a_3})$ (resp. $\finv_2(J_{a_3}$)) a \emph{blue cohorizontal}  (resp. \emph{red cohorizontal}) point.   \end{definition} 
\begin{definition}\label{rb}  As in Definition \ref{bracket}, $B, R\in \Omega\Emb(I,S^1\times B^3;I_0)$ give rise to $F^B, F^R\in \Omega\Omega\Emb(I,S^1\times B^3;I_0)$.  Let $S^B$ and $S^R$ denote the parameter supports for $F^B$ and $F^R$, which by convention lie in the blue and red regions of Figure \ref{fig:FigureCalc35} b).  Let $F^B_u$ denote $F^B|[0,1]\times u$.  By reparametrizing $[0,1]\times [0,1]$ we will assume that for each $u$, $\length(([0,1]\times u)\cap S^B)<\delta$, where $\delta$ is very small and $S^B$ is contained in the $\delta$ neighborhood of the arc from $(1/8,1/8)$ to $(7/8,7/8)$.  
Similar statements hold for $F^R$, where the domain support of $F^R$ is denoted $S^R$ and is contained in the $\delta$ neighborhood of the arc from  $(7/8,1/8)$ to $(1/8,7/8)$.  
Define the \emph{blue slab} (resp. \emph{red slab}) $= ((x,y,z),v)\in E^3$ such that $x\in U_1$ (resp. $y\in U_2$). Define $P_{z_0}=\{(x,y,z,v)\in E^3|z=z_0\}$.    Let $\pi_{E}$ denote the projection of $E^5$ to $E^3$, $\pi_t:E^5\to E^3\times I\times I$ the projection to the first $I$ factor, $\pi_v:E^5\to E^3\times I\times I$ the projection to the second $I$ factor and $\pi_v:(\partial C^3\times [0,1])\times I^2\to [0,1]$ the projection to the $[0,1]$ factor.  Define $\pi_{\pr}$ to be the projection of $D^2\times S^1\times [-1,1]\to D^2\times S^1\times 0$.  \end{definition}


\begin{proposition}  $S_b:=\finv_1(J_{a_3})$ (the blue sphere) and $S_r:=\finv_2(J_{a_3})$ (the red sphere) are disjoint 2-spheres whose union is the standard Hopf link in $B^5$. If $\beta$ is a path from $\partial B^5$ to $S_b$ (resp. $S_r$), then $f_1(\beta)$ (resp. $f_2(\beta)$)  represents the class $p\in \pi_1(S^1\times B^3; I_0)$ (resp. $q\in \pi_1(S^1\times B^3; I_0)$).  \end{proposition}

\begin{remark}  Note that $S_b=F_1^{-1}(J_{a_3}\times (S^1\times B^3)\times (S^1\times B^3))$ and $S_r=F_1^{-1}((S^1\times B^3)\times J_{a_3}\times (S^1\times B^3)).$\end{remark}

\begin{proof}  Let $K_u=S_b\cap E^3\times [0,1]\times u$ and $L_u=S_r \cap E^3\times [0,1]\times u$.  The Proposition follows from the following seven steps.  

\vskip 8pt
\noindent\emph{Step 1:}  If $(\omega,v,t,u)\in S_b\cup S_r$ where $\omega=(x,y,z)$, then either $v>0$ or $u\notin[1/4,3/4]$.  Also either 

i) $v<1/4$, in which case $\alpha^\omega_{v,t,u}$ is a partially shifted $\alpha^\omega_{t,u}$ or 

ii) $z\in [g_{-1},g_0]$,  in which case $\alpha^\omega_{v,t,u}$ is obtained from either a shifted or partially shifted $\alpha^\omega_{t,u}$ or its image under the undo homotopy.

\vskip 8pt
\noindent\emph{Proof of Step 1:}  By construction, the lasso spheres of $B$ and $R $ and those of the undo homotopy project under $\pi_{\pr}$ to their lasso discs.  The $B$ and $R$ lasso discs are disjoint from $a_3$ and their lasso discs and those of the undo homotopies are disjoint when fully shifted left or right.  It follows that if $u\in [.25, .75]$, then $(S_b\cup S_r)\cap C_3\langle I \rangle \times I^2=\emptyset$.  When $z\notin [g_{-1}, g_0]$ a similar argument proves shows that $v<1/4$.\qed

\vskip 8pt
\noindent\emph{Step 2:}  We can assume the following.

i)  Under the left shift the $b_2$ (resp. $r_2$) lasso disc intersects $a_3$ when $v=.1$ (resp. $v=.2$) and under the right shift the $r_1$ (resp. $b_1$) lasso disc intersects $a_3$ when $v=.1$ (resp. $v=.2$).  

ii)  Viewing $B\in\Omega \E$ as a map $[0,1]\times U_1\to S^1\times B^3$, then the projection of $B^{-1}(U_3)$ to $U_1 = \{j_{.3}, j_{.7}\}$.  Similarly the $U_2$ projection of $R^{-1}(U_3)$ equals $\{k_{.3}, k_{.7}\}$.  

iii) Viewing $\nu^B$ as a map $[0,1]\times [0,1]\times U_1\to S^1\times B^3$ where $\nu^B_0=B$ and $\nu^B_1$ is the constant map to $\id_{U_1}$, then the projection of $\nu_u^{B^{-1}}(U_3)$ to $U_1$ consists of two, one or zero points. Here $\nu_u^B$ denotes $\nu^B|u\times [0,1]\times U_1$. In the former case the lower (resp. higher) point in $U_1$ is monotonically increasing (resp. decreasing) and equals $j_{.5}$ when there is exactly one point.  An analogous statement holds for $\nu^R$.

iv)  The lasso discs of $\nu_u^B$ intersect $U_3$ in two, one or zero points.  When there are two points the lower (resp. higher) point in $U_3$ is monotonically increasing (resp. decreasing) in $u$. An analogous statement holds for $\nu_u^R$.

v) Near $\partial E_3\times I^2,\ f_1(x,y,z,v,t,u)=p_1(x)$ and $f_2(x,y,z,v,t,u)=p_2(y)$. \qed

\vskip 8pt
\noindent\emph{Step 3:} $K_u\cap P_z$ is independent of $u\in[1/4,3/4]$ and is $\emptyset$ when $z<j_{.7}$, is one point when $z=j_{.7}$, is two points when $z\in (j_{.7}, 1)$ and an interval when $z=1$.  Furthermore, 

i) If $z=j_{.7}$, then $K_u\cap P_z=(j_{.7}, j_{.7}, j_{.7}, .2)$

ii) If $z\in (j_{.7}, g_{-1})$, then $K_u\cap P_z=(j_{.7}, y_z, z, .2)$, where $y_z=\{j_{.7}, z\}$

iii) if $z\in [g_{-1}, g_0]$, then $K_u\cap P_z=(j_z, y_z, z, v_z)$, where $j_z\in [j_{.3}, j_{.7}]$ is non increasing and $y_z=\{j_z,z\}$

iv) if $z\in (g_0, 1)$, then $K_u\cap P_z=(j_{.3}, y_z, z, .1)$, where $y_z= \{j_{.3},z\}$

v) if $z=1$, then $K_u\cap P_1=(j_{.3}, y_1, 1, .1)$, where $y_1= [j_{.3}, 1]$.  

An analogous statement holds for $L_u\cap P_z$. \qed

\vskip 8pt
\noindent\emph{Step 4:}  $\finv_1(J_{a_3})\cap \finv_2(J_{a_3}) =\emptyset$.

\vskip 8pt
\noindent\emph{Proof of Step 4:} If $u\notin[1/4,3/4]$ and $t\in [0,1] $, then $\alpha^\omega_{v,t,u}=\id$ when restricted to one of $U_1$ or $U_2$ and hence Step 4 follows.  Now suppose that $u\in [1/4,3/4]$.  If $\omega\in C_3\langle I\rangle$, then apply Step 1 to conclude $(K_u\cup L_u)\cap(\omega,0,t,u)=\emptyset$.   If $\omega=(x_0,y_0,z_0)\in \partial C_3\langle I\rangle$, then $(\omega,v,t,u)\in K_u$ (resp. $L_u$) implies $x_0\in \inte(U_1)$ (resp. $y_0\in \inte(U_2)$).  If $z_0\ge g_0$, then $\omega\in U_L\setminus U_R$ and hence by Step 2, $K_u\subset \pi_v^{-1}(.1)$ while $L_u\subset \pi_v^{-1}(.2)$.  If $z_0\le g_{-1}$, then $\omega\in U_R\setminus U_L$ and hence by Step 2, $K_u\subset \pi_v^{-1}(.2)$ while $L_u\subset \pi_v^{-1}(.1)$.   If $z_0\in [g_{-1}, g_0]$, then either $x_0=0$ or $x_0=y_0$ or $y_0=z_0$.  In all three cases at least one of  $x_0\cap \inte(U_1)=\emptyset$ or $y_0 \cap\inte(U_2)=\emptyset$ holds and hence  $K_u\cap L_u=\emptyset$.\qed
\vskip8pt

\noindent\emph{Step 5:} For $u\in [1/4,3/4]$, $\pi_E(K_u\cup L_u)$ is independent of $u$ and equal to a Hopf link.  Also $\pi_t(K_u)\subset N_\delta( u)$ and $\pi_t(L_u)\subset N_\delta (1-u)$.

\vskip 8pt
\noindent\emph{Proof of Step 5:}  The invariance and the Hopf link property follow by Steps 2 and 3.  Figure \ref{fig:FigureCalc36} b) shows $\pi_E(K_u\cup L_u)$, a Hopf link within  the blue and red slabs.  If $(\omega,v,t,u)\in S_b$, then $(t,u)\in S^B$ and hence $\pi_t(\omega,v,t,u)\in N_\delta(u)$.  If $(\omega,v,t,u)\in S_r$, then $(t,u)\in S^R$ and hence $\pi_t(\omega,v,t,u)\in N_\delta(1-u)$.  \qed

\begin{figure}[ht]
$$\includegraphics[width=13cm]{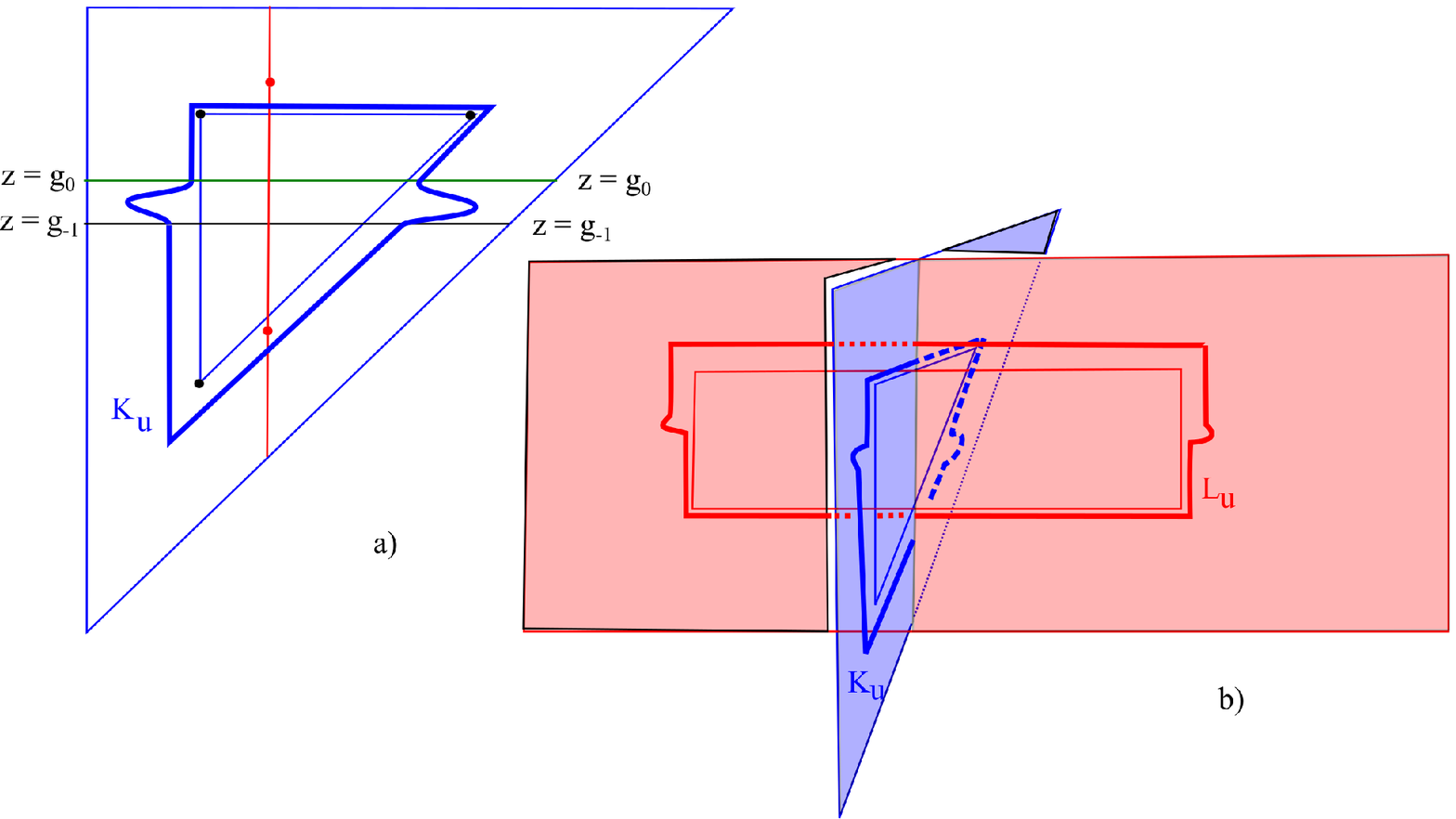}$$
\caption{\label{fig:FigureCalc36}} 
\end{figure}


\vskip 8pt
\noindent\emph{Step 6:}  The surface $D_{1/4}:=S_b\cap E^3\times ([0,1]\times [0,1/4])=$  is a disc that after proper isotopy in  $E^3\times ([1/8,3/8]\times [0,1/4])$ projects under $\pi_E$ to a spanning disc for $K_{1/4}$.    Analogous statements hold for $K_{3/4}$, $L_{1/4}$ and $ L_{3/4}$.
\vskip 8pt
\noindent\emph{Proof of Step 6:} This is a long but routine exercise along the lines of Step 3.  Here one computes the various intersections of $D_{1/4}$ with each $P_z\cap I^2$ and shows that after a small proper isotopy to eliminate arcs projecting to points, the $\pi_E$ projection  of $D_{1/4}$ to $E_3$ is embedded.   \qed

\vskip 8pt
\noindent\emph{Step 7:} If $\beta$ is a path from $\partial B^5$ to $S_b$ (resp. $S_r$), then $f_1(\beta)$ (resp. $f_2(\beta)$)  represents the class $p\in \pi_1(S^1\times B^3)$ (resp. $q\in \pi_1(S^1\times B^3$)).  
\vskip 8pt
\noindent\emph{Proof of Step 7:}  Step 4 follows from the definition of $G(p,q)$.  Here we view $I_0\cup J_{a_3}$ as the basepoint.  \end{proof}

\noindent\emph{Proof of Theorem \ref{gpq whitehead}:} We will show that there is a homotopy of $F_1=(f_1,f_2,f_3)$ to $F_1'=(f_1', f_2', f_3)$ supported away from $S_b\cup S_r$ such that throughout the homotopy $f_3 \equiv a_3$  and each of the homotopies of $f_1$ and $f_2$ are supported away from $J_{a_3}\subset S^1\times B^3$.  Furthermore, there exists disjoint regular neighborhoods $N(S_b), N(S_r)$ such that $f_1=p_1$ off of $N(S_r)$ and $f_2=p_2$ off of $N(S_b)$.  Since each of $S_b$ and $S_r$ have trivial normal bundles and the restriction of $f_1'$ and $f_2'$ to the $B^3$ fibers of $N(S_b)$ and $N(S_r)$ project to degree $\pm 1$ maps to $\partial N(a_3)$, the homotopy can be chosen so that the restriction of $f_1'$ to a $B^3$ fiber is a whisker from $a_1$ to $\partial N(a_3)$ that goes $p$ times around the $S^1$ followed by the standard quotient map of the 3-ball to $\partial N(a_3)$.  The analogous statement holds for $f_2'$ where the whisker starts at $a_2$ and goes q times about $S^1$.  It follows that we have, up to sign independent of $p$ and $q$, the Whitehead product $ t_1^p t_2^q [w_{13}, w_{23}].$

\begin{figure}[ht]
$$\includegraphics[width=13cm]{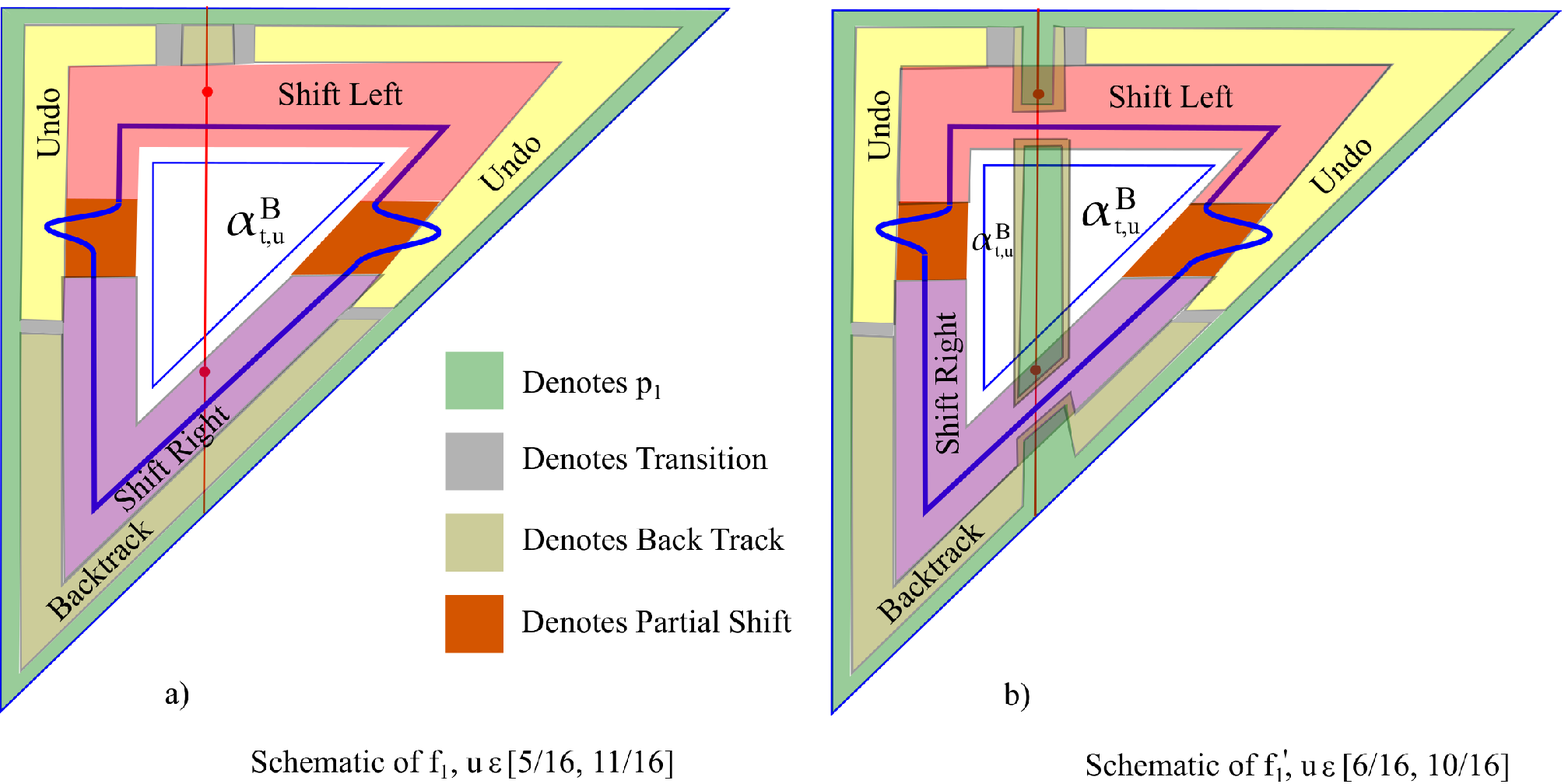}$$
\caption{\label{fig:FigureCalc37}} 
\end{figure}

The mapping $f_1|\textrm{(blue slab)}\times (t,u)$, $(t,u)\in [0,1]\times [1/4, 3/4]$ is schematically shown in Figure \ref{fig:FigureCalc37} a).  The closed subset corresponding to $C_3\langle I\rangle$ is the closed region interior to the blue triangle.  We denote that region by $\alpha^B_{t,u}$ since in that region $f_1(x,y,z,0,t,u) = \alpha^B_{t,u}(x,y,z)$.  Since $f_1|U_1=\id_{U_1}$ near the boundary of $E_3$, that region is denoted   by $p_1$.  In the collar, $\alpha^B_{t,u}|U_1$ is modified by the left, right, undo and backtrack homotopies as indicated in the diagram.  A similar discussion holds for $f_2$, though in that case the diagram is of a square.


We next homotope $f_1$ to $f_1'$ so that in the region $[0,1]\times [3/8,5/8], f_1'$ is defined as in Figure \ref{fig:FigureCalc37} b).  Here we modify $\alpha^B_{t,u}$ or its partially left or right shifted versions by the blue back track homotopy.  Note that the homotopy is done away from where a blue lasso disc crosses $a_3$ under left or right shifting and hence the homotopy is supported away from $J_{a_3}$.    In a similar manner we homotope $f_2$ to $f_2'$.  Since for any $(t,u)$ the corresponding blue bands and lasso discs are disjoint from the red ones, there is no nontrivial intersection between the images of $U_1$ and $U_2$ under the homotopy.  Note that the support of $f_1' $ is now a regular neighborhood of $S_b$ disjoint from a regular neighborhood of $S_r$ which is the support of $f_2'$.  \qed


\vskip 10pt




\newpage
\begin{thebibliography}{AKMR}


\bibitem[AS]{AS} G. Arone \& M. Szymik, \emph{Spaces of knotted circles and exotic smooth structures,}{\tt arXiv: 1909.00978}.



\bibitem[Br]{Br} M. Brown,
\emph{A proof of the generalized Schoenflies theorem},
Bull. Amer. Math. Soc. 66 (1960), 74--76. 

\bibitem[Bu1]{cubes}
R.~Budney \emph{ Little cubes and long knots}, Topology {\bf 46} (2007) 1--27. 

\bibitem[Bu2]{Bu} R. Budney, \emph{A family of embedding spaces}, Geometry \& Topology Monographs, \textbf{13} (2008) 41--83.

\bibitem[BCSS]{BCSS}
R.~Budney, J.~Conant, K.~Scannell, D.~Sinha, 
\emph{New perspectives on self-linking,}
Adv. Math. {\bf 191} (2005) 78--113.



\bibitem[BG2]{BG2} R.~Budney \& D.~Gabai, \emph{Scanning diffeomorphisms}, in preparation.


\bibitem[Ce1]{Ce1} J. Cerf, \emph{Topologie de certains espaces de plongements}, Bull. Soc. Math. France, \textbf{89} (1961), 227--380.

\bibitem[Ce2]{Ce2} J. Cerf, \emph{Sur les Diffeomorphismes de la Sphere de Dimension Trois $(\Gamma_4=0)$}, Springer Lecture Notes, \textbf{53} (1968).

\bibitem[Ce3]{Ce3} J. Cerf, \emph{La stratication naturelle des espaces de fonctions differentiables reels et le theoreme de la pseudo-isotopie}, Inst. Hautes Etudes Sci. Publ. Math. \textbf{39} (1970), 5--173.

\bibitem[CG]{CG} F.~Cohen, S.~Gitler, \emph{On loop spaces of configuration spaces,} T. AMS \textbf{354} No. 5, 1705--1748. 

\bibitem[Da]{Da} J. P. Dax, \emph{Etude homotopique des espaces de plongements}, Ann. Sci. Ecole Norm. Sup. (4) \textbf{5} (1972), 303--377.


\bibitem[E2]{E2} R. D. Edwards, \emph{The 4-Dimensional Light Bulb Theorem (after David Gabai)}, ArXiv:1709.04306.


\bibitem[Fr]{Fr} M. Freedman, \emph{The disk theorem for four-dimensional manifolds}, Proc. Int. Cong. Math., (1983), 647--665.

\bibitem[FQ]{FQ} M. Freedman \& F. Quinn, Topology of 4-manifolds, Princeton Mathematical Series, \textbf{39}, Princeton University Press, Princeton, NJ, 1990.

\bibitem[Ga1]{Ga1} D. Gabai, \emph{The four dimensional light bulb theorem}, J. AMS \textbf{33} (2020), 609--652.

\bibitem[Ga2]{Ga2} D. Gabai, \emph{Self-Referential discs and the light bulb lemma}, to appear, Comment. Math. Helv.

\bibitem[Ga3]{Ga3} D. Gabai, OIST Lecture, January 23, 2019.


\bibitem[Gl]{Gl} H. Gluck, \emph{The embedding of two-spheres in four-spheres}, B. AMS, \textbf{67}(1961), 586--589.



\bibitem[GKW]{GKW} T. Goodwillie, J. Klein \& M. Weiss, \emph{Spaces of Smooth Embeddings, Disjunction and Surgery, Surveys
on Surgery Theory}, vol. 2, Ann. of Math. Stud., \textbf{149}, Princeton, 2001, pp. 221--284.

\bibitem[GW2]{GW2} T.~Goodwillie, M.~Weiss, \emph{Embeddings from the point of view of immersion theory. Part II.}
Geom. Topol. Vol 3, No. 1 (1999) 103--118. 

\bibitem[HW]{HW} A.~Hatcher, \emph{ The second obstruction for pseudo-isotopies,}
Ast\'erisque, tome 6 (1973). Part II. 

\bibitem[Ha1]{Ha1} A.~Hatcher \emph{A proof of the Smale conjecture, ${\rm Diff}(S^{3})\simeq {\rm O}(4),$} Ann. of Math. (2) \textbf{117} (1983), 553--607.

\bibitem[Ha2]{Ha2} A.~Hatcher, \emph{On the diffeomorphism group of $S^1\times S^2$}, 
Proc. Amer. Math. Soc. 83 (1981), no. 2, 427--430.
Also, 2003 revision available from {\tt http://pi.math.cornell.edu/\textasciitilde hatcher/}.

\bibitem[Ha3]{Ha3} A.~Hatcher, \emph{ The second obstruction for pseudo-isotopies,}
Ast\'erisque, tome 6 (1973). Part II. 

\bibitem[HI]{HI} A.~Hatcher, \emph{ Homeomorphisms of sufficiently large $P^2$-irreducible 3-manifolds},
Topology {\bf 15} Issue 4 (1974) pg. 343--347.

\bibitem[Hir]{Hirsch}
M.~Hirsch, \emph{Differential Topology}, Springer GTM 33 (1976). 


\bibitem[Ki]{Ki} R. Kirby, \emph{Problems in low dimensional manifold theory}, Proc. Sympos. Pure Math., \textbf{32} Part 2 (1978), 273--312.

\bibitem[KM]{KM} M. Klug \& M. Miller, \emph{Concordance of surfaces in 4-manifolds and the Freedman-Quinn invariant}, preprint.

\bibitem[Ko1]{Ko1} D. Kosanovic, \emph{A geometric approach to the embedding calculus knot invariants}, 2020 thesis, Rheinischen Friedrich-Wilhelms-Universitat Bonn. 

\bibitem[Ko2]{Ko2} D. Kosanovic, \emph{Embedding calculus and grope cobordism of knots}, ArXiv preprint.



\bibitem[La]{La} F. Laudenbach, \emph{Sur Les 2-spheres d'une variete de dimension 3}, Ann. of Math., \textbf{(2) 97} (1973), 57--81.

\bibitem[Lev]{Lev} J.~Levine, \emph{Inertia Groups of Manifolds and Diffeomorphisms of Spheres,} 
 American Journal of Mathematics Vol 92, No. 1 (Jan. 1970) 243--258.

%
%

\bibitem[Ma1]{Ma1} B. Mazur, \emph{On embeddings of spheres}, Bull. Amer. Math. Soc. \textbf{65} (1959), 59--65.

\bibitem[Ma2] {Ma2} B. Mazur, personal communication.

%
\bibitem[Mi]{Mi} J. Milnor, \emph{Lectures on the h-cobordism theorem}, Princeton University Press, Princeton NJ, 1965.

\bibitem[MM]{MM} J.~Milnor, J.~Moore, \emph{On the structure of Hopf algebras,} Ann. Math. \textbf{81} (2): 211--264.

\bibitem[Mor]{Mor} S. Moriya, \emph{Models for knot spaces and Atiyah duality}, preprint.

\bibitem[Mo]{Mo} M. Morse,
\emph{A reduction of the Schoenflies extension problem},
Bull. Amer. Math. Soc. 66 (1960), 113--115.

\bibitem[MV]{MV} B.~Munson, I.~Voli\'c. \emph{Cubical Homotopy Theory.}
Cambdridge University Press.

\bibitem[Pa]{Pa} R. Palais, \emph{Extending diffeomorphisms}, Proc. Amer. Math. Soc. \textbf{11} (1960), 274--277.


\bibitem[Po]{MP} M.~Polyak, \emph{Invariants of immersions via homology intersections,}
(in preparation). 

\bibitem[PV]{PV} M.~Polyak, O.~Viro, \emph{Gauss Diagrams for Formulas of Vassiliev Invariants,} IMRN (1994), No. 11. 

%
%
%
%

\bibitem[ST]{ST} R. Schneiderman \& P. Teichner, \emph{Homotopy versus Isotopy:  Spheres with duals in 4-manifolds}, arXiv:1904.12350.

\bibitem[Sch]{Sch} H. Schwartz, \emph{Equivalent non-isotopic spheres in 4-manifolds}, J. Topol. \textbf{12} (2019), 1396--1412


\bibitem[Si1]{Si1} D. Sinha, \emph{The topology of the space of knots: Cosimplicial models}, American J. Math., \textbf{131} (2009), 945--980. 
\bibitem[Si2]{Sin2} D.~Sinha, \emph{ Manifold-theoretic compactifications of configuration spaces,} Selecta Math. (N.S.) {\bf 10} (2004), no.3, 391--428.

\bibitem[Sm1]{Sm1} S. Smale, \emph{A classification of immersions of the two-sphere}, Trans. AMS (1957), 281--290.


\bibitem[Sm2]{Sm2} S. Smale, \emph{The classification of immersions of spheres in Euclidean spaces}, Ann. of Math. (2) \textbf{69} (1959), 327--344.

\bibitem[Sm3]{Sm3} S. Smale, \emph{Diffeomorphisms of the 2-Sphere}, Proc. AMS, \textbf{10} (1959), 621--626.

\bibitem[Sp]{Sp} D. Spring, \emph{The golden age of immersion theory in topology}, B. AMS \textbf{42} (2005), 163--180.

\bibitem[St]{St} R. Stong, \emph{Uniqueness of $\pi_1$-negligible embeddings in 4-manifolds: A correction to theorem 10.5 of Freedman and Quinn}, Topology, \textbf{32} (1993) 677--699.


\bibitem[Wa1]{Wa1} T. Watanabe, \emph{Some exotic nontrivial elements of the rational homotopy groups of $\Diff(S^4)$},
arXiv:1812.02448.

\bibitem[Wa2]{Wa2} T. Watanabe, \emph{Theta-graph and diffeomorphisms of some 4-manifolds}, arXiv:2005.09545.



\end{thebibliography}
\enddocument